\documentclass[10pt,reqno]{amsart}
\usepackage[a4paper,margin=1in]{geometry}
\usepackage{amsmath,amssymb,amsfonts,amsthm,mathtools}
\usepackage{mathrsfs}
\usepackage{enumitem}
\usepackage{microtype}
\usepackage{xcolor}
\usepackage{hyperref}
\usepackage{aliascnt}

\hypersetup{
  colorlinks=true,
  linkcolor=blue!55!black,
  citecolor=blue!55!black,
  urlcolor=blue!55!black,
  pdftitle={Operational State Selection in Stochastic Analysis: The Least State for a Calibrated Quadratic Ito Program},
  pdfauthor={Guangqian Zhao},
  pdfsubject={Operational state selection for nondominated Hilbert-valued semimartingales, calibrated quadratic Ito states, common realization, compact capacity cores, Euler-state compatibility, and backward response jets},
  pdfkeywords={operational state selection, response calibration, quadratic Ito state, nondominated stochastic analysis, Schatten ideals, defect-energy transfer, spatial trace tightness, common Euler fields, backward response jet, rough paths}
}

\theoremstyle{plain}
\newtheorem{theorem}{Theorem}[section]
\newaliascnt{proposition}{theorem}
\newtheorem{proposition}[proposition]{Proposition}
\aliascntresetthe{proposition}
\newaliascnt{lemma}{theorem}
\newtheorem{lemma}[lemma]{Lemma}
\aliascntresetthe{lemma}
\newaliascnt{corollary}{theorem}
\newtheorem{corollary}[corollary]{Corollary}
\aliascntresetthe{corollary}

\theoremstyle{definition}
\newaliascnt{definition}{theorem}
\newtheorem{definition}[definition]{Definition}
\aliascntresetthe{definition}
\newaliascnt{remark}{theorem}
\newtheorem{remark}[remark]{Remark}
\aliascntresetthe{remark}
\newaliascnt{example}{theorem}
\newtheorem{example}[example]{Example}
\aliascntresetthe{example}

\usepackage[nameinlink,capitalize]{cleveref}
\crefname{theorem}{Theorem}{Theorems}
\crefname{proposition}{Proposition}{Propositions}
\crefname{lemma}{Lemma}{Lemmas}
\crefname{corollary}{Corollary}{Corollaries}
\crefname{definition}{Definition}{Definitions}
\crefname{remark}{Remark}{Remarks}
\crefname{example}{Example}{Examples}
\Crefname{theorem}{Theorem}{Theorems}
\Crefname{proposition}{Proposition}{Propositions}
\Crefname{lemma}{Lemma}{Lemmas}
\Crefname{corollary}{Corollary}{Corollaries}
\Crefname{definition}{Definition}{Definitions}
\Crefname{remark}{Remark}{Remarks}
\Crefname{example}{Example}{Examples}

\newcommand{\R}{\mathbb R}
\newcommand{\N}{\mathbb N}
\newcommand{\cF}{\mathcal F}
\newcommand{\cP}{\mathcal P}
\newcommand{\Smax}{\mathfrak S_{\Lambda,B}^{0,T}(H)}
\newcommand{\Mmax}{\mathfrak M_{\Gamma}^{0,T}(H)}
\newcommand{\Tr}{\operatorname{Tr}}
\newcommand{\eps}{\varepsilon}
\newcommand{\one}{\mathbf 1}
\newcommand{\dd}{\,d}
\newcommand{\Anti}{\operatorname{Anti}}
\newcommand{\Sym}{\operatorname{Sym}}
\newcommand{\norm}[1]{\left\lVert #1\right\rVert}
\newcommand{\abs}[1]{\left\lvert #1\right\rvert}
\newcommand{\qv}[1]{\langle\!\langle #1\rangle\!\rangle}
\newcommand{\preceqsc}{\preceq_{\cP}^{\mathrm{sc}}}
\newcommand{\equivsc}{\equiv_{\cP}^{\mathrm{sc}}}
\newcommand{\Gito}{\mathbb G^{2,\mathrm I}}

\title[Operational state selection]
{Operational State Selection in Stochastic Analysis:\\
The Least State for a Calibrated Quadratic It\^o Program}
\author{Guangqian Zhao}
\address{School of Mathematical Sciences, University of Science and Technology of China, Hefei, Anhui 230026, China}
\email{zhaoguangqian@mail.ustc.edu.cn}
\date{}


\begin{document}
\raggedbottom

\begin{abstract}
We study the least stable state determined by a causal response program and
its quantitative calibration.  For nondominated families of continuous
Hilbert-valued semimartingale laws with uniform drift and trace-clock bounds,
the calibrated quadratic It\^o program selects a state $\mathcal A_2$ with
canonical coordinates $(X,A,Q)$, where $A$ is Hilbert--Schmidt skew area and
$Q$ is trace-class covariance.  Brownian experiments identify the
Hilbert--Schmidt/operator coefficient gauges and hence the dual
$\mathcal S_2/\mathcal S_1$ state geometry.  A lossless encoder--decoder pair
shows that every joint realization factors through $\mathcal A_2$.

A single sequence of total Borel causal approximants constructs the common
state, while the laws verify its classical semantics.  A dimension-free
defect--energy estimate gives independence from the finite-variation
regularization, and spatial trace tightness yields restart-stable compact
capacity cores in infinite dimensions.  Each fixed higher signature level is
a continuous readout of $\mathcal A_2$.  Rough flows inherit corewise
continuity, while Euler schemes construct a jointly Borel raw-causal It\^o
field outside one parameter-independent polar set.  On compact
covariance-envelope subclasses, covariance responses select a backward first
jet which, under the stated completion and continuation hypotheses, has
modelwise Galtchouk--Kunita--Watanabe semantics and agrees with the solver's
martingale coordinate.
\end{abstract}

\maketitle

Throughout, $0<T<\infty$.  All deterministic bounds, in particular
$\Lambda$ and $B$, are finite and nonnegative unless stated otherwise.

\section{Introduction}
\label{sec:introduction}

\subsection*{The operational state-selection problem}

Fix a causal information system and a specified program of
operations---restart, integration, quadratic response, nonlinear evolution,
conditioning, or infinitesimal testing.  The basic question is
\[
 \boxed{\text{What is the least stable state on which the declared program is
 well posed?}}
\]
Almost-sure equality enters after a state has been specified: laws identify
versions of realized objects, while response requirements determine which
coordinates the state must retain.  The same principle applies to
infinitesimal responses: a derivative is retained as a jet when the value
topology does not recover it continuously.

The construction first descends the declared responses through the causal
information quotient, calibrates their joint profile, and completes the
resulting uniform space.  A coordinate may be removed precisely when it is
continuously recoverable by a stable-causal readout.  Probability enters in
distinct roles: canonical experiments calibrate the response uniformity;
after selection, laws supply modelwise semantics and the null ideal identifying
versions.  An added response leaves the state unchanged exactly
when it descends through the existing response quotient and is uniformly
continuous for the calibrated uniformity.

\subsection*{Central results and applications}

\emph{Operational mechanism.}
\Cref{thm:descent-before-completion} constructs the response congruence and
calibrated Hausdorff completion after causal descent; its corollaries separate
response-null, uncontrolled, recoverable, and retained coordinates.

\emph{I. Calibrated quadratic reconstruction and least-state factorization.}
The analytic reconstruction theorem \Cref{thm:quadratic-response-reconstruction}
combines
\[
 \text{restart algebra}\quad+\quad\text{calibrated }\mathcal S_2/\mathcal S_1
 \text{ geometry}\quad+\quad\text{lossless factorization}
\]
and identifies
\begin{equation}\label{eq:intro-forced-state}
 \boxed{\operatorname{ch}(\mathcal A_2)=(X,A,Q)}.
\end{equation}
If $(U,V)$ is the restarted step-two It\^o signature, then
\begin{equation}\label{eq:intro-lossless-state-readout}
 A=\operatorname{Anti}V,
 \qquad
 Q=U^{\otimes2}-V-V^*.
\end{equation}
Brownian area and covariance experiments select the Hilbert--Schmidt and
operator coefficient norms, hence the dual state geometry
$\mathcal S_2(H)_{\rm sk}\times\mathcal S_1(H)_{\rm sa}$.  The least-state
factorization theorem
\Cref{thm:quadratic-response-representability} shows directly that a causal
source state $S$ realizes the complete quadratic program if and only if
$\mathcal A_2\preceqsc S$.  Its global decoder gives the factor map explicitly
and proves uniqueness up to stable-causal equivalence.  Part~I assumes the
single common realization of
\Cref{def:common-accessibility}.  Part~II constructs this law-independent
realization on $\Smax$, so
\Cref{thm:nondominated-quadratic-reconstruction} gives the same representation
directly on the nondominated Hilbert semimartingale model class.

\emph{II. Common realization and compact perfection.}
The defect-energy theorem \Cref{thm:defect-energy-transfer} makes primitive
reconstruction Lipschitz in the upper regularization-energy pseudometric, hence
scheme-independent in uniform $\mathcal S_2$ when the defect vanishes.
Rough-H\"older and trace-energy convergence require separate estimates and
spatial profiles.  In infinite dimension, the trace-clock obstruction
\Cref{thm:trace-clock-compactness-obstruction} shows that bracket trace control
alone does not give uniform compactness.  The spatial trace-tightness theorem
\Cref{thm:spatial-trace-tight-cores} gives compact restart-stable capacity cores
from a uniform spatial trace profile, strictly weaker than a fixed trace-class
envelope.  Its covariance path space is compact, restart congruent, and
trace-variation isometric to its $L^1$ density class.  Under the same profile,
\Cref{thm:Euler-state-interface} constructs an independent upper-$L^r$
Euler-generated It\^o field and, with additional rough regularity, identifies
it with the bracket-corrected flow off one polar set simultaneously in model,
start time, terminal time, and initial state.

\emph{Backward first-jet application.}
On the compact-envelope branch, \Cref{thm:backward-first-jet-selection} uses an
admissible raw testing core to generate a fiber-valued jet from
cross-covariance responses.  The modelwise covariation identities recover its
GKW energy classes, while clause~\textup{(A2)} places their active section in
one common upper-energy space.
Under the continuation and common-channel interface of
\Cref{def:backward-compatible-interface}, that jet is identified with the
solver's martingale-integral class.  Stopped martingale responses calibrate the global
$\mathbb H_{\mathcal P}^2$ geometry.  For the complete closed linear response
relation, the jet compresses to the inherited value-process topology exactly
when its vertical space vanishes and its value domain is closed; a fixed
nonlinear solver range is governed instead by continuity of its restricted
value-projection inverse.

For a fixed interface and driver, terminal-to-jet continuity governs boundary
enhancement, whereas continuity of the value projection governs evolutionary
compression.

\subsection*{Propagation consequences}

For $1/3<\eta<1/2$, every fixed higher geometric or It\^o signature level is a
continuous readout of $\mathcal A_2$.  Deterministic rough--Young maps give
finite-dimensional Stratonovich and It\^o flows, perfected simultaneously on
compact capacity cores.  Deterministic continuous coefficients with the stated
global state-Lipschitz and linear-growth bounds admit an independent Euler
completion; under the additional autonomous rough-vector-field hypotheses it
agrees with the propagated rough cocycle on one quasi-sure domain.  The
joint exterior--covariance presentation is lossless; its scalar even--odd
energies recover the unitary-invariant spectral data of
\Cref{thm:even-odd-spectral-decoding}, but are not lossless.

\subsection*{Relation to existing frameworks}

Rough-path theory starts from an enhanced driver and proves continuity of
signatures and solution maps.  Wong--Zakai, BDG, and canonical-lift results
connect continuous semimartingales and enhanced martingales to rough paths
\cite{Lyons98,CoutinLejay05,FrizVictoir08,ChevyrevFriz19,FrizHairer20,
FrizZorinKranich23,GrongNilssenSchmeding22}; suitable lifts also yield rough
cocycles and random flows \cite{Kunita90,BailleulRiedelScheutzow17}.  Here both
the enhancement and its operator-ideal topology are outputs of the calibrated
response program.

The least-state viewpoint meets minimal realization
\cite{HoKalman66,Kailath80} and statistical sufficiency
\cite{Blackwell53,LeCam64}, but here sufficiency is relative to a declared
causal program, state comparisons are given by stable-causal readouts, and
calibration fixes the topology.

Pathwise integration, causal functional calculus, and quasi-sure aggregation
construct reference-law-independent objects or aggregate modelwise ones
\cite{Foellmer81,Karandikar95,Nutz12,SonerTouziZhang11,Cohen12,
Oberhauser16,PerkowskiPromel16,BartlKupperNeufeld19,ChiuCont22}.  Here the
response congruence and completed state precede implementation; law-independent
approximants then construct one raw-causal Borel realization whose semantics
and capacity properties are verified by the laws.

Approximation of stochastic integrals and equations by finite-variation paths
has a long history, from Wong--Zakai limits to calculus via regularization and
rough-path corrections
\cite{WongZakai65,RussoVallois07,BerardBergeryVallois11,
FrizOberhauser09,GomesOhashiRussoTeixeira21}.  The defect-energy transfer theorem
gives a scheme-independent criterion: every causal finite-variation
regularization whose first-level defect vanishes in upper energy produces the
same realized primitive in the uniform $\mathcal S_2$ topology.

Capacity rough paths and $G$-Brownian dynamics provide quasi-sure lifts, rough
equations, and Wong--Zakai limits under nonlinear expectation
\cite{DenisHuPeng11,BoedihardjoGengQian16,GengQianYang14,Peng07,
PengZhang17,PengZhang22}.  Our compactness theorem applies to general
nondominated driftless families and uses uniform spatial trace tightness beyond
a trace-clock bound; nondominated capacity methods in finance provide an
earlier benchmark \cite{DenisMartini06}.

Universal measurable SDE representations and nonanticipative derivatives are
further comparison points \cite{Kallenberg96,PrzybylowiczEtAl24,ContFournie13}.
Our forward theorem identifies the propagated cocycle with classical strong
solutions, while the backward theorem constructs a covariance-energy jet from
martingale representation and cross-covariation.

Earlier driver-level and scalar-flow constructions established common capacity
cores and simultaneous It\^o flows in narrower settings
\cite{ZhaoCausalFlows25,ZhaoDriverCores25}.  This framework extends that line
by selecting the Hilbert-valued area--covariance state operationally, proving
its regularization-independent common realization, and identifying spatial
trace tightness as a sufficient compactness condition for perfection.  Operator-ideal
endpoint regularity for infinite-dimensional L\'evy area is treated separately
in \cite{ZhaoLevyArea26}.

\subsection*{Organization and hypotheses}

Part~I develops operational completion, calibrated quadratic reconstruction,
and least-state factorization, comparing signature and norming triangular
presentations.  Part~II constructs the common realization and
spatial-profile compact perfection, with a fixed envelope as a closed
restriction.  Part~III treats signatures, flows, the Euler interface, backward
jets, and exterior-spectral representation.  The appendices collect response
completion, tangent geometry, upper-capacity and common-integral spaces, and
Euler-field completion.

Part~II allows uniformly bounded drift density and absolutely continuous
trace-class bracket density with bounded trace; general spatial-profile compact
perfection is stated for the driftless subclass.  A trace-class envelope is a
sufficient special case used by tangent theory.  Backward selection assumes an
admissible raw testing core and~\textup{(A2)}, while solver identification also
requires~\textup{(B4)}.

\part{Operational state selection and quadratic rigidity}

This part fixes information before probability: responses descend first, their
joint profile determines the Hausdorff completion, and a lossless realization
makes that completion the least stable state.

\section{Operational state selection: descent, calibration, and completion}
\label{sec:response-demand-descent}

\subsection{Causal information and declared operations}

\begin{definition}[Raw causal carrier and causal information presentation]
A \emph{raw causal carrier} is a set $\mathsf U_0$ with truncations
$r_t:\mathsf U_0\to\mathsf U_0$, $0\le t\le T$, satisfying
$r_T=I$ and $r_sr_t=r_{s\wedge t}$.  A \emph{causal information presentation}
$\mathbb I=(\mathscr I_t)_{0\le t\le T}$ is an increasing family of scalar
observables such that
\begin{equation}\label{eq:information-restriction-stability}
 f=f\circ r_t\quad(f\in\mathscr I_t),
 \qquad
 f\circ r_s\in\mathscr I_s
 \quad(0\le s\le t,\ f\in\mathscr I_t).
\end{equation}
The second condition says that a time-$t$ observation evaluated on an
$s$-stopped history is available at time $s$.

The presentation induces
\[
 u\sim_{\mathbb I,t}v
 \quad\Longleftrightarrow\quad
 f(u)=f(v)\ \text{for every }f\in\mathscr I_t,
\]
and the global relation $u\sim_{\mathbb I}v$ when this holds for every $t$.
Write $q_{\mathbb I}:\mathsf U_0\to\mathsf U_{\mathbb I}$ for the global
quotient.  Restriction stability makes
\[
 \bar r_tq_{\mathbb I}(u):=q_{\mathbb I}(r_tu)
\]
well defined, so the information quotient remains a causal carrier.  No
topology, measure, or null ideal is implied by this notation.
\end{definition}

\begin{example}[Filtered path spaces]
On canonical path space one may use the observation map $r_t$, or equivalently
$\mathscr I_t=B_b(\Omega_H,\sigma(r_t))$; then
$x\sim_{\mathbb I,t}y$ exactly when $r_tx=r_ty$.  On a general stochastic
basis $(\Omega,\mathcal F,\mathbb F,P)$, the filtration is an abstract
information structure.  It has the preceding pathwise presentation whenever
the carrier is equipped with measurable truncations generating the intended
$\mathcal F_t$ (or a declared subfiltration).  The law $P$ supplies weights
and null sets; equality $P$-almost surely is equality of realized versions, not
the definition of $\sim_{\mathbb I,t}$.
\end{example}

\begin{definition}[Declared response program]
A \emph{declared response program} on $(\mathsf U_0,\mathbb I)$ consists of
primitive responses
\[
 R_\alpha^{\rm raw}:\mathsf U_0\to\mathsf Y_\alpha,
 \qquad \alpha\in A,
\]
into complete Hausdorff uniform spaces and the algebraic, causal, restart,
dynamical, conditional, or infinitesimal constructors promised to be realized
stably.  Exactly the program-required coordinates are retained.
\end{definition}

\begin{definition}[Response calibration]
\label{def:response-calibration}
A \emph{response calibration} specifies the target uniformities, equivalently
the gauges comparing response errors before completion.  It may be deterministic
or induced by canonical stochastic experiments.  One qualitative response
family may admit inequivalent calibrations and completions; the Brownian
calibration below selects the principal Schatten geometry.
\end{definition}

\begin{definition}[Response status]
Fix $(\mathbb I,\mathfrak R)$.  Inputs are \emph{informationally identified}
when their $q_{\mathbb I}$-images agree.  A descended distinction is
\emph{response-visible} when some declared response separates it.  A candidate
coordinate is \emph{calibration-controlled} when it is uniformly continuous
for the response-generated uniformity; otherwise retaining it changes the
calibrated completion.  Controlled visible data are \emph{retained} when they
survive completion and \emph{propagated} when stable-causally recoverable from
a smaller completed state.  A coordinate is \emph{tangent-generated} by scaled
response germs or infinitesimal cross-responses; it enters a jet only when
demanded and not continuously recoverable from the value topology.
\end{definition}

\subsection{Descent, calibrated completion, and compression}

\begin{definition}[Information descent compatibility]
Let $R^{\rm raw}:\mathsf U_0\to\mathsf Y$ be a primitive response.  If
$\mathsf Y$ has truncations $\rho_t^{\mathsf Y}$, then $R^{\rm raw}$ is
\emph{$\mathbb I$-compatible} if
\[
 u\sim_{\mathbb I,t}v
 \quad\Longrightarrow\quad
 \rho_t^{\mathsf Y}R^{\rm raw}(u)
 =\rho_t^{\mathsf Y}R^{\rm raw}(v)
\]
for every $t$.  For a terminal response the condition is imposed at $T$.  A
constructor is compatible when its output is independent of all information
representatives at the relevant times, and the program is compatible when every
primitive response and declared constructor is compatible.  For a compatible
primitive response, write $R^{\mathbb I}$ for its unique factor through the
global information quotient $q_{\mathbb I}$.
\end{definition}

\begin{definition}[Causally calibrated response program]
\label{def:causally-calibrated-response-program}
A calibrated response program is \emph{causally calibrated} if every primitive
response target $\mathsf Y_\alpha$ carries uniformly continuous truncations
$\rho_t^\alpha$ satisfying
\[
 \rho_T^\alpha=I,
 \qquad
 \rho_s^\alpha\rho_t^\alpha=\rho_{s\wedge t}^\alpha,
\]
and the descended primitive responses are equivariant:
\begin{equation}\label{eq:causal-calibration-equivariance}
 R_\alpha^{\mathbb I}(\bar r_tu)
 =\rho_t^\alpha R_\alpha^{\mathbb I}(u).
\end{equation}
The product response target is equipped with the coordinatewise truncation.
A declared constructor is causally calibrated when it is uniformly continuous
for the chosen response uniformities and commutes with the corresponding
truncations.  In the metric response spaces used below the truncations are
nonexpansive.
\end{definition}

\begin{definition}[Causal spaces and stable-causal maps]
\label{def:causal-spaces}
A \emph{causal topological space} is a Hausdorff topological space
$\mathsf S$ with continuous truncations $\rho_t^{\mathsf S}$ satisfying
\[
 \rho_T^{\mathsf S}=I,
 \qquad
 \rho_s^{\mathsf S}\rho_t^{\mathsf S}
 =\rho_{s\wedge t}^{\mathsf S}.
\]
A \emph{causal uniform space} is a Hausdorff uniform space equipped with
uniformly continuous truncations satisfying the same identities.  A map
$R:\mathsf S\to\mathsf T$ between causal topological spaces is
\emph{stable-causal} if it is continuous and
\begin{equation}\label{eq:stable-causal-readout}
 \rho_t^{\mathsf T}R
 =\rho_t^{\mathsf T}R\rho_t^{\mathsf S},
 \qquad 0\le t\le T.
\end{equation}
Between causal uniform spaces, a stable-causal map is additionally required to
be uniformly continuous.  A stable-causal map used to extract one state from
another is called a \emph{stable-causal readout}.
\end{definition}

\begin{lemma}[Causal structure passes to calibrated completion]
\label{lem:causal-calibrated-completion}
For a causally calibrated response presentation, the response congruence is
invariant under stopping.  The induced truncations on the response quotient
are uniformly continuous and extend uniquely to the response completion,
where they satisfy the same semigroup law.  Every uniformly continuous
coordinate or constructor from the dense operational core into a complete
causal uniform target, and commuting there with stopping, extends uniquely as
a stable-causal map on the completion.
\end{lemma}

\begin{proof}
Let $\iota^0$ be the joint profile embedding of the response quotient.  If
$u\sim_{\mathfrak R}v$, then
\[
 \mathbf R^{\mathbb I}(\bar r_tu)
 =\rho_t\mathbf R^{\mathbb I}(u)
 =\rho_t\mathbf R^{\mathbb I}(v)
 =\mathbf R^{\mathbb I}(\bar r_tv),
\]
so stopping is well defined on the quotient.  There
\[
 \iota^0(\bar r_t[u])=\rho_t\iota^0([u]).
\]
Since the quotient uniformity is the initial uniformity induced by $\iota^0$
and $\rho_t$ is uniformly continuous, $\bar r_t$ is uniformly continuous and
extends uniquely to the completion.  Density and continuity preserve the
semigroup identities.  The same completion argument extends every uniformly
continuous equivariant constructor, and density preserves its stopping
identity, proving stable causality.
\end{proof}

\begin{definition}[Causal random states, realization, and leastness]
\label{def:stable-causal}
Fix a Borel carrier $(\Omega,(r_t)_{0\le t\le T})$ and a nonempty law family
$\cP$.  A \emph{causal Polish state space} is a Polish causal topological
space in the sense of \Cref{def:causal-spaces}.
A Borel state $S:\Omega\to\mathsf S$ is $\cP$-causal if there is a Borel set
$G_S$ with $P(G_S)=1$ for every $P\in\cP$ such that, for all $t$ and all
$\omega,\omega'\in G_S$,
\[
 r_t\omega=r_t\omega'
 \quad\Longrightarrow\quad
 \rho_t^{\mathsf S}S(\omega)=\rho_t^{\mathsf S}S(\omega').
\]
For $\cP$-causal states $S,T$, write $T\preceqsc S$ if
$T=R(S)$ $P$-almost surely for every $P\in\cP$ for some stable-causal
readout $R$, and write $T\equivsc S$ for mutual factorization.  Two factor
maps are identified when they agree on the realized source state under every
$P\in\cP$.

Given a causal Polish response target $\mathsf R_d$ and modelwise responses
$Y^{P,d}$, the state $S$ \emph{stably realizes} $d$ if a stable-causal readout
$\Psi_d:\mathsf S\to\mathsf R_d$ satisfies
$\Psi_d(S)=Y^{P,d}$ $P$-almost surely for every $P\in\cP$.  For a declared
joint response $Y^{P,\mathfrak D}$ in a causal Polish bundle
$\mathsf R_{\mathfrak D}$, a \emph{joint realization} is a single
stable-causal readout $\Psi:\mathsf S\to\mathsf R_{\mathfrak D}$ satisfying
$\Psi(S)=Y^{P,\mathfrak D}$ $P$-almost surely for every $P\in\cP$.  Relative
to a declared response family, a state is \emph{dynamically sufficient} if it
stably realizes every response; it is \emph{jointly sufficient} if it admits
a joint realization of the declared bundle.  A sufficient state $Z$ is
\emph{least} if
$Z\preceqsc S$ for every sufficient state $S$ in the stated sense.
\end{definition}

\begin{theorem}[Operational state-selection theorem]
\label{thm:descent-before-completion}
Let $\mathfrak R$ be a declared response program on
$(\mathsf U_0,\mathbb I)$.
\begin{enumerate}[label=\textup{(\roman*)},leftmargin=2.5em]
\item Every compatible primitive response factors uniquely through the
information quotient.  All algebraic, causal, restart, and uniformly continuous
postcomposition identities descend with it.
\item Let $\mathbf R^{\mathbb I}$ be the descended joint profile and put
\[
 u\sim_{\mathfrak R}v
 \quad\Longleftrightarrow\quad
 \mathbf R^{\mathbb I}(u)=\mathbf R^{\mathbb I}(v).
\]
The quotient by this response congruence carries the Hausdorff initial
uniformity generated by the responses.  The closure of its joint response image
in the product target is the completed response state, and every uniformly
continuous descended constructor extends uniquely to it.  If the program is
causally calibrated in the sense of
\Cref{def:causally-calibrated-response-program}, then the quotient and its
completion inherit the induced truncations, and every equivariant uniformly
continuous constructor extends as a stable-causal map.
\item If a required response fails information descent, no pointwise
realization on the full raw carrier that factors through $q_{\mathbb I}$ can
realize the program.  A probabilistic implementation relative to $\cP$ may
instead impose descent on the realized support; that is a support-relative
program rather than a realization of the full raw one.
\item The calibrated completion is unique in the following precise sense.  Every
complete faithful realization of the descended response presentation whose
uniformity is the initial uniformity generated by the extended responses is
uniquely uniformly isomorphic to the response completion; equivalently, no
second complete geometry can realize the same calibrated profile faithfully.
This is the deterministic uniqueness theorem
\Cref{thm:canonical-response-completion}.
\item Suppose in addition that a probabilistic implementation carries a
$\cP$-causal state $Z$ with stable-causal joint encoder
$\Phi:\mathsf Z\to\mathsf R$ whose value $\Phi(Z)$ is the declared joint
response under every $P\in\cP$, and a global stable-causal decoder
$D_{\mathfrak R}:\mathsf R\to\mathsf Z$ satisfying
$D_{\mathfrak R}\Phi=I_{\mathsf Z}$.  Then, for every $\cP$-causal state $S$,
joint realization of the completed package is equivalent to $Z\preceqsc S$,
and a joint readout $\Psi$ has the factor
\[
 R_S=D_{\mathfrak R}\Psi.
\]
If the decoder factors through one distinguished response,
$D_{\mathfrak R}=D_\star\Pi_\star$, and that response is
$\Pi_\star\Phi(Z)$, then a separate realization $\Psi_\star$ of it already
yields the factor $D_\star\Psi_\star$.
Thus $Z$ is the unique least sufficient realized state up to stable-causal
equivalence; the factor is unique on the realized source state.
\end{enumerate}
\end{theorem}

\begin{proof}
Part~\textup{(i)} is the quotient universal property.  For~\textup{(ii)}, the
response quotient Hausdorffizes the initial pseudouniformity and its injective
profile embedding completes by closure.  For~\textup{(iii)}, every pointwise
state and readout factoring through $q_{\mathbb I}$ is constant on
$q_{\mathbb I}$-fibers, whereas a response that fails information descent is
not; hence no such full-carrier realization exists.  Part~\textup{(iv)} is the uniqueness
clause of \Cref{thm:canonical-response-completion}: completeness closes the
embedded image and density identifies it with the completed range.  For
part~\textup{(v)}, a joint readout $\Psi:\mathsf S\to\mathsf R$ yields the
explicit factor
$D_{\mathfrak R}\Psi:S\to Z$; conversely a factor $R:S\to Z$ yields the
joint readout $\Phi R$.  On realized source states the two assignments are
inverse, because $D_{\mathfrak R}\Phi=I$ and
$\Phi D_{\mathfrak R}\Psi(S)=\Phi(Z)=\Psi(S)$ quasi surely.  Hence the factor
is unique on the realized source.  Finally,
$D_{\mathfrak R}=D_\star\Pi_\star$ turns a separate distinguished-response
realization into the factor $D_\star\Psi_\star$.
\end{proof}

\begin{corollary}[Extension of a calibrated program]
\label{cor:calibrated-program-extension}
There are two distinct pre-completion obstructions.  If a required response
fails information descent, the information presentation or the declared
program must be refined.  If the response descends but is not uniformly
continuous for the chosen calibration, its inclusion requires a refined
calibration and a new completion.  A uniformly continuous derived coordinate,
on the other hand, extends to the existing completion.  A retained coordinate
can be compressed precisely when it is recovered continuously by a
stable-causal readout from the reduced completed state.
\end{corollary}

\begin{corollary}[State-selection alternatives under a fixed calibration]
\label{cor:state-selection-trichotomy}
Assume the response program is causally calibrated.  Let $C^\circ$ be a
candidate coordinate on the descended operational core, with values in a
complete Hausdorff causal uniform space and commuting there with the declared
truncations, and let
$\mathsf Z_{\mathfrak R}$ be the completion generated by the declared
calibrated response profile.  Then the following alternatives are exhaustive.
\begin{enumerate}[label=\textup{(\roman*)},leftmargin=2.5em]
\item If $C^\circ$ is not constant on response-equivalence classes, the
declared program identifies core inputs carrying different $C^\circ$-values.
The coordinate does not descend and is invisible to the declared program.
\item Suppose $C^\circ$ descends to the response quotient but is not uniformly
continuous for the response-generated uniformity.  Then the present
calibration does not control that coordinate.  Retaining it requires adjoining
its uniform structure and recompleting with the refined calibration.  In a
metrizable presentation,
failure is witnessed by pairs $u_n,v_n$ with response distance tending to zero
but with the $C^\circ$-distance bounded away from zero.
\item Suppose $C^\circ$ is uniformly continuous.  It then extends uniquely to
a uniformly continuous stable-causal coordinate
$C:\mathsf Z_{\mathfrak R}\to\mathsf C$.  Relative to a proposed
stable-causal compression $\pi:\mathsf Z_{\mathfrak R}\to\mathsf Y$, the
coordinate is propagated exactly when there is a stable-causal
$\widetilde C:\mathsf Y\to\mathsf C$ such that
$C=\widetilde C\circ\pi$.
\item If the extension in \textup{(iii)} exists but no such factorization
through $\pi$ exists, then $\pi$ is not lossless for any response package
containing $C$.  In particular, a coordinate of the least state may be removed
through $\pi$ only when the whole least state admits stable-causal
recovery from $\pi(\mathsf Z_{\mathfrak R})$.
\end{enumerate}
The same alternatives apply to a tangent or derivative coordinate after its
scale and response calibration have been declared.
\end{corollary}

\begin{proof}
Descent is constancy on response classes, proving~\textup{(i)}.  On the
quotient, uniform completion extends a map into a complete Hausdorff target
exactly when it is uniformly continuous on the dense core; its metric failure
criterion gives the pairs in~\textup{(ii)}.  Equivariance is preserved by
\Cref{lem:causal-calibrated-completion}, and propagation in~\textup{(iii)} is
exactly factorization through $\pi$.  If~\textup{(iv)} failed, projecting a
putative lossless readout onto $C$ would supply the excluded factorization; the
same argument applied to a full decoder proves the least-state assertion.
\end{proof}

\begin{theorem}[Program refinement and the enhancement criterion]
\label{thm:program-refinement-enhancement}
Let $\mathfrak R$ be a causally calibrated response program on a descended
operational core $X$, and let $\mathfrak R^+$ be obtained by adjoining a family
of causally equivariant responses
\[
 S_\beta^\circ:X\longrightarrow\mathsf T_\beta,
 \qquad \beta\in B,
\]
into complete Hausdorff causal uniform spaces.  Write
$\mathsf Z_{\mathfrak R}$ and $\mathsf Z_{\mathfrak R^+}$ for the corresponding
response completions.
\begin{enumerate}[label=\textup{(\roman*)},leftmargin=2.5em]
\item If some $S_\beta^\circ$ is not constant on
$\mathfrak R$-response classes, then the enlarged program refines the response
quotient before any topology is completed.
\item Suppose every $S_\beta^\circ$ descends to the
$\mathfrak R$-response quotient.  Then the enlarged and original response
congruences coincide, and the identity on the common dense quotient is
uniformly continuous from the refined response uniformity to the original one.
It therefore extends uniquely to a stable-causal map
\begin{equation}\label{eq:program-refinement-canonical-map}
 \Pi_{\mathfrak R^+,\mathfrak R}:
 \mathsf Z_{\mathfrak R^+}\longrightarrow\mathsf Z_{\mathfrak R}
\end{equation}
with dense range.
\item Under the hypothesis of~\textup{(ii)}, the following are equivalent:
\begin{enumerate}[label=\textup{(\alph*)},leftmargin=2.5em]
\item every added response $S_\beta^\circ$ is uniformly continuous for the
$\mathfrak R$-generated uniformity;
\item the original and enlarged response-generated uniformities coincide;
\item the map in \eqref{eq:program-refinement-canonical-map} is the unique
uniform isomorphism of the two completions, and every $S_\beta^\circ$ extends
uniquely as a stable-causal readout of $\mathsf Z_{\mathfrak R}$.
\end{enumerate}
\item If one descended added response is not uniformly continuous for the old
calibration, it has no uniformly continuous extension to
$\mathsf Z_{\mathfrak R}$ as a response of the declared calibrated program.
Retaining it requires the refined completion $\mathsf Z_{\mathfrak R^+}$.
In a metrizable presentation the failure is witnessed by pairs $x_n,y_n$ whose
old response distance tends to zero while the added-response distance stays
bounded away from zero.
\end{enumerate}
\end{theorem}

\begin{proof}
Part~\textup{(i)} is the response-congruence definition.  Under~\textup{(ii)},
the enlarged initial uniformity is finer, so the identity to the old quotient
extends by completion to \eqref{eq:program-refinement-canonical-map}; its range
is dense because it contains the common core.  Corewise intertwining of
truncations and \Cref{lem:causal-calibrated-completion} make it stable-causal.

If all additions are old-uniformly continuous, minimality of the initial
uniformity makes the two uniformities equal; the converse is immediate.
Uniqueness of completion and causal equivariance give the isomorphism and
extended readouts in~\textup{(iii)}.  Finally, extension from a dense uniform
subspace to a complete Hausdorff target is equivalent to uniform continuity;
the metric failure criterion proves~\textup{(iv)} and its witness.
\end{proof}

\begin{remark}[Examples of program refinement]
Classical It\^o integration and Brownian $L^2$ BSDEs use
probability-calibrated energy completions on which the relevant stochastic
responses already extend.  A L\'evy-area response, in contrast, is not
continuous for the uniform-path topology and leads to the refined rough-path
completion.  Program refinement is therefore governed by continuity of the
added response in the existing calibration.
\end{remark}

\begin{corollary}[Approximation invariance after descent]
\label{cor:descent-proof-economy}
Once the program has descended, representative independence is settled.
Cauchy estimates, compactness, completion, and universal-limit arguments may be
carried out on the response classes, and continuous descended constructors take
the same value on any two approximation schemes converging to the same class.
\end{corollary}

\begin{remark}[Separation of reductions]
\label{rem:separation-of-reductions}
The following operations are distinct.
\begin{enumerate}[label=\textup{(\roman*)},leftmargin=2.5em]
\item Information identification compares histories through $\mathbb I$.
\item Response congruence identifies information classes not separated by the
program and determines the complete state geometry.
\item Probabilistic version identification compares realized objects modulo a
null ideal; it creates neither of the preceding quotients.
\item Stable compression removes retained data only through continuous
stable-causal factorization.
\item Jet extension adds a scaled infinitesimal response coordinate and is not an
enlargement of the base filtration.
\end{enumerate}
\end{remark}

\begin{proposition}[Lossless-chart factorization]
\label{prop:lossless-chart-factorization}
Let $(\mathsf Z,Z)$ be a $\cP$-causal state.  Let
\[
 \mathcal O:\mathsf Z\to\mathsf R,
 \qquad
 \mathcal D:\mathsf R\to\mathsf Z
\]
be stable-causal maps between causal Polish spaces such that
\begin{equation}\label{eq:lossless-chart-left-inverse}
 \mathcal D\mathcal O=I_{\mathsf Z}.
\end{equation}
Then $\mathcal O$ is a homeomorphism onto the closed retract
\begin{equation}\label{eq:lossless-chart-closed-retract}
 \mathcal O(\mathsf Z)=\operatorname{Fix}(\mathcal O\mathcal D)
 \subseteq\mathsf R.
\end{equation}
This retract is generated by the stable-causal idempotent
$\mathcal O\mathcal D$.
For a $\cP$-causal source state $(\mathsf S,S)$, a joint realization of the
chart is a stable-causal map $\Psi:\mathsf S\to\mathsf R$ satisfying
$\Psi(S)=\mathcal O(Z)$ $P$-almost surely for every $P\in\cP$.  Then:
\begin{enumerate}[label=\textup{(\roman*)},leftmargin=2.5em]
\item Every joint realization yields the stable-causal factor
$\mathcal D\Psi:S\to Z$, and every stable-causal factor $R:S\to Z$ yields
the joint realization $\mathcal O R$.  Thus
\begin{equation}\label{eq:lossless-chart-factor-bijection}
 \Psi\longmapsto\mathcal D\Psi,
 \qquad
 R\longmapsto\mathcal O R
\end{equation}
are inverse after maps agreeing on the realized source state under every law
are identified.  In particular, the source realizes the lossless chart if and
only if $Z\preceqsc S$, and the factor is unique on the realized source.

\item Every stable-causal readout
$F:\mathsf Z\to\mathsf T$ factors through the chart on its image:
\begin{equation}\label{eq:derived-through-lossless-chart}
 F=(F\mathcal D)\mathcal O.
\end{equation}
The factor $F\mathcal D$ is unique on $\mathcal O(\mathsf Z)$.  If in addition
$\mathcal O\mathcal D=I_{\mathsf R}$, it is the unique factor on all of
$\mathsf R$.

\item Suppose a stable-causal map
$\pi:\mathsf Z\to\mathsf Y$ preserves the full chart in the sense that
$\mathcal O=\Psi\pi$ for a stable-causal
$\Psi:\mathsf Y\to\mathsf R$.  Then
\begin{equation}\label{eq:lossless-compression-left-inverse}
 L:=\mathcal D\Psi
 \quad\text{satisfies}\quad
 L\pi=I_{\mathsf Z}.
\end{equation}
Consequently $\pi$ is a homeomorphism onto a closed retract of $\mathsf Y$
generated by the stable-causal idempotent $P:=\pi L$, and
\[
 \pi(\mathsf Z)=\operatorname{Fix}(P).
\]
If $\pi$ is surjective, then $\pi$ is a stable-causal homeomorphism with
inverse $L$.
\end{enumerate}
\end{proposition}

\begin{proof}
The left inverse makes $\mathcal O$ injective with continuous inverse
$\mathcal D|_{\mathcal O(\mathsf Z)}$.  The stable-causal idempotent
$\mathcal O\mathcal D$ has fixed set $\mathcal O(\mathsf Z)$, closed because
$\mathsf R$ is Hausdorff, proving
\eqref{eq:lossless-chart-closed-retract}.

The two maps in \eqref{eq:lossless-chart-factor-bijection} are well defined
because $\Psi(S)=\mathcal O(Z)$ implies $\mathcal D\Psi(S)=Z$, while $R(S)=Z$
implies $\mathcal OR(S)=\mathcal O(Z)$.  Their composites agree with the
original maps on the realized source because $\mathcal D\mathcal O=I$ and
$\mathcal O\mathcal D\Psi(S)=\mathcal O(Z)=\Psi(S)$ under every law.

Equation \eqref{eq:derived-through-lossless-chart} is
$F\mathcal D\mathcal O=F$.  If $G\mathcal O=F$, then for
$r=\mathcal O(z)$ one has
$G(r)=F(z)=F\mathcal D(r)$, proving uniqueness on the chart image and hence
on all of $\mathsf R$ when the chart is onto.

Finally, $L\pi=\mathcal D\Psi\pi=I$, so $L$ inverts $\pi$ on its image.
The stable-causal idempotent $P=\pi L$ has fixed set
$\pi(\mathsf Z)$, which is closed in Hausdorff $\mathsf Y$; surjectivity makes
this image all of $\mathsf Y$.
\end{proof}

\begin{corollary}[Equivalence of lossless response charts]
\label{cor:canonical-equivalence-lossless-charts}
Let, for $i=1,2$,
\[
 \mathcal O_i:\mathsf Z\to\mathsf R_i,
 \qquad
 \mathcal D_i:\mathsf R_i\to\mathsf Z
\]
be stable-causal maps with
$\mathcal D_i\mathcal O_i=I_{\mathsf Z}$.  Then
\begin{equation}\label{eq:canonical-lossless-chart-transports}
 T_{21}:=\mathcal O_2\mathcal D_1:\mathsf R_1\to\mathsf R_2,
 \qquad
 T_{12}:=\mathcal O_1\mathcal D_2:\mathsf R_2\to\mathsf R_1
\end{equation}
restrict to mutually inverse homeomorphisms
\[
 \mathcal O_1(\mathsf Z)\cong\mathcal O_2(\mathsf Z).
\]
These homeomorphisms are the restrictions of the ambient stable-causal maps
$T_{21}$ and $T_{12}$.
They are the unique transports on the chart images that intertwine the two
encoders:
\[
 T_{21}\mathcal O_1=\mathcal O_2,
 \qquad
 T_{12}\mathcal O_2=\mathcal O_1.
\]
If both charts are onto, $T_{21}$ and $T_{12}$ are inverse stable-causal
homeomorphisms of the full response targets.  Thus two lossless presentations
of one completed state differ only by the explicit encoder--decoder change of
coordinates in \eqref{eq:canonical-lossless-chart-transports}.
\end{corollary}

\begin{proof}
On $\mathcal O_1(\mathsf Z)$,
\[
 T_{12}T_{21}\mathcal O_1
 =\mathcal O_1\mathcal D_2\mathcal O_2\mathcal D_1\mathcal O_1
 =\mathcal O_1,
\]
and the symmetric identity holds on $\mathcal O_2(\mathsf Z)$.  The
intertwining formulas follow from $\mathcal D_i\mathcal O_i=I$.  Uniqueness on
the chart images is the uniqueness clause of
\Cref{prop:lossless-chart-factorization}\textup{(ii)}.  Surjectivity identifies
each chart image with its full target.
\end{proof}
\medskip

\subsection{Probability as calibration, realization, and version semantics}

\begin{remark}[The roles of probability]
\label{rem:roles-of-probability}
A law family does not define causal information.  Canonical experiments may
calibrate responses as in \Cref{sec:quadratic-response-specification}; after
selection, a raw construction realizes the state.  Laws then supply polar
sets, version equality, stochastic semantics, and conditional expectations,
while
simultaneous flow or cocycle laws additionally require parameter-stable
perfection.
\end{remark}

\subsection{The degree-two operational test program}

The principal program combines full stopped-path information with restarted
integration, area/covariance probes, and restart composition.  Resolvent defect
is transient finite-scale memory; the later covariance-energy quotient is
tangent-response generated, and quasi-sure equality enters only after
realization.

\section{The least quadratic It\^o state}
\label{sec:intrinsic-state}

For the principal degree-two It\^o program, descent and completion retain
$(X,A,Q)$.  Brownian response magnitudes select its anisotropic geometry as
Hilbert--Schmidt/operator norm duality; their direct-sum laws follow, and
converse rigidity recovers the same split Schatten completion from the
normalized orthogonally natural laws.

Under common accessibility, the intrinsic response syntax and its coherent
completion jointly yield reconstruction and least-state factorization.

\begin{theorem}[Calibrated quadratic reconstruction theorem]
\label{thm:quadratic-response-reconstruction}
Let $H$ be a separable real Hilbert space and $1/3<\eta<1/2$.  Let
$\mathfrak Q_2(H)$ be the principal quadratic response specification defined in
\Cref{sec:quadratic-response-specification,def:coherent-principal-response-target}.
Let $\tau_0=(\eta;\mathcal S_2,\mathcal S_1)$; the temporal exponent $\eta$ is
fixed, while \Cref{thm:Brownian-calibration-rigidity} selects the central ideal
pair.  Assume that the common accessibility hypothesis of
\Cref{def:common-accessibility} holds for $\cP$.
Then the completed deterministic response target is canonically represented by
the stable-causal isometric homeomorphism
\[
 \mathcal O_{\mathfrak Q_2}:
 \mathsf Z_{\eta;\mathcal S_2,1}(H)
 \xrightarrow{\ \cong\ }
 \mathsf R_{\mathfrak Q_2}^{\rm coh},
 \qquad
 \mathcal O_{\mathfrak Q_2}^{-1}=\mathcal D_{\mathfrak Q_2}.
\]
Write $\mathcal A_2:=Z_{\tau_0}$ for the accessible $\cP$-causal realization
in \Cref{def:common-accessibility}.  Its canonical coordinates are
\begin{equation}\label{eq:quadratic-state-canonical-chart}
 \operatorname{ch}_{\tau_0}(\mathcal A_2)=(X,A,Q),
 \qquad
 (A,Q)=\bigl(A^P,\qv{M^P}\bigr)
 \quad P\text{-a.s. for every }P\in\cP.
\end{equation}
The response reconstruction has the following properties.
\begin{enumerate}[label=\textup{(\roman*)},leftmargin=2.5em]
\item \emph{Exact response.}  Under the preceding identification, the
restarted step-two It\^o signature is injective on the completed response
state and has a stable-causal decoder.  In every accessible probabilistic
implementation its common realization agrees modelwise with the classical
signature; equal realized signature fields therefore have quasi-surely equal
$(X,A,Q)$ chart coordinates.
\item \emph{Algebraic rigidity.}  The finite-rank restart responses reconstruct
\[
 (x,a,q)\star(y,b,r)
 =(x+y,a+b+\Anti(x\otimes y),q+r),
\]
the unique central-coordinate degree-two product with additive covariance and
multiplicative It\^o tensor readout.
\item \emph{Analytic rigidity.}  Brownian area energy and rank-one covariance
defect admit the exact calibrations
$\mathfrak c_A(K)=\|K\|_2$ and
$\mathfrak c_Q(S)=\|S\|_{\rm op}$.  Their Pythagorean and max direct-sum laws
are therefore consequences, while the converse orthogonal rigidity theorem
shows that these normalized orthogonally natural laws, together with the
principal normalizations, uniquely recover the Hilbert--Schmidt and operator
coefficient gauges and the dual state completion
\[
 H\times\mathcal S_2(H)_{\rm sk}\times\mathcal S_1(H)_{\rm sa}.
\]
At the covariance endpoint, the finite-rank coefficient completion is
$\mathcal K(H)_{\rm sa}$; the principal program separately uses the full
normal test class $\mathcal L(H)_{\rm sa}$, whose unit ball induces the same
trace-class state gauge.
The induced path topology is $\tau_{\eta;\mathcal S_2,1}$.  The finite-rank
algebra and qualitative point-separation alone do not determine this topology.
\item \emph{Lossless factorization.}  For every
$\cP$-causal state $(\mathsf S,S)$ and every joint realization
$\Psi_S:\mathsf S\to\mathsf R_{\mathfrak Q_2}^{\rm coh}$,
\[
 R_S:=\mathcal D_{\mathfrak Q_2}\Psi_S:
 \mathsf S\longrightarrow\mathsf Z_{\eta;\mathcal S_2,1}(H),
 \qquad R_S(S)=\mathcal A_2\quad\cP\text{-quasi surely}.
\]
Conversely, every stable-causal map
$R:\mathsf S\to\mathsf Z_{\eta;\mathcal S_2,1}(H)$ with
$R(S)=\mathcal A_2$ quasi surely gives the realization
$\mathcal O_{\mathfrak Q_2}R$.  The resulting least-state equivalence is
stated in \Cref{thm:quadratic-response-representability}.
\item \emph{Compression obstructions.}  No noninjective stable-causal quotient
of the full operational state space through which the complete norming map
factors is lossless.  The complete scalar triangular family separates states,
whereas every fixed finite family misses a nonzero finite-rank central
direction in infinite dimension.
\item \emph{Propagation.}  Every finite product and stable-causal
postcomposition of the principal responses is a stable-causal readout of
$\mathsf Z_{\eta;\mathcal S_2,1}(H)$.  Any lossless derived response represents
the same realized state as $\mathcal A_2$ up to stable-causal equivalence.
\end{enumerate}
\end{theorem}

\begin{theorem}[Least-state factorization for the calibrated quadratic program]
\label{thm:quadratic-response-representability}
Under the hypotheses of \Cref{thm:quadratic-response-reconstruction}, every
$\cP$-causal source state $(\mathsf S,S)$ satisfies
\begin{equation}\label{eq:quadratic-response-universal-property}
 S\text{ jointly realizes }\mathfrak Q_2(H)
 \quad\Longleftrightarrow\quad
 \mathcal A_2\preceqsc S.
\end{equation}
If $\Psi_S$ is a joint realization, the factor is the stable-causal map
$\mathcal D_{\mathfrak Q_2}\Psi_S:
\mathsf S\to\mathsf Z_{\eta;\mathcal S_2,1}(H)$.  Conversely, a stable-causal map
$R:\mathsf S\to\mathsf Z_{\eta;\mathcal S_2,1}(H)$ satisfying
$R(S)=\mathcal A_2$ quasi surely produces the realization
$\mathcal O_{\mathfrak Q_2}R$.  These assignments are inverse after maps
agreeing on the realized source under every law are identified.  Hence
$\mathcal A_2$ is the unique least jointly sufficient state up to stable-causal
equivalence.
\end{theorem}

\subsection{The algebraic skeleton and It\^o-defect rigidity}
\label{sec:minimality}

We use standard Schatten-ideal notation, ideal inequalities, and trace dualities
as in \cite{Simon05}.  The topology classification is deterministic.  Accessibility identifies the
intrinsic chart with modelwise It\^o area and bracket and enters the
least-state theorem.  A jointly topologized triangular response recovers the
prescribed operator-ideal topology exactly when its probe families are
norming; failure of norming produces finite-rank invisible directions.

Fix a separable real Hilbert space $H$ and write
\[\Omega_H:=C_0([0,T];H).\]
Fix $\eta\in(1/3,1/2)$.  Let $\mathfrak I(H)$ be a separable Banach ideal of
compact operators such that finite-rank operators are dense, adjoint is
isometric,
\begin{equation}
\label{eq:operational-rank-one}
 \|u\otimes v\|_{\mathfrak I}=\|u\|_H\|v\|_H,
\end{equation}
and the inclusion $\mathcal S_1(H)\hookrightarrow\mathfrak I(H)$ is
continuous.  Put
\[
 E_A(H):=\mathfrak I(H)_{\rm sk},
 \qquad E_Q(H):=\mathcal S_1(H)_{\rm sa},
\]
and define the canonical alternating cocycle
\[
 \omega_H(x,y):=\Anti(x\otimes y)\in E_A(H).
\]

\begin{definition}[Intrinsic It\^o degree-two group]
\label{def:intrinsic-Ito-group}
The \emph{two-channel It\^o degree-two group} is
\begin{equation}\label{eq:intrinsic-Ito-group}
 \Gito_{\mathfrak I,1}(H)
 :=H\times E_A(H)\times E_Q(H)
\end{equation}
with multiplication
\begin{equation}\label{eq:intrinsic-Ito-product}
 (x,a,q)\star(y,b,r)
 :=\bigl(x+y,a+b+\omega_H(x,y),q+r\bigr).
\end{equation}
\end{definition}

Let $\operatorname{Fin}(H)$ denote the finite-rank operators and set
\[
 E_A^{\rm fin}(H):=\operatorname{Fin}(H)_{\rm sk},
 \qquad
 E_Q^{\rm fin}(H):=\operatorname{Fin}(H)_{\rm sa}.
\]
The same product on
\begin{equation}\label{eq:finite-rank-central-skeleton}
 \Gito_{\rm fin}(H)
 :=H\times E_A^{\rm fin}(H)\times E_Q^{\rm fin}(H)
\end{equation}
is independent of the operational type.  Finite-rank density in
$\mathfrak I(H)$ and $\mathcal S_1(H)$ identifies
$\Gito_{\mathfrak I,1}(H)$ with the completion of this common central
skeleton in the $\mathfrak I$- and trace norms.  Thus the algebraic cocycle is
fixed before the observable topology is chosen.

\begin{theorem}[It\^o-defect characterization and algebraic rigidity]
\label{prop:intrinsic-central-extension}
The product \eqref{eq:intrinsic-Ito-product} makes
$\Gito_{\mathfrak I,1}(H)$ a two-step Banach Lie group with inverse
$(x,a,q)^{-1}=(-x,-a,-q)$ and central extension
\[
 0\longrightarrow E_A(H)\oplus E_Q(H)
 \longrightarrow\Gito_{\mathfrak I,1}(H)
 \longrightarrow H\longrightarrow0.
\]
Let $\mathcal T^2_{\mathfrak I}(H)$ be the step-two tensor group with product
$(x,U)\cdot(y,V)=(x+y,U+V+x\otimes y)$ and define
\begin{equation}\label{eq:Ito-defect-subgroup}
 \mathfrak G^{I}_{\mathfrak I,1}(H)
 :=\bigl\{(x,U,q)\in\mathcal T^2_{\mathfrak I}(H)\times E_Q(H):
 U+U^*=x^{\otimes2}-q\bigr\},
\end{equation}
with the product induced from $\mathcal T^2_{\mathfrak I}(H)\times E_Q(H)$.
Then $\mathfrak G^{I}_{\mathfrak I,1}(H)$ is a closed subgroup and
\begin{equation}\label{eq:Ito-defect-isomorphism}
 \Phi_H^I(x,a,q)
 :=\left(x,\frac12x^{\otimes2}+a-\frac12q,q\right)
\end{equation}
is a homeomorphic group isomorphism
$\Gito_{\mathfrak I,1}(H)\to\mathfrak G^{I}_{\mathfrak I,1}(H)$ with
inverse
\[
 (x,U,q)\longmapsto(x,\Anti U,q).
\]
In these coordinates the geometric and covariance readouts are
\begin{align}
 \Theta_H^S(x,a,q)
 &:=\left(x,\frac12x^{\otimes2}+a\right),
 \label{eq:intrinsic-geometric-readout}\\
 \Theta_H^I(x,a,q)
 &:=\left(x,\frac12x^{\otimes2}+a-\frac12q\right),
 \label{eq:intrinsic-Ito-readout}\\
 \Theta_H^Q(x,a,q)&:=q,
 \label{eq:intrinsic-covariance-readout}
\end{align}
and all three maps are continuous homomorphic readouts;
$(\Theta_H^S,\Theta_H^Q)$ is injective.

Moreover, this product is rigid within central-coordinate degree-two
extensions.  Suppose a continuous product on the same underlying space has
the form
\[
 (x,a,q)\star'(y,b,r)
 =\bigl(x+y,a+b+\beta_A(x,y),q+r+\beta_Q(x,y)\bigr),
\]
where $\beta_A$ and $\beta_Q$ are continuous, and assume that the covariance
projection is a homomorphism and that the map in
\eqref{eq:intrinsic-Ito-readout} is a homomorphism into
$\mathcal T^2_{\mathfrak I}(H)$.  Then necessarily
\[
 \beta_Q(x,y)=0,\qquad
 \beta_A(x,y)=\Anti(x\otimes y),
\]
so $\star'=\star$.
\end{theorem}

\begin{proof}
The alternating cocycle identity gives associativity; centrality and the
inverse formula are immediate.  The defect set is closed as the zero set of
$(x,U,q)\mapsto U+U^*-x^{\otimes2}+q$.  It is a subgroup: if
$U+U^*=x^{\otimes2}-q$ and $V+V^*=y^{\otimes2}-r$, then
\[
 (U+V+x\otimes y)+(U+V+x\otimes y)^*
 =(x+y)^{\otimes2}-(q+r).
\]
The ambient inverse is
$(-x,-U+x^{\otimes2},-q)$, and its defect is
\[
 (-U+x^{\otimes2})+(-U+x^{\otimes2})^*
 =x^{\otimes2}+q=(-x)^{\otimes2}-(-q),
\]
so inverses remain in it.  The image of \eqref{eq:Ito-defect-isomorphism}
satisfies the defect identity; conversely that identity gives
$\Sym U=\frac12(x^{\otimes2}-q)$, hence
$U=\frac12x^{\otimes2}+\Anti U-\frac12q$, proving bijectivity and the stated
inverse.  Transport by $\Phi_H^I$ gives the Banach Lie structure, and
multiplicativity follows from
\[
 \frac12(x+y)^{\otimes2}+a+b+\Anti(x\otimes y)-\frac12(q+r)
 =\left(\frac12x^{\otimes2}+a-\frac12q\right)
  +\left(\frac12y^{\otimes2}+b-\frac12r\right)+x\otimes y.
\]
The same identity without the covariance correction proves multiplicativity
of $\Theta_H^S$, while $\Theta_H^Q$ is additive; joint injectivity recovers
$a$ and $q$.

For rigidity, covariance additivity first forces $\beta_Q=0$.  The homomorphism
property of \eqref{eq:intrinsic-Ito-readout} then yields
\[
 \frac12(x+y)^{\otimes2}+a+b+\beta_A(x,y)-\frac12(q+r)
 =\frac12x^{\otimes2}+a-\frac12q
  +\frac12y^{\otimes2}+b-\frac12r+x\otimes y,
\]
and subtraction gives
$\beta_A(x,y)=\frac12(x\otimes y-y\otimes x)=\Anti(x\otimes y)$.
\end{proof}

\begin{remark}[Algebraic rigidity class]
The classification is taken within central-coordinate degree-two products
compatible with the additive covariance channel and the It\^o tensor readout.
Within this class, fixing the two degree-two observables determines the
algebraic cocycle uniquely.
\end{remark}

\subsection{Stochastic calibration and response-forced coefficient geometry}
\label{sec:quadratic-response-specification}

With the algebraic product fixed, Brownian experiments and It\^o isometry now
compute the area and covariance coefficient gauges.  Their Pythagorean/max
block laws and converse rigidity characterize the calibrated norms
(\Cref{cor:calibration-direct-sum-equivalence}); dual completion fixes the state
topology before general semimartingale laws realize the syntax.

\begin{definition}[Brownian response gauges]
\label{def:Brownian-response-calibration}
For $E=\{0\}$ set
$\mathfrak c_A^E(0)=\mathfrak c_Q^E(0)=0$.  Let now $E$ be nonzero and
finite dimensional, and let $B^E$ be standard $E$-valued Brownian motion on
$[0,1]$.  For a unit vector $e\in E$ put $B_t^e=\beta_te$, where
$\beta$ is a standard real Brownian motion.  For
$K\in\operatorname{Skew}(E)$ and $S\in\operatorname{Sym}(E)$ define
\begin{align}
 \mathfrak c_A^E(K)
 &:=\left(2\,\mathbb E\left|\frac12\int_0^1
   \operatorname{Tr}\!\left[K^*\bigl(B_t^E\otimes\dd B_t^E
   -\dd B_t^E\otimes B_t^E\bigr)\right]\right|^2\right)^{1/2},
 \label{eq:Brownian-area-calibration}\\
 \mathfrak c_Q^E(S)
 &:=\sup_{\|e\|=1}
   \left\|
   \operatorname{Tr}\!\left[S(B_1^e\otimes B_1^e)\right]
   -\int_0^1\operatorname{Tr}\!\left[
   S\bigl(B_t^e\otimes\dd B_t^e+\dd B_t^e\otimes B_t^e\bigr)\right]
   \right\|_{L^2(\Omega)}.
 \label{eq:Brownian-defect-calibration}
\end{align}
These are respectively centered area energy and maximal covariance-defect
$L^2$ amplitude over unit rank-one experiments; no operator-ideal norm is
presupposed.
\end{definition}

\begin{theorem}[Exact Brownian calibration]
\label{thm:Brownian-calibration-rigidity}
For every finite-dimensional real Hilbert space $E$,
\begin{equation}\label{eq:Brownian-calibration-identification}
 \boxed{
 \mathfrak c_A^E(K)=\|K\|_{\mathcal S_2},\qquad
 \mathfrak c_Q^E(S)=\|S\|_{\rm op}.}
\end{equation}
Thus the dual state gauges are exactly the Hilbert--Schmidt and trace norms,
and the compatible infinite-dimensional response completion is
\[
 H\times\mathcal S_2(H)_{\rm sk}\times\mathcal S_1(H)_{\rm sa}.
\]
The calibration is orthogonally natural.  For increasing finite-rank
orthogonal projections $P_n\uparrow I$, every
$K\in\mathcal S_2(H)_{\rm sk}$, and every
$S\in\mathcal L(H)_{\rm sa}$,
\begin{equation}\label{eq:Brownian-calibration-exhaustion}
 \sup_n\mathfrak c_A^{P_nH}(P_nKP_n)=\|K\|_2,
 \qquad
 \sup_n\mathfrak c_Q^{P_nH}(P_nSP_n)=\|S\|_{\rm op}.
\end{equation}
\end{theorem}

\begin{proof}
The zero-dimensional case holds by convention.  Assume $E\ne\{0\}$.
For skew $K$, the rank-one trace pairing gives
\[
 \frac12\operatorname{Tr}\!\left[
 K^*(B_t\otimes\dd B_t-\dd B_t\otimes B_t)\right]
 =-\langle KB_t,\dd B_t\rangle.
\]
It\^o isometry and the covariance identity
$\mathbb E(B_t\otimes B_t)=tI_E$ therefore yield
\begin{align*}
 (\mathfrak c_A^E(K))^2
 &=2\int_0^1\mathbb E\|KB_t\|^2\,\dd t
 =2\int_0^1 t\,\operatorname{Tr}(K^*K)\,\dd t
 =\|K\|_{\mathcal S_2}^2.
\end{align*}
For the rank-one process $B_t^e=\beta_te$, the scalar It\^o product formula
holds almost surely in the form
\[
 B_1^e\otimes B_1^e
 -\int_0^1\bigl(B_t^e\otimes\dd B_t^e
                 +\dd B_t^e\otimes B_t^e\bigr)
 =e\otimes e.
\]
Consequently the random variable inside the $L^2$ norm is the deterministic
constant $\operatorname{Tr}(S(e\otimes e))$, and hence
\[
 \mathfrak c_Q^E(S)
 =\sup_{\|e\|=1}|\operatorname{Tr}(S(e\otimes e))|
 =\sup_{\|e\|=1}|\langle Se,e\rangle|
 =\|S\|_{\rm op}.
\]
For finite-rank
$a\in\operatorname{Fin}(H)_{\rm sk}$ and
$q\in\operatorname{Fin}(H)_{\rm sa}$, choose a finite-dimensional subspace
$E$ supporting both operators and put
\[
 \|a\|_{A,*}
 :=\sup_{\mathfrak c_A^E(K)\le1}|\operatorname{Tr}(K^*a)|,
 \qquad
 \|q\|_{Q,*}
 :=\sup_{\mathfrak c_Q^E(S)\le1}|\operatorname{Tr}(Sq)|.
\]
The preceding identities and finite-dimensional duality give
\[
 \|a\|_{A,*}=\|a\|_2,
 \qquad
 \|q\|_{Q,*}=\|q\|_1.
\]
These values do not depend on the supporting space $E$, by orthogonal
naturality.  Completing the two finite-rank state axes therefore gives
$\mathcal S_2(H)_{\rm sk}\times\mathcal S_1(H)_{\rm sa}$; adjoining the
unchanged first-level Hilbert coordinate gives the displayed state completion.
Finally, $P_nKP_n\to K$ in Hilbert--Schmidt norm, while the union of the
$P_nH$ is dense and
$\|S\|_{\rm op}=\sup_{\|e\|=1}|\langle Se,e\rangle|$ for self-adjoint $S$.
These facts prove the two exhaustion formulas.
\end{proof}

\begin{definition}[Orthogonally natural quadratic coefficient architecture]
\label{def:orthogonal-quadratic-architecture}
For each finite-dimensional real Hilbert space $E$, let
$\kappa_A^E$ be a norm on $\operatorname{Skew}(E)$ and let $\kappa_Q^E$ be a
norm on $\operatorname{Sym}(E)$.  The family
$\boldsymbol\kappa=(\kappa_A^E,\kappa_Q^E)_E$ is called an
\emph{orthogonally natural quadratic coefficient architecture} if:
\begin{enumerate}[label=\textup{(A\arabic*)},leftmargin=2.8em]
\item for every orthogonal isomorphism $U:E\to F$,
\[
 \kappa_A^F(UKU^*)=\kappa_A^E(K),\qquad
 \kappa_Q^F(USU^*)=\kappa_Q^E(S);
\]
\item for orthogonal sums,
\begin{align*}
 \kappa_A^{E\oplus F}(K\oplus L)^2
 &=\kappa_A^E(K)^2+\kappa_A^F(L)^2,\\
 \kappa_Q^{E\oplus F}(S\oplus T)
 &=\max\{\kappa_Q^E(S),\kappa_Q^F(T)\};
\end{align*}
\item if $(e_1,e_2)$ is an orthonormal basis of $\mathbb R^2$ and
$J_0=e_1\otimes e_2-e_2\otimes e_1$, then
$\kappa_A^{\mathbb R^2}(J_0)=\sqrt2$, while
$\kappa_Q^{\mathbb R}(1)=1$.
\end{enumerate}
By \Cref{thm:Brownian-calibration-rigidity} the Brownian gauges obey these laws;
the next proposition is their deterministic converse.
\end{definition}

\begin{proposition}[Rigidity under the principal orthogonal response laws]
\label{thm:orthogonal-response-rigidity}
Every orthogonally natural quadratic coefficient architecture is unique.  For
every finite-dimensional real Hilbert space $E$,
\begin{equation}\label{eq:orthogonal-rigidity-coefficients}
 \kappa_A^E(K)=\|K\|_{\mathcal S_2},\qquad
 \kappa_Q^E(S)=\|S\|_{\rm op}.
\end{equation}
Consequently its dual response gauges on finite-rank state coordinates are
\begin{equation}\label{eq:orthogonal-rigidity-state-gauges}
 \sup_{\kappa_A(K)\le1}|\operatorname{Tr}(K^*a)|=\|a\|_2,
 \qquad
 \sup_{\kappa_Q(S)\le1}|\operatorname{Tr}(Sq)|=\|q\|_1.
\end{equation}
For a separable Hilbert space $H$, compatible finite-dimensional exhaustion
therefore produces the principal response completion
\begin{equation}\label{eq:orthogonal-principal-completion}
 H\times\mathcal S_2(H)_{\rm sk}\times\mathcal S_1(H)_{\rm sa}.
\end{equation}
\end{proposition}

\begin{proof}
Let $K\in\operatorname{Skew}(E)$.  The real skew spectral theorem gives an
orthogonal decomposition
\[
 E=E_0\oplus E_1\oplus\cdots\oplus E_m,
 \qquad
 K|_{E_0}=0,\qquad K|_{E_j}=a_jJ_0
\]
after identifying each two-plane $E_j$ with $\mathbb R^2$.  Orthogonal
naturality, homogeneity, and the square-sum law give
\[
 \kappa_A^E(K)^2
 =\sum_{j=1}^m |a_j|^2\kappa_A^{\mathbb R^2}(J_0)^2
 =2\sum_{j=1}^m a_j^2
 =\|K\|_{\mathcal S_2}^2.
\]
If $S\in\operatorname{Sym}(E)$, choose an orthonormal eigenbasis with
eigenvalues $\lambda_1,\ldots,\lambda_n$.  The max direct-sum law and the
one-dimensional normalization yield
\[
 \kappa_Q^E(S)=\max_j|\lambda_j|=\|S\|_{\rm op}.
\]
Finite-dimensional Hilbert--Schmidt self-duality and operator/trace duality
give \eqref{eq:orthogonal-rigidity-state-gauges}.  Increasing
finite-dimensional exhaustion of a separable $H$ yields
\eqref{eq:orthogonal-principal-completion}.
\end{proof}

\begin{corollary}[Brownian calibration dictionary]
\label{prop:Brownian-response-direct-sum}
\label{cor:calibration-direct-sum-equivalence}
For finite-dimensional real Hilbert spaces $E,F$, the Brownian gauges obey
\begin{align}
 \mathfrak c_A^{E\oplus F}(K\oplus L)^2
 &=\mathfrak c_A^E(K)^2+\mathfrak c_A^F(L)^2,
 \label{eq:Brownian-area-block-law}\\
 \mathfrak c_Q^{E\oplus F}(S\oplus T)
 &=\max\{\mathfrak c_Q^E(S),\mathfrak c_Q^F(T)\},
 \label{eq:Brownian-covariance-block-law}
\end{align}
and, for the standard skew block
$J_0=e_1\otimes e_2-e_2\otimes e_1$ on $\mathbb R^2$,
\begin{equation}\label{eq:Brownian-principal-normalizations}
 \mathfrak c_A^{\mathbb R^2}(J_0)=\sqrt2,
 \qquad
 \mathfrak c_Q^{\mathbb R}(1)=1.
\end{equation}
Consequently the following descriptions of the normalized orthogonally
natural quadratic coefficient geometry are equivalent:
\begin{enumerate}[label=\textup{(\roman*)},leftmargin=2.5em]
\item the coefficients are measured by the Brownian response gauges;
\item area blocks obey the $\ell^2$ law, covariance blocks obey the maximum
law, and the preceding principal normalizations hold;
\item the coefficient norms are $\mathcal S_2$ on the skew channel and the
operator norm on the symmetric channel.
\end{enumerate}
Their dual state completion is
$\mathcal S_2(H)_{\rm sk}\times\mathcal S_1(H)_{\rm sa}$.
\end{corollary}

\begin{proof}
\Cref{thm:Brownian-calibration-rigidity} and the elementary norm block formulas
give the laws and normalizations; \Cref{thm:orthogonal-response-rigidity} gives
the converse, and duality gives the stated completion.
\end{proof}

\begin{corollary}[Schatten classification from orthogonal direct-sum laws]
\label{cor:direct-sum-schatten-scale}
Let $1\le r\le\infty$.  Replace the area square-sum law by the $\ell^r$ block
law and normalize a standard skew two-plane by $2^{1/r}$ for $r<\infty$, or by
$1$ for $r=\infty$.  Then the finite-dimensional coefficient norm is
$\|\cdot\|_{\mathcal S_r}$.  The analogous symmetric law, normalized in one
dimension, also gives $\|\cdot\|_{\mathcal S_r}$.  On finite-rank state
coordinates the dual gauge is the Schatten $r'$ gauge; completing finite ranks
gives $\mathcal S_{r'}$ when $r'<\infty$ and $\mathcal K(H)$ at the endpoint
$r'=\infty$.

For the principal pair $(r_A,r_Q)=(2,\infty)$, the state geometry is therefore
$\mathcal S_2/\mathcal S_1$.  At the coefficient endpoint $r_Q=\infty$, the
finite-rank norm completion is $\mathcal K(H)$, while the full normal test space
may be taken to be $\mathcal L(H)=\mathcal S_1(H)^*$; both unit balls induce the
same trace norm on $\mathcal S_1(H)$.
\end{corollary}

\begin{proof}
Use the real skew spectral decomposition into oriented two-planes and the
ordinary spectral decomposition of self-adjoint operators into one-dimensional
blocks.  Orthogonal invariance, homogeneity, and the prescribed block law give
the $\ell^r$ norm of the singular values.  Finite-dimensional Schatten duality
gives the state gauges, and the endpoint completions are the standard ones.
\end{proof}

\begin{proposition}[Qualitative observability alone does not choose a topology]
\label{prop:qualitative-observability-no-topology}
The finite-rank central restart algebra and point-separation by all scalar linear
probes do not determine a Banach completion.  For every pair
$1\le r_A,r_Q\le\infty$, the same finite-dimensional algebraic skeleton supports
orthogonally natural coefficient geometries with Schatten exponents
$(r_A,r_Q)$.  Their dual state gauges have exponents $(r_A',r_Q')$, with compact
operator completion at an $\infty$ endpoint, and are generally inequivalent.
Nevertheless the complete dual probe families separate the same finite-rank
points and reconstruct the same central product.  Hence the principal
$\mathcal S_2/\mathcal S_1$ split is selected by the quantitative Brownian
response laws, not by qualitative separation alone.
\end{proposition}

\begin{proof}
Apply \Cref{cor:direct-sum-schatten-scale} on each axis.  Linear dual balls
separate every finite-rank point independently of the exponent, while the
restart product was fixed in \Cref{prop:intrinsic-central-extension}.
Inequivalence is witnessed by finite-rank diagonal sequences with different
Schatten asymptotics.
\end{proof}

On a separable Hilbert space $H$, use the Hilbert--Schmidt area coefficient
ball and declare the full normal covariance coefficient ball
\[
 \mathbb K_A:=\{K\in\mathcal S_2(H)_{\rm sk}:\|K\|_2\le1\},
 \qquad
 \mathbb K_Q:=\{S\in\mathcal L(H)_{\rm sa}:\|S\|_{\rm op}\le1\},
\]
respectively.  The displayed norms are selected by the Brownian response
gauges.  The full normal class $\mathcal L(H)_{\rm sa}$ is part of the
principal program: the finite-rank operator-norm completion is
$\mathcal K(H)_{\rm sa}$, while both classes induce the trace norm on
$\mathcal S_1(H)_{\rm sa}$.  Equip the balls with the weak Hilbert--Schmidt and weak-star
$\sigma(\mathcal L(H),\mathcal S_1(H))$ topologies.  Both are compact
metrizable.  Let $\Delta_T=\{(s,t):0\le s\le t\le T\}$ and, for a separable
Banach space $E$, let $\mathscr H_{\Delta_T}^{0,\alpha}(E)$ be the
little-H\"older closure of continuous two-index fields vanishing on the
diagonal.  The ambient profile carrier
\begin{equation}\label{eq:principal-joint-response-space}
 \widetilde{\mathsf R}_{\mathfrak Q_2}^{\eta}
 :=C_0^{0,\eta}([0,T];H)
 \times\mathscr H_{\Delta_T}^{0,2\eta}\!\bigl(C(\mathbb K_A)\bigr)
 \times\mathscr H_{\Delta_T}^{0,2\eta}\!\bigl(C(\mathbb K_Q)\bigr)
\end{equation}
serves as the ambient space for response profiles.  The distinguished response
target is the coherent closed range defined in
\Cref{def:coherent-principal-response-target}.

\begin{definition}[Intrinsic principal quadratic response syntax]
\label{def:principal-intrinsic-quadratic-response}
A \emph{principal quadratic response package} on a stopping-stable response
domain is a total causal joint profile
\[
 \mathbf R^{(2)}=(X,\mathcal R_A,\mathcal R_Q)
 \in\widetilde{\mathsf R}_{\mathfrak Q_2}^{\eta}
\]
such that, for every $(s,t)$, the maps in the coefficient variable are the
restrictions to $\mathbb K_A$ and $\mathbb K_Q$ of weakly, respectively
weak-star, continuous linear functionals on
$\mathcal S_2(H)_{\rm sk}$ and $\mathcal L(H)_{\rm sa}$.  Its restarted
responses satisfy, for $s\le u\le t$,
\begin{align}
 \mathcal R_{A;s,t}(K)
 &=\mathcal R_{A;s,u}(K)+\mathcal R_{A;u,t}(K)
   +\operatorname{Tr}\!\left[K^*\Anti
      (X_{s,u}\otimes X_{u,t})\right],
 \qquad K\in\mathbb K_A,
 \label{eq:principal-raw-area-response}\\
 \mathcal R_{Q;s,t}(S)
 &=\mathcal R_{Q;s,u}(S)+\mathcal R_{Q;u,t}(S),
 \qquad S\in\mathbb K_Q.
 \label{eq:principal-raw-cov-response}
\end{align}
The \emph{intrinsic principal quadratic response syntax}, denoted
$\mathfrak Q_2^{\circ}(H)$, is generated under finite products and
stable-causal postcomposition by the restriction of this package to the
smooth finite-rank coherent source class in
\eqref{eq:coherent-core-response-profile}.  Its canonical stable response
completion $\mathfrak Q_2(H)$ is defined in
\Cref{def:coherent-principal-response-target}.
\end{definition}

The two composition identities determine the central response laws.  The
quantitative stochastic calibration determines the coefficient balls, and the
completion theorem reconstructs the spaces on which their scalar profiles act
continuously.

\begin{proposition}[Response characterization of the split topology]
\label{prop:principal-quadratic-forcing}
Let $a\in\operatorname{Fin}(H)_{\rm sk}$ and
$q\in\operatorname{Fin}(H)_{\rm sa}$.  The gauges induced by the two
coefficient classes satisfy
\begin{equation}\label{eq:principal-forced-norms}
 \sup_{K\in\mathbb K_A}|\operatorname{Tr}(K^*a)|=\|a\|_2,
 \qquad
 \sup_{S\in\mathbb K_Q}|\operatorname{Tr}(Sq)|=\|q\|_1.
\end{equation}
Consequently the Hausdorff completion of the common finite-rank central
skeleton in the topology determined by the principal response syntax is
canonically
\begin{equation}\label{eq:principal-response-completion}
 H\times\mathcal S_2(H)_{\rm sk}\times\mathcal S_1(H)_{\rm sa}.
\end{equation}
The stochastic identification of the two central coordinates is given by
\Cref{prop:semimartingale-verifies-principal-response}.
\end{proposition}

\begin{proof}
Hilbert-space duality on $\mathcal S_2(H)_{\rm sk}$ gives the first identity in
\eqref{eq:principal-forced-norms}.  Trace-class duality gives
$\|q\|_1=\sup_{\|S\|_{\rm op}\le1}|\operatorname{Tr}(Sq)|$; self-adjoint tests
suffice for self-adjoint $q$.  Finite-rank density gives
\eqref{eq:principal-response-completion}.
\end{proof}

\begin{proposition}[Semimartingale realization of the principal syntax]
\label{prop:semimartingale-verifies-principal-response}
Let $P$ be a continuous $H$-valued semimartingale law for which the restarted
second level, area, and bracket have the declared regularity.  Put
$U_r^{P,s}=X_r-X_s$ and define
\begin{align}
 \mathcal R_{A;s,t}^{P}(K)
 &:=\frac12\int_s^t
   \operatorname{Tr}\!\left[
   K^*\bigl(U_r^{P,s}\otimes\dd X_r-\dd X_r\otimes U_r^{P,s}\bigr)
   \right],\\
 \mathcal R_{Q;s,t}^{P}(S)
 &:=\operatorname{Tr}\!\left[S(U_t^{P,s}\otimes U_t^{P,s})\right]
   -\int_s^t\operatorname{Tr}\!\left[
   S\bigl(U_r^{P,s}\otimes\dd X_r+\dd X_r\otimes U_r^{P,s}\bigr)
   \right].
 \label{eq:principal-semimartingale-observers}
\end{align}
Then the resulting joint profile is a probabilistic implementation of
\Cref{def:principal-intrinsic-quadratic-response}: it satisfies
\eqref{eq:principal-raw-area-response}--\eqref{eq:principal-raw-cov-response},
and
\begin{equation}\label{eq:principal-modelwise-identification}
 \mathcal R_{A;s,t}^{P}(K)=\operatorname{Tr}(K^*A^P_{s,t}),
 \qquad
 \mathcal R_{Q;s,t}^{P}(S)=\operatorname{Tr}(S Q^P_{s,t}).
\end{equation}
If the principal operational type is accessible, these profiles admit one
jointly continuous version in the ambient carrier
$\widetilde{\mathsf R}_{\mathfrak Q_2}^{\eta}$.
\end{proposition}

\begin{proof}
The antisymmetric second-level Chen identity gives
\eqref{eq:principal-raw-area-response}.  The bracket is additive, giving
\eqref{eq:principal-raw-cov-response}.  The scalar It\^o product formula gives
\eqref{eq:principal-modelwise-identification}.  Accessibility supplies the
joint little-H\"older versions required by the ambient profile carrier.
\end{proof}

\section{Lossless representations, norming obstructions, and minimality}
\label{sec:response-topology-minimality}

The response-generated analytic topology determines the lossless probe
restrictions and excludes proper lossless stable-causal quotients.

\subsection{Calibrated response completion}
\label{sec:probe-forced-completion}

The cocycle precedes topology: Brownian gauges select the principal completion,
while a general stable probe family may induce another central norm.  We first
construct that completion and then classify when it is principal.

Let
\[
 \mathfrak F_A(H):=\operatorname{Fin}(H)_{\rm sk},
 \qquad
 \mathfrak F_Q(H):=\operatorname{Fin}(H)_{\rm sa}.
\]
A pair $L=(L_A,L_Q)$ of linear probe families on these spaces is called
\emph{admissible} if
\begin{equation}\label{eq:probe-forced-norms}
 \|a\|_{L_A}:=\sup_{\lambda\in L_A}|\lambda(a)|,
 \qquad
 \|q\|_{L_Q}:=\sup_{\mu\in L_Q}|\mu(q)|
\end{equation}
are finite separating norms and there are constants $C_A,C_Q<\infty$ such
that, for all $x,y\in H$,
\begin{equation}\label{eq:probe-rank-one-bounds}
 \|\Anti(x\otimes y)\|_{L_A}\le C_A\|x\|\|y\|,
 \qquad
 \|\Sym(x\otimes y)\|_{L_Q}\le C_Q\|x\|\|y\|.
\end{equation}
Write $E_A[L]$ and $E_Q[L]$ for the Banach completions of
$\mathfrak F_A(H)$ and $\mathfrak F_Q(H)$ in these norms.

\begin{proposition}[Response completion of the central skeleton]
\label{thm:response-forced-completion}
Let $L=(L_A,L_Q)$ be admissible.
\begin{enumerate}[label=\textup{(\roman*)},leftmargin=2.4em]
\item The finite-rank product
\[
 (x,a,q)\star(y,b,r)
 =\bigl(x+y,a+b+\Anti(x\otimes y),q+r\bigr)
\]
extends uniquely to a continuous two-step topological group law on
\begin{equation}\label{eq:probe-completed-group}
 \Gito_L(H):=H\times E_A[L]\times E_Q[L].
\end{equation}
This group is the completion of the common finite-rank central skeleton in the
response norm $\|x\|+\|a\|_{L_A}+\|q\|_{L_Q}$.

\item The evaluation maps
\[
 a\longmapsto(\lambda(a))_{\lambda\in L_A},
 \qquad
 q\longmapsto(\mu(q))_{\mu\in L_Q}
\]
extend isometrically from $E_A[L]$ and $E_Q[L]$ onto the closures of their
finite-rank images in the corresponding supremum-norm function spaces.  Thus
the completed central topology is exactly the initial topology generated by
the bundled probes.

\item Suppose $\mathfrak I(H)$ is a separable rank-one-normalized Banach
operator ideal in which finite-rank operators are dense, with
$\mathcal S_1(H)\hookrightarrow\mathfrak I(H)$, and, on
finite-rank skew and self-adjoint operators, the probe norms satisfy
\begin{equation}\label{eq:probe-identification-bounds}
 c_A^{-1}\|a\|_{\mathfrak I}
 \le\|a\|_{L_A}\le C_A'\|a\|_{\mathfrak I},
 \qquad
 c_Q^{-1}\|q\|_1
 \le\|q\|_{L_Q}\le C_Q'\|q\|_1.
\end{equation}
Then the identity on the finite-rank skeleton extends to a canonical
bi-Lipschitz group isomorphism
\[
 \Gito_L(H)\cong
 H\times\mathfrak I(H)_{\rm sk}\times\mathcal S_1(H)_{\rm sa}.
\]
For the complete dual unit balls the identification is isometric.
\end{enumerate}
\end{proposition}

\begin{proof}
The rank-one bounds make the central cocycle continuous, so density uniquely
extends the product to \eqref{eq:probe-completed-group}; associativity and the
inverse pass from the finite-rank skeleton.  Rational finite-rank tensors over
a countable dense subset of $H$ are dense, proving separability.

For~\textup{(ii)}, \eqref{eq:probe-forced-norms} makes each evaluation map an
isometry; its unique completion has range equal to the closure of the original
image.

Under \eqref{eq:probe-identification-bounds} and the stated finite-rank
density, the finite-rank identity maps extend to bi-Lipschitz bijections
from $E_A[L]$ and $E_Q[L]$ onto $\mathfrak I(H)_{\rm sk}$ and
$\mathcal S_1(H)_{\rm sa}$, respectively.  They preserve the cocycle and
hence the group law.  Complete dual balls give isometries by Hahn--Banach.
\end{proof}

\begin{remark}[Response class and central completion]
\label{rem:what-is-forced}
Stable scalar observations determine the central completion; model-class
accessibility is separate (\Cref{sec:canonicalization}).  Thus responses fix
operator-ideal topology, while capacity enters at realization.
\end{remark}

By \Cref{thm:response-forced-completion}, a Banach operational type is an
intrinsic identification of the probe-induced central completion.  Fix one to
formulate accessibility and rough-path closure in its declared ideal.

Let $G\Omega^2_{\eta,\mathfrak I}(H)$ be the little-H\"older closure of smooth
finite-dimensional step-two signatures for the $\eta$/$2\eta$ H\"older metric
on $H\times\mathfrak I$ (the usual superscript $0$ is suppressed).  Completeness
of the ambient space and rational piecewise-linear paths in a countable dense
finite-dimensional union make it Polish; rank-one normalization makes
step-two multiplication continuous.

\begin{definition}[Coordinate-free operational second-order state]
\label{def:operational-state}
The $(\eta;\mathfrak I,1)$ state space
$\mathsf Z_{\eta;\mathfrak I,1}(H)$ consists of continuous paths starting at the group identity
\[
 g:[0,T]\longrightarrow\Gito_{\mathfrak I,1}(H)
\]
such that the geometric readout $\Theta_H^S(g)$ belongs to
$G\Omega^2_{\eta,\mathfrak I}(H)$ and the covariance path
$q_t:=\Theta_H^Q(g_t)$ belongs to
$C_0^{0,2\eta}([0,T];\mathcal S_{1,\rm sa})$.  For
$g_{s,t}:=g_s^{-1}\star g_t$, write its canonical cocycle chart as
\[
 g_{s,t}=(x_{s,t},A_{s,t},q_{s,t}).
\]
The state metric is
\begin{equation}
\label{eq:operational-state-metric}
 d_{\eta;\mathfrak I,1}(g,\widetilde g)
 =|x-\widetilde x|_{\eta;H}
 +|A-\widetilde A|_{2\eta;\mathfrak I}
 +|q-\widetilde q|_{2\eta;\mathcal S_1}.
\end{equation}
Stopping is group-path stopping.  The resulting space is causal Polish.
\end{definition}

\begin{proposition}[Intrinsic chart and representation equivalence]
\label{prop:intrinsic-chart-equivalence}
The cocycle chart
\[
 \operatorname{ch}_\tau(g):=(x,A,q)
\]
is a causal isometric homeomorphism from
$\mathsf Z_{\eta;\mathfrak I,1}(H)$ onto the triple presentation in which
$(1,x,\frac12x^{\otimes2}+A)\in G\Omega^2_{\eta,\mathfrak I}(H)$,
$q\in C_0^{0,2\eta}(\mathcal S_{1,\rm sa})$, and
\begin{equation}
\label{eq:operational-area-Chen}
 A_{s,t}=A_{s,u}+A_{u,t}
 +\Anti(x_{s,u}\otimes x_{u,t}).
\end{equation}
In this chart,
\[
 \mathbb X^S_{s,t}=\frac12x_{s,t}^{\otimes2}+A_{s,t},
 \qquad
 \mathbb X^I_{s,t}=\mathbb X^S_{s,t}-\frac12q_{s,t},
 \qquad
 \widehat j_t=\mathbb X^I_{0,t}.
\]
Thus $(x,A,q)$ is a canonical global chart of the intrinsic group-valued
state.
\end{proposition}

\begin{proof}
Multiplicativity of $g_{s,t}$ is equivalent to additivity of $q$ and
\eqref{eq:operational-area-Chen}.  The analytic requirements and
the metric are exactly those of the triple presentation.  The readout
identities follow from \Cref{prop:intrinsic-central-extension}.
\end{proof}

\begin{definition}[Common accessibility]
\label{def:common-accessibility}
Call $\tau=(\eta;\mathfrak I,\mathcal S_1)$ an
\emph{accessible operational type} for $\cP$ if there is a single Borel
$\cP$-causal group-valued state
\[
Z_\tau:\Omega_H\longrightarrow
 \mathsf Z_{\eta;\mathfrak I,1}(H)
\]
whose canonical chart satisfies
\begin{equation}\label{eq:common-accessibility-chart}
 \operatorname{ch}_\tau(Z_\tau)
 =\bigl(X,A^P,\qv{M^P}\bigr)
 \qquad P\text{-a.s. for every }P\in\cP,
\end{equation}
where $M^P$ is the continuous local-martingale part of the canonical process
and $A^P$ is its It\^o area.  The same map serves all models.  Part~I is
conditional on this common-realization hypothesis; Part~II proves it for
$\Smax$ by a law-independent raw-causal selector whose lawwise semantics are
verified afterwards.
\end{definition}

The ambient operational space allows the two central channels to range over
the full skew $\mathfrak I$-space and self-adjoint trace class.  Positivity and
monotonicity are properties of accessible stochastic states.  This linear
ambient chart allows the norming classification to test every area and
covariance direction.

\subsection{A lossless operator-valued response}
\label{sec:operational-stochastic}

For an accessible type, first treat the operator-valued truncated It\^o
signature, then scalarize it: the probes observe one dynamics rather than
define the state.

Fix an accessible type $\tau$, $P\in\cP$, and a restart time $s$.  Consider
\begin{equation}
\label{eq:primitive-triangular-system}
 \dd U_r^{P,s}=\dd X_r,
 \qquad
 \dd V_r^{P,s}=U_r^{P,s}\otimes\dd X_r,
 \qquad U_s^{P,s}=V_s^{P,s}=0.
\end{equation}
Equivalently, $(1,U,V)$ is the restarted step-two It\^o signature.  The second
equation is understood weakly: for every
$\lambda\in\mathfrak I(H)^*$, the scalar process $\lambda(V)$ solves the
corresponding real It\^o equation.  Accessibility supplies the jointly
continuous $\mathfrak I$-valued version.

\begin{proposition}[Lossless decoding of the step-two It\^o signature]
\label{prop:signature-lossless-identities}
Under an accessible operational type, the modelwise solution of
\eqref{eq:primitive-triangular-system} is
\[
 U_{s,t}^{P,s}=X_{s,t},
 \qquad
 V_{s,t}^{P,s}=\mathbb X^{I,P}_{s,t}
 =A^P_{s,t}+\frac12
   \bigl(X_{s,t}^{\otimes2}-\qv{M^P}_{s,t}\bigr).
\]
It obeys the lossless identities
\begin{equation}
\label{eq:primitive-lossless-identities}
 A^P_{s,t}=\Anti V_{s,t}^{P,s},
 \qquad
 \qv{M^P}_{s,t}
 =(U_{s,t}^{P,s})^{\otimes2}
  -V_{s,t}^{P,s}-(V_{s,t}^{P,s})^*.
\end{equation}
Hence the standard step-two It\^o signature determines both central channels
of the canonical state.
\end{proposition}

\begin{proof}
The first equation gives $U=X_{s,\cdot}$ and the second gives
$V_{s,t}^{P,s}=\int_s^tX_{s,r}\otimes\dd X_r$ weakly.  Antisymmetrization gives
area, the operator It\^o product formula gives the covariance identity in
\eqref{eq:primitive-lossless-identities}, and accessibility supplies the
jointly continuous $\mathfrak I$-valued version.
\end{proof}

\begin{definition}[Restarted It\^o signature response topology]
\label{def:signature-response}
For $z\in\mathsf Z_{\eta;\mathfrak I,1}(H)$ with cocycle chart $(x,A,q)$,
define its restarted step-two It\^o signature response by
\begin{equation}\label{eq:primitive-observation-map}
 \mathcal J_\tau z=(x,V),\qquad
 V_{s,t}:=A_{s,t}+\frac12\bigl(x_{s,t}^{\otimes2}-q_{s,t}\bigr).
\end{equation}
Let $\mathsf R_\tau^{\rm sig}:=\mathcal J_\tau
(\mathsf Z_{\eta;\mathfrak I,1}(H))$.  On this image set
\begin{align}
 \mathfrak q(x,V)_{s,t}
 &:=x_{s,t}^{\otimes2}-V_{s,t}-V_{s,t}^*,
 \label{eq:primitive-defect}\\
 d_\tau^{\rm sig}\bigl((x,V),(\widetilde x,\widetilde V)\bigr)
 &:=|x-\widetilde x|_{\eta;H}
   +|\Anti(V-\widetilde V)|_{2\eta;\mathfrak I}
   +|\mathfrak q(x,V)-\mathfrak q(\widetilde x,\widetilde V)|_{2\eta;\mathcal S_1}.
 \label{eq:signature-response-metric}
\end{align}
Stopping is inherited from the restarted two-index response.  We call this
the \emph{split signature topology}: it records precisely the skew second
level and the symmetric It\^o defect, rather than forcing one norm on the
whole second level.
\end{definition}

\begin{proposition}[Signature representation of the quadratic state]
\label{prop:primitive-lossless-dynamics}
The map $\mathcal J_\tau$ in \eqref{eq:primitive-observation-map} is a
stable-causal isometric homeomorphism from
$\mathsf Z_{\eta;\mathfrak I,1}(H)$ onto
$\mathsf R_\tau^{\rm sig}$.  Its inverse is the stable-causal decoder
\begin{equation}\label{eq:primitive-decoder}
 \mathcal D_\tau^{\rm sig}(x,V)
 :=\operatorname{ch}_\tau^{-1}
 \bigl(x,\Anti V,\mathfrak q(x,V)\bigr).
\end{equation}
If $\tau$ is accessible, then under every $P\in\cP$ the modelwise response
$\mathcal J_\tau Z_\tau$ is exactly the jointly restarted signature system
\eqref{eq:primitive-triangular-system}.  Thus one standard operator-valued
It\^o dynamics is already lossless for the canonical state.
\end{proposition}

\begin{proof}
The antisymmetric part of \eqref{eq:primitive-observation-map} is $A$, because
$x_{s,t}^{\otimes2}$ and $q_{s,t}$ are self-adjoint.  Moreover
\[
 x_{s,t}^{\otimes2}-V_{s,t}-V_{s,t}^*=q_{s,t}.
\]
Thus \eqref{eq:primitive-decoder} is the two-sided inverse, and isometry follows
from
\eqref{eq:operational-state-metric} and
\eqref{eq:signature-response-metric}.  Both formulas commute with stopping and
are stable-causal.  Under accessibility,
\Cref{prop:signature-lossless-identities} identifies $V$ modelwise with the
restarted It\^o second level.
\end{proof}

\begin{definition}[Coherent principal quadratic response target]
\label{def:coherent-principal-response-target}
Fix the coefficient norms reconstructed in
\Cref{thm:orthogonal-response-rigidity} and use the principal full normal
covariance test class, so that
\[
 \mathbb K_A=\{K\in\mathcal S_2(H)_{\rm sk}:\|K\|_2\le1\},
 \qquad
 \mathbb K_Q=\{S\in\mathcal L(H)_{\rm sa}:\|S\|_{\rm op}\le1\},
\]
and let
\[
 \jmath_A(a)(K):=\operatorname{Tr}(K^*a),
 \qquad
 \jmath_Q(q)(S):=\operatorname{Tr}(Sq)
\]
be the corresponding isometric dual embeddings.

Define the \emph{smooth finite-rank compatible response core}
$\mathsf R_{\mathfrak Q_2}^{\circ}$ as follows.  Let $E\subset H$ be
finite dimensional, let $x:[0,T]\to E$ be piecewise $C^1$ with $x_0=0$, and
put
\begin{equation}\label{eq:coherent-core-area}
 A^x_{s,t}:=\Anti\int_s^t x_{s,r}\otimes \dd x_r.
\end{equation}
Let $q:[0,T]\to\operatorname{Fin}(H)_{\rm sa}$ be piecewise $C^1$ with
$q_0=0$.  Associate to $(x,q)$ the response profile
\begin{equation}\label{eq:coherent-core-response-profile}
 \mathcal O_{\mathfrak Q_2}^{\circ}(x,q)
 :=\left(
 x,
 (\jmath_A A^x_{s,t})_{s,t},
 (\jmath_Q(q_t-q_s))_{s,t}
 \right)
 \in\widetilde{\mathsf R}_{\mathfrak Q_2}^{\eta}.
\end{equation}
Their union over finite-dimensional $E$ is
$\mathsf R_{\mathfrak Q_2}^{\circ}$, whose response-metric completion is the
\emph{coherent principal quadratic response target}
\begin{equation}\label{eq:coherent-principal-response-target}
 \mathsf R_{\mathfrak Q_2}^{\rm coh}
 :=\overline{\mathsf R_{\mathfrak Q_2}^{\circ}}^{\,
       \widetilde{\mathsf R}_{\mathfrak Q_2}^{\eta}}.
\end{equation}
Stopping is the restriction of ambient response stopping.  The
\emph{principal quadratic response specification} $\mathfrak Q_2(H)$ is the
intrinsic syntax $\mathfrak Q_2^{\circ}(H)$ of
\Cref{def:principal-intrinsic-quadratic-response} equipped with this canonical
stable response completion.  In particular, neither
$\mathsf Z_{\eta;\mathcal S_2,1}(H)$ nor an It\^o second-level state is used in
\eqref{eq:coherent-principal-response-target}.
\end{definition}

\begin{proposition}[Coherent response completion and global decoding]
\label{prop:coherent-principal-response-target}
The coherent response space
$\mathsf R_{\mathfrak Q_2}^{\rm coh}$ is causal Polish, closed in the ambient
profile carrier $\widetilde{\mathsf R}_{\mathfrak Q_2}^{\eta}$, and invariant
under deterministic stopping.  The core map
\eqref{eq:coherent-core-response-profile} extends uniquely by completion to a
stable-causal isometric homeomorphism
\begin{equation}\label{eq:coherent-principal-encoder}
 \mathcal O_{\mathfrak Q_2}:
 \mathsf Z_{\eta;\mathcal S_2,1}(H)
 \longrightarrow\mathsf R_{\mathfrak Q_2}^{\rm coh},
 \qquad
 \mathcal O_{\mathfrak Q_2}(x,A,q)
 =\bigl(x,(\jmath_AA_{s,t})_{s,t},(\jmath_Qq_{s,t})_{s,t}\bigr).
\end{equation}
Its inverse is the globally defined stable-causal decoder
\begin{equation}\label{eq:coherent-principal-decoder}
 \mathcal D_{\mathfrak Q_2}(x,\alpha,\chi)
 =\operatorname{ch}_{\tau_0}^{-1}(x,A,q),
 \qquad
 A_{s,t}=\jmath_A^{-1}\alpha_{s,t},\qquad
 q_t=\jmath_Q^{-1}\chi_{0,t},
\end{equation}
where $q_{s,t}=q_t-q_s$ and
$\tau_0=(\eta;\mathcal S_2,\mathcal S_1)$.  Thus
\begin{equation}\label{eq:coherent-principal-two-sided-inverse}
 \mathcal D_{\mathfrak Q_2}\mathcal O_{\mathfrak Q_2}=I,
 \qquad
 \mathcal O_{\mathfrak Q_2}\mathcal D_{\mathfrak Q_2}=I
 \quad\text{on }\mathsf R_{\mathfrak Q_2}^{\rm coh}.
\end{equation}
Thus $\mathcal D_{\mathfrak Q_2}$ is defined on the full declared response
target $\mathsf R_{\mathfrak Q_2}^{\rm coh}$.
\end{proposition}

\begin{proof}
The coefficient balls are weak/weak-star compact metrizable, hence their
$C(K)$ spaces and the little-H\"older factors are separable Banach spaces.
Thus the ambient carrier is Polish and
\eqref{eq:coherent-principal-response-target} is closed Polish.

On the response core, the dual identities
\eqref{eq:principal-forced-norms} give exactly
\[
 \|\jmath_A(A-\widetilde A)\|_{C(\mathbb K_A)}=\|A-\widetilde A\|_2,
 \qquad
 \|\jmath_Q(q-\widetilde q)\|_{C(\mathbb K_Q)}=\|q-\widetilde q\|_1.
\]
Hence the core metric is the pullback of $d_{\eta;\mathcal S_2,1}$.  Smooth
finite-dimensional signatures are dense in the geometric component and
finite-rank piecewise-linear paths in
$C_0^{0,2\eta}([0,T];\mathcal S_{1,\rm sa})$.  Completing therefore identifies
the core isometrically with $\mathsf Z_{\eta;\mathcal S_2,1}(H)$ and extends its
observation map to \eqref{eq:coherent-principal-encoder} with range
$\mathsf R_{\mathfrak Q_2}^{\rm coh}$.

Closedness of the $\jmath_A,\jmath_Q$ ranges and passage of Chen/additivity to
limits make \eqref{eq:coherent-principal-decoder} global and inverse to the
completed isometry.  Stopping preserves the smooth core and hence its closure;
both maps commute with it, proving stable causality and
\eqref{eq:coherent-principal-two-sided-inverse}.
\end{proof}

\begin{remark}[Compatible response core]
The ambient carrier also contains profiles violating Chen compatibility,
covariance additivity, or geometric rough-path closure.  The distinguished
target is the response-metric completion of the compatible testing core.
\end{remark}

\begin{proposition}[Universal factorization of finite-dimensional triangular tests]
\label{prop:finite-dimensional-triangular-factorization}
Let $m\in\mathbb N$, let $C\in\mathcal L(H,\mathbb R^m)$, and let
$B\in\mathcal L(\mathfrak I(H),\mathbb R^m)$.  For an accessible type, the
restarted finite-dimensional triangular system
\begin{equation}\label{eq:finite-dimensional-triangular-test}
 \dd Y_r^{P,s}=C\,\dd X_r+B\bigl(U_r^{P,s}\otimes\dd X_r\bigr),
 \qquad Y_s^{P,s}=0,
\end{equation}
where $U^{P,s}$ solves the first equation in
\eqref{eq:primitive-triangular-system}, has the coordinatewise It\^o solution
\begin{equation}\label{eq:finite-dimensional-triangular-factorization}
 Y_{s,t}^{P,s}=C X_{s,t}+B V_{s,t}^{P,s}.
\end{equation}
Consequently its complete restarted response is a stable-causal
readout of $\mathcal J_\tau Z_\tau$.  Thus every coefficient-defined
finite-dimensional triangular It\^o test whose second-order coefficient
extends continuously to $\mathfrak I(H)$ factors through the single
operator-valued step-two signature response.
\end{proposition}

\begin{proof}
Coordinatewise real It\^o integration gives
\eqref{eq:finite-dimensional-triangular-factorization}, with factor map
$(x,V)\mapsto Cx+BV$.  It is split-topology continuous because
\[
 V=\Anti V+\frac12\bigl(x^{\otimes2}-\mathfrak q(x,V)\bigr),
\]
$\mathcal S_1(H)\hookrightarrow\mathfrak I(H)$ continuously, and
$x\mapsto x^{\otimes2}$ is continuous from the first-level H\"older topology
to the corresponding second-level $\mathfrak I$ topology on bounded sets.
It commutes with stopping and is therefore stable-causal.
\end{proof}

\begin{remark}[Operator-valued one-test representation]
\label{rem:operator-valued-one-test}
All finite-dimensional tests factor through the same operator-valued
signature.  The norming theorem below shows that finitely many scalar readouts
cannot recover the full infinite-dimensional central geometry.  The lossless
one-test object is therefore operator valued, with finite-dimensional systems
appearing as its compressions.
\end{remark}

\begin{corollary}[Necessity of the split topology in infinite dimension]
\label{cor:split-topology-necessary}
Assume $H$ is infinite dimensional and take the principal operational type
$\tau=(\eta;\mathcal S_2,\mathcal S_1)$.  Endow the restarted signature
response only with the unsplit metric
\[
 d_{\rm unsplit}((x,V),(\widetilde x,\widetilde V))
 :=|x-\widetilde x|_{\eta;H}
   +|V-\widetilde V|_{2\eta;\mathcal S_2}.
\]
Then the decoder to $\mathsf Z_{\eta;\mathcal S_2,1}(H)$ is not continuous
at the zero response.  More precisely, there are pure-covariance states
$z_n$ with
\[
 d_{\eta;\mathcal S_2,1}(z_n,0)=1,\qquad
 d_{\rm unsplit}(\mathcal J_\tau z_n,0)=\frac1{2\sqrt n}\longrightarrow0.
\]
Thus the trace-class covariance topology cannot be replaced by the ambient
Hilbert--Schmidt topology of the It\^o second level.
\end{corollary}

\begin{proof}
Choose orthonormal vectors $e_1,e_2,\ldots$ and let $P_n$ be the orthogonal
projection onto $\operatorname{span}\{e_1,\ldots,e_n\}$.  Put
$C_n=n^{-1}P_n$, so $\|C_n\|_1=1$ and $\|C_n\|_2=n^{-1/2}$.  Set
\[
 q_t^{(n)}:=T^{2\eta-1}t\,C_n,\qquad 0\le t\le T,
\]
and let $z_n$ be the state with $x=0$, $A=0$, and covariance path
$q^{(n)}$.  Since $2\eta<1$,
$|q^{(n)}|_{2\eta;\mathcal S_1}=1$, while the geometric readout is the
trivial rough path, so $z_n$ belongs to the state space and has unit state
distance from zero.  Its It\^o second level is
$V_{s,t}^{(n)}=-\frac12q_{s,t}^{(n)}$, hence
\[
 |V^{(n)}|_{2\eta;\mathcal S_2}
 =\frac12\|C_n\|_2=\frac1{2\sqrt n}.
\]
This proves the claim.
\end{proof}

\subsection{General dual probes and modelwise scalarization}

The principal coherent target already uses the complete dual balls selected by
the Brownian response gauges and classified by orthogonal rigidity.  We record
the same scalarization for a general declared operational type
$\tau=(\eta;\mathfrak I,\mathcal S_1)$.  Write
$E_A:=\mathfrak I(H)_{\rm sk}$ and $E_Q:=\mathcal S_1(H)_{\rm sa}$.  For
$E=E_A,E_Q$, let $K_E$ be the closed unit ball of $E^*$ with its weak-star
topology.  Since $E$ is separable, $K_E$ is compact metrizable and
\begin{equation}\label{eq:dual-ball-embedding}
 \jmath_E:E\longrightarrow C(K_E),
 \qquad (\jmath_Ev)(\lambda)=\lambda(v),
\end{equation}
is a linear isometry with closed range.

Under an accessible implementation these coordinates are restarted triangular
It\^o responses.  For
$P\in\cP$, put $U_r^{P,s}=X_r-X_s$.  The first-level member is
\begin{equation}\label{eq:first-level-probe}
 \dd U_r^{P,s}=\dd X_r,\qquad U_s^{P,s}=0,
\end{equation}
and for $\lambda\in K_{E_A}$ and $\mu\in K_{E_Q}$ set
\begin{align}
 R^{A,P}_{s,t}(\lambda)
 &:=\frac12\int_s^t
   \lambda\bigl(U_r^{P,s}\otimes\dd X_r
              -\dd X_r\otimes U_r^{P,s}\bigr),
 \label{eq:area-triangular-probe}\\
 R^{Q,P}_{s,t}(\mu)
 &:=\mu(U_t^{P,s}\otimes U_t^{P,s})
  -\int_s^t
   \mu\bigl(U_r^{P,s}\otimes\dd X_r
          +\dd X_r\otimes U_r^{P,s}\bigr).
 \label{eq:cov-triangular-probe}
\end{align}
If $(x,A,q)$ is the canonical chart of the accessible state, the It\^o product
formula gives simultaneously
\begin{equation}\label{eq:joint-response-identification}
 R^{A,P}_{s,t}=\jmath_{E_A}(A_{s,t}),
 \qquad
 R^{Q,P}_{s,t}=\jmath_{E_Q}(q_{s,t}).
\end{equation}
Thus the probes verify the declared operator-ideal state; for the principal
type, their complete bundle is the response-core completion of
\Cref{def:coherent-principal-response-target}.

\subsection{Norming classification and invisible directions}

The full dual balls are maximally norming; compact subfamilies may also be
norming.  Let $L_A\subset K_{E_A}$ and $L_Q\subset K_{E_Q}$ be nonempty
weak-star compact probe families and set
\begin{equation}\label{eq:probe-seminorms}
 \|a\|_{L_A}:=\sup_{\lambda\in L_A}|\lambda(a)|,
 \qquad
 \|q\|_{L_Q}:=\sup_{\mu\in L_Q}|\mu(q)|.
\end{equation}
For $z\in\mathsf Z_{\eta;\mathfrak I,1}(H)$ with chart $(x,A,q)$ define the
restricted bundled observation
\begin{equation}\label{eq:restricted-joint-observation-map}
 \mathcal O_{\tau,L}(z)
 :=\bigl(x,(\lambda(A_{s,t}))_{\lambda\in L_A,s,t},
           (\mu(q_{s,t}))_{\mu\in L_Q,s,t}\bigr),
\end{equation}
and equip its image with
\begin{align}\label{eq:restricted-response-metric}
 d_L(\mathcal O_{\tau,L}z,\mathcal O_{\tau,L}\widetilde z)
 :=&\ |x-\widetilde x|_{\eta;H}
   +\sup_{\lambda\in L_A}|\lambda(A-\widetilde A)|_{2\eta}\\
   &+\sup_{\mu\in L_Q}|\mu(q-\widetilde q)|_{2\eta}.
\end{align}
The response metric induces on the state set the pseudometric
\begin{equation}
\label{eq:operational-pullback-metric}
 d_{{\rm op},L}(z,\widetilde z)
 :=d_L\bigl(\mathcal O_{\tau,L}z,
             \mathcal O_{\tau,L}\widetilde z\bigr).
\end{equation}
Let $\tau_{{\rm op},L}$ be the initial (coarsest response-continuous) topology.
Since the observation commutes with deterministic stopping, all state
truncations are continuous in $\tau_{{\rm op},L}$.  Write
$\tau_{\eta;\mathfrak I,1}$ for the declared topology induced by
\eqref{eq:operational-state-metric}.

The theorem below characterizes equality with
$\tau_{\eta;\mathfrak I,1}$; central finite-rank needles quantify its failure.

\begin{lemma}[Central needles and exact norm detection]
\label{lem:operational-needles}
Fix $I=[s,s+r]\subset[0,T]$, $r>0$, and let
\[
 \chi_I(u)=
 \begin{cases}
 0,&u\le s,\\
 (u-s)/r,&s<u<s+r,\\
 1,&u\ge s+r.
 \end{cases}
\]
For finite-rank $B^*=-B\in E_A$ and $C^*=C\in E_Q$, define
\[
 \iota_I(B,C):=(0,A^{r;B},q^{r;C}),\qquad
 A_u^{r;B}=r^{2\eta}\chi_I(u)B,\qquad
 q_u^{r;C}=r^{2\eta}\chi_I(u)C.
\]
Then $\iota_I(B,C)\in\mathsf Z_{\eta;\mathfrak I,1}(H)$ and
\begin{align}
 d_{\eta;\mathfrak I,1}(\iota_I(B,C),0)
 &=\|B\|_{\mathfrak I}+\|C\|_{\mathcal S_1},
 \label{eq:needle-state-metric}\\
 d_{{\rm op},L}(\iota_I(B,C),0)
 &=\|B\|_{L_A}+\|C\|_{L_Q}.
 \label{eq:needle-operational-metric}
\end{align}
Thus finite-rank central directions embed the two ideal norms and the two
probe seminorms into the state and response geometries with no loss of
constants.
\end{lemma}

\begin{proof}
Each ramp has normalized $2\eta$-H\"older seminorm equal to its coefficient
norm, proving \eqref{eq:needle-state-metric} and, after scalarization,
\eqref{eq:needle-operational-metric}.  It remains to verify geometric
membership of the pure-area ramp.  On a finite-dimensional subspace containing
$\operatorname{ran}B$, the real skew spectral theorem gives
\[
 B=\sum_{\ell=1}^M b_\ell
 (u_\ell\otimes v_\ell-v_\ell\otimes u_\ell)
\]
with orthogonal oriented two-planes.  Subdivide $I$ into $N$ slots and in each
concatenate $M$ smooth closed planar loops, the $\ell$-th having signed area
$r^{2\eta}b_\ell/N$ in the plane
$\operatorname{span}\{u_\ell,v_\ell\}$.  Since each loop has zero first
increment, Chen cross terms vanish; the second level per slot is
$r^{2\eta}B/N$, so the accumulated center agrees with the ramp at mesh points.

The diameter of each loop is $O_B(r^\eta N^{-1/2})$.  Dividing by the slot
length to the power $\eta$ gives the first-level estimate
\[
 |x^N|_{\eta}\le C_B N^{\eta-1/2}\longrightarrow0.
\]
Within a slot the interpolation error is $O_B(r^{2\eta}/N)$; it vanishes at
mesh points, so splitting arbitrary intervals at neighboring mesh points gives
\[
 |A^N-A^{r;B}|_{2\eta}
 \le C_BN^{2\eta-1}\longrightarrow0.
\]
Since $\eta<1/2$, these smooth finite-dimensional signatures converge to
$(0,A^{r;B})$ in $G\Omega^2_{\eta,\mathfrak I}(H)$; the covariance ramp is an
independent $C_0^{0,2\eta}(\mathcal S_{1,\rm sa})$ coordinate.
\end{proof}

The next theorem characterizes when restricted responses recover the declared
path topology; the central needles give the obstruction.

\begin{theorem}[Norming classification of the operational topology]
\label{thm:operational-topology}
Let $\tau=(\eta;\mathfrak I,\mathcal S_1)$ and let
$L=(L_A,L_Q)$ be as above.  Then
\begin{equation}
\label{eq:operational-topology-inclusion}
 \tau_{{\rm op},L}\subseteq\tau_{\eta;\mathfrak I,1}.
\end{equation}
The following are equivalent.
\begin{enumerate}[label=\textup{(\roman*)},leftmargin=2.4em]
\item There are finite constants $c_A,c_Q$ such that
\begin{equation}
\label{eq:norming-lower-bounds}
 \|a\|_{\mathfrak I}\le c_A\|a\|_{L_A}
 \quad(a\in E_A),
 \qquad
 \|q\|_{\mathcal S_1}\le c_Q\|q\|_{L_Q}
 \quad(q\in E_Q).
\end{equation}
\item $d_{{\rm op},L}$ is a metric and is bi-Lipschitz equivalent to
$d_{\eta;\mathfrak I,1}$.
\item $\tau_{{\rm op},L}=\tau_{\eta;\mathfrak I,1}$.
\item The restricted observation $\mathcal O_{\tau,L}$ is a causal
bi-Lipschitz embedding into the observed response topology.  Its image is
closed and Polish, and its inverse on that image is a stable-causal decoder.
\item There exists a continuous decoder on the observed image satisfying
$\mathcal D_{\tau,L}\mathcal O_{\tau,L}=I$ on the whole operational state
space.
\end{enumerate}
If either lower bound in \eqref{eq:norming-lower-bounds} fails, the inclusion
\eqref{eq:operational-topology-inclusion} is strict: there are finite-rank
operational states $z_n$ such that
\begin{equation}
\label{eq:operational-invisible-needles}
 d_{\eta;\mathfrak I,1}(z_n,0)=1,
 \qquad
 d_{{\rm op},L}(z_n,0)\longrightarrow0.
\end{equation}
For the complete dual balls $L_A=K_{E_A}$ and $L_Q=K_{E_Q}$,
\begin{equation}
\label{eq:complete-operational-topology}
 d_{{\rm op},L}=d_{\eta;\mathfrak I,1}
 \quad\text{and}\quad
 \boxed{\tau_{{\rm op},L}=\tau_{\eta;\mathfrak I,1}}.
\end{equation}
Thus the declared area--covariance topology is exactly the topology generated
by the complete triangular It\^o response bundle.
\end{theorem}

\begin{proof}
The unit-ball bounds $\|a\|_{L_A}\le\|a\|_{\mathfrak I}$ and
$\|q\|_{L_Q}\le\|q\|_1$, applied incrementwise, give
\eqref{eq:operational-topology-inclusion}.  Under
\eqref{eq:norming-lower-bounds},
\[
 d_{\eta;\mathfrak I,1}(z,\widetilde z)
 \le C_L d_{{\rm op},L}(z,\widetilde z),
 \qquad C_L:=\max\{1,c_A,c_Q\},
\]
and the reverse inequality has constant one.  Hence \textup{(i)} implies
\textup{(ii)}--\textup{(iii)}.  The resulting causal bi-Lipschitz embedding
has complete, hence closed, image, and its inverse is continuous and commutes
with stopping.  Thus
\textup{(i)}$\Rightarrow$\textup{(iv)}$\Rightarrow$\textup{(v)}.

Conversely, suppose the area lower bound fails.  Choose norm-one
$a_n'\in E_A$ with $\|a_n'\|_{L_A}\to0$.  Approximation by finite-rank
operators, followed by the contractive skew projection and normalization,
gives finite-rank $a_n\in E_A$ with
$\|a_n\|_{\mathfrak I}=1$ and $\|a_n\|_{L_A}\to0$.  With
$z_n=\iota_{[0,T]}(a_n,0)$, \Cref{lem:operational-needles} gives
\[
 d_{\eta;\mathfrak I,1}(z_n,0)=1,
 \qquad d_{{\rm op},L}(z_n,0)\to0.
\]
If the covariance bound fails, trace-norm finite-rank approximation followed
by contractive self-adjoint projection and normalization gives the same result
with $z_n=\iota_{[0,T]}(0,q_n)$.  Thus failure of \textup{(i)} makes the
inclusion strict and precludes every decoder continuous at zero, hence
\textup{(iii)} and \textup{(v)}.  The preceding implications prove the five-way
equivalence and \eqref{eq:operational-invisible-needles}.  For the complete
dual balls, Hahn--Banach gives both norm identities with constant one and
\eqref{eq:complete-operational-topology}.
\end{proof}

\begin{remark}[Ambient and stochastic ranges]
The necessity statement concerns the full operational state space and requires
no support or reachability hypothesis.  Invertibility on a smaller stochastic
range is a separate property.  The least-state theorem uses the ambient
norming implication.
\end{remark}

\begin{corollary}[No finite-dimensional triangular family is lossless in infinite dimension]
\label{cor:no-finite-scalar-lossless-test}
Assume that at least one central axis $E_A$ or $E_Q$ is infinite dimensional.
Then no finite family of coefficient-defined finite-dimensional tests of the
form \eqref{eq:finite-dimensional-triangular-test} is lossless on the full
operational state space.  More precisely, for any finite collection
$(C_j,B_j)_{j=1}^N$ there is a nonzero finite-rank pure central state whose
complete restarted response is zero for every one of these tests.  Likewise,
no finite product family of scalar area/covariance probes is norming.
Consequently no fixed finite-dimensional scalarization admits a continuous
decoder to the full operational state whenever a central axis is infinite
dimensional.
\end{corollary}

\begin{proof}
Combine the second-order coefficient maps into
$\mathbf B:\mathfrak I(H)\to\mathbb R^M$.  If $E_A$ is infinite dimensional,
its finite-rank skew subspace contains nonzero $A_0\in\ker\mathbf B$; the
pure-area ramp of \Cref{lem:operational-needles} has zero first level and
It\^o second level proportional to $A_0$ on every restart, so all
$C_jX+B_jV$ vanish.  If $E_Q$ is infinite dimensional, restrict $\mathbf B$
to finite-rank self-adjoint trace-class operators and choose nonzero
$Q_0$ in its kernel.  Its pure-covariance ramp has $V=-Q/2$ and is likewise
invisible.

For a finite product family of scalar area/covariance probes, the same
finite-dimensional-kernel argument shows that one of the two norming lower
bounds in \eqref{eq:norming-lower-bounds} fails.  The decoder obstruction then
follows from \Cref{thm:operational-topology}.
\end{proof}

For the principal Hilbert--Schmidt area axis, a concrete invisible sequence is
obtained from an orthonormal family by setting
\[
 B_n=2^{-1/2}
 (e_{2n}\otimes e_{2n+1}-e_{2n+1}\otimes e_{2n})
 \in\mathcal S_2(H)_{\rm sk}.
\]
Then $\|B_n\|_2=1$ while every fixed Hilbert--Schmidt probe tends to zero on
$B_n$, whereas the full dual ball recovers the norm exactly.

\begin{proposition}[Point separation by scalar triangular responses]
\label{prop:scalar-triangular-point-separation}
Let $z=(x,A,q)$ and $\widetilde z=(\widetilde x,\widetilde A,\widetilde q)$ be
distinct elements of
$\mathsf Z_{\eta;\mathcal S_2,1}(H)$.  Then there is a scalar restarted
coefficient-defined triangular It\^o system of the form
\eqref{eq:finite-dimensional-triangular-test} whose complete response separates
$z$ and $\widetilde z$.  More precisely:
\begin{enumerate}[label=\textup{(\roman*)},leftmargin=2.5em]
\item a first-level difference is detected by a rank-one map
$C\in\mathcal L(H,\mathbb R)$ and $B=0$;
\item if the first levels agree but the area fields differ, a finite-rank skew
$K$ and
$B_K(T)=\operatorname{Tr}(K^*\operatorname{Anti}T)$ detect the difference;
\item if the first level and area agree but the covariance fields differ, a
finite-rank self-adjoint $S$ and
$B_S(T)=\operatorname{Tr}(S\operatorname{Sym}T)$ detect the difference.
\end{enumerate}
Thus the scalar triangular response class is point separating, although
\Cref{cor:no-finite-scalar-lossless-test} shows that no fixed finite subfamily
is globally lossless in infinite dimension.
\end{proposition}

\begin{proof}
Choose a restart interval on which the first nonzero coordinate difference
appears.  In case~\textup{(i)}, the Riesz representation theorem gives $C$ with
$C(x_{s,t}-\widetilde x_{s,t})\ne0$.  If the first levels agree, the skew part
of the signature difference is $A_{s,t}-\widetilde A_{s,t}$; finite-rank
Hilbert--Schmidt duality gives $K$ in case~\textup{(ii)}.  If the first level and
area agree, then
\[
 \operatorname{Sym}(V_{s,t}-\widetilde V_{s,t})
 =-\frac12(q_{s,t}-\widetilde q_{s,t}).
\]
A nonzero trace-class self-adjoint operator is separated by a finite-rank
self-adjoint trace functional, giving~\textup{(iii)}.  Each $B_K$ and $B_S$ is
continuous on $\mathcal S_2(H)$, so the corresponding response is an admissible
scalar triangular response.
\end{proof}

\subsection{Least-state representation and compression rigidity}

For a norming pair $L=(L_A,L_Q)$ satisfying
\eqref{eq:norming-lower-bounds}, let $\mathfrak T_{\tau,L}$ denote the joint
marked triangular response specification
\eqref{eq:first-level-probe}--\eqref{eq:cov-triangular-probe}, restricted to
$L_A,L_Q$, with its single closed Polish bundled response space
\[
 \mathsf R_{\tau,L}:=\mathcal O_{\tau,L}
 \bigl(\mathsf Z_{\eta;\mathfrak I,1}(H)\bigr).
\]
Joint realization means one stable-causal readout of the whole bundle.

For the accessible type \(\tau=(\eta;\mathfrak I,\mathcal S_1)\), let
\(Y_L^{P,\tau}\) be the modelwise norming triangular response and let
\(Y_{\rm sig}^{P,\tau}\) be the modelwise restarted step-two It\^o signature.
A source state realizes either response when one stable-causal readout produces
the corresponding response under every \(P\in\cP\).  Accessibility gives
\[
 \mathcal O_{\tau,L}(Z_\tau)=Y_L^{P,\tau},
 \qquad
 \mathcal J_\tau(Z_\tau)=Y_{\rm sig}^{P,\tau}
 \quad P\text{-a.s. for every }P\in\cP.
\]

\begin{theorem}[Least-state theorem for norming quadratic responses]
\label{thm:least-state}
Let \(\tau=(\eta;\mathfrak I,\mathcal S_1)\) be accessible, let \(Z_\tau\) be
its canonical group-valued state, and let \(L=(L_A,L_Q)\) satisfy
\eqref{eq:norming-lower-bounds}.  For every \(\cP\)-causal state
\((\mathsf S,S)\),
\begin{equation}\label{eq:signature-response-representability}
 \begin{aligned}
 S\text{ realizes the restarted step-two signature}
 &\quad\Longleftrightarrow\quad Z_\tau\preceqsc S,\\
 S\text{ jointly realizes }\mathfrak T_{\tau,L}
 &\quad\Longleftrightarrow\quad Z_\tau\preceqsc S.
 \end{aligned}
\end{equation}
If \(\Psi_{\rm sig}\) is a signature realization, then
\(\mathcal D_\tau^{\rm sig}\Psi_{\rm sig}\) is the corresponding factor to
\(Z_\tau\).  If \(\Psi_L\) is a triangular realization, the factor is
\(\mathcal D_{\tau,L}\Psi_L\).  Conversely, a factor \(R:S\to Z_\tau\)
produces the two realizations \(\mathcal J_\tau R\) and
\(\mathcal O_{\tau,L}R\).  Thus the signature and the norming triangular
bundle carry the same realized second-order information, and \(Z_\tau\) is the
unique least sufficient state up to stable-causal equivalence.
\end{theorem}

\begin{proof}
The signature chart
\((\mathcal J_\tau,\mathcal D_\tau^{\rm sig})\) from
\Cref{prop:primitive-lossless-dynamics} and the norming triangular chart
\((\mathcal O_{\tau,L},\mathcal D_{\tau,L})\) from
\Cref{thm:operational-topology} are stable-causal and have left inverses.
The two equivalences and the displayed factor maps follow from
\Cref{prop:lossless-chart-factorization}\textup{(i)}.  Norming is essential:
for a nonnorming family, \Cref{thm:operational-topology} supplies central
directions whose responses vanish while their state norm remains one, so no
continuous decoder exists.
\end{proof}

\begin{corollary}[Equivalence of the principal lossless response presentations]
\label{cor:state-independent-quadratic-universality}
Let \(\tau_0=(\eta;\mathcal S_2,\mathcal S_1)\) be accessible and let
\(L=(L_A,L_Q)\) be norming.  For every \(\cP\)-causal state \(S\),
\begin{equation}\label{eq:principal-quadratic-representation}
 \begin{aligned}
 S\text{ realizes the coherent quadratic response}
 &\Longleftrightarrow Z_{\tau_0}\preceqsc S,\\
 S\text{ realizes the restarted step-two signature}
 &\Longleftrightarrow Z_{\tau_0}\preceqsc S,\\
 S\text{ jointly realizes the norming triangular bundle}
 &\Longleftrightarrow Z_{\tau_0}\preceqsc S.
 \end{aligned}
\end{equation}
For any two of their encoder--decoder pairs
\((\mathcal O_i,\mathcal D_i)\) and \((\mathcal O_j,\mathcal D_j)\), the closed
chart images are related by the stable-causal transports
\(\mathcal O_j\mathcal D_i\) and \(\mathcal O_i\mathcal D_j\).  Hence they
encode the same realized second-order state without choosing a preferred
presentation.
\end{corollary}

\begin{proof}
The signature and triangular equivalences are
\Cref{thm:least-state}.  For the coherent quadratic response,
\Cref{prop:principal-quadratic-forcing} fixes the same split completion,
\Cref{prop:semimartingale-verifies-principal-response} identifies its modelwise
coordinates, and
\Cref{prop:coherent-principal-response-target,prop:lossless-chart-factorization}
gives the same factorization criterion.  The transports are those of
\Cref{cor:canonical-equivalence-lossless-charts}.
\end{proof}

\begin{corollary}[One-test principle for second-order completeness]
\label{cor:one-test-second-order-completeness}
Let $\tau$ be accessible and let $\mathfrak D_\tau$ be a causal response
specification such that every $d\in\mathfrak D_\tau$ admits a modelwise
correct stable-causal solution map
$\Phi_d:\mathsf Z_{\eta;\mathfrak I,1}(H)\to\mathsf R_d$.  Suppose one jointly realized response
$d_\star\in\mathfrak D_\tau$ has a stable-causal decoder
\[
 \mathcal D_\star:\mathsf R_{d_\star}\longrightarrow
 \mathsf R_\tau^{\rm sig}
\]
such that
\[
 (\mathcal D_\star\circ\Phi_{d_\star})(Z_\tau)
 =\mathcal J_\tau(Z_\tau)
 \qquad \cP\text{-quasi surely}.
\]  Then for every causal state $S$,
\[
 S\text{ is }\mathfrak D_\tau\text{-dynamically sufficient}
 \quad\Longleftrightarrow\quad
 Z_\tau\preceqsc S.
\]
Consequently $Z_\tau$ is the unique least $\mathfrak D_\tau$-sufficient
state up to stable-causal equivalence.  Thus a response class need not
literally contain the signature system; it is enough that one of its joint
responses stably determines the standard operator-valued step-two It\^o signature.
\end{corollary}

\begin{proof}
If $S$ is $\mathfrak D_\tau$-sufficient, its readout of $d_\star$ composed
with $\mathcal D_\star$ realizes the step-two It\^o signature.  Hence
$Z_\tau\preceqsc S$ by \Cref{thm:least-state}.  Conversely, if
$Z_\tau\preceqsc S$, composing the factor map $S\to Z_\tau$ with the
solution maps $\Phi_d$ realizes every $d\in\mathfrak D_\tau$.
\end{proof}

\begin{corollary}[Rigidity of lossless stable-causal compression]
\label{cor:no-proper-stable-quotient}
Let $\tau=(\eta;\mathfrak I,\mathcal S_1)$ and let
$L=(L_A,L_Q)$ be norming.  Suppose
\[
 \pi:\mathsf Z_{\eta;\mathfrak I,1}(H)\longrightarrow\mathsf S
\]
is a stable-causal map into a causal Polish space and the whole bundled
response factors through $\pi$: there is a stable-causal
$\Psi:\mathsf S\to\mathsf R_{\tau,L}$ such that
\[
 \mathcal O_{\tau,L}=\Psi\circ\pi.
\]
Then $\pi$ has a stable-causal left inverse.  Consequently $\pi$ is a
homeomorphism onto a closed retract of $\mathsf S$ generated by a
stable-causal idempotent.  In
particular, no noninjective continuous causal quotient is lossless for the
complete norming response; if $\pi$ is surjective, it is a stable-causal
homeomorphism.

If $\tau$ is accessible and the quotient state $\pi(Z_\tau)$ jointly realizes
$\mathfrak T_{\tau,L}$, then
\[
 \pi(Z_\tau)\equivsc Z_\tau.
\]
Thus every lossless stable-causal compression of the canonical state belongs
to its stable-causal equivalence class.
\end{corollary}

\begin{proof}
Apply part~\textup{(iii)} of
\Cref{prop:lossless-chart-factorization} with
$\mathcal O=\mathcal O_{\tau,L}$,
$\mathcal D=\mathcal D_{\tau,L}$, and the displayed factorization
$\mathcal O_{\tau,L}=\Psi\pi$.  The left inverse is
$\mathcal D_{\tau,L}\Psi$, and the image of $\pi$ is the closed fixed-point
set of the stable-causal idempotent
$\pi\mathcal D_{\tau,L}\Psi$.  Surjectivity therefore makes $\pi$ an
isomorphism.  In the accessible case, $\pi$ already gives
$\pi(Z_\tau)\preceqsc Z_\tau$, while joint sufficiency and
\Cref{thm:least-state} give the reverse factorization.
\end{proof}

\begin{remark}[Dependence on the operational topology]
Once accessibility is established, the least-state factorization of the
triangular bundle is deterministic.  The extension to
$\mathfrak D_\tau$ assumes that its solution maps are modelwise correct,
continuous, and stopping compatible.  Both the response norm and the
stable-causal preorder depend on $\mathfrak I$, and joint realization means a
single readout of the complete bundle.
\end{remark}

\begin{proof}[Proof of \Cref{thm:quadratic-response-reconstruction}]
Abstract response descent and completion are
\Cref{thm:descent-before-completion}, while
\Cref{prop:coherent-principal-response-target} gives the displayed
stable-causal isometric identification of the coherent completion with
$\mathsf Z_{\eta;\mathcal S_2,1}(H)$.  Injectivity of the signature and its
stable decoder are \Cref{prop:primitive-lossless-dynamics}; accessibility
identifies the fixed state $\mathcal A_2=Z_{\tau_0}$ modelwise with the common
classical realization.  This proves~\textup{(i)}.
Part~\textup{(ii)} is \Cref{prop:intrinsic-central-extension}, with restart
reconstruction supplied by \Cref{prop:quadratic-restart-reconstruction}.
The forward Brownian calculation, deterministic converse rigidity, and their
calibration dictionary are
\Cref{thm:Brownian-calibration-rigidity,thm:orthogonal-response-rigidity,cor:calibration-direct-sum-equivalence}.
The need for quantitative calibration is
\Cref{prop:qualitative-observability-no-topology}; response duality and path
completion are
\Cref{prop:principal-quadratic-forcing,thm:response-forced-completion,thm:operational-topology}.
The covariance endpoint distinction between the finite-rank completion
$\mathcal K(H)_{\rm sa}$ and the separately declared full normal test class
$\mathcal L(H)_{\rm sa}$ is recorded in
\Cref{cor:direct-sum-schatten-scale}; both induce the trace norm on
$\mathcal S_1(H)_{\rm sa}$.
This proves~\textup{(iii)}.
For~\textup{(iv)},
\Cref{prop:coherent-principal-response-target} supplies the global coherent
encoder--decoder pair, and
\Cref{prop:lossless-chart-factorization} gives the two mutually inverse
factorization maps displayed there.  Accessibility and
\Cref{prop:semimartingale-verifies-principal-response} identify the decoded
state and its response chart with the common classical coordinates and the
declared response package under every model.
Part~\textup{(v)} is
\Cref{prop:scalar-triangular-point-separation,cor:no-finite-scalar-lossless-test,cor:no-proper-stable-quotient}.
For~\textup{(vi)}, any finite product and stable-causal postcomposition of
principal coordinates is a stable-causal map
\[
 F:\mathsf Z_{\eta;\mathcal S_2,1}(H)\longrightarrow\mathsf T
\]
into its declared response target.  By
\Cref{prop:lossless-chart-factorization}~\textup{(ii)},
\[
 F=(F\mathcal D_{\mathfrak Q_2})\mathcal O_{\mathfrak Q_2},
\]
so it factors through the coherent chart.  If this derived chart is lossless,
with $D_FF=I_{\mathsf Z_{\eta;\mathcal S_2,1}(H)}$, apply
\Cref{prop:lossless-chart-factorization}~\textup{(i)} to $(F,D_F)$ and to
$(\mathcal O_{\mathfrak Q_2},\mathcal D_{\mathfrak Q_2})$; realizability of
either lossless chart is equivalent to stable-causal factorization through
the operational state space.  Evaluating the factorization at the fixed common
realization gives the corresponding statement for $\mathcal A_2$.  Apart from
Brownian calibration, the model-class input is accessibility, verified in
Part~II.
\end{proof}

\begin{proof}[Proof of \Cref{thm:quadratic-response-representability}]
Part~\textup{(iv)} of
\Cref{thm:quadratic-response-reconstruction} supplies the assignments
\[
 \Psi_S\longmapsto
 \mathcal D_{\mathfrak Q_2}\Psi_S,
 \qquad
 R\longmapsto\mathcal O_{\mathfrak Q_2}R.
\]
Here $R$ ranges over stable-causal maps
$\mathsf S\to\mathsf Z_{\eta;\mathcal S_2,1}(H)$ satisfying
$R(S)=\mathcal A_2$ quasi surely.  The composite starting from $R$ is the
identity because
$\mathcal D_{\mathfrak Q_2}\mathcal O_{\mathfrak Q_2}
=I_{\mathsf Z_{\eta;\mathcal S_2,1}(H)}$; the composite starting from
$\Psi_S$ agrees with the original realization on the realized source because
\[
 \Psi_S(S)=\mathcal O_{\mathfrak Q_2}(\mathcal A_2)
 \qquad\cP\text{-quasi surely}.
\]
This proves \eqref{eq:quadratic-response-universal-property}.  Leastness is the
forward implication, and $\mathcal A_2$ is jointly sufficient through
$\mathcal O_{\mathfrak Q_2}$.  Mutual least-state factorization gives
uniqueness up to $\equivsc$.
\end{proof}

\part{Common realization and compact perfection}

\section{Common causal realization and construction independence}
\label{sec:canonicalization}

The response state is selected before a law is imposed.  We construct one
total Borel raw-causal realization on the nondominated path space and identify
its classical version under every law.  Pathwise stochastic integration and
quasi-sure aggregation provide related model-independent mechanisms
\cite{Karandikar95,Nutz12,SonerTouziZhang11,Cohen12,Oberhauser16,
PerkowskiPromel16,BartlKupperNeufeld19}, while causal functional calculus gives
a stopped-path viewpoint \cite{ChiuCont22}; here quasi-sure equality is used
only after selecting the common raw field.

The defect-energy principle identifies all total Borel raw-causal
finite-variation schemes with vanishing upper-energy defect; resolvents and
localized kernels are examples.
Restart-stable compact domains then separate the limiting state from
finite-scale memory and support downstream perfection.

\subsection{Abstract causal canonicalization}

The first step is an upper-capacity canonicalization principle independent of
Hilbert and rough-path structure.

\begin{definition}[Selection-gauged causal target]
\label{def:selection-gauged-target}
Let $\Omega$ be a standard Borel space with Borel truncations $r_t$ satisfying
$r_T=I$ and $r_sr_t=r_{s\wedge t}$.  A \emph{realization-gauged causal target}
is a Polish space $(\mathsf Y,\tau_{\rm ref})$ with Borel truncations
$\rho_t$ satisfying the same identities and a complete metric $d_{\rm sel}$
on the underlying set such that:
\begin{enumerate}[label=\textup{(\roman*)},leftmargin=2.4em]
 \item $(y,z)\mapsto d_{\rm sel}(y,z)$ is Borel for
 $\tau_{\rm ref}\otimes\tau_{\rm ref}$;
 \item $d_{\rm sel}(y_n,y)\to0$ implies $y_n\to y$ in $\tau_{\rm ref}$;
 \item $d_{\rm sel}(\rho_ty,\rho_tz)\le d_{\rm sel}(y,z)$ for every $t$.
\end{enumerate}
A Borel map $F:\Omega\to\mathsf Y$ is \emph{raw causal} when
$\rho_tF=\rho_tFr_t$ for every $t$.
\end{definition}

\begin{lemma}[Upper-capacity $L^0$ completeness and causal closedness]
\label{lem:upper-capacity-L0-completeness}
Let $\cP$ be a nonempty family of laws on $\Omega$ and put
$c_{\cP}(E)=\sup_{P\in\cP}P^*(E)$.  Let
$(\mathsf Y,d_{\rm sel})$ be a realization-gauged causal target.  For Borel
sets, ``full-capacity'' will always mean that the complement is polar:
$G$ is $\cP$-full precisely when $c_{\cP}(G^c)=0$.  The weaker equality
$c_{\cP}(G)=1$ is not used as a substitute.  For Borel
maps $F,G:\Omega\to\mathsf Y$, put
\begin{equation}\label{eq:upper-capacity-L0-metric}
 d_{0,\cP}(F,G)
 :=\sup_{P\in\cP}E^P\bigl[1\wedge d_{\rm sel}(F,G)\bigr].
\end{equation}
After quotienting by quasi-sure equality, the resulting metric space of Borel
maps is complete.  Moreover:
\begin{enumerate}[label=\textup{(\roman*)},leftmargin=2.5em]
\item $d_{0,\cP}(F_n,F)\to0$ if and only if $F_n\to F$ in upper capacity for
$d_{\rm sel}$;
\item if $(F_n)$ is $d_{0,\cP}$-Cauchy and every $F_n$ is raw causal, a
deterministic subsequence converges pointwise outside one Borel polar set, and
the limit has a Borel representative $F^{\mathrm c}$ which is
$\cP$-causal on one Borel full-capacity set;
\item the limit is unique quasi surely; if $F_n\to F^P$ in $P$-probability for
every $P\in\cP$, then $F^{\mathrm c}=F^P$ $P$-almost surely.
\end{enumerate}
\end{lemma}

\begin{proof}
For $0<\delta\le1$,
\[
 c_{\cP}(d_{\rm sel}(F,G)>\delta)
 \le \delta^{-1}d_{0,\cP}(F,G),
\]
while for every $0<\varepsilon<1$,
\[
 d_{0,\cP}(F,G)
 \le \varepsilon+c_{\cP}(d_{\rm sel}(F,G)>\varepsilon).
\]
This proves~\textup{(i)}.

Let $(F_n)$ be $d_{0,\cP}$-Cauchy.  Choose a subsequence $(F_{n_k})$ such that
$d_{0,\cP}(F_{n_{k+1}},F_{n_k})\le2^{-3k}$ and put
\begin{equation}\label{eq:canonicalization-tail-event}
 E_k:=\bigcup_{j\ge k}
 \{d_{\rm sel}(F_{n_{j+1}},F_{n_j})>2^{-j}\}.
\end{equation}
Markov's inequality and subadditivity give
$c_{\cP}(E_k)\le\sum_{j\ge k}2^{-2j}$.  Hence the subsequence is summably
Cauchy outside a polar set and converges there in the complete metric
$d_{\rm sel}$.  On $E_k^c$ its limit satisfies
\begin{equation}\label{eq:canonicalization-subsequence-rate}
 d_{\rm sel}(F_{n_k},F^{\mathrm c})
 \le\sum_{j\ge k}2^{-j}=2^{-k+1}.
\end{equation}
The $d_{\rm sel}$-Cauchy locus is Borel; on it convergence in
$d_{\rm sel}$ implies convergence in the reference Polish topology, so the
pointwise limit is Borel.  Extend it by a fixed base point off that locus.
The preceding tail estimate gives upper-capacity convergence of the
subsequence, and the Cauchy property plus the triangle inequality gives
upper-capacity convergence of the full sequence.  Part~\textup{(i)} then gives
$d_{0,\cP}$ convergence and proves completeness.

If the $F_n$ are raw causal and $x,y$ belong to the common Cauchy locus
with $r_tx=r_ty$, then
\[
 \rho_tF_n(x)=\rho_tF_n(r_tx)
 =\rho_tF_n(r_ty)=\rho_tF_n(y).
\]
Contractivity of $\rho_t$ passes this equality to the pointwise limit, proving
$\cP$-causality on the common full-capacity locus.  Uniqueness follows from
\[
 c_{\cP}(d_{\rm sel}(F,G)>\delta)
 \le c_{\cP}(d_{\rm sel}(F,F_n)>\delta/2)
   +c_{\cP}(d_{\rm sel}(F_n,G)>\delta/2).
\]
Modelwise identification is uniqueness of limits in probability.  This proves
\textup{(ii)}--\textup{(iii)}.
\end{proof}

\begin{proposition}[Causal canonicalization principle]
\label{thm:causal-canonicalization-principle}
Let $\cP$ be a nonempty family of probability laws on $\Omega$ and put
$c_{\cP}(E)=\sup_{P\in\cP}P^*(E)$.  Let $F_n^\theta:\Omega\to\mathsf Y$ be
total Borel raw-causal maps into a realization-gauged causal target.  Assume
that, for every $\delta>0$,
\begin{equation}
\label{eq:abstract-upper-capacity-Cauchy}
 \lim_{N\to\infty}\sup_{m,n\ge N}
 c_{\cP}\!\left(d_{\rm sel}(F_m^\theta,F_n^\theta)>\delta\right)=0.
\end{equation}
Then there are a deterministic subsequence $(n_k)$, a Borel $\cP$-polar set
$N$, and a Borel map $F^{\theta,\mathrm c}:\Omega\to\mathsf Y$ such that
\[
 d_{\rm sel}(F_{n_k}^\theta,F^{\theta,\mathrm c})\longrightarrow0
 \qquad\text{on }\Omega\setminus N.
\]
In fact the \emph{full sequence} converges to the same selector in upper
capacity:
\begin{equation}\label{eq:canonicalization-full-capacity-convergence}
 c_{\cP}\left(d_{\rm sel}(F_n^\theta,F^{\theta,\mathrm c})>\delta\right)
 \longrightarrow0
 \qquad(\delta>0).
\end{equation}
The selector is $\cP$-causal and its quasi-sure class is the unique Borel
$\cP$-causal upper-capacity limit of the sequence.  If, for every
$P\in\cP$, the full sequence converges in $P$-probability for $d_{\rm sel}$
to a modelwise primitive $F^P$, then
\[
 F^{\theta,\mathrm c}=F^P\qquad P\text{-a.s.}
\]
Two schemes satisfying \eqref{eq:abstract-upper-capacity-Cauchy} and having
the same modelwise primitives yield the same quasi-sure realization class.  For
a countable family of schemes with the same modelwise primitives, all resulting
selectors agree outside one common Borel polar set.
\end{proposition}

\begin{proof}
By \Cref{lem:upper-capacity-L0-completeness}\textup{(i)},
\eqref{eq:abstract-upper-capacity-Cauchy} is exactly the Cauchy criterion for
$d_{0,\cP}$.  Completeness and causal closedness give a Borel representative
$F^{\theta,\mathrm c}$, a deterministic quasi-surely convergent subsequence,
and convergence of the full sequence in upper capacity.  The same lemma gives
quasi-sure uniqueness and modelwise identification.  Applying uniqueness to
two schemes with the same modelwise primitives proves scheme independence;
a countable union of their Borel mismatch polar sets gives simultaneous
synchronization for a countable family.
\end{proof}

Thus moment rates and summable schedules verify
\eqref{eq:abstract-upper-capacity-Cauchy}, while the modelwise primitive fixes
the resulting quasi-sure state independently of the scheme.

\subsection{The semimartingale model class}

Let $H$ be a nonzero separable real Hilbert space and
\[
 \Omega_H:=C_0([0,T];H).
\]
For the stochastic sections write $\Omega:=\Omega_H$ and let $X$
be the coordinate process with raw filtration $\mathbb F^0$.  Let
$\Smax=\mathfrak S_{\Lambda,B}^{0,T}(H)$ be the nonempty family of laws $P$
for which
\begin{equation}\label{eq:hilbert-decomposition}
 X_t=M_t^P+\int_0^t b_r^P\,\dd r,
 \qquad \|b_r^P\|_H\le B,
\end{equation}
and the continuous local-martingale part has operator bracket
\begin{equation}\label{eq:hilbert-qv-density}
 \dd\qv{M^P}_t=a_t^P\,\dd t,
 \qquad a_t^P\in\mathcal S_1(H)_+,
 \qquad \Tr a_t^P\le\Lambda.
\end{equation}
All inequalities are understood $\dd t\otimes P$-almost everywhere.  Identify
\[
 H\widehat\otimes_2H\simeq\mathcal S_2(H),
 \qquad x\otimes y:h\mapsto\langle y,h\rangle x,
\]
so tensor flip is the Hilbert adjoint.  The inclusion
$\mathcal S_1(H)\hookrightarrow\mathcal S_2(H)$ is continuous.  The
standard Hilbert-space stochastic integral and operator-bracket calculus is
used
\cite{Metivier82,DaPratoZabczyk14}.

For a random variable $Z$ and a continuous process $Y$ set
\[
 \|Z\|_{\mathcal L_{\Smax}^q}:=\sup_{P\in\Smax}\|Z\|_{L^q(P)},
 \qquad
 \|Y\|_{\mathbb S_{\Smax}^q}:=
 \sup_{P\in\Smax}\Big\|\sup_{t\le T}\|Y_t\|\Big\|_{L^q(P)}.
\]
Write $c_{\Smax}(E)=\sup_{P\in\Smax}P^*(E)$.

\begin{definition}[Raw causality]
\label{def:raw-causality}
A map $F:[0,T]\times\Omega_H\to E$ is raw causal if
$F_t(x)=F_t(r_tx)$, where $(r_tx)_u=x_{u\wedge t}$.  The same convention is
used for increment fields.
\end{definition}

\begin{lemma}[Dimension-free first- and second-level moments]
\label{lem:enhanced-moments}
For every $0<\beta<1/2$ there are $q_0$ and $C=C(\beta,\Lambda,B,T)$,
independent of $H$, such that for $q\ge q_0$,
\[
 \sup_{P\in\Smax}\|X\|_{L^q(P;C^\beta(H))}\le C\sqrt q,
 \qquad
 \sup_{P\in\Smax}\|\mathbf X^{S,P}\|_{L^q(P;\mathcal R_\beta)}
 \le C\sqrt q,
\]
where $\mathbf X^{S,P}$ is the classical Hilbert--Schmidt Stratonovich lift
and $\mathcal R_\beta$ denotes the homogeneous step-two radius
\[
 \|\mathbf x\|_{\mathcal R_\beta}
 :=\|x\|_{\beta\text{-H\"ol}}
   +\|\mathbb x\|_{2\beta\text{-H\"ol}}^{1/2}.
\]
\end{lemma}

\begin{proof}
Fix $P\in\Smax$ and put $h=t-s$.  Hilbert-space BDG and
\eqref{eq:hilbert-qv-density}, followed by the deterministic drift bound, give
for $q\ge2$
\begin{equation}\label{eq:enhanced-first-increment}
 \|X_{s,t}\|_{L^q(P;H)}\le C\sqrt q\,h^{1/2}.
\end{equation}
Write the classical Stratonovich second level as
\[
 \mathbb X^{S,P}_{s,t}
 =\int_s^t X_{s,r}\otimes\dd M_r^P
  +\int_s^t X_{s,r}\otimes b_r^P\,\dd r
  +\frac12\qv{M^P}_{s,t}.
\]
For the martingale term use the covariance-weighted identity
\[
 \|(u\mapsto X_{s,r}\otimes u)(a_r^P)^{1/2}\|_
 {\mathcal S_2(H,\mathcal S_2(H))}^2
 =\|X_{s,r}\|_H^2\Tr a_r^P.
\]
Maximal Hilbert-space BDG, \eqref{eq:enhanced-first-increment}, and
$\Tr a_r^P\le\Lambda$ yield
\[
 \left\|\int_s^t X_{s,r}\otimes\dd M_r^P\right\|_{L^q(P;\mathcal S_2)}
 \le C\sqrt{q\Lambda}
      \left(\int_s^t\|X_{s,r}\|_{L^q(P;H)}^2\,\dd r\right)^{1/2}
 \le Cq h.
\]
The drift term is bounded by $CB\sqrt q\,h^{3/2}$ and
$\|\qv{M^P}_{s,t}\|_2\le\Tr\qv{M^P}_{s,t}\le\Lambda h$.  Hence
\begin{equation}\label{eq:enhanced-second-increment}
 \|\mathbb X^{S,P}_{s,t}\|_{L^q(P;\mathcal S_2)}\le Cq h.
\end{equation}

Apply the deterministic chaining estimates in the proof of
\Cref{lem:dyadic-rough-reconstruction} to the multiplicative functional
$\mathbf X^{S,P}$ and the zero functional, with moment exponent $2q$ and
increment exponent $1/2$.  Once $q$ is large enough that
$\beta+1/(2q)<1/2$, \eqref{eq:enhanced-first-increment}--
\eqref{eq:enhanced-second-increment} give
\[
 \bigl\|\|X\|_{\beta\text{-H\"ol}}\bigr\|_{L^{2q}(P)}
 \le C\sqrt q,
 \qquad
 \bigl\|\|\mathbb X^{S,P}\|_{2\beta\text{-H\"ol}}\bigr\|_{L^q(P)}
 \le Cq.
\]
Taking the square root in the second estimate gives the claimed homogeneous
rough-path radius.  Relabeling $2q$ as $q$ and taking the supremum over $P$
completes the proof.
\end{proof}

\subsection{Causal resolvent construction}

For $x\in\Omega_H$ and $0<\eps\le T$ define
\begin{equation}\label{eq:causal-resolvent}
 Y_t^\eps(x)=\frac1\eps\int_0^t e^{-(t-r)/\eps}x_r\,\dd r,
 \qquad D_t^\eps(x)=x_t-Y_t^\eps(x).
\end{equation}
Then $\eps\dot Y^\eps=D^\eps$.  Put
\begin{align}
 J_t^\eps(x)
 &:=Y_t^\eps(x)\otimes x_t-\int_0^t\dd Y_r^\eps(x)\otimes x_r,
 \label{eq:causal-resolvent-tensor}\\
 Q_t^\eps(x)
 &:=x_t^{\otimes2}-J_t^\eps(x)-(J_t^\eps(x))^*,
 \label{eq:finite-resolvent-bracket}\\
 \mathsf E_t^\eps(x)
 &:=\frac2\eps\int_0^t(D_r^\eps(x))^{\otimes2}\,\dd r.
 \label{eq:accumulated-defect-energy}
\end{align}
All three maps are total, continuous in the raw uniform topology, and raw
causal.  Moreover
\begin{equation}\label{eq:resolvent-energy-identity}
 Q_t^\eps=(D_t^\eps)^{\otimes2}+\mathsf E_t^\eps.
\end{equation}
For every deterministic $s$,
\begin{equation}\label{eq:finite-stopping}
 J_t^\eps(r_sx)=J_{t\wedge s}^\eps(x),
 \qquad
 Q_t^\eps(r_sx)=Q_{t\wedge s}^\eps(x).
\end{equation}
For $s\in[0,T]$ set
\begin{equation}\label{eq:raw-shift}
 (\theta_sx)_v=x_{s+(v\wedge(T-s))}-x_s.
\end{equation}
The exact restart identities
\begin{align}
 D_{s+v}^\eps(x)&=D_v^\eps(\theta_sx)+e^{-v/\eps}D_s^\eps(x),
 \label{eq:D-restart}\\
 Y_{s+v}^\eps(x)-x_s&=Y_v^\eps(\theta_sx)-e^{-v/\eps}D_s^\eps(x)
 \label{eq:Y-restart}
\end{align}
hold pathwise.

Finite-variation approximation has classically been studied through
Wong--Zakai limits, regularization procedures, and rough-path corrections
\cite{WongZakai65,RussoVallois07,BerardBergeryVallois11,
FrizOberhauser09,GomesOhashiRussoTeixeira21}.  The defect-energy transfer
principle controls the second-order error by the
upper energy of the first-level regularization defect, uniformly over the
choice of kernel or approximation profile.

\begin{lemma}[Dimension-free tensor error transform]
\label{lem:dimension-free-tensor-error-transform}
Let $P\in\Smax$, let $q\ge2$, and let $Z$ be a predictable continuous
$H$-valued process such that
\[
 \left\|\left(\int_0^T\|Z_r\|_H^2\,\dd r\right)^{1/2}
 \right\|_{L^q(P)}<\infty.
\]
Define the $\mathcal S_2(H)$-valued error transform
\begin{equation}\label{eq:defect-error-transform}
 \mathcal E^P(Z)_t:=\int_0^t Z_r\otimes\dd X_r.
\end{equation}
Then
\begin{equation}\label{eq:defect-error-transform-bound}
 \|\mathcal E^P(Z)\|_{\mathbb S^q(P;\mathcal S_2)}
 \le C\bigl(\sqrt{q\Lambda}+B\sqrt T\bigr)
 \left\|\left(\int_0^T\|Z_r\|_H^2\,\dd r\right)^{1/2}
 \right\|_{L^q(P)},
\end{equation}
where $C$ is numerical and independent of $H,P,q$, and $Z$.
\end{lemma}

\begin{proof}
For $h\in H$ put $\Gamma_r^Zh:=Z_r\otimes h$.  The rank-one tensor identity
gives
\[
 \|\Gamma_r^Z(a_r^P)^{1/2}\|_{\mathcal S_2(H,\mathcal S_2(H))}^2
 =\|Z_r\|_H^2\Tr a_r^P.
\]
Maximal Hilbert-space BDG and $\Tr a_r^P\le\Lambda$ bound the martingale
part by the first term on the right of
\eqref{eq:defect-error-transform-bound}.  For the drift part,
\[
 \sup_{t\le T}\left\|\int_0^tZ_r\otimes b_r^P\,\dd r\right\|_2
 \le B\sqrt T\left(\int_0^T\|Z_r\|_H^2\,\dd r\right)^{1/2}.
\]
Adding the two estimates proves the claim.
\end{proof}

\begin{theorem}[Defect-energy transfer for the Hilbert--Schmidt primitive]
\label{thm:defect-energy-transfer}
Let $Y:\Omega_H\to C([0,T];H)$ be a total Borel raw-causal map whose
sample paths have finite variation, put $D:=X-Y$, and define
\begin{align}
 J_t^Y&:=Y_t\otimes X_t-\int_0^t\dd Y_r\otimes X_r,
 \label{eq:general-regularization-tensor}\\
 Q_t^Y&:=X_t^{\otimes2}-J_t^Y-(J_t^Y)^*.
 \label{eq:general-regularization-defect}
\end{align}
For $q\ge2$ set
\begin{align}
 \mathfrak d_q(Y)
 &:=\sup_{P\in\Smax}
 \left\|\left(\int_0^T\|X_r-Y_r\|_H^2\,\dd r\right)^{1/2}\right\|_{L^q(P)},
 \label{eq:general-defect-energy}\\
 \mathfrak d_q(Y,\widetilde Y)
 &:=\sup_{P\in\Smax}
 \left\|\left(\int_0^T\|Y_r-\widetilde Y_r\|_H^2\,\dd r\right)^{1/2}\right\|_{L^q(P)}.
 \label{eq:pairwise-regularization-energy}
\end{align}
Then $J^Y$ and $Q^Y$ are total Borel continuous-path fields satisfying the
exact stopping identities
\begin{equation}\label{eq:general-primitive-exact-stopping}
 J_t^Y(r_sx)=J_{t\wedge s}^Y(x),
 \qquad
 Q_t^Y(r_sx)=Q_{t\wedge s}^Y(x)
 \qquad(s,t\in[0,T]).
\end{equation}
In particular, they are raw causal as path-valued maps.
Under every $P\in\Smax$, let $\mathcal E^P$ be the transform in
\eqref{eq:defect-error-transform}.  Then, up to indistinguishability,
\begin{align}
 J^Y-J^P&=-\mathcal E^P(D),
 \label{eq:general-defect-J-error}\\
 Q^Y-\qv{M^P}
 &=\mathcal E^P(D)+\mathcal E^P(D)^*.
 \label{eq:general-defect-Q-error}
\end{align}
Moreover,
\begin{equation}\label{eq:general-defect-transfer-bound}
 \sup_{P\in\Smax}\left(
 \|J^Y-J^P\|_{\mathbb S^q(P;\mathcal S_2)}
 +\|Q^Y-\qv{M^P}\|_{\mathbb S^q(P;\mathcal S_2)}\right)
 \le C\bigl(\sqrt{q\Lambda}+B\sqrt T\bigr)\mathfrak d_q(Y),
\end{equation}
and for any two such regularizations,
\begin{equation}\label{eq:pairwise-defect-transfer-bound}
 \sup_{P\in\Smax}\left(
 \|J^Y-J^{\widetilde Y}\|_{\mathbb S^q(P;\mathcal S_2)}
 +\|Q^Y-Q^{\widetilde Y}\|_{\mathbb S^q(P;\mathcal S_2)}\right)
 \le C\bigl(\sqrt{q\Lambda}+B\sqrt T\bigr)
       \mathfrak d_q(Y,\widetilde Y).
\end{equation}
Here $C$ is numerical and independent of $H,P,q$ and the regularizations.
Thus $\mathfrak d_q(\cdot,\cdot)$ is an upper-energy pseudometric on
regularizations modulo zero $\mathfrak d_q$-distance, and the
$C([0,T];\mathcal S_2)^2$ primitive reconstruction is Lipschitz for this
pseudometric in the upper-$L^q$ seminorm
\[
 \sup_{P\in\Smax}
 \|\,\cdot\,\|_{L^q(P;C([0,T];\mathcal S_2)^2)}
\]
before any limiting scheme is chosen.
\end{theorem}

\begin{proof}
On the bounded-variation locus, the indefinite Stieltjes integral in
\eqref{eq:general-regularization-tensor} is the uniform-in-time limit of its
dyadic Riemann-sum paths.  Each Riemann-sum map is Borel with values in
$C([0,T];\mathcal S_2(H))$, so the limit is a Borel continuous-path field.
For $t\le s$, raw causality of $Y$ gives
$J_t^Y(r_sx)=J_t^Y(x)$.  For $t>s$, split the Stieltjes integral at $s$ and
use that $r_sx$ is constant afterwards:
\[
 \begin{aligned}
 J_t^Y(r_sx)
 &=Y_t(r_sx)\otimes x_s
   -\int_0^s\dd Y_r(r_sx)\otimes x_r
   -\bigl(Y_t(r_sx)-Y_s(r_sx)\bigr)\otimes x_s\\
 &=J_s^Y(x).
 \end{aligned}
\]
This proves the first identity in
\eqref{eq:general-primitive-exact-stopping}; the second follows algebraically
from \eqref{eq:general-regularization-defect}.  Thus the primitive pair is a
raw-causal map into the continuous-path target even though the regularizer
$Y$ itself need not obey exact stopping.
Raw causality makes $Y$ adapted under every $P$; its continuous
finite-variation paths are therefore predictable.  Integration by parts gives
\[
 J_t^Y=\int_0^tY_r\otimes\dd X_r.
\]
Subtracting $J^P=\int X\otimes\dd X$ proves
\eqref{eq:general-defect-J-error}.  The operator It\^o product formula gives
\[
 X_t^{\otimes2}=J_t^P+(J_t^P)^*+\qv{M^P}_t,
\]
which, together with \eqref{eq:general-regularization-defect}, proves
\eqref{eq:general-defect-Q-error}.

If $\mathfrak d_q(Y)=\infty$, the one-sided estimate is vacuous; otherwise
apply \Cref{lem:dimension-free-tensor-error-transform} with $Z=D$.
The two
defect identities, together with $\|A+A^*\|_2\le2\|A\|_2$, give
\eqref{eq:general-defect-transfer-bound}.  For two regularizations,
\[
 J^Y-J^{\widetilde Y}=\mathcal E^P(Y-\widetilde Y),
 \qquad
 Q^Y-Q^{\widetilde Y}
 =-\mathcal E^P(Y-\widetilde Y)
  -\mathcal E^P(Y-\widetilde Y)^*.
\]
If $\mathfrak d_q(Y,\widetilde Y)=\infty$, the pairwise estimate is again
vacuous; otherwise a second application of the lemma with
$Z=Y-\widetilde Y$ proves
\eqref{eq:pairwise-defect-transfer-bound}.
\end{proof}

\begin{corollary}[Scheme-independent canonicalization in the primitive topology]
\label{cor:defect-energy-scheme-independence}
Let $(Y^n)_n$ be total Borel raw-causal finite-variation regularizations with
$\mathfrak d_q(Y^n)\to0$ for one $q\ge2$.  Then
$(J^{Y^n},Q^{Y^n})$ is Cauchy in upper capacity in
$C([0,T];\mathcal S_2)^2$, every modelwise limit is
$(J^P,\qv{M^P})$, and every causal canonicalization furnished by
\Cref{thm:causal-canonicalization-principle} agrees quasi surely with the
common primitive $(\widehat J,Q)$ constructed below.  Trace-norm and
$\eta$-H\"older upgrades are treated under separate spatial-tail and modulus
hypotheses.
\end{corollary}

\begin{proof}
For two indices $m,n$,
$\mathfrak d_q(Y^m,Y^n)\le
\mathfrak d_q(Y^m)+\mathfrak d_q(Y^n)$, so
\eqref{eq:pairwise-defect-transfer-bound} gives the upper-capacity Cauchy
criterion.  The one-sided estimate
\eqref{eq:general-defect-transfer-bound} gives the modelwise limit, and the
causal canonicalization principle identifies all resulting selectors with
the same quasi-sure realization class.
\end{proof}

\begin{proposition}[Dimension-free Hilbert resolvent estimate]
\label{prop:fubini-rate}
For every $P\in\Smax$, up to indistinguishability,
\begin{align}
 D_t^\eps&=\int_0^t e^{-(t-r)/\eps}\,\dd X_r,
 \label{eq:resolvent-stochastic-convolution}\\
 J_t^\eps&=\int_0^tY_r^\eps\otimes\dd X_r.
 \label{eq:stochastic-fubini}
\end{align}
Let
$J_t^P:=\int_0^tX_r\otimes\dd X_r\in\mathcal S_2(H)$.  Then
\begin{align}
 J^\eps-J^P&=-\int_0^\cdot D_r^\eps\otimes\dd X_r,
 \label{eq:resolvent-J-error}\\
 Q^\eps-\qv{M^P}
 &=\int_0^\cdot D_r^\eps\otimes\dd X_r
  +\left(\int_0^\cdot D_r^\eps\otimes\dd X_r\right)^*.
 \label{eq:resolvent-Q-error}
\end{align}
For every $q\ge2$,
\begin{equation}\label{eq:tensor-rate}
 \sup_{P\in\Smax}
 \left(\|J^\eps-J^P\|_{\mathbb S^q(P;\mathcal S_2)}
 +\|Q^\eps-\qv{M^P}\|_{\mathbb S^q(P;\mathcal S_2)}\right)
 \le Cq\sqrt\eps,
\end{equation}
where $C=C(\Lambda,B,T)$ is independent of $H,P,q,\eps$.
Furthermore
\begin{align}
 \sup_{t\le T}\|D_t^\eps\|_{\mathcal L_{\Smax}^q}
 &\le C(\sqrt{q\eps}+B\eps),
 \label{eq:D-pointwise-moment}\\
 \sup_{P\in\Smax}
 \left\|\left(\int_0^T\|D_r^\eps\|_H^2\,\dd r\right)^{1/2}\right\|_{L^q(P)}
 &\le C\sqrt{q\eps}.
 \label{eq:D-energy-moment}
\end{align}
For every $0<\beta_0<1/2$,
\begin{align}
 \sup_{P\in\Smax}
 \bigl\|\|D^\eps\|_{\infty;H}\bigr\|_{L^q(P)}
 &\le C_{\beta_0}\sqrt q\,\eps^{\beta_0},
 \label{eq:D-maximal-moment}\\
 \sup_{P\in\Smax}
 \bigl\|\|D^\eps\|_{\infty;H}^2\bigr\|_{L^q(P)}
 &\le C_{\beta_0}q\,\eps^{2\beta_0}.
 \label{eq:D-maximal-square-moment}
\end{align}
\end{proposition}

\begin{proof}
Semimartingale integration by parts applied to the deterministic convolution
in \eqref{eq:causal-resolvent} gives
\[
 D_t^\eps=\int_0^t e^{-(t-r)/\eps}\,\dd X_r,
 \qquad
 J_t^\eps=\int_0^tY_r^\eps\otimes\dd X_r.
\]
This proves \eqref{eq:resolvent-stochastic-convolution}--
\eqref{eq:stochastic-fubini}.

Write $D^\eps=D^{\eps,M}+D^{\eps,b}$ according to
\eqref{eq:hilbert-decomposition}.  Hilbert-space BDG and
\eqref{eq:hilbert-qv-density} yield, uniformly in $t$ and $P$,
\begin{align*}
 \|D_t^{\eps,M}\|_{L^q(P)}
 &\le C\sqrt q\left(
   \int_0^t e^{-2(t-r)/\eps}\Tr a_r^P\,\dd r
   \right)^{1/2}
 \le C\sqrt{q\Lambda\eps},\\
 \|D_t^{\eps,b}\|_H
 &\le B\int_0^te^{-(t-r)/\eps}\,\dd r
 \le B\eps.
\end{align*}
This is \eqref{eq:D-pointwise-moment}.  Minkowski's integral inequality gives
\begin{align*}
 \left\|\left(\int_0^T\|D_r^\eps\|_H^2\,\dd r\right)^{1/2}
 \right\|_{L^q(P)}
 &\le\left(\int_0^T\|D_r^\eps\|_{L^q(P;H)}^2\,\dd r\right)^{1/2}\\
 &\le C_{\Lambda,B,T}\sqrt{q\eps},
\end{align*}
where $B\eps\le B\sqrt{T\eps}$ for $0<\eps\le T$.  This proves
\eqref{eq:D-energy-moment}.

Apply \Cref{thm:defect-energy-transfer} with $Y=Y^\eps$.  The exact identities
\eqref{eq:general-defect-J-error}--\eqref{eq:general-defect-Q-error} become
\eqref{eq:resolvent-J-error}--\eqref{eq:resolvent-Q-error}, and
\eqref{eq:D-energy-moment} gives \eqref{eq:tensor-rate}.

Finally, \Cref{lem:resolvent-boundary-layer} gives pathwise
\[
 \|D^\eps(x)\|_{\infty;H}
 \le C_{\beta_0}\eps^{\beta_0}
      \|x\|_{\beta_0\text{-H\"ol}}.
\]
Together with \Cref{lem:enhanced-moments}, first at exponent $q$ and then at
$2q$, this proves
\eqref{eq:D-maximal-moment}--\eqref{eq:D-maximal-square-moment}.
\end{proof}

\begin{proposition}[Sharpness of the $\sqrt\eps$ scale]
\label{prop:Brownian-rate-sharpness}
If $e\in H$ is a unit vector and $X_t=\sigma W_t e$, then for every $t>0$
\[
 \lim_{\eps\downarrow0}\eps^{-1/2}
 \|Q_t^\eps-\sigma^2t(e\otimes e)\|_{L^2(\mathcal S_2)}
 =\sigma^2\sqrt{2t},
\]
and the primitive error has limit $\sigma^2\sqrt{t/2}$.  Hence the exponent
$1/2$ in \eqref{eq:tensor-rate} is dimension-free and sharp.
\end{proposition}

\begin{proof}
Write $D_t^\eps=d_t^\eps e$.  Then
\[
 \dd d_t^\eps=\sigma\,\dd W_t-\eps^{-1}d_t^\eps\,\dd t,
 \qquad
 E[(d_t^\eps)^2]=\frac{\sigma^2\eps}{2}
                  (1-e^{-2t/\eps}).
\]
The identity
$Q_t^\eps-\sigma^2t(e\otimes e)
 =2\sigma\int_0^td_r^\eps\,\dd W_r\,(e\otimes e)$ gives
\[
 \eps^{-1}E\|Q_t^\eps-\sigma^2t(e\otimes e)\|_2^2
 =2\sigma^4\left(t-\frac\eps2(1-e^{-2t/\eps})\right)
 \longrightarrow2\sigma^4t.
\]
The identity
$J_t^\eps-J_t^P=-\sigma\int_0^td_r^\eps\,\dd W_r\,(e\otimes e)$ and It\^o
isometry give the rescaled variance limit $\sigma^4t/2$.
\end{proof}

\begin{lemma}[Causal Borel limit selector]
\label{lem:causal-Borel-limit}
Let $E$ be a separable Banach space with base point $0$ and define
\[
 \operatorname{Lim}_E:E^{\N}\longrightarrow E,
 \qquad
 \operatorname{Lim}_E((z_n)_n)
 :=\begin{cases}
    \lim_nz_n,&(z_n)_n\text{ converges in }E,\\
    0,&\text{otherwise}.
   \end{cases}
\]
Then $\operatorname{Lim}_E$ is Borel.  Suppose
$F_n:[0,T]\times\Omega_H\to E$ are Borel and
\[
 F_n(t,r_sx)=F_n(t\wedge s,x)
 \qquad(n\ge1).
\]
The coordinatewise selector
$F(t,x):=\operatorname{Lim}_E((F_n(t,x))_n)$ is Borel and satisfies the same
stopping identity.  On every Borel set on which $(F_n)_n$ converges uniformly
in $C([0,T];E)$, $F$ is the resulting continuous path-valued limit.
\end{lemma}

\begin{proof}
The convergence set in $E^{\N}$ is Borel by the Cauchy criterion, and the
limit map is Borel there; extension by the base point proves the first claim.
For fixed $(s,t,x)$, the sequences
$(F_n(t,r_sx))_n$ and $(F_n(t\wedge s,x))_n$ coincide, so applying
$\operatorname{Lim}_E$ gives the stopping identity.  The last assertion is
immediate from uniform convergence.
\end{proof}

\begin{lemma}[Borel currying into continuous-path carriers]
\label{lem:Borel-continuous-section-currying}
Let $D$ be a compact metrizable space, let $E$ be Polish, and let
$F:D\times\Omega_H\to E$ be Borel.  If $G\subseteq\Omega_H$ is Borel and
$d\mapsto F(d,x)$ is continuous for every $x\in G$, then
\[
 x\longmapsto F(\cdot,x):G\longrightarrow C(D;E)
\]
is Borel for the uniform topology.  Consequently its extension by any fixed
continuous base-point path on $G^c$ is a Borel $C(D;E)$-valued map.  The same
statement applies to time-indexed fields with $D=[0,T]$ and increment fields
with $D=\Delta_T$.
\end{lemma}

\begin{proof}
Choose a countable dense subset $D_0\subset D$ and a bounded compatible
complete metric $d_E$ on $E$.  For continuous $f,g:D\to E$,
\[
 \sup_{d\in D}d_E(f(d),g(d))
 =\sup_{d\in D_0}d_E(f(d),g(d)).
\]
Hence the Borel sigma-field of $C(D;E)$ for the uniform topology is generated
by the evaluation maps at points of $D_0$.  Every coordinate
$x\mapsto F(d,x)$, $d\in D_0$, is Borel, and continuity of the sections makes
these coordinates determine the whole path.  This proves the first assertion;
the base-point extension is Borel because $G$ is Borel.
\end{proof}

Fix $\eps_n=T2^{-n}$ and define, for every
$(t,x)\in[0,T]\times\Omega_H$,
\begin{equation}\label{eq:coordinatewise-causal-selector}
 \widehat J_t(x)
 :=\operatorname{Lim}_{\mathcal S_2(H)}
    \bigl((J_t^{\eps_n}(x))_{n\ge1}\bigr).
\end{equation}
Let $G_J$ be the set on which $(J^{\eps_n})$ is uniformly Cauchy in
$C([0,T];\mathcal S_2(H))$.

\begin{lemma}[Exact-stopping primitive selector]
\label{lem:exact-stopping-primitive-selector}
The set $G_J$ is Borel, stopping-stable, and $\Smax$-full.  The field
$\widehat J$ in \eqref{eq:coordinatewise-causal-selector} is total Borel and
raw causal, and it satisfies on the whole raw space
\begin{equation}\label{eq:coordinatewise-selector-stopping}
 \widehat J_t(r_sx)=\widehat J_{t\wedge s}(x)
 \qquad(s,t\in[0,T],\ x\in\Omega_H).
\end{equation}
On $G_J$, $\widehat J$ is the uniform continuous-path limit of
$(J^{\eps_n})$.  Under every $P\in\Smax$ it is indistinguishable from the
classical primitive $J^P$.
\end{lemma}

\begin{proof}
The Cauchy criterion makes $G_J$ Borel.  Estimate
\eqref{eq:tensor-rate} verifies the upper-capacity Cauchy criterion of
\Cref{thm:causal-canonicalization-principle}; Markov's inequality along the
dyadic scale and capacity Borel--Cantelli give
$c_{\Smax}(G_J^c)=0$.  Exact finite-scale stopping gives, for every $s$,
\[
 \sup_{t\le T}
 \|J_t^{\eps_n}(r_sx)-J_t^{\eps_m}(r_sx)\|_2
 \le
 \sup_{u\le T}
 \|J_u^{\eps_n}(x)-J_u^{\eps_m}(x)\|_2.
\]
Hence $G_J$ is stopping-stable.  Borel measurability and
\eqref{eq:coordinatewise-selector-stopping} follow from
\Cref{lem:causal-Borel-limit}; on $G_J$ the coordinatewise selector equals the
uniform limit.  Finally, for fixed $P$, the one-sided estimate
\eqref{eq:tensor-rate} and the same dyadic Borel--Cantelli argument give
uniform almost-sure convergence of $J^{\eps_n}$ to $J^P$.  Continuity of both
paths makes the identification simultaneous in time.
\end{proof}

Define
\begin{align}
 Q_t&:=X_t^{\otimes2}-\widehat J_t-\widehat J_t^*,
 \label{eq:common-Q}\\
 \mathbb X^I_{s,t}&:=\widehat J_t-\widehat J_s-X_s\otimes X_{s,t},
 \label{eq:common-I}\\
 A_{s,t}&:=\Anti\mathbb X^I_{s,t},
 \label{eq:common-area}\\
 \mathbb X^S_{s,t}&:=\mathbb X^I_{s,t}+\frac12Q_{s,t}
 =\frac12X_{s,t}^{\otimes2}+A_{s,t}.
 \label{eq:common-S}
\end{align}
The canonical group-valued state has cocycle chart
\begin{equation}\label{eq:common-second-order-state}
 \operatorname{ch}_\tau(\mathcal A_2)=(X,A,Q).
\end{equation}
The primitive representation $(X,\widehat J,Q)$ is recovered from
\[
 \widehat J_t=\frac12(X_t^{\otimes2}-Q_t)+A_{0,t},
\]
which follows from \eqref{eq:common-Q}--\eqref{eq:common-area}.

\begin{lemma}[Common weak-geometric H\"older locus]
\label{lem:common-weak-geometric-locus}
For $\alpha\in(1/3,1/2)$ define
\begin{equation}\label{eq:common-alpha-rough-locus}
 G_\alpha^{\rm wg}
 :=\left\{x\in G_J:
 \|X(x)\|_{\alpha\text{-H\"ol}}
 +\|\mathbb X^S(x)\|_{2\alpha\text{-H\"ol}}^{1/2}<\infty\right\}.
\end{equation}
Then $G_\alpha^{\rm wg}$ is Borel, invariant under deterministic stopping,
and has full $\Smax$-capacity.  On this set
$\mathbf X^S=(1,X,\mathbb X^S)$ is a weakly geometric
$\alpha$-H\"older rough path in the Hilbert--Schmidt tensor norm.  For every
$\eta\in(1/3,\alpha)$,
\begin{equation}\label{eq:weak-to-geometric-tuning}
 \mathbf X^S(x)\in G\Omega^2_\eta(H;\mathcal S_2),
 \qquad x\in G_\alpha^{\rm wg}.
\end{equation}
Consequently
\begin{equation}\label{eq:common-all-exponent-geometric-locus}
 G_{\rm geo}:=
 \bigcap_{\alpha\in\mathbb Q\cap(1/3,1/2)}G_\alpha^{\rm wg}
\end{equation}
is Borel, stopping-stable, has full capacity, and satisfies
\eqref{eq:weak-to-geometric-tuning} for every $1/3<\eta<1/2$.
\end{lemma}

\begin{proof}
The H\"older seminorms in \eqref{eq:common-alpha-rough-locus} are countable
suprema over rational pairs, so the locus is Borel.  The definitions
\eqref{eq:common-I}--\eqref{eq:common-S} give, for every raw path and every
$s\le u\le t$,
\begin{align*}
 \mathbb X^S_{s,t}
 &=\mathbb X^S_{s,u}+\mathbb X^S_{u,t}
   +X_{s,u}\otimes X_{u,t},\\
 \operatorname{Sym}\mathbb X^S_{s,t}
 &=\frac12X_{s,t}^{\otimes2}.
\end{align*}
Thus finite H\"older radius already implies weak geometricity.  Exact stopping
of $\widehat J$ and $Q$ gives stopping of the whole lift and preserves the
finite-radius condition.

By \Cref{lem:exact-stopping-primitive-selector}, under every fixed
$P\in\Smax$ the selector $\widehat J$ agrees uniformly in time with $J^P$
outside a $P$-null set.  Hence $\mathbf X^S$ agrees there with
the classical Stratonovich lift $\mathbf X^{S,P}$.  The moment estimate
\Cref{lem:enhanced-moments} therefore gives
$P(G_\alpha^{\rm wg})=1$ for every $P$, and the complement is polar.

For Hilbert spaces equipped at level two with the Schatten-$2$ norm,
\cite[Theorem~1.1]{GrongNilssenSchmeding22} states that every weakly geometric
$\alpha$-H\"older rough path belongs to the geometric
$\eta$-H\"older closure for every $1/3<\eta<\alpha$.  This proves
\eqref{eq:weak-to-geometric-tuning}.  The last assertion follows by a
countable intersection and, for a prescribed $\eta<1/2$, choosing a rational
$\alpha\in(\eta,1/2)$.
\end{proof}

\begin{theorem}[Common realization of the Hilbert second-order state]
\label{cor:common-representative}
Each of the coordinate fields
\[
 \widehat J,\qquad Q,\qquad A,\qquad \mathbb X^I,\qquad \mathbb X^S
\]
is total Borel and raw causal, with its natural time or increment index.  On
$G_J$, for every fixed $x$, the algebraic finite-scale fields built from
$J^{\eps_n}(x)$ and $Q^{\eps_n}(x)$ converge to these coordinates uniformly
in time, and uniformly on $\Delta_T$ for the increment fields.  The
compatibility loci and compact cores introduced below upgrade this pathwise
uniform convergence to rough/H\"older convergence and continuity with respect
to the raw path variable.
They obey, for every raw path,
\begin{equation}\label{eq:selector-stopping}
 \widehat J_t(r_sx)=\widehat J_{t\wedge s}(x),\qquad
 Q_t(r_sx)=Q_{t\wedge s}(x).
\end{equation}
Under every $P\in\Smax$ they agree, simultaneously in time, with
the classical semimartingale primitive, operator bracket, It\^o area, and
Stratonovich lift.  There is a Borel stopping-stable full-capacity set $G_*$
such that for every $x\in G_*$ and every $1/3<\eta<1/2$,
\[
 (1,X(x),\mathbb X^S(x))\in G\Omega^2_\eta(H;\mathcal S_2),
 \qquad
 Q(x)\in C_0([0,T];\mathcal S_1(H)_+)
 \text{ and is trace-norm Lipschitz},
\]
with
\begin{equation}\label{eq:Q-trace-Lipschitz}
 \|Q_t-Q_s\|_{\mathcal S_1}\le\Lambda|t-s|.
\end{equation}
In particular $Q\in C_0^{0,\theta}([0,T];\mathcal S_1)$ for every
$0<\theta<1$.  Consequently the operational type
$(\eta;\mathcal S_2,\mathcal S_1)$ is accessible for $\Smax$: on $G_*$ use
$(X,A,Q)$, and on its Borel complement use a fixed base point of
$\mathsf Z_{\eta;\mathcal S_2,1}(H)$.  The resulting total Borel state is
$\Smax$-causal on the common domain and has the same quasi-sure coordinates.
Write $(\widehat J^{\mathrm c},Q^{\mathrm c})$ for the corresponding
continuous-path totalization: it equals $(\widehat J,Q)$ on $G_*$ and the
zero paths on $G_*^c$.  This is the single common map required by
\Cref{def:common-accessibility}; hence the accessibility hypothesis of
\Cref{thm:quadratic-response-reconstruction} holds for $\Smax$.
\end{theorem}

\begin{proof}
Modelwise identification of $\widehat J$ is
\Cref{lem:exact-stopping-primitive-selector}; the
remaining coordinates then agree algebraically with the modelwise bracket,
It\^o area, and Stratonovich lift.  Equation \eqref{eq:selector-stopping} is
\eqref{eq:coordinatewise-selector-stopping}, and the same identity propagates
to all algebraic readouts.  The common geometric rough-path domain is
$G_{\rm geo}$ from \Cref{lem:common-weak-geometric-locus}.

Fix increasing finite-rank orthogonal projections $P_N\uparrow I_H$ and let
$G_Q$ be the set of $x\in G_J$ such that, for every rational $0\le s<t\le T$,
\begin{equation}\label{eq:common-Q-rational-locus}
 Q_{s,t}(x)\in\mathcal S_2(H)_{+},
 \qquad
 \sup_N\Tr(P_NQ_{s,t}(x)P_N)\le\Lambda(t-s).
\end{equation}
Positivity is closed in $\mathcal S_2(H)_{\rm sa}$, and every finite-rank
trace in \eqref{eq:common-Q-rational-locus} is Hilbert--Schmidt continuous;
hence $G_Q$ is Borel.  Modelwise bracket identification gives
$P(G_Q)=1$ for every $P\in\Smax$.

For $x\in G_Q$, continuity of $Q$ in $\mathcal S_2$ extends positivity and
\eqref{eq:common-Q-rational-locus} from rational to arbitrary pairs.  If
$A\ge0$, then
\[
 \Tr A=\sup_N\Tr(P_NAP_N),
\]
with the value $+\infty$ allowed.  Thus every increment $Q_{s,t}(x)$ is
trace class and
\[
 \|Q_t(x)-Q_s(x)\|_1=\Tr Q_{s,t}(x)\le\Lambda(t-s).
\]
This proves the asserted trace-norm Lipschitz property.  Exact stopping then
gives, for $0\le u\le v\le T$ and $s\in[0,T]$,
\[
 Q_{u,v}(r_sx)=Q_{u\wedge s,v\wedge s}(x).
\]
The all-time positivity and trace bounds are therefore preserved by stopping,
and \Cref{lem:exact-stopping-primitive-selector} gives $r_sx\in G_J$.
Consequently $G_Q$ is stopping-stable.

Set $G_*:=G_{\rm geo}\cap G_Q$.  It is Borel, stopping-stable, and has polar
complement.  The fields $(t,x)\mapsto\widehat J_t(x)$ and
$((s,t),x)\mapsto\mathbb X^S_{s,t}(x)$ are jointly Borel.  On $G_J$ their
sections are continuous, so
\Cref{lem:Borel-continuous-section-currying} makes
$x\mapsto(X(x),\mathbb X^S(x),Q(x))$ Borel in the uniform continuous-path
carrier.  For fixed $\eta$, the natural injections
\[
 G\Omega^2_\eta(H;\mathcal S_2)\longrightarrow
 C(\Delta_T;H\oplus\mathcal S_2),
 \qquad
 C_0^{0,2\eta}([0,T];\mathcal S_1(H)_{\rm sa})
 \hookrightarrow C([0,T];\mathcal S_2(H)_{\rm sa})
\]
are continuous injections between Polish spaces.  Lusin--Souslin makes their
images and inverse maps Borel.  Since $G_*$ consists of geometric lifts and
trace-norm Lipschitz covariance paths (hence little $2\eta$-H\"older paths), the
state is Borel there in the declared topology.  Base-point totalization on
$G_*^c$ gives the claimed fields $(\widehat J^{\mathrm c},Q^{\mathrm c})$.
Stopping stability and fullness of $G_*$ make this totalization quasi-surely
causal.  Its values on $G_*^c$ make the fields total; raw-path continuity is
the corewise conclusion of \Cref{thm:spatial-trace-tight-cores}.
\end{proof}

\begin{corollary}[Quantitative scheme-independent convergence to the common primitive]
\label{cor:defect-energy-common-capacity}
Let $Y^n$ be any sequence of total Borel raw-causal finite-variation
regularizations as in \Cref{thm:defect-energy-transfer}, and fix $q\ge2$.
For the Borel continuous-path totalization
$(\widehat J^{\mathrm c},Q^{\mathrm c})$ from
\Cref{cor:common-representative}, put
\[
 \Delta_n
 :=\|J^{Y^n}-\widehat J^{\mathrm c}\|_{\infty;\mathcal S_2}
   +\|Q^{Y^n}-Q^{\mathrm c}\|_{\infty;\mathcal S_2}.
\]
Then, for every $\delta>0$,
\begin{equation}\label{eq:defect-energy-common-capacity-rate}
 c_{\Smax}(\Delta_n>\delta)
 \le
 \left[
 \frac{C(\sqrt{q\Lambda}+B\sqrt T)
       \mathfrak d_q(Y^n)}{\delta}
 \right]^q.
\end{equation}
Consequently $\mathfrak d_q(Y^n)\to0$ implies that the full sequence
$(J^{Y^n},Q^{Y^n})$ converges in upper capacity to $(\widehat J,Q)$ in the
product uniform $C([0,T];\mathcal S_2)^2$ topology.  Convergence in
$\tau_{\eta;\mathcal S_2,\mathcal S_1}$ is obtained under the separate
spatial-tail and modulus hypotheses.
\end{corollary}

\begin{proof}
Under each $P\in\Smax$, \Cref{cor:common-representative} identifies
$(\widehat J^{\mathrm c},Q^{\mathrm c})$ with
$(J^P,\qv{M^P})$ up to indistinguishability.  Hence
the $L^q(P)$ estimate of \Cref{thm:defect-energy-transfer} applies directly
to $\Delta_n$.  Markov's inequality and the supremum over $P$ give
\eqref{eq:defect-energy-common-capacity-rate}.
\end{proof}

\begin{remark}[Corewise continuity and deterministic locality obstructions]
\label{rem:locality-obstruction-capacity-cores}
Chevyrev and Ferrucci prove that, for H\"older exponents at most $1/2$, a
local homogeneous rough-path lift cannot be defined on the entire deterministic
H\"older path space, and they obtain related boundedness obstructions for
L\'evy-area assignments \cite{ChevyrevFerrucci26}.  The common selector here is
total Borel and raw causal on the raw path space and continuous on
model-supported compact capacity cores; the rough-path identities hold on one
full-capacity domain.  Its continuity regime is therefore the capacity-core
topology rather than the full deterministic H\"older path space.
\end{remark}

\begin{theorem}[Least quadratic state for nondominated Hilbert semimartingales]
\label{thm:nondominated-quadratic-reconstruction}
Let $\Smax$ be the nondominated Hilbert semimartingale model class of
\Cref{sec:canonicalization}.  For $1/3<\eta<1/2$, put
$\tau_0=(\eta;\mathcal S_2,\mathcal S_1)$.  The quadratic response
specification $\mathfrak Q_2(H)$ admits the total $\Smax$-causal
reconstruction
\[
 \mathcal A_2:\Omega_H\longrightarrow
 \mathsf Z_{\eta;\mathcal S_2,1}(H),
 \qquad
 \operatorname{ch}_{\tau_0}(\mathcal A_2)=(X,A,Q),
\]
without an additional accessibility hypothesis.  The total quasi-sure state
constructed in Part~II is the unique least sufficient state for
$\mathfrak Q_2(H)$ on $\Smax$: a causal source realizes the complete quadratic
program if and only if it admits a stable-causal readout of $\mathcal A_2$.
Its area and covariance coordinates have the
$\mathcal S_2/\mathcal S_1$ state geometry selected by Brownian calibration
and the corresponding orthogonal direct-sum laws.  Every stable-causal derived
representation of the complete quadratic program factors through this state.
\end{theorem}

\begin{proof}
The common representative theorem \Cref{cor:common-representative} supplies
accessibility.  Apply
\Cref{thm:quadratic-response-reconstruction,thm:quadratic-response-representability}.
\end{proof}

\begin{corollary}[Representation equivalence of the common state]
\label{prop:state-universal-property}
On $G_*$ the presentations
\[
 (X,A,Q),\qquad (X,\widehat J,Q),\qquad
 (X,\mathbb X^I,\mathbb X^S)
\]
are related by continuous algebraic readouts in their natural Hilbert--Schmidt
and trace-class topologies.  In particular
\[
 \widehat J_t=\frac12(X_t^{\otimes2}-Q_t)+A_{0,t},
 \qquad
 Q_t=2(\mathbb X^S_{0,t}-\mathbb X^I_{0,t}).
\]
\end{corollary}

\begin{proof}
This is the algebra in \eqref{eq:common-I}--\eqref{eq:common-S} at $s=0$.
\end{proof}

\subsection{Linear naturality and projective consistency}

The intrinsic central extension is functorial under bounded linear Hilbert
maps in the Hilbert--Schmidt/trace-class type, and the stochastic selector
respects this functoriality.  Finite-dimensional compressions converge in the
declared state topology.

For separable real Hilbert spaces $H,K$ and $L\in\mathcal L(H,K)$, write
$\ell_Lx=Lx$ and define
\begin{equation}
\label{eq:linear-state-pushforward}
 L_\sharp(x,A,q)
 :=\bigl(Lx,(LA_{s,t}L^*)_{s,t},(Lq_tL^*)_t\bigr).
\end{equation}

\begin{lemma}[Filtration reduction of bounded Hilbert characteristics]
\label{lem:filtration-reduction-characteristics}
Let $Y$ be a continuous $K$-valued process which is adapted to a filtration
$\mathbb G=(\mathcal G_t)$ contained in a filtration
$\mathbb F=(\mathcal F_t)$.  Suppose that in $\mathbb F$
\[
 Y_t=N_t+\int_0^t c_r\,\dd r,
 \qquad \|c_r\|_K\le C,
\]
where $N$ is a continuous local martingale, and suppose its pathwise operator
quadratic variation is absolutely continuous,
\[
 [Y]_t=\int_0^t\alpha_r\,\dd r,
 \qquad
 \alpha_r\in\mathcal S_1(K)_+,
 \qquad \Tr\alpha_r\le\Lambda'
\]
for $\dd r\otimes P$-almost every $(r,\omega)$.  Then, in the usual
augmentation of $\mathbb G$, $Y$ is a continuous special semimartingale
\[
 Y_t=\widetilde N_t+\int_0^t\widetilde c_r\,\dd r
\]
with $\widetilde N$ a continuous local martingale,
$\|\widetilde c_r\|_K\le C$, and
\[
 \dd[\widetilde N]_t=\widetilde\alpha_t\,\dd t,
 \qquad
 \widetilde\alpha_t\in\mathcal S_1(K)_+,
 \qquad
 \Tr\widetilde\alpha_t\le\Lambda'.
\]
\end{lemma}

\begin{proof}
Stricker's theorem makes the $\mathbb G$-adapted $Y$ a continuous special
$\mathbb G$-semimartingale \cite{Stricker77,FoellmerProtter11}.  The dual
predictable projection of $A_t:=\int_0^t c_r\,\dd r$ has density
$\widetilde c=E[c\mid\mathcal P(\mathbb G)]$ under $\dd t\otimes P$; this is
the Bochner form of the standard predictable-projection construction
\cite[Chapter~3]{Metivier82}; see also \cite{Protter05}.  Uniqueness of the special
decomposition identifies $\int\widetilde c\,\dd r$ as the new
finite-variation characteristic, and conditional Jensen gives
$\|\widetilde c\|_K\le C$ almost everywhere.

Quadratic variation is pathwise, so $[\widetilde N]=[Y]$, a continuous
$\mathbb G$-predictable increasing path.  For $n\ge1$ set
\[
 \alpha_t^{(n)}
 :=n\bigl([Y]_t-[Y]_{(t-1/n)^+}\bigr).
\]
These predictable positive trace-class processes have trace at most
$\Lambda'$.  Banach-valued Lebesgue differentiation gives convergence in
$L^1_{\rm loc}(\dd t\otimes P;\mathcal S_1)$ to the derivative of $[Y]$;
an a.e. convergent subsequence defines predictable $\widetilde\alpha$.
Closedness preserves positivity and the trace bound.
\end{proof}

\begin{proposition}[Linear naturality and projective consistency]
\label{thm:linear-naturality}
Let $H,K$ be separable real Hilbert spaces and $L\in\mathcal L(H,K)$.
Then $L_\sharp$ is a continuous causal group-path map from
$\mathsf Z_{\eta;\mathcal S_2,1}(H)$ to
$\mathsf Z_{\eta;\mathcal S_2,1}(K)$ and commutes with the primitive,
It\^o, geometric, and stopping readouts.

If $P\in\mathfrak S_{\Lambda,B}^{0,T}(H)$ and
$P_L=(\ell_L)_\#P$, then
\[
 P_L\in\mathfrak S_{\|L\|^2\Lambda,\,\|L\|B}^{0,T}(K)
\]
and, writing $\mathcal A_2^H,\mathcal A_2^K$ for the two canonical selectors,
\begin{equation}
\label{eq:canonical-linear-naturality}
 \mathcal A_2^K(\ell_Lx)=L_\sharp\mathcal A_2^H(x)
\end{equation}
on the common state domain of $\mathcal A_2^H$, hence $P$-almost surely.

Let $(\Pi_n)$ be finite-rank orthogonal projections on $H$ with
$\Pi_n\to I_H$ strongly, and let $\mathfrak I(H)$ be any operator ideal
satisfying the hypotheses of \Cref{def:operational-state}.  Then, for every
$z\in\mathsf Z_{\eta;\mathfrak I,1}(H)$,
\begin{equation}
\label{eq:projective-state-convergence}
 d_{\eta;\mathfrak I,1}((\Pi_n)_\sharp z,z)\longrightarrow0.
\end{equation}
Consequently, on the common domain of the Hilbert--Schmidt realization
constructed here,
\[
 \mathcal A_2^H(\Pi_nx)=(\Pi_n)_\sharp\mathcal A_2^H(x)
 \longrightarrow\mathcal A_2^H(x)
\]
in the Hilbert--Schmidt/trace-class state topology.
\end{proposition}

\begin{proof}
The identity
\[
 L\,\Anti(x\otimes y)L^*=\Anti(Lx\otimes Ly)
\]
shows that $L_\sharp$ is a homomorphism of the intrinsic central extensions.
The ideal inequalities
\[
 \|LAL^*\|_{\mathcal S_2(K)}\le\|L\|^2\|A\|_{\mathcal S_2(H)},
 \qquad
 \|LqL^*\|_{\mathcal S_1(K)}\le\|L\|^2\|q\|_{\mathcal S_1(H)}
\]
and $\|Lx\|_K\le\|L\|\|x\|_H$ give continuity.  Smooth signatures are mapped to smooth signatures, so geometric closure is
preserved.  All algebraic
readouts and stopping commute with $L_\sharp$.

In the original filtration under $P$, the image process $LX$ is a continuous
special semimartingale with finite-variation density $Lb^P$ and pathwise
quadratic variation
\[
 [LX]_t=\int_0^tLa_r^PL^*\,\dd r,
 \qquad
 \Tr(La_r^PL^*)\le\|L\|^2\Tr a_r^P.
\]
Apply \Cref{lem:filtration-reduction-characteristics} with
$C=\|L\|B$ and $\Lambda'=\|L\|^2\Lambda$ to the usual augmented
canonical filtration generated by $LX$.  It follows that the transformed law
belongs to
$\mathfrak S_{\|L\|^2\Lambda,\,\|L\|B}^{0,T}(K)$.

The causal resolvent and its tensor commute pathwise with $L$:
\[
 Y^{\eps,K}(\ell_Lx)=LY^{\eps,H}(x),
 \quad
 J^{\eps,K}(\ell_Lx)=LJ^{\eps,H}(x)L^*,
 \quad
 Q^{\eps,K}(\ell_Lx)=LQ^{\eps,H}(x)L^*.
\]
Passing to the defining limits proves
\eqref{eq:canonical-linear-naturality}.

For the last assertion, the ideal property makes each
$(\Pi_n)_\sharp$ contractive in the declared state metric.  Smooth
finite-dimensional signatures together with smooth finite-rank
$\mathcal S_1$-valued paths are dense in
$\mathsf Z_{\eta;\mathfrak I,1}(H)$.  On every such finite-dimensional state,
strong convergence $\Pi_n\to I_H$ is uniform on the finitely many spatial and
operator directions involved and hence gives convergence in the state
metric.  Contractivity and a three-term approximation argument prove
\eqref{eq:projective-state-convergence}; the canonical-state statement
follows from \eqref{eq:canonical-linear-naturality} with $L=\Pi_n$.
\end{proof}

\begin{corollary}[Projective observability without finite truncation]
\label{cor:projective-observability}
Assume $H$ is infinite dimensional and let $(\Pi_n)$ be finite-rank orthogonal
projections with $\Pi_n\to I_H$ strongly.  Then the map
\[
 z\longmapsto\bigl((\Pi_n)_\sharp z\bigr)_{n\ge1}
\]
is injective on $\mathsf Z_{\eta;\mathfrak I,1}(H)$, and every state is
recovered by the state-topology limit
\[
 z=\lim_{n\to\infty}(\Pi_n)_\sharp z.
\]
No single finite-rank compression $(\Pi_n)_\sharp$ is injective on the full
operational state space, hence no fixed finite truncation is lossless.  On the
common stochastic domain the same assertions hold for the canonical state
$\mathcal A_2$.
\end{corollary}

\begin{proof}
The convergence is \eqref{eq:projective-state-convergence}.  If two states
have the same complete compression sequence, taking the state-topology limit
gives equality, proving injectivity of the projective observation.  For a
fixed finite-rank $\Pi_n$, choose a unit vector $e\in\ker\Pi_n$ and a nonzero
pure-covariance ramp with coefficient $e\otimes e$.  It is a nonzero
operational state but its $\Pi_n$-compression is zero.  The stochastic
statement follows from \eqref{eq:canonical-linear-naturality}.
\end{proof}

\subsection{Construction independence under localized kernels}

Localized causal convolutions give the same state as the exponential
resolvent: their estimates merely verify the defect-energy hypothesis, after
which canonicalization is scheme-independent.  This differs logically from
kernelwise regularization and rough-approximation results
\cite{RussoVallois07,BerardBergeryVallois11,FrizOberhauser09,
GomesOhashiRussoTeixeira21}.

For $0<\eps\le T$, let
$\mathsf K_\eps\in C^1([0,T];\mathcal L(H))$ have absolutely continuous
first derivative, satisfy $\mathsf K_\eps(0)=I_H$, and define
\[
 \ell_\eps^2:=\int_0^T\|\mathsf K_\eps(u)\|_{\rm op}^2\,\dd u.
\]
For $x\in\Omega$ put
\begin{align}
 \Delta_t^{\eps,\mathsf K}(x)
 &:=x_t+\int_0^t\mathsf K_\eps'(t-r)x_r\,\dd r,
 \label{eq:localized-kernel-defect}\\
 Y_t^{\eps,\mathsf K}(x)&:=x_t-\Delta_t^{\eps,\mathsf K}(x),\\
 J_t^{\eps,\mathsf K}(x)
 &:=Y_t^{\eps,\mathsf K}(x)\otimes x_t
   -\int_0^t\dd Y_r^{\eps,\mathsf K}(x)\otimes x_r,
 \label{eq:localized-kernel-tensor}\\
 Q_t^{\eps,\mathsf K}(x)
 &:=x_t^{\otimes2}-J_t^{\eps,\mathsf K}(x)
   -J_t^{\eps,\mathsf K}(x)^*.
\end{align}
The Stieltjes integral is pathwise well defined because
$Y^{\eps,\mathsf K}$ is absolutely continuous.

\begin{corollary}[Localized-kernel universality of the common state]
\label{thm:localized-kernel-universality}
There is a constant $C=C(\Lambda,B,T)$ such that, for every $q\ge2$ and
$P\in\Smax$,
\begin{align}
 \sup_{t\le T}\|\Delta_t^{\eps,\mathsf K}\|_{L^q(P)}
 &\le C\sqrt q\,\ell_\eps,
 \label{eq:localized-defect-rate}\\
 \|J^{\eps,\mathsf K}-J^P\|_{\mathbb S^q(P;\mathcal S_2)}
 +\|Q^{\eps,\mathsf K}-\qv{M^P}\|_{\mathbb S^q(P;\mathcal S_2)}
 &\le Cq\,\ell_\eps,
 \label{eq:localized-tensor-rate}
\end{align}
where $J^P=\int_0^\cdot X_r\otimes\dd X_r$.
If $\ell_\eps\le L\sqrt\eps$ and $\eps_n=T2^{-n}$, define the dyadic
selectors coordinatewise by \Cref{lem:causal-Borel-limit}, with
$J^{\eps_n,\mathsf K}$ in place of $J^{\eps_n}$.  These coordinate fields are
total Borel and raw causal and satisfy
\[
 \widehat J^{\mathsf K}=\widehat J,
 \qquad Q^{\mathsf K}=Q
 \qquad c_{\Smax}\text{-quasi surely},
\]
simultaneously in time.  Any countable family of such kernels can be
synchronized outside one Borel polar set.

If, for some $0<\delta<1$,
\[
 \sup_{0<t\le T}t^\delta\|\mathsf K_\eps(t)\|_{\rm op}
 +\int_0^T u^\delta\|\mathsf K_\eps'(u)\|_{\rm op}\,\dd u
 \le L_\delta\eps^\delta,
\]
then on every raw $\delta$-H\"older ball
$\{\|x\|_{\delta\text{-H\"ol}}\le R\}$,
\[
 \|\Delta^{\eps,\mathsf K}(x)\|_\infty
 \le L_\delta R\eps^\delta.
\]
\end{corollary}

\begin{proof}
Deterministic integration by parts gives
\[
 \Delta_t^{\eps,\mathsf K}(x)
 =\mathsf K_\eps(t)x_t
  +\int_0^t\mathsf K_\eps'(u)(x_{t-u}-x_t)\,\dd u,
\]
which proves raw continuity and causality.  Under $P$, semimartingale
integration by parts yields
\[
 \Delta_t^{\eps,\mathsf K}
 =\int_0^t\mathsf K_\eps(t-r)\,\dd X_r,
 \qquad
 J_t^{\eps,\mathsf K}
 =\int_0^tY_r^{\eps,\mathsf K}\otimes\dd X_r.
\]
For every $t\le T$, Hilbert-space BDG and
\eqref{eq:hilbert-qv-density} give
\[
 \|\Delta_t^{\eps,\mathsf K,M}\|_{L^q(P)}
 \le C\sqrt{q\Lambda}\,\ell_\eps,
\]
while Cauchy--Schwarz gives
\[
 \|\Delta_t^{\eps,\mathsf K,b}\|_H
 \le B\sqrt T\,\ell_\eps.
\]
This proves \eqref{eq:localized-defect-rate}.  Minkowski's inequality gives
\begin{equation}\label{eq:localized-defect-energy}
 \sup_{P\in\Smax}
 \left\|\left(\int_0^T
       \|\Delta_r^{\eps,\mathsf K}\|_H^2\,\dd r
       \right)^{1/2}\right\|_{L^q(P)}
 \le C\sqrt q\,\ell_\eps.
\end{equation}
Apply \Cref{thm:defect-energy-transfer} with
$Y=Y^{\eps,\mathsf K}$ and $D=\Delta^{\eps,\mathsf K}$.  The exact error
identities and \eqref{eq:localized-defect-energy} give
\eqref{eq:localized-tensor-rate}.

For the dyadic schedule, let
$a_n=\|J^{\eps_{n+1},\mathsf K}-J^{\eps_n,\mathsf K}\|_\infty$.
By \eqref{eq:localized-tensor-rate} and
$\ell_{\eps}\le L\sqrt\eps$, for any fixed $q_*>2$,
\[
 \sup_{P\in\Smax}\|a_n\|_{L^{q_*}(P)}
 \le C_{q_*,L}(\sqrt{\eps_n}+\sqrt{\eps_{n+1}}).
\]
For $r_n=2^{-n\theta}$ with $0<\theta<1/2-1/q_*$, Markov and capacity
Borel--Cantelli give $a_n\le r_n$ eventually off one polar set; hence uniform
convergence.  \Cref{lem:causal-Borel-limit} gives total raw-causal fields.
Modelwise $L^q$ convergence identifies their primitive with $J^P$, while
\Cref{thm:causal-canonicalization-principle,cor:defect-energy-scheme-independence}
identifies it with $\widehat J$ quasi surely; $Q$ follows algebraically, and
the same principle synchronizes countably many kernels.  The displayed
integration-by-parts formula gives the final weighted H\"older bound.
\end{proof}

\subsection{Finite-scale memory is not a state invariant}
\label{sec:restart-memory}

Construction independence concerns the limiting state.  The finite-dimensional
restart memory carried by a particular regularization is an implementation
feature and may vary with the approximation scheme.

For a scalar causal convolution kernel $k$ write
\[
 D_t^k(x)=k(t)x_0+\int_0^t k(t-r)\,\dd x_r
\]
formally at the semimartingale level, with the equivalent integration-by-parts
pathwise realization used above.  Say that $k$ has an exact $m$-mode
(time-homogeneous linear) restart if its future convolution state can be
encoded by an $m$-dimensional linear state satisfying a constant-coefficient
ODE between input increments.  Equivalently, its impulse response is a
matrix-exponential coefficient, in the standard finite-dimensional linear
realization sense \cite{HoKalman66,Brockett70,Kailath80}.

\begin{proposition}[State--memory separation]
\label{prop:state-memory-separation}
There are two normalized localized kernel families, both satisfying the
hypotheses of \Cref{thm:localized-kernel-universality} and therefore selecting
the same reconstructed state $\mathcal A_2$, with different exact restart
complexities:
\begin{enumerate}[label=\textup{(\roman*)},leftmargin=2.4em]
\item the exponential kernel $k_\eps(t)=e^{-t/\eps}$ has a one-mode exact
restart;
\item any nonzero $C^2$ compactly supported profile
$k_\eps(t)=\phi(t/\eps)$ has no finite-dimensional exact time-homogeneous
linear restart.
\end{enumerate}
Thus finite-scale restart memory is not an invariant of the response-
reconstructed state.
\end{proposition}

\begin{proof}
The exponential defect is governed by one scalar linear mode, so \textup{(i)}
is immediate.  A finite-dimensional time-homogeneous linear realization has
an impulse response of the form $c^*e^{tA}b$ and is therefore real analytic
for $t>0$.  A nonzero compactly supported profile vanishes on an open tail and
cannot have this form, proving \textup{(ii)}.  Both kernel families have
$L^2$ radius $O(\sqrt\eps)$, so
\Cref{thm:localized-kernel-universality} identifies their quasi-sure limit
with the same $\mathcal A_2$.
\end{proof}

\begin{corollary}[Rigidity of one-mode homogeneous restart]
\label{cor:one-mode-restart-rigidity}
Let $k:[0,\infty)\to\mathbb R$ be a normalized scalar impulse response,
$k(0)=1$.  If its convolution state admits an exact one-mode
time-homogeneous linear restart, then
\[
 k(t)=e^{at}
\]
for some $a\in\mathbb R$.  If the mode is contractive and dissipative, then
$k(t)=e^{-\lambda t}$ for some $\lambda\ge0$.  If its generator is strictly
dissipative---equivalently, if the one-mode semigroup is asymptotically
stable---then $\lambda>0$.  Thus the normalized contractive one-mode kernels
are exactly the exponentials $e^{-\lambda t}$ with $\lambda\ge0$, and the
strictly dissipative ones are those with $\lambda>0$.
\end{corollary}

\begin{proof}
A one-mode time-homogeneous linear realization has scalar state equation
$\dot d=ad$ between input increments.  Its impulse response is therefore
$c e^{at}b$.  The normalization $k(0)=1$ gives $cb=1$ and hence
$k(t)=e^{at}$.  Contractivity and dissipation give $a\le0$; strict
dissipation, or equivalently asymptotic stability in one dimension, gives
$a<0$.
\end{proof}

Thus incoming memory closes the finite-scale response but is propagated and
dissipated, rather than retained in the scheme-independent limiting state; cf.
\Cref{rem:separation-of-reductions}.

\subsubsection{Exponential memory and temporal naturality}

For the exponential one-mode realization, the incoming defect closes restart
exactly and exposes the seam mechanism.

\begin{definition}[Resolvent with incoming memory]
\label{def:augmented-resolvent-state}
For an incoming memory $d\in H$ and a fresh increment path
$\eta\in\Omega$, define, for $0\le v\le T$,
\begin{align}
 \mathscr D_v^\eps(d,\eta)
 &:=e^{-v/\eps}d+D_v^\eps(\eta),
 \label{eq:augmented-defect}\\
 \mathscr Y_v^\eps(d,\eta)
 &:=\eta_v-\mathscr D_v^\eps(d,\eta)
   =Y_v^\eps(\eta)-e^{-v/\eps}d.
 \label{eq:augmented-smoothed-path}
\end{align}
The pair $(\mathscr Y^\eps,\mathscr D^\eps)$ is the resolvent
state with incoming memory.  Its initial values are
$\mathscr D_0^\eps=d$ and $\mathscr Y_0^\eps=-d$.
\end{definition}

\begin{proposition}[Exact restart law with memory]
\label{prop:augmented-restart-cocycle}
Let $u,v\ge0$ with $u+v\le T$.  Then
\begin{align}
 \mathscr D_{u+v}^\eps(d,\eta)
 &=\mathscr D_v^\eps\bigl(
     \mathscr D_u^\eps(d,\eta),\theta_u\eta
   \bigr),
 \label{eq:augmented-defect-cocycle}\\
 \mathscr Y_{u+v}^\eps(d,\eta)-\eta_u
 &=\mathscr Y_v^\eps\bigl(
     \mathscr D_u^\eps(d,\eta),\theta_u\eta
   \bigr).
 \label{eq:augmented-Y-cocycle}
\end{align}
In particular, for every raw path $x$, seam $s$, and
$0\le v\le T-s$,
\begin{align}
 D_{s+v}^\eps(x)
 &=\mathscr D_v^\eps\bigl(D_s^\eps(x),\theta_sx\bigr),
 \label{eq:augmented-original-D}\\
 Y_{s+v}^\eps(x)-x_s
 &=\mathscr Y_v^\eps\bigl(D_s^\eps(x),\theta_sx\bigr).
 \label{eq:augmented-original-Y}
\end{align}
Thus the boundary layer obtained by restarting with zero memory is exactly
the discarded term $e^{-v/\eps}D_s^\eps(x)$.
\end{proposition}

\begin{proof}
The usual defect restart identity gives
$D_{u+v}^\eps(\eta)=D_v^\eps(\theta_u\eta)
+e^{-v/\eps}D_u^\eps(\eta)$.  Adding
$e^{-(u+v)/\eps}d$ proves
\eqref{eq:augmented-defect-cocycle}.  Since
$\mathscr Y=\eta-\mathscr D$, subtracting $\eta_u$ gives
\eqref{eq:augmented-Y-cocycle}.  The last two identities are the same
calculation with incoming memory $D_s^\eps(x)$ and future increment
$\theta_sx$.
\end{proof}

\begin{example}[Necessity of the incoming defect]
\label{ex:discarded-incoming-defect}
Take the deterministic path $x_t=t$.  Its resolvent is
\[
 Y_t^\eps=t-\eps(1-e^{-t/\eps}),
 \qquad
 D_t^\eps=\eps(1-e^{-t/\eps}).
\]
Stop the input at time $s$ and keep it constant afterwards.  The exact
post-$s$ resolvent is
\[
 Y_{s+v}^\eps(r_sx)
 =x_s-e^{-v/\eps}D_s^\eps(x),
 \qquad v\ge0.
\]
A restart from the fresh zero increment with zero incoming memory gives the
constant value $x_s$.  The discrepancy is
$-e^{-v/\eps}D_s^\eps(x)$, the boundary layer carried by the incoming defect.
\end{example}

The defect closes finite-scale restart and disappears in the limit.

For a path $\xi$ beginning at zero, set
\begin{equation}\label{eq:resolvent-A-operator}
 \mathcal A_v^\eps(\xi)
 :=e^{-v/\eps}\xi_v
   +\frac1\eps\int_0^ve^{-u/\eps}\xi_u\,\dd u.
\end{equation}
This is the integration-by-parts realization of
$\int_0^ve^{-u/\eps}\dd\xi_u$ and is defined for every continuous $\xi$.

\begin{lemma}[Exponential boundary layers]
\label{lem:resolvent-boundary-layer}
Let $0<\alpha<\beta<1$, and suppose
$\norm{x}_{\beta\text{-H\"ol}}\le R$.  Then
\begin{align}
 \norm{Y^\eps(x)}_{\beta\text{-H\"ol}}&\le R,
 &\norm{D^\eps(x)}_\infty&\le C_\beta R\eps^\beta,
 \notag\\
 \norm{D^\eps(x)}_{\alpha\text{-H\"ol}}
 &\le C_{\alpha,\beta}R\eps^{\beta-\alpha},
 &\sup_{v\le T}|\mathcal A_v^\eps(x)|
 &\le C_\beta R\eps^\beta.
 \label{eq:deterministic-resolvent-bounds}
\end{align}
Let
\[
 B_t=\one_{\{t>a\}}(1-e^{-(t-a)/\eps})u,
 \qquad |u|\le U\eps^\beta,
\]
and let $W$ be a bounded-variation path with
$\norm{W}_{\beta\text{-H\"ol}}\le M$.  If
$1/3<\alpha<\beta<1/2$, then
\begin{equation}\label{eq:exponential-layer-rough}
 \rho_\alpha(S_2(W+B),S_2(W))
 \le C U(1+M+U)\eps^{\beta-\alpha}.
\end{equation}
For a fixed finite sum of layers, the same estimate holds with the total
amplitude and a constant depending on their number.
\end{lemma}

\begin{proof}
With $k_\eps(a)=\eps^{-1}e^{-a/\eps}\one_{\{a\ge0\}}$ and the zero
extension $\bar x$,
\[
 Y_t^\eps(x)=\int_0^\infty k_\eps(a)\bar x_{t-a}\,\dd a.
\]
Translation invariance of the H\"older seminorm gives the first bound, and
\[
 |D_t^\eps(x)|
 =\left|\int_0^\infty
  k_\eps(a)(x_t-\bar x_{t-a})\,\dd a\right|
 \le R\int_0^\infty k_\eps(a)a^\beta\,\dd a
 \le C_\beta R\eps^\beta.
\]
Since $\norm{D^\eps}_{\beta\text{-H\"ol}}\le2R$, H\"older interpolation
gives the third bound.  The last follows from
\eqref{eq:resolvent-A-operator} and $|x_u|\le Ru^\beta$.

On an interval $[p,q]$ of length $h$, the variation of $B$ is at most
$|u|\min\{1,h/\eps\}$.  Expanding the two signatures and integrating one
cross term by parts gives
\[
 |\mathbb S_2(W+B)_{p,q}-\mathbb S_2(W)_{p,q}|
 \le C\bigl(Mh^\beta V_{p,q}+V_{p,q}^2\bigr).
\]
For $h\le\eps$, use $V_{p,q}\le U\eps^{\beta-1}h$; for
$h\ge\eps$, use $V_{p,q}\le U\eps^\beta$.  Dividing by
$h^{2\alpha}$ in the two regimes yields
$C(MU+U^2)\eps^{2(\beta-\alpha)}$.  The first-level quotient is bounded
by $CU\eps^{\beta-\alpha}$.  This proves
\eqref{eq:exponential-layer-rough}; finite sums follow with their total
amplitude, including the quadratic cross terms.
\end{proof}

\begin{lemma}[Uniform resolvent seam]
\label{lem:resolvent-seam}
Let $1/3<\alpha<\beta<1/2$ and
$\norm{x}_{\beta\text{-H\"ol}}\le R$.  For $0\le v\le T-s$, put
\[
 \mathbb I_{s,s+v}^\eps(x)
 :=J_{s+v}^\eps(x)-J_s^\eps(x)-x_s\otimes x_{s,s+v}.
\]
Then
\begin{equation}\label{eq:resolvent-J-seam}
 \mathbb I_{s,s+v}^\eps(x)-J_v^\eps(\theta_sx)
 =-D_s^\eps(x)\otimes\mathcal A_v^\eps(\theta_sx),
\end{equation}
and the right-hand side is bounded by $CR^2\eps^{2\beta}$, uniformly in
$(s,v)$.  Moreover,
\begin{equation}\label{eq:resolvent-seam-rate}
 \sup_{s\in[0,T]}
 \rho_\alpha\!\left(
  S_2(Y_{s+\cdot}^\eps(x)-Y_s^\eps(x)),
  S_2(Y^\eps(\theta_sx))
 \right)
 \le CR(1+R)\eps^{\beta-\alpha},
\end{equation}
where the distance is taken on $[0,T-s]$.
\end{lemma}

\begin{proof}
The restart identity gives
\begin{equation}\label{eq:Y-increment-restart}
 Y_{s+v}^\eps(x)-Y_s^\eps(x)
 =Y_v^\eps(\theta_sx)+(1-e^{-v/\eps})D_s^\eps(x).
\end{equation}
The correction is one exponential boundary layer, so
\Cref{lem:resolvent-boundary-layer} proves
\eqref{eq:resolvent-seam-rate}.  Interpreting
\eqref{eq:causal-resolvent-tensor} as the integration-by-parts integral
$\int Y^\eps\otimes\dd x$, subtract the restarted integral over
$[s,s+v]$.  Formula \eqref{eq:Y-restart} leaves
\[
 -D_s^\eps\otimes\int_0^ve^{-u/\eps}\,\dd(\theta_sx)_u,
\]
which is \eqref{eq:resolvent-J-seam}.  The bound follows from
\eqref{eq:deterministic-resolvent-bounds}.
\end{proof}

For a fresh future path $y\in\Omega$, define
\[
 (\lambda_sy)_t:=
 \begin{cases}
  0,&t\le s,\\
  y_{t-s},&t>s,
 \end{cases}
 \qquad
 x\otimes_sy:=r_sx+\lambda_sy.
\]
Only $y|_{[0,T-s]}$ is used.

\begin{lemma}[Stopping, shift, and concatenation formulas]
\label{lem:resolvent-operations}
Let $L=T-s$.  For $v\in[0,T]$,
\begin{align}
 &Y_v^\eps(\theta_sx)
 -\bigl(Y_{s+(v\wedge L)}^\eps(x)-Y_s^\eps(x)\bigr)
 \notag\\
 &\qquad
 =-(1-e^{-v/\eps})D_s^\eps(x)
 +\one_{\{v\ge L\}}(1-e^{-(v-L)/\eps})D_T^\eps(x).
 \label{eq:fixed-horizon-resolvent-shift}
\end{align}
Furthermore,
\begin{equation}\label{eq:resolvent-stopping-layer}
 Y_t^\eps(r_sx)-(r_sY^\eps(x))_t
 =\one_{\{t>s\}}(1-e^{-(t-s)/\eps})D_s^\eps(x),
\end{equation}
and the delayed path satisfies
\begin{equation}\label{eq:resolvent-delay-covariance}
 (Y^\eps,J^\eps,Q^\eps,\mathsf E^\eps)_t(\lambda_sy)
 =\begin{cases}
  (0,0,0,0),&t\le s,\\
  (Y^\eps,J^\eps,Q^\eps,\mathsf E^\eps)_{t-s}(y),&t>s.
 \end{cases}
\end{equation}
If $z=x\otimes_sy$, then, for $t\ge s$,
\begin{align}
 J_t^\eps(z)
 &=J_t^\eps(r_sx)+J_t^\eps(\lambda_sy)
   +x_s\otimes y_{t-s}
   -D_s^\eps(x)\otimes\mathcal A_{t-s}^\eps(y),
 \label{eq:finite-resolvent-J-concatenation}\\
 Q_t^\eps(z)
 &=Q_t^\eps(r_sx)+Q_t^\eps(\lambda_sy)
 \notag\\
 &\quad
  +D_s^\eps(x)\otimes\mathcal A_{t-s}^\eps(y)
  +\mathcal A_{t-s}^\eps(y)\otimes D_s^\eps(x).
 \label{eq:finite-resolvent-Q-concatenation}
\end{align}
Finally, if
\[
 \widetilde Y_t^\eps:=
 \begin{cases}
  Y_t^\eps(x),&t\le s,\\
  Y_s^\eps(x)+Y_{t-s}^\eps(y),&t>s,
 \end{cases}
\]
then
\begin{equation}\label{eq:resolvent-concatenation-layer}
 Y_t^\eps(z)-\widetilde Y_t^\eps
 =\one_{\{t>s\}}(1-e^{-(t-s)/\eps})D_s^\eps(x).
\end{equation}
\end{lemma}

\begin{proof}
For $v\le L$, \eqref{eq:fixed-horizon-resolvent-shift} is
\eqref{eq:Y-increment-restart} with the sides reversed.  After $L$, the
shifted path is stopped and its defect decays exponentially; its value at
$L$ is $D_T^\eps-e^{-L/\eps}D_s^\eps$, which gives the terminal layer.
The same calculation proves \eqref{eq:resolvent-stopping-layer}.  Changing
variables in the convolution and in
\eqref{eq:causal-resolvent-tensor} proves
\eqref{eq:resolvent-delay-covariance}.

For concatenation, use linearity of $Y^\eps$ and expand
\eqref{eq:causal-resolvent-tensor}.  The only mixed term is the past defect
paired with the future increment, and
\[
 \int_s^te^{-(r-s)/\eps}\,\dd y_{r-s}
 =\mathcal A_{t-s}^\eps(y).
\]
This gives \eqref{eq:finite-resolvent-J-concatenation}; substituting it in
\eqref{eq:finite-resolvent-bracket} gives
\eqref{eq:finite-resolvent-Q-concatenation}.  The final identity follows
from linearity and the stopped-resolvent formula.
\end{proof}

\begin{lemma}[Trace-energy seam bounds]
\label{lem:trace-energy-seams}
Let $1/3<\alpha<\beta<1/2$.  If
$\|x\|_{\beta\text{-H\"ol}}\le R$ and
$\|y\|_{\beta\text{-H\"ol}}\le S$, then, uniformly in the seam $s$ and
$0<\eps\le T$,
\begin{align}
 \|\mathsf E^\eps(r_sx)-r_s\mathsf E^\eps(x)\|_{\infty;\mathcal S_1}
 &\le C R^2\eps^{2\beta},
 \label{eq:energy-stopping-seam}\\
 \|\mathsf E^\eps(\theta_sx)-\theta_s\mathsf E^\eps(x)\|_{\infty;\mathcal S_1}
 &\le C R^2\eps^{2\beta},
 \label{eq:energy-shift-seam}\\
 \sup_{t\ge s}\left\|
  \mathsf E_t^\eps(x\otimes_sy)-\mathsf E_s^\eps(x)
  -\mathsf E_{t-s}^\eps(y)\right\|_1
 &\le C(R^2+RS)\eps^{2\beta}.
 \label{eq:energy-concatenation-seam}
\end{align}
Here $(r_sF)_t=F_{t\wedge s}$ and
$(\theta_sF)_v=F_{s+(v\wedge(T-s))}-F_s$, followed by constant stopping at
the remaining horizon.
\end{lemma}

\begin{proof}
For $t>s$, the stopped defect is
$D_t^\eps(r_sx)=e^{-(t-s)/\eps}D_s^\eps(x)$, whence
\[
 \mathsf E_t^\eps(r_sx)-\mathsf E_s^\eps(x)
 =(1-e^{-2(t-s)/\eps})(D_s^\eps(x))^{\otimes2}.
\]
This gives \eqref{eq:energy-stopping-seam}.  On the remaining horizon,
\eqref{eq:D-restart} gives
\[
 D_v^\eps(\theta_sx)=D_{s+v}^\eps(x)-e^{-v/\eps}D_s^\eps(x).
\]
Expanding the square in the definition of $\mathsf E^\eps$ yields
\eqref{eq:energy-shift-seam}; the two cross terms are bounded using
$\|u\otimes v\|_1=\|u\|\|v\|$, the integral of
$\eps^{-1}e^{-v/\eps}$, and
\eqref{eq:deterministic-resolvent-bounds}.  For the concatenated path,
\[
 D_{s+v}^\eps(x\otimes_sy)
 =D_v^\eps(y)+e^{-v/\eps}D_s^\eps(x).
\]
The same expansion gives \eqref{eq:energy-concatenation-seam}.  The stopped
terminal pieces have the same exponential form and satisfy the identical
bound.
\end{proof}

\begin{proposition}[Shift-equivariance and concatenation of the common state]
\label{thm:shift-equivariant-enhancement}
The domain $G_J$ is stable under deterministic stopping, every fixed-horizon
shift $\theta_s$, and finite temporal concatenation.  Let $x\in G_J$,
$s,v\in[0,T]$, and $u=s+(v\wedge(T-s))$.  Then
\begin{align}
 \widehat J_v(\theta_sx)
 &=\mathbb X^I_{s,u}(x),
 \label{eq:shift-J}\\
 Q_v(\theta_sx)&=Q_{s,u}(x),
 \label{eq:shift-Q}\\
 \mathbb X^\bullet_{a,b}(\theta_sx)
 &=\mathbb X^\bullet_{s+(a\wedge(T-s)),
                     s+(b\wedge(T-s))}(x),
 \quad 0\le a\le b\le T,\quad \bullet\in\{I,S\}.
 \label{eq:shift-lift}
\end{align}
If $x,y\in G_J$ are respectively a stopped past and a delayed future at
$s$, then $z=x+y\in G_J$ and
\begin{equation}\label{eq:J-concatenation}
 \widehat J_t(z)=\widehat J_t(x)+\widehat J_t(y)+x_s\otimes y_t.
\end{equation}
Consequently $Q(z)=Q(x)+Q(y)$, and both second levels concatenate by
Chen's rule.
\end{proposition}

\begin{proof}
For a continuous path, the approximate-identity bounds implicit in
\eqref{eq:causal-resolvent} give, uniformly in $s$,
\[
 |D_s^\eps(x)|+\sup_v|\mathcal A_v^\eps(\theta_sx)|\longrightarrow0.
\]
Indeed, split the exponential integrals at a fixed $\delta>0$, use the
modulus of continuity on $[0,\delta]$, and then the factor
$e^{-\delta/\eps}$.  Hence \eqref{eq:resolvent-J-seam} and its terminal
version show that $J^{\eps_n}(\theta_sx)$ is uniformly Cauchy whenever
$J^{\eps_n}(x)$ is.  Passing to the limit yields \eqref{eq:shift-J}; the
identities for $Q$ and the second levels are algebraic.  Stopping follows
from \eqref{eq:finite-stopping}.  For concatenation, the error term in
\eqref{eq:finite-resolvent-J-concatenation} tends to zero uniformly, which
gives \eqref{eq:J-concatenation}; the remaining conclusions follow from
\eqref{eq:common-Q} and Chen's relation.  Iteration treats finitely many
seams.
\end{proof}

\begin{proposition}[Common bracket-variation and operator-envelope domains]
\label{prop:common-bracket-variation-domain}
Let $G_Q$ be the domain constructed in the proof of
\Cref{cor:common-representative}; equivalently,
\begin{equation}\label{eq:bracket-domain}
 \begin{split}
 G_Q=\{x\in G_J:\;&Q_t(x)-Q_s(x)\in\mathcal S_1(H)_+,\\
 &\Tr(Q_t(x)-Q_s(x))\le\Lambda(t-s)
 \text{ for all rational }0\le s<t\le T\}.
 \end{split}
\end{equation}
Fix a countable dense $\mathbb Q$-linear subspace $H_{\mathbb Q}\subset H$ and, for
$\Gamma\in\mathcal S_1(H)_+$, define
\begin{equation}\label{eq:common-Gamma-envelope-domain}
 G_Q^\Gamma
 :=\left\{x\in G_J:
 0\le\langle Q_{s,t}(x)h,h\rangle
 \le(t-s)\langle\Gamma h,h\rangle
 \text{ for all rational }s<t\text{ and }h\in H_{\mathbb Q}\right\}.
\end{equation}
Then $G_Q$ and $G_Q^\Gamma$ are Borel and
$c_{\Smax}(G_Q^c)=0$.  For every model subclass
$\mathcal P_\Gamma\subset\Smax$ satisfying
$0\preceq a_t^P\preceq\Gamma$ for $\dd t\otimes P$-almost every
$(t,\omega)$ and every $P\in\mathcal P_\Gamma$,
\[
 c_{\mathcal P_\Gamma}((G_Q^\Gamma)^c)=0.
\]
On $G_Q$, $Q$ is trace-norm Lipschitz and has bounded variation.  On
$G_Q^\Gamma$, for all $0\le s<t\le T$,
\begin{equation}\label{eq:common-Gamma-envelope-all-times}
 0\preceq Q_{s,t}\preceq(t-s)\Gamma,
 \qquad
 \|Q_t-Q_s\|_1\le(t-s)\Tr\Gamma.
\end{equation}
Both domains are stable under stopping, every shift, and finite temporal
concatenation.
\end{proposition}

\begin{proof}
The proof of \Cref{cor:common-representative} already gives the Borelness,
fullness, and trace-norm Lipschitz property of $G_Q$.  The definition of
$G_Q^\Gamma$ is countable and Borel.  Under every
$P\in\mathcal P_\Gamma$,
\[
 Q_{s,t}=\int_s^t a_r^P\,\dd r,
 \qquad 0\preceq a_r^P\preceq\Gamma,
\]
so its countably many defining inequalities hold almost surely.

On $G_J$, $Q$ is continuous in $\mathcal S_2$.  Continuity in the time pair
and density of $H_{\mathbb Q}$ extend the quadratic-form inequalities to all
times and all $h\in H$.  A positive operator dominated by the trace-class
operator $(t-s)\Gamma$ is trace class, and its trace norm is at most
$(t-s)\Tr\Gamma$.  This proves
\eqref{eq:common-Gamma-envelope-all-times}.  Stability follows from
\eqref{eq:selector-stopping}, \eqref{eq:shift-Q}, and covariance additivity
under concatenation; for an increment crossing a seam, the two dominated
positive increments add and their time lengths sum to $t-s$.
\end{proof}

\section{Compact perfection: spatial tightness and restart-stable cores}
\label{sec:restart-approximation-geometry}

Capacity rough paths and nonlinear-expectation dynamics supply quasi-sure
enhancements in structured settings
\cite{BoedihardjoGengQian16,GengQianYang14,PengZhang17,PengZhang22}, and
quantitative scalar/driver cores support simultaneous flows
\cite{ZhaoCausalFlows25,ZhaoDriverCores25}.  In the present Hilbert setting,
however, a trace clock allows covariance to escape through orthogonal spatial
directions; the next obstruction identifies the missing control.

\begin{theorem}[Trace-clock obstruction to compact capacity cores]
\label{thm:trace-clock-compactness-obstruction}
Assume that $H$ is infinite dimensional and let $(e_n)$ be an orthonormal
sequence.  Let $\mathbb P_n$ be the law on $\Omega_H$ of
\[
 X_t=W_t e_n,
\]
where $W$ is a standard real Brownian motion.  Then
$\mathbb P_n\in\mathfrak S_{1,0}^{0,T}(H)$ for every $n$, and for every compact
$K\subset C_0([0,T];H)$,
\begin{equation}\label{eq:trace-clock-capacity-one-outside-compacts}
 \sup_{n\ge1}\mathbb P_n(K^c)
 =c_{\mathfrak S_{1,0}^{0,T}(H)}(K^c)=1.
\end{equation}
Consequently the trace-clock class is not uniformly tight and admits no compact
capacity core.  Moreover, no $\Gamma\in\mathcal S_1(H)_+$ satisfies
\[
 e_n\otimes e_n\preceq\Gamma
 \qquad\text{for every }n.
\]
\end{theorem}

\begin{proof}
Each law has bracket density $a_t^{\mathbb P_n}=e_n\otimes e_n$, hence trace
one.  Let $K\subset C_0([0,T];H)$ be compact.  Its endpoint image
$K_T:=\{x_T:x\in K\}$ is compact in $H$, and therefore
\[
 \sup_{y\in K_T}|\langle y,e_n\rangle|\longrightarrow0.
\]
Set $r_n:=\sup_{y\in K_T}|\langle y,e_n\rangle|$.  On $\{X\in K\}$,
\[
 |W_T|=|\langle X_T,e_n\rangle|\le r_n,
\]
so
\[
 \mathbb P_n(K)\le \mathbb P(|W_T|\le r_n)\longrightarrow0.
\]
Hence $\sup_n\mathbb P_n(K^c)=1$; the capacity equality follows because every
$\mathbb P_n$ belongs to the trace-clock class.  If a trace-class $\Gamma$
dominated every $e_n\otimes e_n$, then
$\langle\Gamma e_n,e_n\rangle\ge1$ for all $n$, contradicting
$\operatorname{Tr}\Gamma<\infty$.
\end{proof}

We next place the common realization on restart-stable compact capacity cores.
A fixed envelope $0\preceq a_t^P\preceq\Gamma$ is one source of the spatial
tightness missing in the preceding obstruction; the construction begins with
a dyadic rough-approximation estimate.

For a continuous step-two multiplicative functional
$\mathbf x=(1,x,\mathbb x)$ set
\[
 \rho_\alpha(\mathbf x,\widetilde{\mathbf x})
 :=\sup_{s<t}\frac{|x_{s,t}-\widetilde x_{s,t}|}{|t-s|^\alpha}
  +\sup_{s<t}\frac{|\mathbb x_{s,t}-\widetilde{\mathbb x}_{s,t}|}
        {|t-s|^{2\alpha}}.
\]

\begin{lemma}[Dyadic reconstruction with Chen cross terms]
\label{lem:dyadic-rough-reconstruction}
Let $\mathbf x=(1,x,\mathbb x)$ and
$\widetilde{\mathbf x}=(1,\widetilde x,\widetilde{\mathbb x})$ be random
continuous step-two multiplicative functionals.  Suppose that, for some
$r\ge4$, $\delta>0$, $\gamma>\alpha+1/r$, and every $s<t$,
\begin{align*}
 \norm{x_{s,t}-\widetilde x_{s,t}}_{L^r}
 &\le a_r\delta|t-s|^\gamma,\\
 \norm{\mathbb x_{s,t}-\widetilde{\mathbb x}_{s,t}}_{L^{r/2}}
 &\le b_r\delta^2|t-s|^{2\gamma},\\
 \norm{x_{s,t}}_{L^r}+\norm{\widetilde x_{s,t}}_{L^r}
 &\le c_r|t-s|^{1/2},
\end{align*}
and $1/2>\alpha+1/r$.  Then
\begin{equation}\label{eq:dyadic-reconstruction-moment}
 \norm{\rho_\alpha(\mathbf x,\widetilde{\mathbf x})}_{L^{r/2}}
 \le C\bigl(a_r\delta+b_r\delta^2+a_rc_r\delta\bigr).
\end{equation}
Here $C$ depends only on $(\alpha,T)$ and the reciprocals of the two
positive gaps in the assumptions; in particular, it is uniform in $r$ when
those gaps are bounded below.
\end{lemma}

\begin{proof}
An affine time change reduces the proof to $T=1$.  Write
$t_i^k=i2^{-k}$ and introduce the adjacent-increment maxima
\begin{align*}
 A_k&:=\max_{0\le i<2^k}
       |(x-\widetilde x)_{t_i^k,t_{i+1}^k}|,\\
 B_k&:=\max_{0\le i<2^k}
       |(\mathbb x-\widetilde{\mathbb x})_{t_i^k,t_{i+1}^k}|,\\
 C_k&:=\max_{0\le i<2^k}
       \bigl(|x_{t_i^k,t_{i+1}^k}|
             +|\widetilde x_{t_i^k,t_{i+1}^k}|\bigr).
\end{align*}
The elementary maximum inequality
$\|\max_i|Z_i|\|_{L^p}\le
 (\sum_i\|Z_i\|_{L^p}^p)^{1/p}$ gives
\begin{align}
 \|A_k\|_{L^r}
 &\le a_r\delta\,2^{-k(\gamma-1/r)},
 &\|B_k\|_{L^{r/2}}
 &\le b_r\delta^2\,2^{-k(2\gamma-2/r)},
 \label{eq:dyadic-maxima-error}\\
 \|C_k\|_{L^r}
 &\le c_r\,2^{-k(1/2-1/r)}.
 \label{eq:dyadic-maxima-base}
\end{align}

The deterministic chaining estimate is as follows.  Every open interval $(s,t)$
has a disjoint dyadic decomposition, ordered from left to right, containing
at most two intervals of each generation $k\ge k_0$, where
$2^{-k_0-1}<t-s\le2^{-k_0}$.  Endpoint pieces are obtained as limits;
continuity permits passage to those limits.  The identities below are first
applied to finite contiguous truncations.  If the weighted series diverge,
the resulting pathwise bounds are immediate; otherwise, let the truncations
exhaust the interval.  Additivity of the first level therefore yields
\begin{equation}\label{eq:dyadic-first-chain}
 \sup_{s<t}\frac{|(x-\widetilde x)_{s,t}|}{|t-s|^\alpha}
 \le C\sum_{k\ge0}2^{\alpha k}A_k.
\end{equation}
For an ordered collection of dyadic blocks $(I_j)$ in such a decomposition,
iterating Chen's identity gives
\begin{align*}
 (\mathbb x-\widetilde{\mathbb x})_{s,t}
 &=\sum_j(\mathbb x-\widetilde{\mathbb x})_{I_j}\\
 &\quad+\sum_{i<j}\Bigl(
   (x-\widetilde x)_{I_i}\otimes x_{I_j}
   +\widetilde x_{I_i}\otimes(x-\widetilde x)_{I_j}
 \Bigr).
\end{align*}
Consequently,
\begin{equation}\label{eq:dyadic-second-chain}
 \sup_{s<t}
 \frac{|(\mathbb x-\widetilde{\mathbb x})_{s,t}|}{|t-s|^{2\alpha}}
 \le C\sum_{k\ge0}2^{2\alpha k}B_k
   +C\left(\sum_{k\ge0}2^{\alpha k}A_k\right)
       \left(\sum_{k\ge0}2^{\alpha k}C_k\right).
\end{equation}
Indeed, the three unweighted sums over generations $k\ge k_0$ contain at
most two terms per generation, and multiplication by
$|t-s|^{-\alpha}$ or $|t-s|^{-2\alpha}$ is absorbed by the displayed
weights.

Minkowski's inequality, \eqref{eq:dyadic-maxima-error}--
\eqref{eq:dyadic-maxima-base}, and H\"older's inequality for the product in
\eqref{eq:dyadic-second-chain} now show
\begin{align*}
 \left\|\sum_k2^{\alpha k}A_k\right\|_{L^r}
 &\le a_r\delta\sum_k2^{-k(\gamma-\alpha-1/r)},\\
 \left\|\sum_k2^{2\alpha k}B_k\right\|_{L^{r/2}}
 &\le b_r\delta^2\sum_k2^{-k(2\gamma-2\alpha-2/r)},\\
 \left\|\sum_k2^{\alpha k}C_k\right\|_{L^r}
 &\le c_r\sum_k2^{-k(1/2-\alpha-1/r)}.
\end{align*}
All three series converge under the stated assumptions.  Combining these
bounds with \eqref{eq:dyadic-first-chain}--
\eqref{eq:dyadic-second-chain}, using monotonicity from $L^r$ to
$L^{r/2}$ for the first level, proves
\eqref{eq:dyadic-reconstruction-moment}; lower bounds for the two exponent
gaps make the geometric-series constants uniform in $r$.
\end{proof}

\begin{lemma}[Dyadic geometric resolvent bound]
\label{lem:dyadic-geometric-resolvent-bound}
Let $1/3<\alpha<1/2$ and
$0<\vartheta<1/2-\alpha$.  For the smooth resolvent path $Y^\eps$ and the
canonical geometric lift $\mathbf X^{S,P}$,
\begin{equation}\label{eq:dyadic-geometric-resolvent-rate}
 \sup_{P\in\Smax}
 \|\rho_\alpha(S_2(Y^\eps),\mathbf X^{S,P})\|_{L^q(P)}
 \le Cq\eps^\vartheta,
 \qquad q\ge q_0.
\end{equation}
The estimate is used only to choose summable dyadic core gauges.
\end{lemma}

\begin{proof}
Set $h=t-s$ and fix $0<\vartheta<1/2-\alpha$.  Choose $q_0$ so large that
\[
 \frac1{2q_0}<\frac12-\alpha-\vartheta,
\]
and put $\gamma:=1/2-\vartheta$.  For $q\ge q_0$, apply
\Cref{lem:dyadic-rough-reconstruction} with $r=2q$ and
$\delta=\eps^\vartheta$.

The stochastic convolution identity gives
\[
 \|D_{s,t}^\eps\|_{L^{2q}(P)}
 \le C\sqrt q\,(h^{1/2}\wedge\eps^{1/2}).
\]
For $0<\gamma<1/2$,
\[
 h^{1/2}\wedge\eps^{1/2}
 \le \eps^{1/2-\gamma}h^\gamma
 =\delta h^\gamma,
\]
so the first-level hypothesis holds with $a_r\le C\sqrt q$.

Let $\mathbb Y^\eps$ be the second level of $S_2(Y^\eps)$.  The following
second-level estimate is used in the compact-core gauges.  First,
\[
 \|Y_{s,t}^\eps\|_{L^{2q}(P;H)}
 \le C\sqrt q\left(h^{1/2}\wedge\frac{h}{\sqrt\eps}\right).
 \tag{*}
\]
Indeed, the $h^{1/2}$ bound follows from
$Y_{s,t}^\eps=X_{s,t}-D_{s,t}^\eps$, while
$\eps\dot Y^\eps=D^\eps$ and
\eqref{eq:D-pointwise-moment} give the $h/\sqrt\eps$ bound; the bounded-drift
terms are absorbed into $C=C(\Lambda,B,T)$.  Since both lifts are geometric,
\begin{align*}
 \Sym\bigl(\mathbb Y_{s,t}^\eps-\mathbb X_{s,t}^{S,P}\bigr)
 &=\frac12\left((Y_{s,t}^\eps)^{\otimes2}
                 -X_{s,t}^{\otimes2}\right).
\end{align*}
Using $X_{s,t}=Y_{s,t}^\eps+D_{s,t}^\eps$, the preceding first-level bounds
and H\"older's inequality therefore yield
\begin{equation}\label{eq:resolvent-second-symmetric-local}
 \left\|\Sym\bigl(\mathbb Y_{s,t}^\eps
                  -\mathbb X_{s,t}^{S,P}\bigr)\right\|_{L^q(P;\mathcal S_2)}
 \le Cq\,(h\wedge\sqrt{\eps h}).
\end{equation}

For the antisymmetric part, write $D_{s,r}^\eps=D_r^\eps-D_s^\eps$.
Since $dY_r^\eps=\eps^{-1}D_r^\eps\,dr$, integration by parts gives the exact
identity
\begin{align}
 \Anti\bigl(\mathbb Y_{s,t}^\eps-\mathbb X_{s,t}^{S,P}\bigr)
 &=\Anti\bigl(D_t^\eps\otimes Y_{s,t}^\eps\bigr)
   -\Anti\int_s^tD_{s,r}^\eps\otimes dM_r^P \notag\\
 &\quad
   -\Anti\int_s^tD_{s,r}^\eps\otimes b_r^P\,dr.
 \label{eq:resolvent-second-antisymmetric-decomposition}
\end{align}
To verify it, expand $X_{s,r}=Y_{s,r}^\eps+D_{s,r}^\eps$ in the classical
Stratonovich area.  The cross integral
$\int_s^tY_{s,r}^\eps\otimes dD_r^\eps$ has antisymmetric part
$\Anti(Y_{s,t}^\eps\otimes D_t^\eps)$, because
$dY_r^\eps\otimes D_r^\eps=\eps^{-1}D_r^\eps{}^{\otimes2}dr$ is symmetric;
the remaining increment term is the two integrals displayed above.  Notice
that the It\^o--Stratonovich correction is symmetric and hence disappears
after $\Anti$.

The boundary term in
\eqref{eq:resolvent-second-antisymmetric-decomposition} is bounded by
\eqref{eq:D-pointwise-moment} and $(*)$:
\[
 \|D_t^\eps\otimes Y_{s,t}^\eps\|_{L^q(P;\mathcal S_2)}
 \le Cq\,(h\wedge\sqrt{\eps h}).
\]
For the martingale term, the covariance-weighted Hilbert--Schmidt identity
and maximal BDG give
\begin{align*}
 \left\|\int_s^tD_{s,r}^\eps\otimes dM_r^P\right\|_{L^q(P;\mathcal S_2)}
 &\le C\sqrt{q\Lambda}
       \left(\int_s^t
       \|D_{s,r}^\eps\|_{L^{2q}(P;H)}^2\,dr\right)^{1/2}\\
 &\le Cq\left(\int_0^h(u\wedge\eps)\,du\right)^{1/2}
 \le Cq\,(h\wedge\sqrt{\eps h}).
\end{align*}
Finally,
\[
 \left\|\int_s^tD_{s,r}^\eps\otimes b_r^P\,dr\right\|_{L^q(P;\mathcal S_2)}
 \le CB\sqrt q\int_0^h(\sqrt u\wedge\sqrt\eps)\,du
 \le Cq\,(h\wedge\sqrt{\eps h}),
\]
where $h\le T$ is used in the last step.  Combining this estimate with
\eqref{eq:resolvent-second-symmetric-local} proves
\begin{equation}\label{eq:resolvent-second-local-increment}
 \|\mathbb Y_{s,t}^\eps-\mathbb X_{s,t}^{S,P}\|_{L^q(P;\mathcal S_2)}
 \le Cq\,(h\wedge\sqrt{\eps h}).
\end{equation}
Since $\gamma>1/4$,
\[
 h\wedge\sqrt{\eps h}
 \le \eps^{1-2\gamma}h^{2\gamma}
 =\delta^2h^{2\gamma},
\]
so the second-level hypothesis holds with $b_r\le Cq$.  Finally,
\Cref{lem:enhanced-moments} gives the base increment bound with
$c_r\le C\sqrt q$.

The inequalities
$\gamma>\alpha+1/(2q)$ and $1/2>\alpha+1/(2q)$ follow from the choice of
$q_0$.  Therefore \Cref{lem:dyadic-rough-reconstruction} yields
\[
 \|\rho_\alpha(S_2(Y^\eps),\mathbf X^{S,P})\|_{L^q(P)}
 \le C\bigl(\sqrt q\,\delta+q\delta^2+q\delta\bigr)
 \le Cq\eps^\vartheta,
\]
uniformly in $P$, after enlarging $C$ for $0<\eps\le T$.  This is
\eqref{eq:dyadic-geometric-resolvent-rate}.
\end{proof}

\subsection{Spatial trace tightness and compact-envelope realization}

\label{sec:master-core}

By \Cref{thm:trace-clock-compactness-obstruction}, the additional intrinsic
compactness datum is a uniform spatial trace profile.

\begin{definition}[Uniform spatial trace-tightness profile]
\label{def:uniform-spatial-trace-tightness}
Let $\varnothing\ne\mathcal P_{\rm sp}\subset
\mathfrak S_{\Lambda,0}^{0,T}(H)$.  A sequence of finite-rank orthogonal
projections $P_m\uparrow I_H$, together with numbers
$\vartheta_m\downarrow0$, is a \emph{uniform spatial trace-tightness profile}
if, with $R_m:=I-P_m$,
\begin{equation}\label{eq:uniform-spatial-trace-profile}
 \sup_{P\in\mathcal P_{\rm sp}}
 \left\|\Tr(R_ma_t^PR_m)\right\|_{L^\infty(\dd t\otimes P)}
 \le\vartheta_m,
 \qquad m\ge0.
\end{equation}
After passing to a subsequence and relabeling, one may assume
$\vartheta_m\le2^{-8m}$.  Writing $r_m=\operatorname{rank}P_m$, define
\begin{equation}\label{eq:spatial-trace-profile-modulus}
 \omega_{\boldsymbol\vartheta}(\eps)
 :=\inf_{m\ge0}\left(\sqrt{r_m\eps}+\sqrt{\vartheta_m}\right).
\end{equation}
Then $\omega_{\boldsymbol\vartheta}(\eps)\downarrow0$ as $\eps\downarrow0$.
\end{definition}

\begin{example}[Spatial trace tightness without a dominating envelope]
\label{ex:spatial-trace-tight-no-envelope}
Let $H$ be infinite dimensional, let $(e_n)$ be an orthonormal basis, and let
$\mathbb P_n$ be the law of $X_t=n^{-1/2}W_te_n$.  For the projection onto
$\operatorname{span}\{e_1,\ldots,e_m\}$,
\[
 \sup_{n\ge1}\Tr\bigl((I-P_m)a^{\mathbb P_n}(I-P_m)\bigr)
 =\sup_{n>m}\frac1n=\frac1{m+1}\longrightarrow0.
\]
Thus the family has a uniform spatial trace-tightness profile.  However, if
one positive trace-class operator $\Gamma$ dominated every
$n^{-1}e_n\otimes e_n$, then
$\langle\Gamma e_n,e_n\rangle\ge1/n$ for all $n$, contradicting
$\Tr\Gamma<\infty$.  Hence spatial trace tightness is strictly weaker than
the existence of a common trace-class operator envelope.
\end{example}

\begin{lemma}[Compact restart geometry of covariance profiles]
\label{lem:covariance-profile-compact-restart}
Fix a uniform spatial trace-tightness profile
$(P_m,\vartheta_m)$ and put $R_m=I-P_m$.  Let
$\mathscr Q_{\Lambda,\boldsymbol\vartheta}$ be the set of continuous paths
$q:[0,T]\to\mathcal S_1(H)_{\rm sa}$ with $q_0=0$ such that, for every
$0\le s\le t\le T$ and every $m$,
\begin{equation}\label{eq:abstract-covariance-profile-class}
 0\preceq q_{s,t},\qquad
 \Tr q_{s,t}\le\Lambda(t-s),\qquad
 \Tr(R_mq_{s,t}R_m)\le\vartheta_m(t-s).
\end{equation}
Then:
\begin{enumerate}[label=\textup{(\roman*)},leftmargin=2.5em]
\item $\mathscr Q_{\Lambda,\boldsymbol\vartheta}$ is compact in
$C([0,T];\mathcal S_1(H))$.  For every $\nu>1$, it is also compact in the
$\nu$-variation topology, and the identity from the uniform topology to the
$\nu$-variation topology is continuous on this class.
\item The class is closed under deterministic stopping, fixed-horizon future
shift, and covariance concatenation
\[
 (q\otimes_s\widetilde q)_t
 :=q_{t\wedge s}+\widetilde q_{(t-s)_+},
\]
with the usual constant-extension convention.  In particular, the spatial
trace profile is a restart congruence: no enlargement of
$(\Lambda,\boldsymbol\vartheta)$ is needed under these operations.
\end{enumerate}
\end{lemma}

\begin{proof}
Every path in the class is increasing in the positive cone and
\[
 \|q_t-q_s\|_1=\Tr q_{s,t}\le\Lambda|t-s|,
\]
so the family is uniformly trace-norm Lipschitz.  For a positive trace-class
operator $B$ and an orthogonal projection $P$ with $R=I-P$, the four-block
decomposition and trace-ideal Cauchy--Schwarz give
\begin{equation}\label{eq:positive-block-tail}
 \|B-PBP\|_1
 \le2\sqrt{\Tr B\,\Tr(RBR)}+\Tr(RBR).
\end{equation}
Applying this to $B=q_t$ gives
\[
 \|q_t-P_mq_tP_m\|_1
 \le 2\sqrt{\Lambda T\,\vartheta_mT}+\vartheta_mT.
\]
The compressed values lie in a bounded subset of the finite-dimensional
space $P_m\mathcal S_1(H)P_m$; hence the set of all values $q_t$ is uniformly
totally bounded in trace norm.  Banach-valued Arzel\`a--Ascoli now gives
relative compactness in $C([0,T];\mathcal S_1)$.  The conditions in
\eqref{eq:abstract-covariance-profile-class} are closed under uniform
trace-norm convergence, so the class is compact.

For any Banach-valued bounded-variation paths $f,g$ and $\nu>1$,
\begin{equation}\label{eq:uniform-variation-interpolation}
 \|f-g\|_{\nu\text{-var}}
 \le
 (2\|f-g\|_\infty)^{1-1/\nu}
 (\|f\|_{1\text{-var}}+\|g\|_{1\text{-var}})^{1/\nu}.
\end{equation}
Every $q$ in the class has total trace variation at most $\Lambda T$;
therefore uniform convergence implies $\nu$-variation convergence on the
class, proving the first part.

Stopping and future shift restrict increments of $q$, so all three
bounds in \eqref{eq:abstract-covariance-profile-class} are preserved.  For a
concatenated increment crossing the seam, positivity and additivity split it
into one increment of $q$ and one of $\widetilde q$; the trace and every tail
trace then add with total time length $t-s$.  This proves the restart claim.
\end{proof}

\begin{proposition}[Covariance path--density isomorphism]
\label{prop:covariance-path-density-isomorphism}
Fix a spatial trace profile $(P_m,\vartheta_m)$ and write $R_m=I-P_m$.
Let $\mathscr A_{\Lambda,\boldsymbol\vartheta}$ be the space of Lebesgue-a.e.
equivalence classes of strongly measurable
$a:[0,T]\to\mathcal S_1(H)_{\rm sa}$ satisfying, for almost every $t$,
\begin{equation}\label{eq:covariance-density-profile-class}
 0\preceq a_t,
 \qquad
 \Tr a_t\le\Lambda,
 \qquad
 \Tr(R_ma_tR_m)\le\vartheta_m\quad(m\ge0).
\end{equation}
Equip $\mathscr A_{\Lambda,\boldsymbol\vartheta}$ with the inherited
$L^1([0,T];\mathcal S_1)$ metric
\begin{equation}\label{eq:covariance-density-L1-metric}
 d_{\rm den}(a,\widetilde a)
 :=\int_0^T\|a_t-\widetilde a_t\|_1\,\dd t.
\end{equation}
This is a closed, hence complete, subset of
$L^1([0,T];\mathcal S_1(H)_{\rm sa})$.
The Bochner integration map
\begin{equation}\label{eq:covariance-density-integration-map}
 \mathcal I_{\rm cov}:\mathscr A_{\Lambda,\boldsymbol\vartheta}
 \longrightarrow\mathscr Q_{\Lambda,\boldsymbol\vartheta},
 \qquad
 (\mathcal I_{\rm cov}a)_t:=\int_0^t a_r\,\dd r,
\end{equation}
is a bijection.  When $\mathscr Q_{\Lambda,\boldsymbol\vartheta}$ is equipped
with the trace-variation metric, it is an isometric bijection; in particular
that metric space is complete.  More explicitly, if
$q=\mathcal I_{\rm cov}a$ and
$\widetilde q=\mathcal I_{\rm cov}\widetilde a$, then
\begin{align}
 \|q-\widetilde q\|_{1\text{-var};\mathcal S_1}
 &=\int_0^T\|a_t-\widetilde a_t\|_1\,\dd t,
 \label{eq:covariance-density-variation-isometry}\\
 \|q-\widetilde q\|_{\infty;\mathcal S_1}
 &\le\int_0^T\|a_t-\widetilde a_t\|_1\,\dd t.
 \label{eq:covariance-density-uniform-contraction}
\end{align}
Thus the derivative class is the exact $L^1$ tangent coordinate of the
covariance path when the latter is equipped with trace variation.

The bijection intertwines restart operations.  For $s\in[0,T]$, define, up to
Lebesgue-a.e. equality,
\begin{align}
 (r_sa)_t&:=\one_{\{t<s\}}a_t,
 \label{eq:density-stopping-operation}\\
 (\theta_sa)_v&:=\one_{\{v<T-s\}}a_{s+v},
 \label{eq:density-shift-operation}\\
 (a\otimes_s\widetilde a)_t
 &:=\one_{\{t<s\}}a_t+\one_{\{t>s\}}\widetilde a_{t-s}.
 \label{eq:density-concatenation-operation}
\end{align}
Then
\begin{equation}\label{eq:covariance-density-restart-intertwining}
 \mathcal I_{\rm cov}(r_sa)=r_s\mathcal I_{\rm cov}(a),
 \qquad
 \mathcal I_{\rm cov}(\theta_sa)=\theta_s\mathcal I_{\rm cov}(a),
 \qquad
 \mathcal I_{\rm cov}(a\otimes_s\widetilde a)
 =\mathcal I_{\rm cov}(a)\otimes_s\mathcal I_{\rm cov}(\widetilde a).
\end{equation}
All three density operations preserve
$\mathscr A_{\Lambda,\boldsymbol\vartheta}$.

If $\Gamma\in\mathcal S_1(H)_+$, the additional pathwise order condition
\begin{equation}\label{eq:covariance-path-Gamma-order}
 0\preceq q_{s,t}\preceq(t-s)\Gamma
 \qquad(s\le t)
\end{equation}
is equivalent under \eqref{eq:covariance-density-integration-map} to
$0\preceq a_t\preceq\Gamma$ for almost every $t$.
\end{proposition}

\begin{proof}
If $a$ satisfies \eqref{eq:covariance-density-profile-class}, then positivity
and the two trace bounds integrate over every interval, so
$\mathcal I_{\rm cov}a\in\mathscr Q_{\Lambda,\boldsymbol\vartheta}$.
Conversely, every $q\in\mathscr Q_{\Lambda,\boldsymbol\vartheta}$ is
$\Lambda$-Lipschitz in trace norm.  The trace class has the Radon--Nikodym
property, hence $q$ is Bochner differentiable almost everywhere and
\[
 q_t=\int_0^t\dot q_r\,\dd r.
\]
Choose a countable dense set in $H$.  Differentiating the scalar increasing
paths $t\mapsto\langle q_t h,h\rangle$ on one common full-measure set gives
$\dot q_t\succeq0$.  Since trace and the maps
$B\mapsto\Tr(R_mBR_m)$ are continuous linear functionals on
$\mathcal S_1(H)$, differentiation of the corresponding scalar Lipschitz
paths gives
\[
 \Tr\dot q_t\le\Lambda,
 \qquad
 \Tr(R_m\dot q_tR_m)\le\vartheta_m
\]
for every $m$ on a common full-measure set.  Thus
$\dot q\in\mathscr A_{\Lambda,\boldsymbol\vartheta}$.  Bochner
differentiation also proves uniqueness of the density class, hence
\eqref{eq:covariance-density-integration-map} is bijective.

The density class is closed in $L^1$.  Indeed, if $a^n\to a$ in $L^1$, pass
to a subsequence converging in trace norm almost everywhere.  The positive
cone is trace-norm closed, while trace and every tail trace
$B\mapsto\Tr(R_mBR_m)$ are continuous; hence
\eqref{eq:covariance-density-profile-class} passes to the limit almost
everywhere.  Thus $a\in\mathscr A_{\Lambda,\boldsymbol\vartheta}$.

For an absolutely continuous Banach-valued path, total variation equals the
integral of the norm of its Bochner derivative.  Applied to
$q-\widetilde q$, this gives
\eqref{eq:covariance-density-variation-isometry}; the uniform estimate follows
by integration.  The identities
\eqref{eq:covariance-density-restart-intertwining} are immediate from the
three density formulas and Bochner integration, and the profile bounds are
preserved pointwise almost everywhere.

Finally, if $a_t\preceq\Gamma$ almost everywhere, integration gives
\eqref{eq:covariance-path-Gamma-order}.  Conversely, differentiate
\[
 0\le\langle q_{s,t}h,h\rangle
 \le(t-s)\langle\Gamma h,h\rangle
\]
for a countable dense set of $h$ and extend by continuity of quadratic forms;
this gives $0\preceq\dot q_t\preceq\Gamma$ almost everywhere.
\end{proof}

The operator envelope
\[
 0\preceq a_t^P\preceq\Gamma,
 \qquad \Gamma\in\mathcal S_1(H)_+,
\]
is a convenient sufficient realization of this profile and is retained in the
main construction because the later tangent theory uses the compact order
interval $[0,\Gamma]$.  The compactness theorem is proved at the intrinsic
spatial-profile level; the fixed-envelope cores are then recovered as a
closed operator-order subfamily.

Fix $\Gamma\in\mathcal S_1(H)_+$ and let
\[
 \Mmax=\mathfrak M_\Gamma^{0,T}(H)
 :=\{P\in\mathfrak S_{\Tr\Gamma,0}^{0,T}(H):
       0\preceq a_t^P\preceq\Gamma\}.
\]
Whenever an earlier ambient estimate is invoked on this fixed-envelope
family, it is understood with the specialized parameters
$(\Lambda,B)=(\Tr\Gamma,0)$, so that
$\Mmax\subset\mathfrak S_{\Tr\Gamma,0}^{0,T}(H)$.  The raw resolvent,
primitive, and Borel limit selectors are law-independent; this specialization
changes only the upper norms and the family over which they are taken.
Choose finite-rank orthogonal projections $P_m\uparrow I_H$ and write
$r_m:=\operatorname{rank}P_m$ so that
\begin{equation}\label{eq:Gamma-tail-profile}
 \tau_m:=\Tr((I-P_m)\Gamma)\le2^{-8m},
 \qquad m\ge0.
\end{equation}
The projections encode the compactness profile; the canonical state is
independent of this choice.  Define the spectral compactness
modulus
\begin{equation}\label{eq:trace-energy-modulus}
 \omega_\Gamma(\eps)
 :=\inf_{m\ge0}\left(\sqrt{r_m\eps}+\sqrt{\tau_m}\right).
\end{equation}
Then $\omega_\Gamma(\eps)\downarrow0$ as $\eps\downarrow0$.

\begin{lemma}[Trace-energy approximation under spatial trace tightness]
\label{cor:trace-energy-spatial-profile}
Let $\mathcal P_{\rm sp}$ have a profile
$(P_m,\vartheta_m)$ in the sense of
\Cref{def:uniform-spatial-trace-tightness}.  Then, for every $q\ge2$ and
$0<\eps\le T$,
\begin{equation}\label{eq:trace-energy-spatial-profile-rate}
 \sup_{P\in\mathcal P_{\rm sp}}
 \left\|\sup_{t\le T}
 \|\mathsf E_t^\eps-Q_t\|_{\mathcal S_1}\right\|_{L^q(P)}
 \le C_{\Lambda,T}q\,\omega_{\boldsymbol\vartheta}(\eps).
\end{equation}
\end{lemma}

\begin{proof}
Put $R_m=I-P_m$.  On the finite-dimensional block,
\begin{align*}
 \|P_m(\mathsf E_t^\eps-Q_t)P_m\|_1
 &\le\sqrt{r_m}\,\|\mathsf E_t^\eps-Q_t\|_2\\
 &\le\sqrt{r_m}\left(
   \|Q_t^\eps-Q_t\|_2+\|D_t^\eps\|_H^2\right).
\end{align*}
Choose $\beta_0\in(1/4,1/2)$.  By
\eqref{eq:tensor-rate} and \eqref{eq:D-maximal-square-moment},
\begin{align*}
 &\sup_{P\in\mathcal P_{\rm sp}}
 \left\|\sup_{t\le T}
   \|P_m(\mathsf E_t^\eps-Q_t)P_m\|_1\right\|_{L^q(P)}\\
 &\qquad\le C\sqrt{r_m}\,q
   \bigl(\sqrt\eps+\eps^{2\beta_0}\bigr)
 \le C_{\beta_0,T}q\sqrt{r_m\eps}.
\end{align*}
The last inequality uses $2\beta_0>1/2$ and $0<\eps\le T$.

Apply the block estimate \eqref{eq:positive-block-tail} with $P=P_m$.
The spatial profile gives, under every $P\in\mathcal P_{\rm sp}$,
\[
 \Tr Q_T\le T\Lambda,
 \qquad \Tr(R_mQ_TR_m)\le T\vartheta_m.
\]
The resolvent commutes with $R_m$, so
\[
 R_m\mathsf E_t^\eps(X)R_m
 =\mathsf E_t^\eps(R_mX).
\]
Apply \eqref{eq:D-energy-moment} to $X$ and to the projected martingale
$R_mX$, whose bracket trace density is bounded by $\vartheta_m$.  At exponent
$2q$ this yields
\[
 \sup_{P\in\mathcal P_{\rm sp}}
 \|\Tr\mathsf E_T^\eps\|_{L^q(P)}\le C_{\Lambda,T}q,
 \qquad
 \sup_{P\in\mathcal P_{\rm sp}}
 \|\Tr(R_m\mathsf E_T^\eps R_m)\|_{L^q(P)}
 \le C_{\Lambda,T}q\vartheta_m.
\]
Since both operator paths are increasing, \eqref{eq:positive-block-tail}
and H\"older's inequality control their maximal tails by
$C_{\Lambda,T}q\sqrt{\vartheta_m}$.  Combining the finite block and the two
tails, then minimizing over $m$, proves
\eqref{eq:trace-energy-spatial-profile-rate}.
\end{proof}

\begin{corollary}[Trace-energy approximation under a compact covariance envelope]
\label{lem:trace-energy-approximation}
For every $q\ge2$ and $0<\eps\le T$,
\begin{equation}\label{eq:trace-energy-Lq-rate}
 \sup_{P\in\Mmax}
 \left\|\sup_{t\le T}
   \|\mathsf E_t^\eps-Q_t\|_{\mathcal S_1}
 \right\|_{L^q(P)}
 \le C_\Gamma q\,\omega_\Gamma(\eps).
\end{equation}
In particular, $\mathsf E^\eps\to Q$ uniformly in time and in upper
probability for the trace norm.
\end{corollary}

\begin{proof}
The order bound $a_t^P\preceq\Gamma$ induces the spatial profile
$\vartheta_m=\tau_m=\Tr((I-P_m)\Gamma)$ and
$\Lambda=\Tr\Gamma$.  Hence
$\omega_{\boldsymbol\vartheta}=\omega_\Gamma$, and
\Cref{cor:trace-energy-spatial-profile} gives
\eqref{eq:trace-energy-Lq-rate}.
\end{proof}

Define the spatial compactness gauge
\begin{equation}\label{eq:spatial-tail-gauge}
 \mathfrak T_\Gamma(x):=
 \sup_{m\ge0}2^m\|(I-P_m)x\|_{\infty;H}.
\end{equation}

Fix
\begin{equation}\label{eq:core-exponents}
 \frac13<\alpha<\beta<\frac12,
 \qquad
 0<\gamma<\frac12,
 \qquad
 0<\kappa<\min\{\tfrac12-\alpha,\beta-\alpha\}.
\end{equation}
For the dyadic scales $\eps_n=T2^{-n}$ choose a strictly increasing
sequence $(n_j)$ and put $\delta_j:=\eps_{n_j}$ so that
\begin{equation}\label{eq:trace-adapted-scale}
 \omega_\Gamma(\delta_j)+\delta_j^\kappa+\delta_j^{2\beta}
 \le2^{-8j},\qquad j\ge0.
\end{equation}
This is possible by \eqref{eq:trace-energy-modulus}.  Set
\begin{align}
 a_n(x)&:=\|J^{\eps_{n+1}}(x)-J^{\eps_n}(x)\|_{\infty;\mathcal S_2},
 \label{eq:gauge-tensor-increment}\\
 b_n(x)&:=\rho_\alpha\bigl(S_2(Y^{\eps_{n+1}}(x)),
                           S_2(Y^{\eps_n}(x))\bigr),
 \label{eq:gauge-signature-increment}\\
 c_j(x)&:=\|\mathsf E^{\delta_{j+1}}(x)
             -\mathsf E^{\delta_j}(x)\|_{\infty;\mathcal S_1}.
 \label{eq:gauge-trace-energy-increment}
\end{align}
Define
\begin{equation}\label{eq:Hilbert-core-gauge}
 \mathfrak R_\Gamma(x):=
 \max\left\{1,\|x\|_{\beta\text{-H\"ol}},\mathfrak T_\Gamma(x),
 \sup_n\eps_n^{-\gamma/2}a_n(x)^{1/2},
 \sup_n\eps_n^{-\kappa/2}b_n(x)^{1/2},
 \sup_j2^{2j}c_j(x)^{1/2}\right\}
\end{equation}
on the paths for which the resulting limits satisfy
\begin{equation}\label{eq:Hilbert-core-compatibility}
 \mathbb X^S_{s,t}=\widehat J_t-\widehat J_s-X_s\otimes X_{s,t}
                   +\frac12Q_{s,t},
 \qquad
 0\preceq Q_{s,t}\preceq(t-s)\Gamma,
 \qquad
 \mathsf E^{\delta_j}\longrightarrow Q
 \text{ in }C([0,T];\mathcal S_1),
\end{equation}
and set $\mathfrak R_\Gamma=\infty$ otherwise.  Put
\[
 \mathcal C_R:=\{x:\mathfrak R_\Gamma(x)\le R\},
 \qquad
 G_*^\Gamma:=\bigcup_{R\ge1}\mathcal C_R.
\]

\begin{lemma}[Common full-capacity $\Gamma$-compatibility event]
\label{lem:common-Gamma-compatibility-event}
Let
\begin{align*}
 G_\Gamma^{\rm comp}:=G_Q^\Gamma\cap G_\alpha^{\rm wg}
 &\cap\left\{J^{\eps_n}\longrightarrow\widehat J
       \text{ in }C([0,T];\mathcal S_2)\right\}\\
 &\cap\left\{S_2(Y^{\eps_n})\longrightarrow\mathbf X^S
       \text{ in }\rho_\alpha\right\}\\
 &\cap\left\{(\mathsf E^{\delta_j})_j
       \text{ is Cauchy in }C([0,T];\mathcal S_1)\right\}\\
 &\cap\left\{\|\mathsf E^{\delta_j}-Q\|_{\infty;\mathcal S_2}
       \longrightarrow0\right\}.
\end{align*}
Then $G_\Gamma^{\rm comp}$ is Borel and
\[
 c_{\Mmax}\bigl((G_\Gamma^{\rm comp})^c\bigr)=0.
\]
On this single event all relations in
\eqref{eq:Hilbert-core-compatibility} hold simultaneously.
\end{lemma}

\begin{proof}
The first two factors are Borel and full by
\Cref{prop:common-bracket-variation-domain,lem:common-weak-geometric-locus}.
The remaining factors are Borel by the Cauchy criterion in their separable
complete targets and, for uniform convergence, a supremum over rational times;
the rough-path coordinate is Borel by the currying--Lusin--Souslin argument in
\Cref{cor:common-representative}.  Their fullness follows from
\eqref{eq:tensor-rate}, \eqref{eq:dyadic-geometric-resolvent-rate}, and
\Cref{lem:trace-energy-approximation,eq:trace-adapted-scale} by Markov and
capacity Borel--Cantelli.  The trace-energy Cauchy limit is $Q$ because
$\mathcal S_1\hookrightarrow\mathcal S_2$ continuously and the
$\mathcal S_2$ limit is unique.  The algebraic identity follows from the
definitions of the common It\^o and geometric coordinates together with the
declared convergences; $G_Q^\Gamma$ supplies the order inequality.
\end{proof}

\begin{lemma}[Spatial-profile compact containment]
\label{lem:spatial-profile-compact-containment}
Let $\mathcal P_{\rm sp}$ have a profile $(P_m,\vartheta_m)$ relabeled so
that $\vartheta_m\le2^{-8m}$, and put
\[
 \mathfrak T_{\boldsymbol\vartheta}(x)
 :=\sup_{m\ge0}2^m\|(I-P_m)x\|_{\infty;H}.
\]
For every $R<\infty$, the set
\[
 \{x:\|x\|_{\beta\text{-H\"ol}}\le R,
       \mathfrak T_{\boldsymbol\vartheta}(x)\le R\}
\]
is compact in $C_0([0,T];H)$.  Moreover
\begin{equation}\label{eq:spatial-profile-path-tail-moment}
 \sup_{P\in\mathcal P_{\rm sp}}
 \left\|2^m\|(I-P_m)X\|_\infty\right\|_{L^q(P)}
 \le C_{\Lambda,T}\sqrt q\,2^{-3m},
 \qquad q\ge2, m\ge0,
\end{equation}
and therefore
\[
 c_{\mathcal P_{\rm sp}}
 \bigl(\mathfrak T_{\boldsymbol\vartheta}>R\bigr)
 \le Ce^{-cR^2},
 \qquad R\ge1.
\]
\end{lemma}

\begin{proof}
For fixed $m$, the projected paths $P_mx$ lie in a finite-dimensional space
and are uniformly bounded and equicontinuous, hence are relatively compact.
A diagonal subsequence converges for every $m$, while the definition of
$\mathfrak T_{\boldsymbol\vartheta}$ gives the uniform tail estimate
$\|(I-P_m)x\|_\infty\le R2^{-m}$.  This proves relative compactness.  The
uniform limit inherits both the $\beta$-H\"older bound and every spatial-tail
inequality, so the displayed set is closed and hence compact.

Under $P\in\mathcal P_{\rm sp}$, $(I-P_m)X$ is a Hilbert martingale with
bracket trace at most $T\vartheta_m$.  Hilbert maximal BDG gives
\[
 \sup_{P\in\mathcal P_{\rm sp}}
 \left\|2^m\|(I-P_m)X\|_\infty\right\|_{L^q(P)}
 \le C_{\Lambda,T}\sqrt q\,2^m\sqrt{\vartheta_m}
 \le C_{\Lambda,T}\sqrt q\,2^{-3m},
\]
which is \eqref{eq:spatial-profile-path-tail-moment}.  Optimizing Markov's
inequality and summing over $m$ proves the capacity estimate.
\end{proof}

\begin{corollary}[Compact containment from a trace-class envelope]
\label{lem:Hilbert-compact-containment}
For every $R<\infty$, the set
\[
 \{x:\|x\|_{\beta\text{-H\"ol}}\le R,
       \mathfrak T_\Gamma(x)\le R\}
\]
is compact in $C_0([0,T];H)$.  Moreover
\begin{equation}\label{eq:spatial-tail-capacity}
 c_{\Mmax}(\mathfrak T_\Gamma>R)\le Ce^{-cR^2},
 \qquad R\ge1.
\end{equation}
\end{corollary}

\begin{proof}
The envelope induces the profile
$\vartheta_m=\tau_m=\Tr((I-P_m)\Gamma)$, for which
$\mathfrak T_{\boldsymbol\vartheta}=\mathfrak T_\Gamma$.
Apply \Cref{lem:spatial-profile-compact-containment}; its capacity estimate is
exactly \eqref{eq:spatial-tail-capacity} for $\mathcal P_{\rm sp}=\Mmax$.
\end{proof}

\begin{lemma}[Spatially trace-tight upgrade from Hilbert--Schmidt to trace norm]
\label{lem:dominated-trace-continuity}
Let $A_n,A\in\mathcal S_1(H)_+$ satisfy
\[
 \sup_n\Tr A_n+\Tr A<\infty.
\]
Suppose there are finite-rank orthogonal projections $P_m\uparrow I$ such that,
with $R_m=I-P_m$,
\begin{equation}\label{eq:trace-tight-positive-family}
 \lim_{m\to\infty}
 \left(\sup_n\Tr(R_mA_nR_m)+\Tr(R_mAR_m)\right)=0.
\end{equation}
If $\|A_n-A\|_{\mathcal S_2}\to0$, then
$\|A_n-A\|_{\mathcal S_1}\to0$.  The conclusion is uniform for families
with a common trace bound, a common tail modulus in
\eqref{eq:trace-tight-positive-family}, and a common Hilbert--Schmidt
convergence modulus.  In particular, the hypothesis holds whenever
$0\preceq A_n,A\preceq C\Gamma$ for one
$\Gamma\in\mathcal S_1(H)_+$.
\end{lemma}

\begin{proof}
The block estimate \eqref{eq:positive-block-tail}, the common trace bound, and
\eqref{eq:trace-tight-positive-family}
make the tails of $A_n$ and $A$ uniformly small.  On the finite-dimensional
range of $P_m$, Hilbert--Schmidt convergence implies trace-norm convergence.
The three-term estimate
\[
 \|A_n-A\|_1
 \le\|A_n-P_mA_nP_m\|_1
   +\|P_m(A_n-A)P_m\|_1
   +\|A-P_mAP_m\|_1
\]
proves the assertion by first fixing $m$ and then letting $n\to\infty$.
If $A_n,A\preceq C\Gamma$, the required tail modulus follows from
$\Tr(R_mA_nR_m)\vee\Tr(R_mAR_m)\le C\Tr(R_m\Gamma R_m)$.
\end{proof}

\begin{lemma}[Uniform finite-scale closure principle]
\label{lem:uniform-finite-scale-closure}
Let $X$ be a topological space, let $(Y,d)$ be complete, and let
$F_n:X\to Y$ be continuous.  Suppose $x^k\to x$ and that, for some
$r(n)\downarrow0$,
\begin{equation}\label{eq:uniform-finite-scale-Cauchy-tail}
 \sup_k\sup_{m\ge n}d(F_m(x^k),F_n(x^k))\le r(n).
\end{equation}
Then $(F_n(x))_n$ is Cauchy, and if $F_\infty(x^k)$ and $F_\infty(x)$ denote
the corresponding limits, then
\[
 F_\infty(x^k)\longrightarrow F_\infty(x).
\]
More generally, if \eqref{eq:uniform-finite-scale-Cauchy-tail} holds uniformly
on a set $K$, then $F_n$ converges uniformly on $K$ to a continuous limit; the
same conclusion holds on the closure of $K$ whenever the same tail modulus is
stable under limits.
\end{lemma}

\begin{proof}
For fixed $m\ge n$, continuity gives
$d(F_m(x),F_n(x))\le r(n)$ after passing $k\to\infty$, so $(F_n(x))$ is
Cauchy.  For every $n$,
\[
 d(F_\infty(x^k),F_\infty(x))
 \le r(n)+d(F_n(x^k),F_n(x))+r(n).
\]
First let $k\to\infty$ and then $n\to\infty$.  The uniform statement is the
same estimate with the supremum over $K$.
\end{proof}

\begin{lemma}[Borel and closed compatibility locus]
\label{lem:core-compatibility-locus}
The set of paths on which the three finite-scale sequences in
\eqref{eq:gauge-tensor-increment}--\eqref{eq:gauge-trace-energy-increment}
are Cauchy and satisfy \eqref{eq:Hilbert-core-compatibility} is Borel.  If
$x^k\to x$ uniformly and
\[
 \sup_k\mathfrak R_\Gamma(x^k)<\infty,
\]
then $x$ belongs to the same compatibility locus and the primitive,
signature, and trace-energy limits of $x^k$ converge uniformly, without
subsequence extraction, to the corresponding limits of $x$.  In particular,
the algebraic compatibility relations and the covariance envelope pass to
the limit.
\end{lemma}

\begin{proof}
Finite-scale continuity and the Cauchy criterion make the three Cauchy loci
Borel.  Their targets are Polish, so on each locus the limit map is Borel;
extend it by a fixed base point off that locus.  After composing the
trace-class limit with the continuous injection
\[
 C([0,T];\mathcal S_1)\hookrightarrow
 C([0,T];\mathcal S_2),
\]
equality with the declared $Q$-coordinate is therefore Borel.  The algebraic
identities and order inequalities are inverse images of closed sets at
rational time pairs; continuity extends them to all pairs.  Hence the full
compatibility locus is Borel.

Suppose now that $x^k\to x$ uniformly and
$\mathfrak R_\Gamma(x^k)\le R$.  For the primitive, signature, and trace-energy
approximants, the weighted Cauchy coordinates provide a deterministic tail
modulus uniform in $k$, while every fixed finite-scale map is raw-uniform
continuous.  Apply \Cref{lem:uniform-finite-scale-closure} in the corresponding
complete targets.  It yields existence of the limits at $x$ and convergence,
without subsequence extraction, of the limiting primitive, signature, and
trace-energy paths.  The trace-energy application is taken directly in
$C([0,T];\mathcal S_1)$.  At finite scale,
\[
 \mathsf E^\delta=Q^\delta-D^\delta\otimes D^\delta.
\]
The common $\beta$-H\"older bound and
\Cref{lem:resolvent-boundary-layer} make the last term vanish uniformly as
$\delta\downarrow0$.  Hence the trace-energy limit agrees in
$\mathcal S_2$ with the algebraic covariance read from the primitive and
signature limits.  Its $\mathcal S_1$ representative is therefore the same
operator; equivalently one may use
\Cref{lem:dominated-trace-continuity} under the common $\Gamma$-envelope.
Positivity and $0\preceq Q_{s,t}\preceq(t-s)\Gamma$ are closed in trace
norm.  All relations in \eqref{eq:Hilbert-core-compatibility} consequently
pass to the limit.
\end{proof}

\begin{proposition}[Gauge-wise stopping, shift, and concatenation estimates]
\label{prop:gaugewise-restart}
Let $x\in\mathcal C_R$, $y\in\mathcal C_S$, and $s\in[0,T]$.  There is a
constant $C$, depending only on the fixed exponents and $T$, such that
\begin{align}
 \|r_sx\|_{\beta\text{-H\"ol}}
 &\le R,
 &\|\theta_sx\|_{\beta\text{-H\"ol}}
 &\le2R,
 &\|x\otimes_sy\|_{\beta\text{-H\"ol}}
 &\le C(R+S),
 \label{eq:gauge-path-operations}\\
 \mathfrak T_\Gamma(r_sx)
 &\le\mathfrak T_\Gamma(x),
 &\mathfrak T_\Gamma(\theta_sx)
 &\le2\mathfrak T_\Gamma(x),
 &\mathfrak T_\Gamma(x\otimes_sy)
 &\le\mathfrak T_\Gamma(x)+\mathfrak T_\Gamma(y).
 \label{eq:gauge-spatial-operations}
\end{align}
For every $n,j$,
\begin{align}
 a_n(r_sx)
 &\le a_n(x),
 \notag\\
 a_n(\theta_sx)
 &\le C\bigl(a_n(x)+R^2\eps_n^{2\beta}\bigr),
 \notag\\
 a_n(x\otimes_sy)
 &\le C\bigl(a_n(x)+a_n(y)+(R+S)^2\eps_n^{2\beta}\bigr),
 \label{eq:gauge-a-operations}\\
 b_n(r_sx)+b_n(\theta_sx)
 &\le C\bigl(b_n(x)+R(1+R)\eps_n^{\beta-\alpha}\bigr),
 \notag\\
 b_n(x\otimes_sy)
 &\le C\bigl((1+S)b_n(x)+(1+R)b_n(y)
       +(1+R+S)^2\eps_n^{\beta-\alpha}\bigr),
 \label{eq:gauge-b-operations}\\
 c_j(r_sx)+c_j(\theta_sx)
 &\le C\bigl(c_j(x)+R^2\delta_j^{2\beta}\bigr),
 \notag\\
 c_j(x\otimes_sy)
 &\le C\bigl(c_j(x)+c_j(y)
       +(R^2+RS+S^2)\delta_j^{2\beta}\bigr).
 \label{eq:gauge-c-operations}
\end{align}
The estimates
\eqref{eq:gauge-path-operations}--\eqref{eq:gauge-c-operations} are
envelope-free: they remain valid, with the same constants, for any two path
classes on which the corresponding numerical gauge coordinates are bounded by
$R,S$, and with $\mathfrak T_\Gamma$ replaced by the spatial gauge associated
with the same projections $(P_m)$.  Compatibility-locus preservation follows
separately from the algebraic identities used below.
Consequently, with
\[
 \Gamma_*(R,S):=1+C_*(1+R+S)^2,
\]
one has, uniformly in $s$,
\[
 r_s\mathcal C_R\cup\theta_s\mathcal C_R
 \subset\mathcal C_{\Gamma_*(R,0)},
 \qquad
 x\otimes_sy\in\mathcal C_{\Gamma_*(R,S)}.
\]
The stopping inclusion remains valid after substituting a Borel stopping time
for $s$.
\end{proposition}

\begin{proof}
Stopping is H\"older-contractive; a fixed-horizon shift has one
constant-extension boundary, and a seam-crossing increment splits at $s$.
This proves \eqref{eq:gauge-path-operations}.  Since $P_m$ commutes with all
three path operations, the triangle inequality gives
\eqref{eq:gauge-spatial-operations}.

Exact stopping of $J^\eps$, the shift formula
\eqref{eq:resolvent-J-seam}, and the concatenation identity
\eqref{eq:finite-resolvent-J-concatenation} give
\eqref{eq:gauge-a-operations}, using
\eqref{eq:deterministic-resolvent-bounds} for the boundary terms.

For any of the three operations $\mathsf O_s$, equations
\eqref{eq:resolvent-stopping-layer},
\eqref{eq:fixed-horizon-resolvent-shift}, and
\eqref{eq:resolvent-concatenation-layer} show that
$Y^\eps(\mathsf O_sx)$ differs from $\mathsf O_sY^\eps(x)$ by at most two
exponential layers, of amplitude $CR\eps^\beta$ (one input) or
$C(R+S)\eps^\beta$ (two inputs).  Thus
\Cref{lem:resolvent-boundary-layer,lem:resolvent-seam} gives, uniformly in
$s$,
\begin{align*}
 \rho_\alpha\bigl(S_2(Y^\eps(r_sx)),
                  r_sS_2(Y^\eps(x))\bigr)
 &\le CR(1+R)\eps^{\beta-\alpha},\\
 \rho_\alpha\bigl(S_2(Y^\eps(\theta_sx)),
                  \theta_sS_2(Y^\eps(x))\bigr)
 &\le CR(1+R)\eps^{\beta-\alpha},\\
 \rho_\alpha\bigl(S_2(Y^\eps(x\otimes_sy)),
                  S_2(Y^\eps(x))\otimes_sS_2(Y^\eps(y))\bigr)
 &\le C(1+R+S)^2\eps^{\beta-\alpha}.
\end{align*}
Contractivity of stopping, the adjacent-scale triangle inequality, and the
explicit cross term in step-two concatenation give
\[
 b_n(r_sx)+b_n(\theta_sx)
 \le C\{b_n(x)+R(1+R)\eps_n^{\beta-\alpha}\},
\]
while local Lipschitz continuity on the two rough balls supplies the factors
$1+S$ and $1+R$ in the two-input bound
\eqref{eq:gauge-b-operations}.  These factors are the contribution of the
cross increment and cannot be dropped.  All bounds are uniform in seam and
scale.  \Cref{lem:trace-energy-seams} at adjacent scales gives
\eqref{eq:gauge-c-operations}.

On $\mathcal C_R$,
\[
 a_n(x)\le R^2\eps_n^\gamma,
 \qquad
 b_n(x)\le R^2\eps_n^\kappa,
 \qquad
 c_j(x)\le R^22^{-4j}.
\]
The seam terms are summable because
\[
 \beta-\frac\gamma2>0,
 \qquad
 \beta-\alpha-\kappa>0,
 \qquad
 2^{2j}\delta_j^\beta\le2^{-2j}
\]
by \eqref{eq:core-exponents} and \eqref{eq:trace-adapted-scale}.  Applying
$\sqrt{u+v+w}\le\sqrt u+\sqrt v+\sqrt w$ in
\eqref{eq:gauge-a-operations}--\eqref{eq:gauge-c-operations} gives, uniformly
in seam and scale,
\begin{align*}
 \sup_n\eps_n^{-\gamma/2}a_n(x\otimes_sy)^{1/2}
 &\le C_T(1+R+S),\\
 \sup_n\eps_n^{-\kappa/2}b_n(x\otimes_sy)^{1/2}
 &\le C_T(1+R+S)^{3/2},\\
 \sup_j2^{2j}c_j(x\otimes_sy)^{1/2}
 &\le C(1+R+S).
\end{align*}
The respective seam contributions are
$(1+R+S)\eps_n^{\beta-\gamma/2}$,
$(1+R+S)\eps_n^{(\beta-\alpha-\kappa)/2}$, and
$(1+R+S)2^{2j}\delta_j^\beta$, so the preceding gaps control their suprema.
Together with \eqref{eq:gauge-path-operations} and
\eqref{eq:gauge-spatial-operations}, every coordinate of
$\mathfrak R_\Gamma(x\otimes_sy)$ is therefore bounded by
$C_T(1+R+S)^2$.  The same calculation with one input gives the stopping and
shift bounds.  Enlarging one fixed constant $C_*$ gives
\[
 \mathfrak R_\Gamma(r_sx)\vee
 \mathfrak R_\Gamma(\theta_sx)
 \le\Gamma_*(R,0),
 \qquad
 \mathfrak R_\Gamma(x\otimes_sy)\le\Gamma_*(R,S),
\]
which is the asserted quadratic reindexing.

\Cref{thm:shift-equivariant-enhancement} preserves primitive/signature
compatibility and the covariance envelope, while the trace-energy seam errors
vanish along $(\delta_j)$.  Thus all operations preserve the compatibility
locus of \Cref{lem:core-compatibility-locus}.  Since the estimates are uniform
in the deterministic seam and $x\mapsto r_{\tau(x)}x$ is Borel for a Borel
stopping time $\tau$, substitution $s=\tau(x)$ proves the final assertion
without an uncountable intersection.
\end{proof}

\begin{theorem}[Compact perfection under uniform spatial trace tightness]
\label{thm:spatial-trace-tight-cores}
Let $\mathcal P_{\rm sp}\subset\mathfrak S_{\Lambda,0}^{0,T}(H)$ have a
uniform spatial trace-tightness profile
$(P_m,\vartheta_m)$, relabeled so that
$\vartheta_m\le2^{-8m}$.  Choose a scale $\delta_j\downarrow0$ such that
\[
 \omega_{\boldsymbol\vartheta}(\delta_j)
 +\delta_j^\kappa+\delta_j^{2\beta}\le2^{-8j}.
\]
Put
\[
 \mathfrak T_{\boldsymbol\vartheta}(x)
 :=\sup_{m\ge0}2^m\|(I-P_m)x\|_{\infty;H},
\]
let $a_n,b_n,c_j$ be the finite-scale increments in
\eqref{eq:gauge-tensor-increment}--\eqref{eq:gauge-trace-energy-increment},
and define
\begin{equation}\label{eq:spatial-profile-core-gauge}
 \mathfrak R_{\boldsymbol\vartheta}(x)
 :=\max\left\{1,\|x\|_{\beta\text{-H\"ol}},
 \mathfrak T_{\boldsymbol\vartheta}(x),
 \sup_n\eps_n^{-\gamma/2}a_n(x)^{1/2},
 \sup_n\eps_n^{-\kappa/2}b_n(x)^{1/2},
 \sup_j2^{2j}c_j(x)^{1/2}\right\}
\end{equation}
on the locus where the finite-scale limits exist,
$\mathsf E^{\delta_j}\to Q$ in
$C([0,T];\mathcal S_1)$, and the following profile compatibility replaces
the operator-order clause in \eqref{eq:Hilbert-core-compatibility}:
\begin{equation}\label{eq:spatial-profile-Q-compatibility}
 Q_{s,t}\in\mathcal S_1(H)_+,
 \qquad \Tr Q_{s,t}\le\Lambda(t-s),
 \qquad \Tr(R_mQ_{s,t}R_m)\le\vartheta_m(t-s)
 \quad(m\ge0).
\end{equation}
Set $\mathfrak R_{\boldsymbol\vartheta}=\infty$ off this locus and let
$\mathcal C_R^{\boldsymbol\vartheta}$ be its sublevels.  Then there are
constants $C,c,C_*>0$, depending only on the fixed parameters and the spatial
profile, such that the following hold.
\begin{enumerate}[label=\textup{(\roman*)},leftmargin=2.5em]
\item For every $R\ge1$, $\mathcal C_R^{\boldsymbol\vartheta}$ is compact in
the raw uniform topology and
\[
 c_{\mathcal P_{\rm sp}}
 \bigl((\mathcal C_R^{\boldsymbol\vartheta})^c\bigr)
 \le Ce^{-cR^2}.
\]
\item On every $\mathcal C_R^{\boldsymbol\vartheta}$ and for every
$0<\zeta<1$, the map
\[
 x\longmapsto (X(x),A(x),Q(x))
\]
is continuous with values in
\[
 C_0([0,T];H)\times C^{0,2\alpha}(\mathcal S_2(H)_{\rm sk})
 \times C^{0,\zeta}(\mathcal S_1(H)_{\rm sa}).
\]
For every $\nu>1$, $x\mapsto Q(x)$ is also continuous in the
$\nu$-variation topology, and every covariance path satisfies
$[Q(x)]_{1;\mathcal S_1}\le\Lambda$.
\item With $\Gamma_*(R,S)=1+C_*(1+R+S)^2$, uniformly in the seam $s$,
\[
 r_s\mathcal C_R^{\boldsymbol\vartheta}
 \cup\theta_s\mathcal C_R^{\boldsymbol\vartheta}
 \subset\mathcal C_{\Gamma_*(R,0)}^{\boldsymbol\vartheta},
 \qquad
 x\in\mathcal C_R^{\boldsymbol\vartheta},\ y\in\mathcal C_S^{\boldsymbol\vartheta}
 \Longrightarrow
 x\otimes_sy\in\mathcal C_{\Gamma_*(R,S)}^{\boldsymbol\vartheta}.
\]
The stopping inclusion remains valid for Borel stopping times, and the limiting
state obeys the stopping, shift, and Chen concatenation identities on the
resulting full-capacity domain.
\item The finite-scale approximations converge uniformly on each core:
\begin{align*}
 \sup_{x\in\mathcal C_R^{\boldsymbol\vartheta}}
 \|J^{\eps_n}(x)-\widehat J(x)\|_\infty
 &\le C_R\eps_n^\gamma,\\
 \sup_{x\in\mathcal C_R^{\boldsymbol\vartheta}}
 \rho_\alpha(S_2(Y^{\eps_n}(x)),\mathbf X^S(x))
 &\le C_R\eps_n^\kappa,\\
 \sup_{x\in\mathcal C_R^{\boldsymbol\vartheta}}
 \|\mathsf E^{\delta_j}(x)-Q(x)\|_{\infty;\mathcal S_1}
 &\le C_R2^{-4j}.
\end{align*}
\end{enumerate}
The fixed-envelope result is recovered by taking
$\vartheta_m=\Tr((I-P_m)\Gamma)$ and then imposing the closed
operator-order restriction made explicit below.
\end{theorem}

\begin{proof}
Steps 1--5 depend only on the spatial trace profile.  Operator domination
enters in the fixed-envelope specialization.

\emph{Step 1: Borel compatibility and upper-capacity gauge bounds.}
The Borel argument of \Cref{lem:core-compatibility-locus} applies verbatim:
replace the operator-order test by the countable rational-time profile tests,
using trace-norm continuity of
$A\mapsto\Tr(R_mAR_m)$.  Indeed, positivity and the total trace bound on
rational increments first give
\[
 \|Q_t-Q_s\|_1=\Tr(Q_t-Q_s)\le\Lambda(t-s),
 \qquad s<t\ \text{rational}.
\]
The rational $\mathcal S_1$-Lipschitz bound has a unique trace-norm continuous
extension.  Its image in $\mathcal S_2$ equals the declared continuous
covariance coordinate $Q$ by uniqueness, and trace-norm continuity then
extends positivity and every tail inequality to arbitrary times.  Thus the
rational tests define exactly the declared profile locus.

We first verify that the compatibility locus on which the gauge is finite is
common and full.  The modelwise bracket identity gives, simultaneously for
every rational $s<t$ and every $m$,
\[
 \Tr(R_mQ_{s,t}R_m)
 =\int_s^t\Tr(R_ma_r^PR_m)\,\dd r
 \le\vartheta_m(t-s),
 \qquad
 \Tr Q_{s,t}\le\Lambda(t-s).
\]
By \eqref{eq:trace-energy-spatial-profile-rate}, Markov's inequality, and
capacity Borel--Cantelli, the chosen sequence
$(\mathsf E^{\delta_j})_j$ converges in
$C([0,T];\mathcal S_1)$ outside one polar set.  Its image under the continuous
inclusion $\mathcal S_1\hookrightarrow\mathcal S_2$ is the already selected
common covariance $Q$, so uniqueness of the $\mathcal S_2$ limit identifies
the trace-class limit with $Q$.  Thus $Q$ is trace-norm continuous on this
same full-capacity domain.  The functionals
$A\mapsto\Tr(R_mAR_m)$ are trace-norm continuous, and the rational-time
inequalities above therefore extend to every $s<t$ without enlarging the
exceptional set.  Hence the full compatibility locus has full capacity.

The first-level H\"older estimate, the tensor increment estimate for $a_n$,
and the geometric resolvent estimate for $b_n$ use only the trace clock
$\Lambda$ and therefore hold uniformly on $\mathcal P_{\rm sp}$.  Choose
$\kappa<\chi<1/2-\alpha$.  The tensor and geometric estimates give
\begin{align*}
 \sup_{P\in\mathcal P_{\rm sp}}
 \bigl\|\eps_n^{-\gamma/2}a_n^{1/2}\bigr\|_{L^{2q}(P)}
 &\le C\sqrt q\,\eps_n^{(1-2\gamma)/4},\\
 \sup_{P\in\mathcal P_{\rm sp}}
 \bigl\|\eps_n^{-\kappa/2}b_n^{1/2}\bigr\|_{L^{2q}(P)}
 &\le C\sqrt q\,\eps_n^{(\chi-\kappa)/2}.
\end{align*}
Both exponents are positive, so geometric summation yields the
profile-independent estimate
\begin{equation}\label{eq:master-two-gauge-moments}
 \sup_{P\in\mathcal P_{\rm sp}}
 \left\|\sup_n\eps_n^{-\gamma/2}a_n^{1/2}
       +\sup_n\eps_n^{-\kappa/2}b_n^{1/2}
 \right\|_{L^{2q}(P)}
 \le C\sqrt q.
\end{equation}
The spatial tail estimate
\eqref{eq:spatial-profile-path-tail-moment} from
\Cref{lem:spatial-profile-compact-containment} shows that
$\mathfrak T_{\boldsymbol\vartheta}(X)$ has a uniform $C\sqrt q$ moment
bound.  The trace-energy estimate
\Cref{cor:trace-energy-spatial-profile} and the adapted choice of
$(\delta_j)$ give
\[
 \sup_{P\in\mathcal P_{\rm sp}}\|c_j\|_{L^q(P)}
 \le Cq2^{-8j},
 \qquad
 \sup_{P\in\mathcal P_{\rm sp}}
 \left\|\sup_j2^{2j}c_j^{1/2}\right\|_{L^{2q}(P)}
 \le C\sqrt q.
\]
Together with \eqref{eq:spatial-profile-path-tail-moment} and the
first-level H\"older moments this now yields, on a common full-capacity
domain,
\[
 \sup_{P\in\mathcal P_{\rm sp}}
 \|\mathfrak R_{\boldsymbol\vartheta}(X)\|_{L^{2q}(P)}
 \le C\sqrt q,
 \qquad q\ge q_0.
\]
Optimizing Markov's inequality in $q$ therefore gives
\begin{equation}\label{eq:spatial-profile-core-tail}
 c_{\mathcal P_{\rm sp}}
 \bigl((\mathcal C_R^{\boldsymbol\vartheta})^c\bigr)
 \le Ce^{-cR^2},
 \qquad R\ge1.
\end{equation}

\emph{Step 2: compactness and corewise continuity.}
Let $x^k\in\mathcal C_R^{\boldsymbol\vartheta}$.  The compact-containment
part of \Cref{lem:spatial-profile-compact-containment} gives, along a
subsequence, $x^k\to x$ uniformly.  Fixed-scale continuity and the weighted
Cauchy tails in \eqref{eq:spatial-profile-core-gauge} let
\Cref{lem:uniform-finite-scale-closure} pass the primitive, lift, and
trace-energy limits to $x$ in their three complete targets.  The last
convergence is directly in trace norm, and the closed profile class
$\mathscr Q_{\Lambda,\boldsymbol\vartheta}$ of
\Cref{lem:covariance-profile-compact-restart} preserves all covariance
conditions.  For every $x^k$ the remaining compatibility identity is
\[
 \mathbb X^S_{s,t}(x^k)
 =\widehat J_t(x^k)-\widehat J_s(x^k)
  -X_s(x^k)\otimes X_{s,t}(x^k)+\frac12Q_{s,t}(x^k).
\]
The four coordinates converge in their uniform $H$, $\mathcal S_2$, and
$\mathcal S_1$ topologies, so this identity passes to $x$.  Thus the sublevel
is compact.  Its gauge bounds telescope:
\[
 \sum_{k\ge n}a_k\le C_R\eps_n^\gamma,
 \qquad
 \sum_{k\ge n}b_k\le C_R\eps_n^\kappa,
 \qquad
 \sum_{\ell\ge j}c_\ell\le C_R2^{-4j}.
\]
These are the rates in part~\textup{(iv)} and, by the same closure argument,
give corewise continuity.  Since profile covariances are trace-norm Lipschitz
with constant $\Lambda$, uniform convergence and the interpolation inequality
\[
 [Q^k-Q]_{\theta;\mathcal S_1}
 \le
 \bigl(2\|Q^k-Q\|_{\infty;\mathcal S_1}\bigr)^{1-\theta}
 [Q^k-Q]_{1;\mathcal S_1}^{\theta}
 \le
 (2\Lambda)^\theta
 \bigl(2\|Q^k-Q\|_{\infty;\mathcal S_1}\bigr)^{1-\theta}
\]
give continuity in every $C^{0,\theta}$, $0<\theta<1$; continuity in every
$\nu$-variation topology, $\nu>1$, follows from
\Cref{lem:covariance-profile-compact-restart}.

\emph{Step 3: restart stability.}
The first-level spatial gauge is compatible with the path operations because
$P_m$ acts only on the Hilbert coordinate:
\[
 \mathfrak T_{\boldsymbol\vartheta}(r_sx)
 \le\mathfrak T_{\boldsymbol\vartheta}(x),
 \quad
 \mathfrak T_{\boldsymbol\vartheta}(\theta_sx)
 \le2\mathfrak T_{\boldsymbol\vartheta}(x),
 \quad
 \mathfrak T_{\boldsymbol\vartheta}(x\otimes_sy)
 \le\mathfrak T_{\boldsymbol\vartheta}(x)
     +\mathfrak T_{\boldsymbol\vartheta}(y).
\]
The envelope-free clause of \Cref{prop:gaugewise-restart} and
\Cref{lem:trace-energy-seams} give the $a_n,b_n,c_j$ bounds, with no
$\Gamma$-dependent constant; the profile class itself is restart-closed by
\Cref{lem:covariance-profile-compact-restart}.  The weighted estimates yield
the uniform polynomial reindexing in part~\textup{(iii)}.  Uniformity in $s$
permits pointwise substitution of a Borel stopping time, and passage to the
limits gives the exact stopping, shift, and Chen identities.

Finally, if $a_t^P\preceq\Gamma$, then
\[
 \Tr((I-P_m)a_t^P(I-P_m))
 \le\Tr((I-P_m)\Gamma),
\]
so the $\Gamma$-envelope construction is the stated special case.
\end{proof}

For the remainder of the spatial-profile dynamics, write
\begin{equation}\label{eq:spatial-profile-full-domain}
 G_*^{\boldsymbol\vartheta}
 :=\bigcup_{n\ge1}\mathcal C_n^{\boldsymbol\vartheta}.
\end{equation}
By part~\textup{(i)}, its complement is
$\mathcal P_{\rm sp}$-polar, and parts~\textup{(ii)}--\textup{(iii)} give a
restart-stable compact exhaustion with corewise continuous state.  For a
$\Gamma$-induced profile it may be intersected with $G_*^\Gamma$.

\begin{corollary}[Fixed-envelope restriction of the spatial-profile capacity cores]
\label{thm:master-causal-cores}
The sublevels $\mathcal C_R$ of \eqref{eq:Hilbert-core-gauge} are
compact in the raw uniform topology, and
\begin{equation}\label{eq:master-core-tail}
 c_{\Mmax}(\mathcal C_R^c)\le Ce^{-cR^2}.
\end{equation}
On every $\mathcal C_R$, for every $0<\theta<1$, the map
\[
 x\longmapsto\bigl(X(x),A(x),Q(x)\bigr)
\]
is continuous in
\[
 C_0([0,T];H)\times C^{0,2\alpha}(\mathcal S_2(H)_{\rm sk})
 \times C^{0,\theta}(\mathcal S_1(H)_{\rm sa}).
\]
For every $\nu>1$, the covariance component is also continuous as a map
\begin{equation}\label{eq:core-Q-nu-variation-continuity}
 x\longmapsto Q(x):\mathcal C_R
 \longrightarrow C^{\nu\text{-var}}([0,T];\mathcal S_1(H)_{\rm sa}).
\end{equation}
Every covariance path in the core is trace-norm Lipschitz, with
\begin{equation}\label{eq:core-uniform-Q-Lipschitz}
 [Q(x)]_{1;\mathcal S_1}\le\Tr\Gamma,
 \qquad x\in\mathcal C_R.
\end{equation}
Moreover,
\begin{align}
 \sup_{x\in\mathcal C_R}
 \|J^{\eps_n}(x)-\widehat J(x)\|_\infty
 &\le C_R\eps_n^\gamma,
 \label{eq:J-corewise-rate}\\
 \sup_{x\in\mathcal C_R}
 \rho_\alpha(S_2(Y^{\eps_n}(x)),\mathbf X^S(x))
 &\le C_R\eps_n^\kappa,
 \label{eq:geometric-corewise-rate}\\
 \sup_{x\in\mathcal C_R}
 \|\mathsf E^{\delta_j}(x)-Q(x)\|_{\infty;\mathcal S_1}
 &\le C_R2^{-4j}.
 \label{eq:trace-energy-corewise-rate}
\end{align}
Furthermore there is $C_*$ such that, with
\begin{equation}\label{eq:master-reindex}
 \Gamma_*(R,S)=1+C_*(1+R+S)^2.
\end{equation}
uniformly in the seam $s$,
\begin{equation}\label{eq:master-stop-shift}
 r_s\mathcal C_R\cup\theta_s\mathcal C_R
 \subset\mathcal C_{\Gamma_*(R,0)},
 \qquad
 x\in\mathcal C_R,\ y\in\mathcal C_S
 \Longrightarrow x\otimes_sy\in\mathcal C_{\Gamma_*(R,S)}.
\end{equation}
The first inclusion remains valid after pointwise substitution of a Borel
stopping time.  The limiting state obeys the exact stopping, shift, and Chen
concatenation identities on $G_*^\Gamma$.
\end{corollary}

\begin{proof}
Set
\[
 \vartheta_m:=\Tr((I-P_m)\Gamma),
 \qquad \Lambda:=\Tr\Gamma.
\]
Then $\omega_{\boldsymbol\vartheta}=\omega_\Gamma$, so the scale
\eqref{eq:trace-adapted-scale} is admissible for
\Cref{thm:spatial-trace-tight-cores}.  Set
$\mathcal F_\Gamma:=G_\Gamma^{\rm comp}$, the global (not fixed-radius)
$\Gamma$-compatibility locus.  Then
\[
 c_{\Mmax}(\mathcal F_\Gamma^c)=0
\]
by \Cref{lem:common-Gamma-compatibility-event}.  On $\mathcal F_\Gamma$ the
numerical coordinates of
$\mathfrak R_\Gamma$ and
$\mathfrak R_{\boldsymbol\vartheta}$ coincide, while
$0\preceq Q_{s,t}\preceq(t-s)\Gamma$ implies all profile tail inequalities.
Hence
\[
\mathcal C_R
 =\mathcal C_R^{\boldsymbol\vartheta}\cap\mathcal F_\Gamma
\]
is the closed operator-order subfamily of the profile core.  Indeed, the
profile core already enforces the common convergence and algebraic
compatibility conditions, while the remaining operator-order inequalities are
closed under trace-norm convergence.  Moreover,
\[
 c_{\Mmax}(\mathcal C_R^c)
 \le c_{\Mmax}\bigl((\mathcal C_R^{\boldsymbol\vartheta})^c\bigr)
    +c_{\Mmax}(\mathcal F_\Gamma^c)
 \le Ce^{-cR^2},
\]
which proves \eqref{eq:master-core-tail}.  The exhaustion
$\bigcup_R\mathcal C_R$ has full capacity, while each fixed sublevel has the
displayed exponential complement bound.  The profile theorem gives
compactness, both stated continuities,
and the three rates.  Its Lipschitz constant
$\Lambda=\Tr\Gamma$ proves
\eqref{eq:core-uniform-Q-Lipschitz}.

For inputs in the fixed-envelope cores, the proof of
\Cref{prop:gaugewise-restart} preserves the compatibility identities and the
operator envelope, while its numerical estimates give the stated quantitative
reindexing.  Hence
\eqref{eq:master-stop-shift} and its Borel stopping-time version follow.
The exact limiting identities are those of
\Cref{thm:shift-equivariant-enhancement}.  Thus the fixed-envelope result is a
closed fixed-envelope restriction of the full-capacity spatial-profile
exhaustion.
\end{proof}

\begin{corollary}[Corewise trace-variation convergence]
\label{cor:core-bracket-variation}
On $G_*^\Gamma$, $Q$ has finite variation in trace norm and
\[
 \|Q_t-Q_s\|_1\le(t-s)\Tr\Gamma.
\]
For every $R<\infty$ and $1<\nu<\infty$,
\begin{equation}\label{eq:trace-energy-variation-rate}
 \sup_{x\in\mathcal C_R}
 \|\mathsf E^{\delta_j}(x)-Q(x)\|_{\nu\text{-var};\mathcal S_1}
 \le C_{R,\nu}2^{-4j(1-1/\nu)}.
\end{equation}
More generally, if two Banach-valued paths have uniformly bounded total
variation, then their $\nu$-variation distance is bounded by a constant
times their uniform distance to the power $1-1/\nu$.
\end{corollary}

\begin{proof}
The first assertion follows from
$0\preceq Q_{s,t}\preceq(t-s)\Gamma$.  Each
$\mathsf E^{\delta_j}$ is positive and increasing, so its trace variation is
$\Tr\mathsf E_T^{\delta_j}$.  The uniform trace convergence
\eqref{eq:trace-energy-corewise-rate} bounds these terminal traces uniformly
on $\mathcal C_R$.  The interpolation estimate
\eqref{eq:uniform-variation-interpolation}, applied to
$f=\mathsf E^{\delta_j}-Q$, proves
\eqref{eq:trace-energy-variation-rate}.
\end{proof}

\part{Propagation consequences and backward state selection}

\section{Propagation consequences: higher signatures and flows}
\label{sec:propagation-consequences}

\subsection{Higher-signature propagation}
\label{sec:all-order-generation}

Once the quadratic state is commonly realized, higher signatures are
stable-causal readouts on the full operational state space.
Under the vector-field assumptions below, rough--Young flows are continuous
on the stated $\alpha$-regular realization cores; they become full-state
stable-causal readouts when $p>1/\eta$.  The Euler field of
\Cref{sec:Euler-state-interface} is instead constructed by first-level
completion and agrees with the corewise propagated flow under the additional
regularity in \Cref{thm:Euler-state-interface}\textup{(iii)}.  For
$1/3<\eta<1/2$, degree two is the last primitive rough level, so every fixed
higher geometric or It\^o level is generated continuously.  We work on
\[
 \mathsf Z_{\eta;\mathcal S_2,1}(H).
\]

\subsubsection{Hilbert tensor algebras and mixed rough--Young equations}

For $N\ge0$, put $H^{\otimes_2 0}:=\R$ and let
\[
 T^{(N)}_2(H):=\bigoplus_{k=0}^N H^{\otimes_2 k}
\]
be the truncated Hilbert tensor algebra with the truncated concatenation
product.  We write $\pi_k$ for the projection onto degree $k$ and identify
$\mathcal S_2(H)$ with $H\otimes_2H$ through the rank-one convention used
throughout the paper.

Let $z\in\mathsf Z_{\eta;\mathcal S_2,1}(H)$ have cocycle chart $(x,A,q)$,
and write
\[
 \mathbf x^S=(1,x,\mathbb x^S),
 \qquad
 \mathbb x^S_{s,t}=\frac12x_{s,t}^{\otimes2}+A_{s,t}.
\]
By definition, $\mathbf x^S$ belongs to the little $\eta$-H\"older geometric
rough-path closure in the Hilbert tensor norm.  The additive field $q$ is the
increment field of a $2\eta$-H\"older $\mathcal S_1(H)_{\rm sa}$-valued path,
and hence also of a $2\eta$-H\"older $H^{\otimes_2 2}$-valued path.  Since
$3\eta>1$, a path controlled by $x$ can be integrated against $q$ in the
Young sense.

For a restart time $s$, consider the triangular equations in $T^{(N)}_2(H)$
\begin{align}
 \dd \mathbf S^{S,N}_{s,r}
 &=\mathbf S^{S,N}_{s,r}\otimes \dd\mathbf x^S_r,
 &\mathbf S^{S,N}_{s,s}&=1,
 \label{eq:all-order-geometric-equation}\\
 \dd \mathbf S^{I,N}_{s,r}
 &=\mathbf S^{I,N}_{s,r}\otimes \dd\mathbf x^S_r
   -\frac12\mathbf S^{I,N}_{s,r}\otimes \dd q_r,
 &\mathbf S^{I,N}_{s,s}&=1.
 \label{eq:all-order-Ito-equation}
\end{align}
The first is the geometric signature equation.  In the second, right
multiplication by $\dd q$ raises degree by two and is a Young term.
Equivalently, with
$S^{I,(-1)}:=0$ and $S^{I,(0)}:=1$, its homogeneous levels satisfy
\begin{equation}
 \label{eq:Ito-level-recursion}
 \dd S^{I,(k)}_{s,r}
 =S^{I,(k-1)}_{s,r}\otimes \dd\mathbf x^S_r
  -\frac12S^{I,(k-2)}_{s,r}\otimes \dd q_r,
 \qquad 1\le k\le N.
\end{equation}
Thus $q$ transports the correction from degree $k-2$.

For a Banach space $E$, recall that an $E$-valued path $Y$ is
\emph{$x$-controlled} if there is a path
$Y':[s,T]\to\mathcal L(H,E)$ such that
\[
 Y_{u,v}=Y'_u x_{u,v}+R^Y_{u,v},
 \qquad |Y'|_{\eta}+|R^Y|_{2\eta}<\infty.
\]
We use the usual controlled-path norm with the initial values included.

\begin{lemma}[Triangular mixed rough--Young closure]
\label{lem:mixed-rough-Young-closure}
Fix $N\ge2$ and $R<\infty$.  Suppose
\[
 \|\mathbf x^S\|_{\eta;[s,T]}+|q|_{2\eta;[s,T]}\le R.
\]
Set $U^{(-1)}:=0$, $U^{(0)}:=1$, and define recursively
\begin{equation}
\label{eq:mixed-triangular-recursion}
 U^{(k)}_{s,t}
 :=\int_s^t U^{(k-1)}_{s,r}\otimes\dd\mathbf x^S_r
   -\frac12\int_s^t U^{(k-2)}_{s,r}\otimes\dd q_r,
 \qquad 1\le k\le N.
\end{equation}
Then every level is well defined and
$U^{(k)}_{s,\cdot}$ is $x$-controlled with Gubinelli derivative
\[
 h\longmapsto U^{(k-1)}_{s,\cdot}\otimes h.
\]
For fixed $N$ the controlled norms of all levels are bounded by a constant
$C_{N,R}$, uniformly in the restart time $s$.  On bounded sets, the map
$(\mathbf x^S,q)\mapsto(U^{(1)},\ldots,U^{(N)})$ is locally Lipschitz in the
rough-path/controlled-path metric and the $2\eta$-H\"older norm of $q$.

The same conclusions hold in variation topology: if $2<p<3$ and
$1\le\nu\le p/2$, then on sets with bounded rough $p$-variation norm of
$\mathbf x^S$ and bounded $\nu$-variation norm of $q$, the solution map is
locally Lipschitz for the rough $p$-variation metric and the
$\nu$-variation norm, with all levels controlled in the $p$-variation sense.
On the common domain the H\"older and variation constructions coincide by
levelwise uniqueness.
\end{lemma}

\begin{proof}
Induct on $k$, the case $k=1$ being the rough integral of a constant.
If $U^{(k-1)}$ is controlled with derivative $U^{(k-2)}$, the first integral
in \eqref{eq:mixed-triangular-recursion} is a controlled rough integral and
the second is Young because $\eta+2\eta>1$.  Moreover,
\[
 U^{(k)}_{s;u,v}
 =U^{(k-1)}_{s,u}\otimes x_{u,v}+O(|v-u|^{2\eta});
\]
and the Young term is $O(|v-u|^{2\eta})$.  This gives the stated derivative;
the standard rough/Young estimates and their difference versions close the
induction, yielding $C_{N,R}$ and local Lipschitz continuity.  No
order-uniform constant is asserted.

For the variation formulation, use controlled rough paths of finite
$p$-variation.  Each previously constructed level has finite $p$-variation;
the correction integral is Young because
$1/p+1/\nu>1$, and its increments have finite $\nu$-variation, hence finite
$p/2$-variation because $\nu\le p/2$.  It therefore belongs to the controlled
remainder.  The corresponding rough and Young difference estimates give the
stated local Lipschitz bound by the same induction.
\end{proof}

\begin{lemma}[Dyadic reconstruction of a controlled remainder]
\label{lem:dyadic-controlled-remainder}
Let $E$ be a separable Banach space.  Let $X:[0,T]\to H$,
$Y':[0,T]\to\mathcal L(H,E)$, and
$R:\Delta_T\to E$ be random continuous fields with $R_{t,t}=0$ and
\begin{equation}\label{eq:controlled-remainder-coboundary}
 R_{s,t}-R_{s,u}-R_{u,t}=Y'_{s,u}X_{u,t},
 \qquad s\le u\le t.
\end{equation}
Suppose that for some $q>1$ and all $s<t$,
\begin{align}
 \|X_{s,t}\|_{L^{2q}}&\le C_X|t-s|^{1/2},
 \label{eq:controlled-dyadic-X-moment}\\
 \|Y'_{s,t}\|_{L^{2q}}&\le C_{Y'}|t-s|^{1/2},
 \label{eq:controlled-dyadic-derivative-moment}\\
 \|R_{s,t}\|_{L^q}&\le C_R|t-s|.
 \label{eq:controlled-dyadic-remainder-moment}
\end{align}
If $0<\alpha<1/2$ and
\begin{equation}\label{eq:controlled-dyadic-exponent-gap}
 2\alpha<1-\frac1q,
\end{equation}
then, after modification on one null set,
$X$ and $Y'$ are $\alpha$-H\"older and $R$ is $2\alpha$-H\"older.  Moreover
\begin{equation}\label{eq:controlled-dyadic-moment-bound}
 \|[X]_\alpha\|_{L^{2q}}
 +\|[Y']_\alpha\|_{L^{2q}}
 +\|[R]_{2\alpha}\|_{L^q}
 \le C_{\alpha,q,T}
 \bigl(C_X+C_{Y'}+C_R+C_XC_{Y'}\bigr).
\end{equation}
Thus $Y$, defined by
$Y_{s,t}=Y'_sX_{s,t}+R_{s,t}$, is an $X$-controlled path with derivative
$Y'$.
\end{lemma}

\begin{proof}
Use dyadic times $t_i^n=iT2^{-n}$ and set
\begin{align*}
 A_n&:=\max_i|X_{t_i^n,t_{i+1}^n}|,\\
 B_n&:=\max_i\|Y'_{t_i^n,t_{i+1}^n}\|,\\
 C_n&:=\max_i\|R_{t_i^n,t_{i+1}^n}\|.
\end{align*}
Put
\[
 \rho_q:=\frac12-\frac1{2q},
 \qquad
 \sigma_q:=1-\frac1q.
\]
The elementary maximum inequality gives
\begin{align*}
 \|A_n\|_{L^{2q}}&\le C_XT^{1/2}2^{-n\rho_q},\\
 \|B_n\|_{L^{2q}}&\le C_{Y'}T^{1/2}2^{-n\rho_q},\\
 \|C_n\|_{L^q}&\le C_RT2^{-n\sigma_q}.
\end{align*}
By \eqref{eq:controlled-dyadic-exponent-gap}, the weighted series
\[
 \mathcal A:=\sum_n2^{\alpha n}A_n,
 \qquad
 \mathcal B:=\sum_n2^{\alpha n}B_n,
 \qquad
 \mathcal C:=\sum_n2^{2\alpha n}C_n
\]
converge respectively in $L^{2q},L^{2q},L^q$, and dyadic chaining gives
\[
 [X]_\alpha\le C\mathcal A,
 \qquad
 [Y']_\alpha\le C\mathcal B.
\]

For the remainder, decompose any interval $(s,t)$ into an ordered disjoint
family of dyadic blocks, with at most two blocks of each generation.  Iterating
\eqref{eq:controlled-remainder-coboundary} over a finite truncation of this
family gives
\[
 R_{s,t}
 =\sum_jR_{I_j}
  +\sum_{j\ge2}Y'_{s,a_j}X_{I_j},
 \qquad I_j=[a_j,b_j].
\]
The first sum is bounded by $C\mathcal C|t-s|^{2\alpha}$.  Since
$\sum_j|I_j|^\alpha\le C_\alpha|t-s|^\alpha$, the second satisfies
\[
 \sum_{j\ge2}\|Y'_{s,a_j}\|\,|X_{I_j}|
 \le C_\alpha[Y']_\alpha[X]_\alpha|t-s|^{2\alpha}.
\]
Passing through finite truncations yields
\[
 [R]_{2\alpha}
 \le C_\alpha\bigl(\mathcal C+\mathcal A\mathcal B\bigr).
\]
H\"older's inequality proves \eqref{eq:controlled-dyadic-moment-bound}.
Finally set $Y_{s,t}:=Y'_sX_{s,t}+R_{s,t}$.  By
\eqref{eq:controlled-remainder-coboundary} and additivity of $X$,
\[
 Y_{s,t}-Y_{s,u}-Y_{u,t}
 =Y'_sX_{u,t}+Y'_{s,u}X_{u,t}-Y'_uX_{u,t}=0.
\]
The increments are additive and, after choosing $Y_0$, define the asserted
controlled path.
\end{proof}

Fix $P\in\Smax$ and a restart time $s$.  Define the classical Hilbert-tensor
It\^o tower recursively by
\begin{equation}\label{eq:classical-Ito-tensor-tower}
 \mathrm I_{s,t}^{P,(0)}:=1,\qquad
 \mathrm I_{s,t}^{P,(k)}
 :=\int_s^t\mathrm I_{s,r}^{P,(k-1)}\otimes\dd X_r,
 \qquad k\ge1.
\end{equation}

\begin{lemma}[Fixed-order moment scaling for the classical It\^o tower]
\label{lem:classical-Ito-tower-moments}
For every integer $k\ge0$ and $q\ge2$, there is
$C_{k,q}=C_{k,q}(\Lambda,B,T)$, independent of $P$ and $s$, such that
for $s\le t\le T$,
\begin{equation}\label{eq:classical-Ito-tower-maximal-moment}
 \left\|\sup_{s\le r\le t}
 \|\mathrm I_{s,r}^{P,(k)}\|_2\right\|_{L^q(P)}
 \le C_{k,q}|t-s|^{k/2}.
\end{equation}
For $k\ge0$ and $s\le u\le v\le T$,
\begin{equation}\label{eq:classical-Ito-tower-increment-moment}
 \|\mathrm I_{s,v}^{P,(k)}-\mathrm I_{s,u}^{P,(k)}\|_{L^q(P)}
 \le C_{k,q}|v-u|^{1/2}.
\end{equation}
In particular every level is a well-defined continuous Hilbert-tensor
semimartingale with moments of every fixed order.
\end{lemma}

\begin{proof}
Induct on $k$, writing $h=t-s$.  The tensoring operator
$\Gamma_u z:=\mathrm I_{s,u}^{P,(k-1)}\otimes z$ satisfies
\[
 \|\Gamma_u(a_u^P)^{1/2}\|_2^2
 =\|\mathrm I_{s,u}^{P,(k-1)}\|_2^2\Tr a_u^P.
\]
Maximal Hilbert-space BDG, the trace-clock bound, and
\eqref{eq:classical-Ito-tensor-tower} therefore give
\begin{align*}
 &\left\|\sup_{s\le r\le t}\left\|
   \int_s^r\mathrm I_{s,u}^{P,(k-1)}\otimes\dd M_u^P
  \right\|_2\right\|_{L^q(P)}\\
 &\qquad\le C\sqrt{q\Lambda h}\,
 \left\|\sup_{s\le u\le t}
 \|\mathrm I_{s,u}^{P,(k-1)}\|_2\right\|_{L^q(P)}
 \le C_{k,q}h^{k/2}.
\end{align*}
The drift part is bounded by
\[
 B h\left\|\sup_{s\le u\le t}
 \|\mathrm I_{s,u}^{P,(k-1)}\|_2\right\|_{L^q(P)}
 \le B\sqrt T\,C_{k-1,q}h^{k/2}.
\]
This proves \eqref{eq:classical-Ito-tower-maximal-moment} by induction.

For increments, apply the same estimates on $[u,v]$ to
\[
 \mathrm I_{s,v}^{P,(k)}-\mathrm I_{s,u}^{P,(k)}
 =\int_u^v\mathrm I_{s,r}^{P,(k-1)}\otimes\dd X_r
\]
using the bound on $[s,T]$.  The martingale and drift terms are respectively
$O(|v-u|^{1/2})$ and $O(|v-u|)$, proving
\eqref{eq:classical-Ito-tower-increment-moment}.
\end{proof}

\begin{proposition}[Tensor covariation recursion for the classical It\^o tower]
\label{prop:tensor-Ito-bracket-recursion}
For $k\ge2$, the scalar covariations obtained by contracting
$\mathrm I^{P,(k-1)}$ and $X$ against algebraic simple tensors admit a unique
$H^{\otimes_2 k}$-valued finite-variation representative, denoted by
$\qv{\mathrm I^{P,(k-1)},X}$, and it satisfies
\begin{equation}\label{eq:tensor-Ito-bracket-recursion}
 \dd\qv{\mathrm I^{P,(k-1)},X}_r
 =\mathrm I_{s,r}^{P,(k-2)}\otimes\dd Q_r.
\end{equation}
Here $H^{\otimes_{\rm alg}j}$ denotes the algebraic tensor product.  Thus
``tensor covariation'' in this proposition means this contraction-defined
representative; the strong Hilbert-norm conversion is proved separately in
\Cref{lem:strong-Hilbert-rough-Stratonovich-tower}.
\end{proposition}

\begin{proof}
The construction and all fixed-order integrability statements are supplied by
\Cref{lem:classical-Ito-tower-moments}.  The martingale part of
$\mathrm I^{P,(k-1)}$ is
\[
 \int_s^\cdot \mathrm I_{s,r}^{P,(k-2)}\otimes\dd M_r^P.
\]
Let $u\in H^{\otimes_{\rm alg}(k-2)}$ and $v,w\in H$.  The scalar stochastic-integral
bracket rule gives
\begin{align*}
 &\dd\Bigl[
   \langle\mathrm I^{P,(k-1)},u\otimes v\rangle,
   \langle X,w\rangle
  \Bigr]_r\\
 &\qquad
 =\langle\mathrm I_{s,r}^{P,(k-2)},u\rangle
   \,\dd\langle Q_rv,w\rangle
 =\left\langle
   \mathrm I_{s,r}^{P,(k-2)}\otimes\dd Q_r,
   u\otimes v\otimes w
  \right\rangle.
\end{align*}
Algebraic simple tensors are dense in $H^{\otimes_2 k}$.  Moreover the right
side of \eqref{eq:tensor-Ito-bracket-recursion} has finite variation because
\[
 \|\mathrm I_{s,r}^{P,(k-2)}\otimes\dd Q_r\|_2
 \le\|\mathrm I_{s,r}^{P,(k-2)}\|_2\,\|\dd Q_r\|_1.
\]
The scalar identities therefore determine the displayed finite-variation
candidate uniquely.
\end{proof}

\begin{lemma}[Strong Hilbert rough--Stratonovich conversion for the It\^o tower]
\label{lem:strong-Hilbert-rough-Stratonovich-tower}
Fix $P\in\Smax$, a restart time $s$, and $k\ge1$, and set
$\mathrm I^{P,(-1)}:=0$.  For every
$1/3<\alpha<1/2$, after modification on one $P$-null set,
$\mathrm I_{s,\cdot}^{P,(k-1)}$ is controlled by $X$, with derivative
\[
 h\longmapsto \mathrm I_{s,\cdot}^{P,(k-2)}\otimes h
\]
and a $2\alpha$-H\"older remainder.  On the modelwise canonical rough-path
domain one has, in $H^{\otimes_2 k}$ and simultaneously in the terminal time,
\begin{equation}\label{eq:strong-Hilbert-rough-Stratonovich-conversion}
 \int_s^\cdot
 \mathrm I_{s,r}^{P,(k-1)}\otimes\dd\mathbf X_r^{S,P}
 =
 \int_s^\cdot
 \mathrm I_{s,r}^{P,(k-1)}\otimes\dd X_r
 +\frac12\int_s^\cdot
 \mathrm I_{s,r}^{P,(k-2)}\otimes\dd Q_r.
\end{equation}
Consequently the rough integral is the Hilbert-valued Stratonovich integral,
and
\begin{equation}\label{eq:classical-mixed-Ito-recursion}
 \int_s^t\mathrm I_{s,r}^{P,(k-1)}\otimes\circ\dd X_r
 -\frac12\int_s^t\mathrm I_{s,r}^{P,(k-2)}\otimes\dd Q_r
 =\mathrm I_{s,t}^{P,(k)}.
\end{equation}
On the modelwise canonical rough-path domain, the first integral in this last
display is the rough integral against $\mathbf X^{S,P}$ and the second is the
Young--Stieltjes integral.  Thus the classical It\^o tower solves the mixed
triangular equation \eqref{eq:Ito-level-recursion}.
\end{lemma}

\begin{proof}
For $k=1$, the integrand $\mathrm I^{P,(0)}\equiv1$ is controlled with zero
derivative and zero remainder.  Both sides of
\eqref{eq:strong-Hilbert-rough-Stratonovich-conversion} then equal
$X_{s,\cdot}$.  Assume henceforth $k\ge2$.

The increment decomposition
\[
 \mathrm I_{s;u,v}^{P,(k-1)}
 =\mathrm I_{s,u}^{P,(k-2)}\otimes X_{u,v}
  +R^{P,(k-1)}_{u,v},
 \qquad
 R^{P,(k-1)}_{u,v}
 :=\int_u^v
    (\mathrm I_{s,r}^{P,(k-2)}-\mathrm I_{s,u}^{P,(k-2)})
    \otimes\dd X_r
\]
gives the controlled structure quantitatively.  By \eqref{eq:classical-Ito-tower-increment-moment}, for every fixed
order and every sufficiently large $q$,
\[
 \|\mathrm I_{s,r}^{P,(k-2)}-\mathrm I_{s,u}^{P,(k-2)}\|_{L^{2q}}
 \le C_{k,q}|r-u|^{1/2}.
\]
Maximal BDG for the martingale part of $R^{P,(k-1)}$ and the bounded drift
then give
\[
 \|R^{P,(k-1)}_{u,v}\|_{L^q}\le C_{k,q}|v-u|.
\]
Its exact coboundary identity is
\begin{equation}\label{eq:Ito-tower-remainder-coboundary}
 R^{P,(k-1)}_{u,w}
 -R^{P,(k-1)}_{u,v}
 -R^{P,(k-1)}_{v,w}
 =\bigl(\mathrm I_{s,v}^{P,(k-2)}
        -\mathrm I_{s,u}^{P,(k-2)}\bigr)\otimes X_{v,w}.
\end{equation}
This follows by expanding at $v$.  Fix $1/3<\alpha<1/2$ and choose $q$ with
\[
 2\alpha<1-\frac1q.
\]
Apply \Cref{lem:dyadic-controlled-remainder} with
\[
 Y'_u:h\longmapsto\mathrm I_{s,u}^{P,(k-2)}\otimes h.
\]
The preceding moment estimates verify the hypotheses, so on one full set
$\mathrm I^{P,(k-1)}$ is controlled by $X$ with derivative
$h\mapsto\mathrm I^{P,(k-2)}\otimes h$ and a $2\alpha$-H\"older remainder.
For the fixed $k$, intersect over rational
$\alpha\in(1/3,1/2)$.  When the tower is used simultaneously through a fixed
order $N$, intersect additionally over $1\le j\le N$; a larger rational
exponent implies every smaller one.  Thus all required structures hold on one
$P$-full set, without a
two-parameter Kolmogorov theorem for the nonadditive remainder.

For the strong conversion in the Hilbert tensor norm, put
\[
 G_rh:=\mathrm I_{s,r}^{P,(k-1)}\otimes h,
 \qquad
 G'_r(h_1\otimes h_2)
 :=\mathrm I_{s,r}^{P,(k-2)}\otimes h_1\otimes h_2.
\]
The preceding argument shows that $(G,G')$ is controlled by $X$.  Let
$\Pi^n=\{s=t_0^n<\cdots<t_{m_n}^n=T\}$ be deterministic partitions with
$|\Pi^n|\to0$.  For $r\in[s,T]$, define the stopped sums
\begin{align*}
 S_r^{n,1}
 &:=\sum_{i:t_i^n<r}
 G_{t_i^n}X_{t_i^n,t_{i+1}^n\wedge r},\\
 S_r^{n,2}
 &:=\sum_{i:t_i^n<r}
 G'_{t_i^n}\mathbb X^{I,P}_{t_i^n,t_{i+1}^n\wedge r}.
\end{align*}
After localization, predictable-step approximation and Hilbert-space BDG give
\[
 S^{n,1}
 \longrightarrow
 \int_s^\cdot G_r\,\dd X_r
 \quad\text{in probability in }
 C([s,T];H^{\otimes_2 k}).
\]
Write $X=M^P+\int_0^\cdot b_u^P\,\dd u$.  On the same stopped set, the
identity
\begin{align*}
 S_{r,\mathrm{mart}}^{n,2}
 &:=\sum_{i:t_i^n<r}\int_{t_i^n}^{t_{i+1}^n\wedge r}
   G'_{t_i^n}\bigl(X_{t_i^n,u}\otimes\dd M_u^P\bigr),\\
 S_{r,\mathrm{drift}}^{n,2}
 &:=\sum_{i:t_i^n<r}\int_{t_i^n}^{t_{i+1}^n\wedge r}
   G'_{t_i^n}\bigl(X_{t_i^n,u}\otimes b_u^P\bigr)\,\dd u
\end{align*}
gives $S^{n,2}=S_{\mathrm{mart}}^{n,2}+S_{\mathrm{drift}}^{n,2}$.
Hilbert-space BDG, the trace clock, fixed-order moments, and the bounded drift
then give
\begin{align*}
 E^P\sup_{r\in[s,T]}
   \|S_{r,\mathrm{mart}}^{n,2}\|_2^2
 &\le C_R\sum_i|t_{i+1}^n-t_i^n|^2
 \le C_RT|\Pi^n|,\\
 E^P\sup_{r\in[s,T]}
   \|S_{r,\mathrm{drift}}^{n,2}\|_2
 &\le C_R\sum_i|t_{i+1}^n-t_i^n|^{3/2}
 \le C_RT|\Pi^n|^{1/2}.
\end{align*}
Hence, after removing the localization,
\[
 S^{n,2}
 \longrightarrow0
 \quad\text{in probability in }
 C([s,T];H^{\otimes_2 k}).
\]
Since $Q$ has finite trace variation and $G'$ is continuous, the corresponding
stopped Riemann--Stieltjes sums
\[
 S_r^{n,Q}
 :=\frac12\sum_{i:t_i^n<r}
 G'_{t_i^n}Q_{t_i^n,t_{i+1}^n\wedge r}
\]
satisfy
\[
 S^{n,Q}
 \longrightarrow
 \frac12\int_s^\cdot
 \mathrm I_{s,r}^{P,(k-2)}\otimes\dd Q_r
\]
uniformly in $C([s,T];H^{\otimes_2 k})$.  Using
$\mathbb X^{S,P}=\mathbb X^{I,P}+Q/2$ in the compensated sums gives
\[
 \int_s^\cdot
 \mathrm I_{s,r}^{P,(k-1)}\otimes\dd\mathbf X_r^{S,P}
 =
 \int_s^\cdot
 \mathrm I_{s,r}^{P,(k-1)}\otimes\dd X_r
 +\frac12\int_s^\cdot
 \mathrm I_{s,r}^{P,(k-2)}\otimes\dd Q_r.
\]
The right side is the Hilbert-valued Stratonovich integral, proving the claim.
\end{proof}

\begin{corollary}[All-order signature propagation from the least state]
\label{thm:all-order-generation}
Let $1/3<\eta<1/2$ and $N\ge2$.

\textup{(i) Geometric tower.}
For every $z\in\mathsf Z_{\eta;\mathcal S_2,1}(H)$,
\eqref{eq:all-order-geometric-equation} has a unique solution.  Its increments
form the unique multiplicative degree-$N$ geometric extension of
$\mathbf x^S$; in particular,
\[
 \pi_1\mathbf S^{S,N}_{s,t}=x_{s,t},
 \qquad
 \pi_2\mathbf S^{S,N}_{s,t}=\mathbb x^S_{s,t}.
\]

\textup{(ii) It\^o tower.}
Equation~\eqref{eq:all-order-Ito-equation} has a unique mixed rough--Young
solution.  Its increments form a multiplicative tensor family, and its first
two levels are
\[
 \pi_1\mathbf S^{I,N}_{s,t}=x_{s,t},
 \qquad
 \pi_2\mathbf S^{I,N}_{s,t}
 =\mathbb x^S_{s,t}-\frac12q_{s,t}
 =\mathbb x^I_{s,t}.
\]
The family is compatible under degree truncation:
\[
 \pi_{\le M}\mathbf S^{S,N}=\mathbf S^{S,M},
 \qquad
 \pi_{\le M}\mathbf S^{I,N}=\mathbf S^{I,M},
 \qquad 2\le M\le N.
\]

\textup{(iii) Stability and naturality.}
For each fixed $N$, both solution maps are locally Lipschitz on bounded
subsets of the state topology into the corresponding inhomogeneous rough
path topology.  They are causal, commute with deterministic stopping,
future shift, and finite concatenation, and are natural under bounded linear
maps between separable Hilbert spaces.

\textup{(iv) Semimartingale identification.}
If $z=Z_\tau$ is the accessible canonical state and $P$ is one of its
semimartingale laws, then, $P$-almost surely, the homogeneous coordinates of
$\mathbf S^{S,N}$ and $\mathbf S^{I,N}$ are respectively the classical
Stratonovich and It\^o iterated integrals up to order $N$.
\end{corollary}

\begin{proof}
The Hilbert geometric extension theorem gives the unique solution of
\eqref{eq:all-order-geometric-equation}; uniqueness yields truncation
compatibility and multiplicativity.

For the It\^o equation, apply
\Cref{lem:mixed-rough-Young-closure} degree by degree to
\eqref{eq:Ito-level-recursion}.  It gives existence, uniqueness, and local
Lipschitz continuity in each fixed inhomogeneous tensor topology.  Autonomy,
concatenation, and uniqueness give the Chen, stopping, and shift identities;
applying $L^{\otimes k}$ levelwise gives linear naturality.

Under fixed $P$, uniqueness identifies the geometric equation with the
classical Stratonovich tower.  For the It\^o tower
$(\mathrm I^{P,(k)})_{k\le N}$, the bracket recursion and strong conversion
in \Cref{prop:tensor-Ito-bracket-recursion,lem:strong-Hilbert-rough-Stratonovich-tower}
show levelwise that the Stratonovich correction
$\frac12\mathrm I^{P,(k-2)}\otimes\dd Q$ is cancelled by the Young term in
\eqref{eq:Ito-level-recursion}.  Hence
\[
 S^{I,(k)}=\mathrm I^{P,(k)},\qquad 0\le k\le N,
\]
for $s=0$.  Chen's relation and truncated tensor inversion determine all
restarted increments; continuity makes the identity simultaneous in
$0\le s\le t\le T$ on the same full set.
\end{proof}

\begin{corollary}[Corewise all-order continuity]
\label{cor:core-all-order-continuity}
Fix $N\ge2$ and $R<\infty$, and assume that the operational exponent satisfies
$1/3<\eta\le\alpha$.  On $\mathcal C_R$, the geometric and It\^o
all-order readouts
\[
 x\longmapsto \mathbf S^{S,N}(x),
 \qquad
 x\longmapsto \mathbf S^{I,N}(x)
\]
are continuous.  The degree-$N$ geometric signatures of the smooth resolvent
paths converge uniformly on $\mathcal C_R$ to $\mathbf S^{S,N}$, and the
corresponding bracket-corrected towers formed with the trace-energy
approximations converge uniformly to $\mathbf S^{I,N}$.  More precisely, the
towers driven by $S_2(Y^{\eps_n})$ and $\mathsf E^{\delta_j}$ converge jointly
as $n,j\to\infty$, and hence along every prescribed diagonal.
For any exponent $1/3<\eta<1/2$, this hypothesis can be arranged by choosing
$\eta\le\alpha<\beta<1/2$ before constructing the cores.
\end{corollary}

\begin{proof}
Continuity of the exact readouts follows from
\Cref{thm:master-causal-cores,thm:all-order-generation}.  For the approximation
claim choose
\[
 1/\alpha<p<3,
 \qquad 1<\nu<p/2.
\]
The geometric rate in \Cref{thm:master-causal-cores}, the uniform
trace-energy rate and common total-variation bound in
\Cref{cor:core-bracket-variation}, and
\Cref{lem:homogeneous-holder-variation-transfer} give joint convergence of
the two drivers in the mixed $p$-/\,$\nu$-variation topology.  Now
\Cref{lem:mixed-rough-Young-closure} and the locally Lipschitz geometric
extension yield both asserted uniform all-order limits.
\end{proof}

\begin{corollary}[Third-order propagation identity]
\label{cor:third-level-derived}
For $z=(x,A,q)\in\mathsf Z_{\eta;\mathcal S_2,1}(H)$, let
$S^{S,(3)}$ and $S^{I,(3)}$ be the third homogeneous levels furnished by
\Cref{thm:all-order-generation}.  Define the mixed transport coordinates
\begin{equation}
 \label{eq:mixed-third-order-transports}
 B^{q x}_{s,t}:=\int_s^t q_{s,u}\otimes \dd x_u,
 \qquad
 B^{x q}_{s,t}:=\int_s^t x_{s,u}\otimes \dd q_u,
\end{equation}
where both integrals are Young integrals.  Then
\begin{equation}
 \label{eq:third-order-Ito-Stratonovich-state}
 S^{S,(3)}_{s,t}
 =S^{I,(3)}_{s,t}
  +\frac12B^{q x}_{s,t}
  +\frac12B^{x q}_{s,t}.
\end{equation}
Moreover,
\begin{align}
 B^{q x}_{s,t}
 &=B^{q x}_{s,u}+B^{q x}_{u,t}+q_{s,u}\otimes x_{u,t},
 \label{eq:q-x-Chen}\\
 B^{x q}_{s,t}
 &=B^{x q}_{s,u}+B^{x q}_{u,t}+x_{s,u}\otimes q_{u,t}.
 \label{eq:x-q-Chen}
\end{align}
Thus the third-order coordinates are transported from the existing first- and
second-order channels and are derived restart coordinates of the step-two
state.
\end{corollary}

\begin{proof}
At degree three, \eqref{eq:all-order-Ito-equation} gives
\[
 S^{I,(3)}_{s,t}
 =\int_s^t S^{I,(2)}_{s,u}\otimes\circ\dd x_u
  -\frac12\int_s^t x_{s,u}\otimes\dd q_u.
\]
Since $S^{S,(2)}=S^{I,(2)}+q/2$,
\[
 S^{S,(3)}_{s,t}
 =\int_s^t S^{I,(2)}_{s,u}\otimes\circ\dd x_u
  +\frac12\int_s^t q_{s,u}\otimes\dd x_u,
\]
which proves \eqref{eq:third-order-Ito-Stratonovich-state}.  Splitting either
Young integral at $u$ and using additivity of $x$ and $q$ proves
\eqref{eq:q-x-Chen}--\eqref{eq:x-q-Chen}.
\end{proof}

\begin{remark}[The threshold $\eta=1/3$]
\label{rem:last-primitive-level}
For $1/3<\eta<1/2$, every fixed higher level is a continuous readout of the
whole step-two path, not of one isolated increment $(x_{s,t},A_{s,t})$.  At
or below $1/3$, step-two sewing fails on the full H\"older class, so a
degree-three enhancement or selecting structure is generally required.  This
concerns universal sufficiency, not every individual path.
\end{remark}

\begin{corollary}[All-order generation from the least state]
\label{cor:all-order-minimality}
Let $\tau=(\eta;\mathcal S_2,\mathcal S_1)$ be accessible, write
$\operatorname{ch}_\tau(z)=(x(z),A(z),q(z))$, and put
\[
 \mathsf Q_\eta
 :=C_0^{0,2\eta}\bigl([0,T];\mathcal S_1(H)_{\rm sa}\bigr)
\]
with its natural truncations.  The projective
families
\[
 \mathbf S^{S,\infty}:=(\mathbf S^{S,N})_{N\ge2},
 \qquad
 \mathbf S^{I,\infty}:=(\mathbf S^{I,N})_{N\ge2},
\]
endowed with the projective-limit topologies generated by their finite-degree
Hilbert-tensor projections, are stable-causal readouts of $Z_\tau$.

\textup{(i) Geometric quotient.}
Let $\pi_{\rm geom}(z):=(x(z),A(z))$, with the image topology and induced
truncations.  The geometric tower factors through $\pi_{\rm geom}$ and is a
lossless representation of precisely this quotient.  Its degree-one and
degree-two coordinates give the stable-causal decoder
\[
 x=\pi_1\mathbf S^{S,\infty},
 \qquad
 A=\Anti\pi_2\mathbf S^{S,\infty}.
\]
Consequently a causal source state $S$ realizes the geometric tower if and only
if $\pi_{\rm geom}(Z_\tau)\preceqsc S$.  The covariance path is not determined
by this tower.

\textup{(ii) Lossless augmented towers.}
For $\bullet\in\{S,I\}$, define
\[
 \mathcal T^\bullet(z)
 :=\bigl(\mathbf S^{\bullet,\infty}(z),q(z)\bigr)
\]
and equip its image with the topology inherited from the product of the
projective tower space and $\mathsf Q_\eta$.  Then $\mathcal T^\bullet$ is a
stable-causal homeomorphism onto its image.  A stable-causal decoder is
\[
 \mathcal D^\bullet(\mathbf s,q)
 :=\operatorname{ch}_\tau^{-1}
   \bigl(\pi_1\mathbf s,\Anti\pi_2\mathbf s,q\bigr),
 \qquad
 \mathcal D^\bullet\mathcal T^\bullet=I.
\]
Thus both augmented towers are lossless representations of the complete
state.

\textup{(iii) Algebraic recovery from the It\^o tower.}
The unaugmented It\^o tower determines $q$ algebraically through
\[
 q=x^{\otimes2}-\pi_2\mathbf S^{I,\infty}
   -\bigl(\pi_2\mathbf S^{I,\infty}\bigr)^*.
\]
In infinite dimension this formula does not in general define a continuous
decoder from the standard projective Hilbert-tensor topology to the
$\mathcal S_1$ state topology.  The explicit $q$ coordinate in
$\mathcal T^I$ retains this trace-class topology.  It may be omitted when the
declared covariance range has equivalent $\mathcal S_2$ and $\mathcal S_1$
topologies, in particular when $H$ is finite dimensional.

\textup{(iv) Least-state consequence.}
For every causal source state $S$ and either $\bullet\in\{S,I\}$,
\[
 S\text{ jointly realizes }\mathcal T^\bullet(Z_\tau)
 \quad\Longleftrightarrow\quad
 Z_\tau\preceqsc S.
\]
Hence $Z_\tau$ remains the unique least sufficient state for every response
program containing either augmented tower and whose remaining responses are
stable-causal readouts of $Z_\tau$.
\end{corollary}

\begin{proof}
The stable-causal readout assertions and degree-truncation compatibility are
\Cref{thm:all-order-generation}.  For the geometric tower,
\[
 \pi_1\mathbf S^{S,\infty}=x,
 \qquad
 \pi_2\mathbf S^{S,\infty}=\frac12x^{\otimes2}+A.
\]
Thus its antisymmetric degree-two coordinate recovers $A$, while its defining
equation is independent of $q$.  This proves~\textup{(i)}.

For the It\^o tower,
\[
 \pi_1\mathbf S^{I,\infty}=x,
 \qquad
 \pi_2\mathbf S^{I,\infty}
 =\frac12x^{\otimes2}+A-\frac12q.
\]
Since $q$ is self-adjoint, antisymmetrization again recovers $A$, and the
symmetric defect gives the formula in~\textup{(iii)}.  Retaining $q$ in its
declared $\mathsf Q_\eta$ topology makes $\mathcal D^\bullet$ continuous; the
decoder commutes with truncations and is a left inverse of
$\mathcal T^\bullet$.  This proves~\textup{(ii)}.

To see the topological obstruction in infinite dimension, choose an
orthonormal sequence $(e_j)$ and set
\[
 B_n:=\frac1n\sum_{j=1}^n e_j\otimes e_j,
 \qquad
 q_t^n:=\frac tT B_n,
 \qquad
 z_n:=\operatorname{ch}_\tau^{-1}(0,0,q^n).
\]
Then
\[
 \|B_n\|_{\mathcal S_1}=1,
 \qquad
 \|B_n\|_{\mathcal S_2}=n^{-1/2}\longrightarrow0.
\]
Thus $z_n$ does not converge to the zero state in the declared
$\mathcal S_1$ topology.  Its geometric tower is trivial, whereas the odd
It\^o levels vanish and, for every fixed $k\ge1$,
\[
 \pi_{2k}\mathbf S^{I,\infty}_{s,t}(z_n)
 =\frac{(-1/2)^k}{k!}\bigl(q_{s,t}^n\bigr)^{\otimes k}.
\]
Every fixed It\^o level therefore converges to the trivial level in its Hilbert
tensor norm.  The unaugmented It\^o towers converge projectively to the
identity, so no continuous inverse into the trace-class state topology exists
in general.  This proves the qualification in~\textup{(iii)}.

Finally, the encoder--decoder pairs
$(\mathcal T^\bullet,\mathcal D^\bullet)$ identify joint realization of either
augmented tower with stable-causal factorization through $Z_\tau$.
\Cref{prop:lossless-chart-factorization,thm:least-state} gives~\textup{(iv)}.
\end{proof}

\subsection{Coherent flow propagation}

\label{sec:dynamical-propagation}

The deterministic rough/Young map propagates $(\mathbf X^S,Q)$ into
$E=\mathbb R^m$.  A lift-to-cocycle viewpoint
\cite{BailleulRiedelScheutzow17} complements universal measurable SDE
representations \cite{Kallenberg96,PrzybylowiczEtAl24}.  Here the flow is a
downstream state response: compact capacity cores support one jointly
measurable, parameter-continuous cocycle, identified modelwise with the
classical It\^o flow.

\paragraph{Rough-path metric.}
For a geometric step-two $p$-rough path
$\mathbf z=(z,\mathbb z)$, use the standard homogeneous radius
\[
 \|\mathbf z\|_{p\text{-var}}
 :=\|z\|_{p\text{-var}}
   +\|\mathbb z\|_{(p/2)\text{-var}}^{1/2}
\]
and homogeneous distance
\begin{equation}\label{eq:homogeneous-pvar-distance}
 d_{p\text{-var}}^{\rm hom}(\mathbf z,\widetilde{\mathbf z})
 :=\|z-\widetilde z\|_{p\text{-var}}
   +\|\mathbb z-\widetilde{\mathbb z}\|_{(p/2)\text{-var}}^{1/2}.
\end{equation}
State estimates use the unrooted second level; flow estimates use this
homogeneous metric.

\begin{lemma}[H\"older/BV transfer to the mixed-driver topology]
\label{lem:homogeneous-holder-variation-transfer}
Let $1/3<\alpha<1/2$ and $p>1/\alpha$.  For any two continuous step-two
multiplicative functionals,
\begin{equation}\label{eq:homogeneous-holder-variation-transfer}
 d_{p\text{-var}}^{\rm hom}(\mathbf x,\widetilde{\mathbf x})
 \le C_{\alpha,p,T}\left(
  \rho_\alpha(\mathbf x,\widetilde{\mathbf x})
  +\rho_\alpha(\mathbf x,\widetilde{\mathbf x})^{1/2}
 \right).
\end{equation}
Consequently an $O(\eps^\kappa)$ estimate in the unrooted H\"older distance
implies an $O(\eps^{\kappa/2})$ estimate in the homogeneous rough-path
metric.

If, in addition, $q$ and $\widetilde q$ are Banach-valued paths of finite
variation, then for every $\nu>1$,
\begin{equation}\label{eq:uniform-BV-to-nuvar-transfer}
 \|q-\widetilde q\|_{\nu\text{-var}}
 \le
 \bigl(2\|q-\widetilde q\|_\infty\bigr)^{1-1/\nu}
 \bigl(\|q\|_{1\text{-var}}+
       \|\widetilde q\|_{1\text{-var}}\bigr)^{1/\nu}.
\end{equation}
Consequently, on a common total-variation ball, convergence in the unrooted
$\alpha$-H\"older rough-path distance together with uniform convergence of
the finite-variation channel implies convergence for
\[
 d_{p\text{-var}}^{\rm hom}(\mathbf x,\widetilde{\mathbf x})
 +\|q-\widetilde q\|_{\nu\text{-var}}.
\]
Alternatively, if $q$ and $\widetilde q$ are $2\eta$-H\"older and
$\nu>1/(2\eta)$, then
\begin{equation}\label{eq:holder-to-nuvar-transfer}
 \|q-\widetilde q\|_{\nu\text{-var}}
 \le T^{2\eta}[q-\widetilde q]_{2\eta}.
\end{equation}
The Young restrictions enter only at the application: one takes
$\nu\le p/2$ in the triangular tensor equation and
$\nu<p/(p-1)$ in the finite-output mixed RDE.
\end{lemma}

\begin{proof}
For every partition, the first-level H\"older estimate gives
\[
 \|x-\widetilde x\|_{p\text{-var}}
 \le T^\alpha[x-\widetilde x]_\alpha.
\]
Applying the partition estimate to the second level gives
\[
 \|\mathbb x-\widetilde{\mathbb x}\|_{(p/2)\text{-var}}
 \le T^{2\alpha}
      [\mathbb x-\widetilde{\mathbb x}]_{2\alpha}.
\]
Taking the square root in the second term proves
\eqref{eq:homogeneous-holder-variation-transfer}.
For a partition $\Pi$, put $f=q-\widetilde q$.  Then
\[
 \sum_{[u,v]\in\Pi}\|f_{u,v}\|^\nu
 \le (2\|f\|_\infty)^{\nu-1}
      \bigl(\|q\|_{1\text{-var}}+\|\widetilde q\|_{1\text{-var}}\bigr).
\]
Taking the supremum over $\Pi$ and then the $\nu$-th root proves
\eqref{eq:uniform-BV-to-nuvar-transfer}.  If $f=q-\widetilde q$ is
$2\eta$-H\"older, then
\[
 \sum_{[u,v]\in\Pi}\|f_{u,v}\|^\nu
 \le [f]_{2\eta}^\nu
      \sum_{[u,v]\in\Pi}|v-u|^{2\eta\nu}
 \le [f]_{2\eta}^\nu T^{2\eta\nu},
\]
which proves \eqref{eq:holder-to-nuvar-transfer} and the final assertions.
\end{proof}

\subsubsection{Rough Stratonovich and bracket-corrected It\^o representations}

Throughout this subsubsection, fix a nonempty martingale family
$\mathcal P_{\rm sp}\subset\mathfrak S_{\Lambda,0}^{0,T}(H)$ with the
uniform spatial trace-tightness profile $(P_m,\vartheta_m)$ of
\Cref{def:uniform-spatial-trace-tightness}.  The common perfection domain is
$G_*^{\boldsymbol\vartheta}$ from
\eqref{eq:spatial-profile-full-domain}.  Thus the stochastic estimates below
use the trace clock, while the spatial profile supplies common compact
perfection; domination by one trace-class operator is not assumed.

Fix $m\ge1$ and set $E=\R^m$.  The driver is Hilbert valued and the solution
space is finite dimensional.  Write
\[
 V:E\to\mathcal L(H,E),
 \qquad V_0:E\to E,
\]
and, for $h,k\in H$, define on rank-one tensors
\begin{equation}\label{eq:finite-output-second-coefficient}
 \mathcal W_V(y)(h\otimes k)
 :=D(V(\cdot)k)(y)[V(y)h].
\end{equation}

\begin{lemma}[Automatic Hilbert--Schmidt compatibility]
\label{lem:finite-output-S2-compatibility}
If $V\in\operatorname{Lip}^\rho(E;\mathcal L(H,E))$ with $\rho>1$, then
\eqref{eq:finite-output-second-coefficient} has a unique extension
\[
 \widetilde{\mathcal W}_V(y)\in
 \mathcal L(\mathcal S_2(H),E),
\]
and
\[
 y\longmapsto\widetilde{\mathcal W}_V(y)
 \in\operatorname{Lip}^{\rho-1}
 \bigl(E;\mathcal L(\mathcal S_2(H),E)\bigr).
\]
Its restriction to $\mathcal S_1(H)$ defines the trace contraction
\[
 \langle\mathcal W_V(y),q\rangle_{\rm tr}
 :=\widetilde{\mathcal W}_V(y)q,
 \qquad q\in\mathcal S_1(H).
\]
\end{lemma}

\begin{proof}
Let $(e_i)_{i=1}^m$ be the standard basis and write
$V(y)h=(\langle v_i(y),h\rangle_H)_{i=1}^m$, where
$v_i(y)=V(y)^*e_i$.  For $T\in\mathcal S_2(H)$ set
\[
 [\widetilde{\mathcal W}_V(y)T]_i
 :=\sum_{a=1}^m
   \langle T,v_a(y)\otimes D_av_i(y)\rangle_{\mathcal S_2}.
\]
On rank-one tensors this is exactly
\eqref{eq:finite-output-second-coefficient}.  The finite sum gives the
Hilbert--Schmidt bound, and the asserted regularity follows from the
$\operatorname{Lip}^\rho$ bounds on $V$.  Uniqueness follows from finite-rank
density in $\mathcal S_2(H)$.
\end{proof}

\begin{proposition}[Finite-dimensional rough/finite-variation stability]
\label{prop:finite-output-mixed-stability}
Fix $p\in(2,3)$ and $1<\nu<p/(p-1)$.  Let $V_0$ be bounded globally
Lipschitz.
\begin{enumerate}[label=\textup{(\roman*)},leftmargin=2.4em]
\item If $V\in\operatorname{Lip}^\rho$ for some $\rho>p$, the solution map of
\[
 \dd Z=V_0(Z)\,\dd t+V(Z)\,\dd\mathbf z
\]
is locally Lipschitz on bounded geometric $p$-rough-path balls for
$d_{p\text{-var}}^{\rm hom}$.
\item If $V\in\operatorname{Lip}^{\rho+1}$ for some $\rho>p$, the solution map
of
\[
 \dd Z=V_0(Z)\,\dd t+V(Z)\,\dd\mathbf z
 -\frac12\langle\mathcal W_V(Z),\dd q\rangle_{\rm tr}
\]
is locally Lipschitz on bounded sets for
\[
 d_{p\text{-var}}^{\rm hom}(\mathbf z,\widetilde{\mathbf z})
 +\|q-\widetilde q\|_{\nu\text{-var};\mathcal S_1}.
\]
Both solution maps are causal under terminal restriction and obey the usual
subinterval-restriction and concatenation laws.  They do not commute with
stopping of $(\mathbf z,q)$ alone when $V_0\ne0$; exact stop preservation is
recovered by adjoining the finite-variation clock channel $\ell_t=t$, with
its canonical Young augmentation, and stopping $(\ell,\mathbf z,q)$
simultaneously.  More precisely,
for every driver radius $R$ and compact initial set $K\subset\R^m$, there is
$C_{R,K}$ such that solutions $Z,\widetilde Z$ starting in $K$ satisfy
\begin{equation}\label{eq:finite-output-mixed-stability}
 \|Z-\widetilde Z\|_\infty
 \le C_{R,K}\left(
 d_{p\text{-var}}^{\rm hom}(\mathbf z,\widetilde{\mathbf z})
 +\|q-\widetilde q\|_{\nu\text{-var};\mathcal S_1}
 +|Z_0-\widetilde Z_0|\right),
\end{equation}
whenever the two rough drivers and the two finite-variation drivers have
radius at most $R$.  In part~\textup{(i)}, the $q$-term is omitted.
\end{enumerate}
\end{proposition}

\begin{proof}
By \Cref{lem:finite-output-S2-compatibility}, the second coefficient extends
to the Hilbert tensor completion.  For two drivers set
\[
 \omega(s,t):=
 \|\mathbf z\|_{p\text{-var};[s,t]}^p
 +\|\widetilde{\mathbf z}\|_{p\text{-var};[s,t]}^p
 +\|q\|_{\nu\text{-var};[s,t]}^\nu
 +\|\widetilde q\|_{\nu\text{-var};[s,t]}^\nu
 +(t-s).
\]
On an interval where $\omega(s,t)$ is below a fixed threshold, the controlled
rough-path fixed-point estimate and the Young--Loeve inequality give
\[
 \|Z-\widetilde Z\|_{\mathcal D;[s,t]}
 \le C_R\left(
 |Z_s-\widetilde Z_s|
 +d_{p\text{-var};[s,t]}^{\rm hom}
    (\mathbf z,\widetilde{\mathbf z})
 +\|q-\widetilde q\|_{\nu\text{-var};[s,t]}
 \right),
\]
where $\mathcal D$ is the controlled norm.  Here $1/p+1/\nu>1$ is the Young
condition and $\operatorname{Lip}^{\rho+1}$ controls
$\widetilde{\mathcal W}_V$.  A greedy partition has uniformly bounded length
on a driver ball; iteration proves \eqref{eq:finite-output-mixed-stability}.
Taking $q=\widetilde q=0$ gives part~\textup{(i)}, while uniqueness gives
terminal causality, restriction, and concatenation; see
\cite{Lyons98,FrizHairer20,GrongNilssenSchmeding22}.
\end{proof}

\begin{proposition}[Common Hilbert-driven Stratonovich flow]
\label{thm:common-Stratonovich-flow}
Fix $p\in(1/\alpha,3)$.  Assume $V_0$ is bounded globally Lipschitz and
$V\in\operatorname{Lip}^\rho(\R^m;\mathcal L(H,\R^m))$ for some $\rho>p$.
For $x\in G_*^{\boldsymbol\vartheta}$, let $\Phi^S_{s,t}(x,y)$ solve
\begin{equation}\label{eq:common-Stratonovich-SDE}
 \dd Z_t=V_0(Z_t)\,\dd t+V(Z_t)\,\dd\mathbf X_t^S,
 \qquad Z_s=y.
\end{equation}
Then $\Phi^S$ is jointly Borel, corewise jointly continuous, raw causal in the
terminal time, and satisfies the flow and future-shift cocycle laws (with the
shift variable restricted to the remaining interval $0\le v\le T-s$).  Under
every $P\in\mathcal P_{\rm sp}$ each deterministic section is the classical Stratonovich
strong solution driven by the $H$-valued martingale $X$.
\end{proposition}

\begin{proof}
Apply \Cref{prop:finite-output-mixed-stability}(i) to the common lift.
\Cref{lem:homogeneous-holder-variation-transfer} carries the core topology
to the homogeneous $p$-variation topology.  Deterministic continuity then
gives measurability, corewise continuity, and the cocycle; the common
Stratonovich lift gives the modelwise identity.
\end{proof}

\begin{proposition}[Bracket-corrected finite-dimensional It\^o flow]
\label{thm:common-Ito-flow}
Fix $p\in(1/\alpha,3)$ and choose
\begin{equation}\label{eq:mixed-variation-exponents}
 1<\nu<\frac p{p-1}.
\end{equation}
Assume $V_0$ is bounded globally Lipschitz and
$V\in\operatorname{Lip}^{\rho+1}(\R^m;\mathcal L(H,\R^m))$ for some
$\rho>p$.  For $x\in G_*^{\boldsymbol\vartheta}$ define $\Psi$ by
\begin{equation}\label{eq:common-Ito-mixed-RDE}
 \dd Z
 =V_0(Z)\,\dd t+V(Z)\,\dd\mathbf X^S
 -\frac12\langle\mathcal W_V(Z),\dd Q\rangle_{\rm tr}.
\end{equation}
Then $\Psi$ is jointly Borel, nonanticipative under terminal restriction,
corewise jointly continuous, and obeys the restriction, concatenation, and
cocycle laws.  This nonanticipativity is the identity used in
\Cref{def:stable-causal}; exact stop preservation is stronger and, when
$V_0\ne0$, requires the clock augmentation described in
\Cref{prop:finite-output-mixed-stability}.  Under every
$P\in\mathcal P_{\rm sp}$, $\Psi$ equals the classical strong It\^o solution
\begin{equation}\label{eq:common-Ito-SDE-from-rough-flow}
 \dd Z_t=V_0(Z_t)\,\dd t+V(Z_t)\,\dd X_t.
\end{equation}
\end{proposition}

\begin{proof}
Use \Cref{prop:finite-output-mixed-stability}(ii).  On each core the
covariance paths lie in a common total-variation ball, so
\Cref{lem:homogeneous-holder-variation-transfer} carries the state topology
to the mixed-driver topology.  Under fixed $P$, the rough integral is
Stratonovich, and by
\eqref{eq:finite-output-second-coefficient} the Hilbert-valued
It\^o--Stratonovich conversion reads
\[
 V(Z)\circ\dd X
 =V(Z)\,\dd X
  +\frac12\langle\mathcal W_V(Z),\dd Q\rangle_{\rm tr}.
\]
The term subtracted in \eqref{eq:common-Ito-mixed-RDE} cancels this correction.
The two uniqueness statements then identify the fields simultaneously in time.
\end{proof}

\begin{corollary}[Corewise and full-state flow regimes]
\label{cor:full-state-flow-readout}
The preceding two flow propositions require only $p>1/\alpha$ and therefore
give continuous propagated responses on the $\alpha$-regular profile cores.
If instead
\[
 \frac1\eta<p<3,
\]
then the Stratonovich solution-field map is a stable-causal
readout on the whole state space
$\mathsf Z_{\eta;\mathcal S_2,1}(H)$.  Under the It\^o vector-field
assumptions, choose additionally
\begin{equation}\label{eq:full-state-mixed-variation-window}
 \frac1{2\eta}<\nu<\frac p{p-1}.
\end{equation}
Then the bracket-corrected It\^o solution-field map is likewise a continuous
stable-causal readout on the whole operational state space.  In both cases
the target carries the compact-parameter local-uniform topology.  These
full-state statements concern nonanticipativity under terminal restriction;
exact stop preservation still requires the clock augmentation when
$V_0\ne0$.
\end{corollary}

\begin{proof}
Every state has an $\eta$-H\"older geometric rough coordinate and a
$2\eta$-H\"older covariance coordinate.  Apply
\Cref{lem:homogeneous-holder-variation-transfer} with $\alpha=\eta$.
For the It\^o equation also use
\eqref{eq:holder-to-nuvar-transfer}.  The interval in
\eqref{eq:full-state-mixed-variation-window} is nonempty because
$\eta>1/3$ and $p<3$.  The deterministic stability theorem gives continuity
in the full state topology, locally uniformly in the flow parameters;
uniqueness gives terminal nonanticipativity and the cocycle laws.
\end{proof}

\begin{remark}[Finite-dimensional output]
The driver and its second-order state are Hilbert-valued, while the nonlinear
solution space is $\R^m$.  In this setting the second coefficient extends
automatically to $\mathcal S_2(H)$ by
\Cref{lem:finite-output-S2-compatibility}.
\end{remark}

\begin{corollary}[It\^o dynamics as a propagated response]
\label{thm:state-propagated-Ito-dynamics}
Retain the assumptions of \Cref{thm:common-Ito-flow} and let
$\varnothing\ne\cP\subset\mathcal P_{\rm sp}$.  For every fixed starting pair
$(s,y)\in[0,T]\times\mathbb R^m$ and every $1\le r<\infty$, the state-based
section
\[
 x\longmapsto \Psi_{s,\cdot}(x,y)
\]
has a representative in
$\mathbb L_{\cP}^r(C([0,T];\mathbb R^m))$.  More precisely:
\begin{enumerate}[label=\textup{(\roman*)},leftmargin=2.5em]
\item on the common full-capacity domain $G_*^{\boldsymbol\vartheta}$ the full field
$(x,s,t,y)\mapsto\Psi_{s,t}(x,y)$ is jointly Borel, nonanticipative under
terminal restriction, locally jointly continuous on each profile core
$\mathcal C_R^{\boldsymbol\vartheta}$, and satisfies the exact
flow, terminal-causality, and future-shift identities simultaneously in all
parameters, with every future shift restricted to the remaining time
interval; for $V_0\ne0$, exact stop preservation refers to the
clock-augmented field;
\item under every $P\in\cP$, each deterministic section is the classical strong
It\^o solution of \eqref{eq:common-Ito-SDE-from-rough-flow};
\item the fixed-parameter regular upper-$L^r$ class is obtained directly from
state propagation; the independent Euler construction and its equality with
this class are established in \Cref{thm:Euler-state-interface};
\item with the solution field equipped with the compact-parameter
local-uniform topology, every stable-causal functional of that
field is a corewise continuous propagated response.  Under the additional
full-state roughness conditions of \Cref{cor:full-state-flow-readout}, it is a
stable-causal readout of $\mathcal A_2$ and therefore inherits
construction independence from the common state.
\end{enumerate}
Thus these dynamics are state-propagated; the Euler construction below identifies the
same field independently.
\end{corollary}

\begin{proof}
Parts~\textup{(i)}--\textup{(ii)} follow by composing the common state with
the mixed solution map, using
\Cref{thm:spatial-trace-tight-cores,prop:finite-output-mixed-stability,thm:common-Ito-flow};
uniqueness gives the flow, terminal-causality, and future-shift identities.

For part~\textup{(iii)}, fix $(s,y)$ and $r<\infty$.  The trace-clock bound
$\Tr a_t^P\le\Lambda$ and the global bounded/Lipschitz vector-field
hypotheses give, by the classical BDG--Gronwall estimate, a number $r'>r$ such
that
\[
 \sup_{P\in\cP}
 E^P\!\left[\sup_{t\in[s,T]}|\Psi_{s,t}(X,y)|^{r'}\right]<\infty.
\]
Hence the $r$-th powers have uniformly vanishing tails, while part~\textup{(i)}
gives continuity on every $\mathcal C_R^{\boldsymbol\vartheta}$.  Since
$c_{\cP}((\mathcal C_R^{\boldsymbol\vartheta})^c)
 \le c_{\mathcal P_{\rm sp}}((\mathcal C_R^{\boldsymbol\vartheta})^c)
 \to0$, this is a
quasi-continuous representative.  The characterization in
\Cref{thm:quasi-continuous-characterization} therefore yields
$\Psi_{s,\cdot}(\cdot,y)\in
\mathbb L_{\cP}^r(C([0,T];\mathbb R^m))$.

Part~\textup{(iv)} is closure under composition on the profile cores; its
full-state assertion is \Cref{cor:full-state-flow-readout}, together with
construction independence of the selected state.
\end{proof}

\subsection{Independent Euler completion and state equivalence}
\label{sec:Euler-state-interface}

The rough--Young solution map above propagates the selected state.  An
independent construction of the same It\^o dynamics is obtained from
first-level Euler schemes.  Pathwise Euler convergence under rough and
stochastic integrability conditions is developed in
\cite{AllanKwossekLiuPromel25}; here the schemes are completed in upper-$L^r$
uniformly over a nondominated law family.  The two constructions operate at
different regularity levels: global state-Lipschitz continuity suffices for
the common Euler/It\^o solution, whereas continuity as a readout of
$(\mathbf X^S,Q)$ requires the stronger rough vector-field hypotheses.

For a deterministic starting pair $(s,y)$ and a partition
$\pi=\{s=t_0<\cdots<t_N=T\}$, let $\ell_\pi(u)=t_i$ on
$[t_i,t_{i+1})$ and put $Y_t^{\pi;s,y}(x)=y$ for $t\le s$.  For $t\ge s$
define the raw continuous Euler map by
\begin{equation}\label{eq:raw-Euler-starting-pair}
\begin{split}
 Y_t^{\pi;s,y}(x)
 &=y+\int_s^t
 b\bigl(\ell_\pi(u),Y_{\ell_\pi(u)}^{\pi;s,y}(x)\bigr)\,\dd u\\
 &\quad+\sum_{i=0}^{N-1}
 \sigma\bigl(t_i,Y_{t_i}^{\pi;s,y}(x)\bigr)
 \bigl(x_{t\wedge t_{i+1}}-x_{t\wedge t_i}\bigr).
\end{split}
\end{equation}
Finite induction shows that this map is continuous in the raw uniform
topology and causal in the terminal time.

For $R<\infty$, write
\begin{equation}\label{eq:Euler-field-Banach-space}
 \mathscr B_R
 :=C\bigl([0,T]\times\overline B_R;
          C([0,T];\mathbb R^m)\bigr),
\end{equation}
where a solution started from $(s,y)$ is frozen at $y$ on $[0,s]$.

\begin{theorem}[Euler-generated It\^o field and state equivalence]
\label{thm:Euler-state-interface}
Let
$\varnothing\ne\mathcal P_{\rm sp}
 \subset\mathfrak S_{\Lambda,0}^{0,T}(H)$
have the uniform spatial trace-tightness profile $(P_m,\vartheta_m)$ fixed
above.  Let
\[
 b:[0,T]\times\mathbb R^m\to\mathbb R^m,
 \qquad
 \sigma:[0,T]\times\mathbb R^m\to\mathcal L(H,\mathbb R^m)
\]
be deterministic and continuous.  Assume, uniformly in time, that they are
globally Lipschitz in the state variable and have linear growth, with the
operator norm used for $\sigma$.
\begin{enumerate}[label=\textup{(\roman*)},leftmargin=2.5em]
\item For every deterministic $(s,y)$, every $1\le r<\infty$, and every
sequence of deterministic partitions $|\pi_n|\to0$, the schemes
$Y^{\pi_n;s,y}$ converge, independently of the partition sequence, in
\[
 \mathbb L_{\mathcal P_{\rm sp}}^r
 \bigl(C([0,T];\mathbb R^m)\bigr)
\]
to a common It\^o solution $Y^{s,y}$.  Under every
$P\in\mathcal P_{\rm sp}$ its projection is the unique classical strong
solution
\begin{equation}\label{eq:general-common-Ito-SDE}
 Y_t^{s,y}
 =y+\int_s^t b(u,Y_u^{s,y})\,\dd u
   +\int_s^t\sigma(u,Y_u^{s,y})\,\dd X_u,
 \qquad t\in[s,T].
\end{equation}
The stochastic term is the common integral of
\Cref{thm:uniform-BDG-extension} and its modelwise projections are understood
in the corresponding bracket-energy quotient.

\item There is a total jointly Borel raw-causal field
\[
 \Phi:\Omega\times\Delta_T\times\mathbb R^m\longrightarrow\mathbb R^m
\]
whose fixed-parameter sections represent the limits in part~\textup{(i)}.
For every $R$, the polar-equivalence class of the frozen fixed-parameter
family admits a compatible Borel representative
$\Phi^R:\Omega\to\mathscr B_R$ such that
\begin{equation}\label{eq:Euler-field-all-upper-Lr}
 \Phi^R\in\bigcap_{1\le r<\infty}
 \mathbb L_{\mathcal P_{\rm sp}}^r(\mathscr B_R).
\end{equation}
One diagonal of continuous finite-parameter interpolation fields, each built
from finitely many maps \eqref{eq:raw-Euler-starting-pair}, selects $\Phi$.
There are a Borel full-capacity set $G_{\rm E}$ and an increasing compact core
such that, simultaneously in every model and parameter,
$(s,t,y)\mapsto\Phi_{s,t}(x,y)$ is locally continuous on $G_{\rm E}$ and the
raw-path/parameter field is jointly continuous on compact parameter sets over
each core.  On $G_{\rm E}$, $\Phi$ and every $\Phi^R$ agree throughout their
parameter domains, and
\begin{equation}\label{eq:Euler-perfect-flow-law}
 \Phi_{s,t}(x,y)
 =\Phi_{u,t}\bigl(x,\Phi_{s,u}(x,y)\bigr)
\end{equation}
for every $x\in G_{\rm E}$, $0\le s\le u\le t\le T$, and
$y\in\mathbb R^m$.

\item Suppose additionally that
\begin{equation}\label{eq:Euler-rough-coefficient-interface}
 b(t,y)=V_0(y),\qquad \sigma(t,y)=V(y),
\end{equation}
where $V_0$ is bounded globally Lipschitz and, for some
$p_{\rm rp}\in(1/\alpha,3)$ and $\rho>p_{\rm rp}$,
\[
 V\in\operatorname{Lip}^{\rho+1}
 \bigl(\mathbb R^m;\mathcal L(H,\mathbb R^m)\bigr).
\]
Let $\Psi$ be the bracket-corrected state flow of
\Cref{thm:common-Ito-flow} on $G_*^{\boldsymbol\vartheta}$, totalized by
$\Psi_{s,t}(x,y)=y$ on its complement.  There is one Borel
$\mathcal P_{\rm sp}$-polar set $N_{\rm id}$, independent of
$P,s,t,y$.  Define
\begin{equation}\label{eq:Euler-state-common-domain}
 G_{\rm id}
 :=\bigl(G_{\rm E}\cap G_*^{\boldsymbol\vartheta}\bigr)
   \setminus N_{\rm id},
\end{equation}
Then $G_{\rm id}$ is Borel with polar complement, and
\begin{equation}\label{eq:Euler-state-field-equality}
 \Phi_{s,t}(x,y)=\Psi_{s,t}(x,y)
 \qquad
 (x\in G_{\rm id},\ 0\le s\le t\le T,\ y\in\mathbb R^m).
\end{equation}
Thus the independent first-level Euler construction and the corewise
continuous state propagation produce the same common It\^o cocycle.  If
$p_{\rm rp}>1/\eta$, the latter is the full-state continuous readout of
\Cref{cor:full-state-flow-readout}.
\end{enumerate}
\end{theorem}

\begin{proof}
Part~\textup{(i)} is \Cref{lem:trace-clock-Euler-sections}.  The spatial
profile gives uniform tightness and the explicit compact capacity cores of
\Cref{thm:spatial-trace-tight-cores}; hence
\Cref{lem:Euler-parameter-field-completion,lem:one-set-Euler-perfection}
give part~\textup{(ii)}, with $G_{\rm E}:=G_{\rm fl}$.  Under
\eqref{eq:Euler-rough-coefficient-interface},
\Cref{thm:common-Ito-flow} identifies $\Psi$ under each model with the same
classical strong solution as $\Phi$.  Both fields are locally continuous in
all flow parameters outside one polar set.  The comparison clause of
\Cref{lem:one-set-Euler-perfection}, which is the field-comparison application
of \Cref{lem:dense-skeleton-one-set-perfection}, gives part~\textup{(iii)}.
\end{proof}

\begin{remark}[Euler completion and parameterwise realization]
Part~\textup{(ii)} selects the field from finite families of Euler sections,
deterministic parameter interpolation, and one fast diagonal argument.  It
combines fixed-parameter upper-$L^r$ convergence with a single total field;
the parameter sections away from the interpolation nodes are supplied by the
completed field.  Equation \eqref{eq:Euler-perfect-flow-law} holds at
deterministic times on one common polar complement and permits Borel
substitution for the intermediate time.
\end{remark}

\begin{remark}[State, coherent flow, and derivative jet]
\label{rem:dynamics-propagation-not-reconstruction}
The nonlinear flow is a deterministic corewise continuous response of the
realized state; its global Borel version and common cocycle domain are
realization properties.  Under the full-state conditions of
\Cref{cor:full-state-flow-readout}, it is a stable-causal readout.
A declared flow derivative is propagated only when continuously recoverable
from the relevant topology, and otherwise requires its own jet completion.
\end{remark}

\section{Covariance tangents and backward first-jet selection}
\label{sec:covariance-backward-selection}

Throughout, the model class is $\Mmax$.  The pathwise covariance coordinate
$Q$ realizes simultaneously the Hilbertian form behind the
Kunita--Watanabe/GKW projection \cite{KunitaWatanabe67,Galtchouk76}.
Covariance responses first quotient its varying null spaces and then complete
the resulting fibre jet in common upper energy.  This precedes quasi-sure
identification and differs from single-law functional It\^o inversion
\cite{ContFournie13}.  The covariance channel supplies the tangent fibres;
skew area remains indispensable to forward rough dynamics.

\subsection{Covariance density and energy geometry}
\label{sec:covariance-tangent-tomography}

Fix the compact-envelope regime of \Cref{sec:master-core} and write
\begin{equation}\label{eq:Gamma-order-interval}
 \mathcal I_\Gamma
 :=\{a\in\mathcal S_1(H)_+:0\preceq a\preceq\Gamma\}.
\end{equation}
It is trace-norm compact: if $P_n\uparrow I$ is
finite-rank and $R_n=I-P_n$, positivity and trace-ideal
Cauchy--Schwarz give, uniformly for $a\in\mathcal I_\Gamma$,
\[
 \|a-P_naP_n\|_1
 \le2\sqrt{\Tr(\Gamma)\Tr(R_n\Gamma R_n)}
       +\Tr(R_n\Gamma R_n)\longrightarrow0.
\]
Finite-dimensional compactness then applies to the compressions.  No separate
tomography of arbitrary covariance sets is needed.

On $G_*^\Gamma$, \Cref{thm:master-causal-cores} gives a trace-norm
Lipschitz covariance path.  We first isolate the deterministic differentiation
mechanism used to choose its predictable density.

\begin{proposition}[Predictable differentiation of causal trace-class paths]
\label{prop:canonical-predictable-differentiation}
Let $q:\Omega_H\to C([0,T];\mathcal S_2(H)_{\rm sa})$ be total Borel and raw
causal, and let $G$ be a Borel stopping-stable set on which every $q(x)$ takes
values in $\mathcal S_1(H)_{\rm sa}$ and is Bochner absolutely continuous in
trace norm.  Put $\Delta_n=T2^{-n}$ and $t_k^n=k\Delta_n$, and define lagged
dyadic averages
\[
 (\mathscr D_nq)_t
 :=\begin{cases}
 0,&0\le t\le\Delta_n,\\[1mm]
 \Delta_n^{-1}\bigl(q_{t_k^n}-q_{t_{k-1}^n}\bigr),
 &t_k^n<t\le t_{k+1}^n,
 \end{cases}
\]
as an $\mathcal S_2(H)_{\rm sa}$-valued process.  Let
$\iota_1:\mathcal S_1(H)_{\rm sa}\hookrightarrow
\mathcal S_2(H)_{\rm sa}$ denote the continuous inclusion and set
\[
 L_1(q):=\bigl\{(t,x):(\mathscr D_nq)_t(x)\in
 \iota_1(\mathcal S_1(H)_{\rm sa})\ \text{for every }n\bigr\}.
\]
Define the totalized trace-class derivative by
\begin{equation}\label{eq:canonical-predictable-differentiation-totalization}
 (\mathscr Dq)_t(x):=
 \begin{cases}
 \operatorname{Lim}_{\mathcal S_1(H)_{\rm sa}}
 \bigl((\iota_1^{-1}(\mathscr D_nq)_t(x))_{n\ge1}\bigr),
 &(t,x)\in L_1(q),\\
 0,&(t,x)\notin L_1(q),
 \end{cases}
\end{equation}
where the Borel selector of \Cref{lem:causal-Borel-limit} returns zero when
the displayed sequence does not converge.  Then $\mathscr Dq$ is a total
Borel predictable
$\mathcal S_1(H)_{\rm sa}$-valued process and, for every $x\in G$,
\begin{equation}\label{eq:canonical-predictable-differentiation}
 q_t(x)=\int_0^t(\mathscr Dq)_r(x)\,\dd r,
 \qquad
 \mathscr Dq(x)=\dot q(x)\quad\text{a.e.}
\end{equation}
If the path field $q$ obeys exact stopping, shift, or concatenation identities,
then, on the corresponding loci where the transformed paths also belong to
$G$, the derivative classes obey the restriction, time-shift, and
concatenation identities almost everywhere.  The same conclusion holds for
substitution of a bounded raw stopping time whenever the stopped path remains
in $G$.
\end{proposition}

\begin{proof}
Each lagged average is elementary predictable as an $\mathcal S_2$-valued
process.  By Lusin--Souslin, $\iota_1$ has Borel image and Borel inverse;
hence $L_1(q)$ is predictable and the pulled-back averages are Borel there.
The selector in
\eqref{eq:canonical-predictable-differentiation-totalization} therefore gives
a total predictable $\mathcal S_1$-valued process.  At Bochner Lebesgue points
the lagged intervals lie in $(t-2\Delta_n,t)$ and their averages converge in
trace norm to $\dot q_t$, proving
\eqref{eq:canonical-predictable-differentiation}.
Differentiating the exact path identities gives the equivariance statements;
for random stopping the seam is fixed once the raw path is fixed.
\end{proof}

Apply this construction to the common covariance path.

\begin{proposition}[Common predictable covariance density]
\label{prop:common-predictable-covariance-density}
Let $\Delta_n=T2^{-n}$ and $t_k^n=k\Delta_n$.  Define the
$\mathcal S_2(H)_{\rm sa}$-valued elementary predictable process
\begin{equation}\label{eq:common-covariance-density-approximants}
 \mathfrak a_t^{(n)}
 :=\begin{cases}
 0,&0\le t\le \Delta_n,\\[1mm]
 \Delta_n^{-1}\bigl(Q_{t_k^n}-Q_{t_{k-1}^n}\bigr),
 &t_k^n<t\le t_{k+1}^n,\quad 1\le k<2^n.
 \end{cases}
\end{equation}
Let $\mathscr DQ$ be the total trace-class derivative furnished by
\Cref{prop:canonical-predictable-differentiation}; its approximants are
\eqref{eq:common-covariance-density-approximants}.  Define
\begin{equation}\label{eq:common-covariance-density}
 \mathfrak a_t
 :=\begin{cases}
 (\mathscr DQ)_t,&(\mathscr DQ)_t\in\mathcal I_\Gamma,\\
 0,&\text{otherwise.}
 \end{cases}
\end{equation}
Then $\mathfrak a$ is a total Borel predictable
$\mathcal I_\Gamma$-valued process.  For every $x\in G_*^\Gamma$, it is a
predictable representative of the canonical derivative class associated with
$Q(x)$ by \Cref{prop:covariance-path-density-isomorphism}.  In particular,
$Q(x)$ is Bochner absolutely continuous in $\mathcal S_1(H)$ and
\begin{equation}\label{eq:common-covariance-pathwise-density}
 Q_t(x)=\int_0^t\mathfrak a_r(x)\,\dd r,
 \qquad
 0\preceq\mathfrak a_r(x)\preceq\Gamma
 \quad\text{for a.e. }r.
\end{equation}
Under every $P\in\Mmax$,
\begin{equation}\label{eq:common-density-modelwise}
 \mathfrak a_t=a_t^P
 \qquad \dd t\otimes P\text{-a.e.}
\end{equation}
It is the same density under every law and is unique in the common product
class: for any other such $\widetilde{\mathfrak a}$,
\begin{equation}\label{eq:common-density-upper-product-uniqueness}
 \sup_{P\in\Mmax}
 E^P\int_0^T
 \|\mathfrak a_t-\widetilde{\mathfrak a}_t\|_1\,\dd t=0.
\end{equation}
\end{proposition}

\begin{proof}
The totalization in \Cref{prop:canonical-predictable-differentiation} is
Borel and predictable; restriction to the Borel closed interval
$\mathcal I_\Gamma$ gives \eqref{eq:common-covariance-density}.  On
$G_*^\Gamma$ the derivative lies in that interval almost everywhere, and
\Cref{prop:covariance-path-density-isomorphism} yields
\eqref{eq:common-covariance-pathwise-density} on $G_*^\Gamma$.

Under $P$, $Q_t=\int_0^ta_r^P\,\dd r$ in $\mathcal S_1$; uniqueness of
Bochner derivatives gives \eqref{eq:common-density-modelwise}, and the same
argument integrated over the product space gives
\eqref{eq:common-density-upper-product-uniqueness}.
\end{proof}

\begin{corollary}[Restart covariance of the common density]
\label{cor:common-density-restart-covariance}
Let $x,y\in G_*^\Gamma$ and let $s\in[0,T]$.  Up to equality for Lebesgue
almost every time,
\begin{align}
 \mathfrak a_t(r_sx)
 &=\one_{\{t<s\}}\mathfrak a_t(x),
 \label{eq:common-density-stopping}\\
 \mathfrak a_v(\theta_sx)
 &=\one_{\{v<T-s\}}\mathfrak a_{s+v}(x),
 \label{eq:common-density-shift}\\
 \mathfrak a_t(x\otimes_sy)
 &=\one_{\{t<s\}}\mathfrak a_t(x)
   +\one_{\{t>s\}}\mathfrak a_{t-s}(y).
 \label{eq:common-density-concatenation}
\end{align}
If $\tau$ is a bounded raw stopping time and $x\in G_*^\Gamma$, then
\begin{equation}\label{eq:common-density-random-stopping}
 \mathfrak a_t(r_\tau x)
 =\one_{\{t<\tau(x)\}}\mathfrak a_t(x)
 \qquad\text{for a.e. }t.
\end{equation}
Thus deterministic restart and random stopping transport the covariance-energy
fibres by the corresponding restriction and time-shift rules, with the same
common density representative.
\end{corollary}

\begin{proof}
Differentiate the exact stopping, shift, and concatenation identities using
\Cref{prop:canonical-predictable-differentiation}; at the pathwise seam
$s=\tau(x)$ this also gives \eqref{eq:common-density-random-stopping}.
\end{proof}

\begin{definition}[Covariance-energy fiber]
\label{def:covariance-energy-fiber}
For $a\in\mathcal S_1(H)_+$ define
\begin{equation}\label{eq:covariance-energy-fiber}
 \mathscr E_a
 :=\overline{H/\ker a}^{\,\|\cdot\|_a},
 \qquad
 \|[h]\|_a^2:=\langle ah,h\rangle_H.
\end{equation}
The map
\begin{equation}\label{eq:covariance-energy-fiber-embedding}
 \iota_a:[h]\longmapsto a^{1/2}h
\end{equation}
extends isometrically onto
$\overline{\operatorname{Ran}a^{1/2}}\subset H$.  Elements of $\mathscr E_a$
are called covariance-energy vectors.
\end{definition}

\begin{corollary}[Covariance-energy response quotient]
\label{cor:covariance-fiber-energy-quotient}
For $h,k\in H$ the following are equivalent:
\[
 [h]=[k]\text{ in }\mathscr E_a,
 \qquad
 \langle a(h-k),h-k\rangle=0,
 \qquad
 a^{1/2}h=a^{1/2}k.
\]
Consequently every continuous response built only from covariance pairings
$\langle a\cdot,\cdot\rangle$ descends uniquely to $\mathscr E_a$, and two
directions have the same response exactly when their difference lies in
$\ker a$.  Thus $\mathscr E_a$ is the response quotient for first-order
covariance-energy directions, defined before quasi-sure identification.
\end{corollary}

\begin{proof}
The equivalence is immediate from the definition of the seminorm and the
isometric embedding $[h]\mapsto a^{1/2}h$.  A covariance pairing vanishes on
$\ker a$ in each argument, so it is representative-independent and factors
through the quotient.  Exactness follows because the norm itself is one of the
descended responses.
\end{proof}

\begin{lemma}[Borel realization of the covariance-energy bundle]
\label{lem:Borel-covariance-energy-bundle}
Let
\begin{equation}\label{eq:Borel-covariance-energy-bundle}
 \mathscr E_\Gamma^\sharp
 :=\left\{(a,u)\in\mathcal I_\Gamma\times H:
 u\in\overline{\operatorname{Ran}a^{1/2}}\right\}.
\end{equation}
Then $\mathscr E_\Gamma^\sharp$ is a Borel subset of
$\mathcal I_\Gamma\times H$.  Via the fiberwise isometries $\iota_a$, it gives
the disjoint union $\bigsqcup_{a\in\mathcal I_\Gamma}\mathscr E_a$ a
canonical standard Borel Hilbert-bundle structure.  In particular, if
$a_t$ is a predictable Borel $\mathcal I_\Gamma$-valued process, a section
$U_t\in\mathscr E_{a_t}$ is predictable Borel exactly when
$(a_t,\iota_{a_t}U_t)$ is predictable Borel as an
$\mathscr E_\Gamma^\sharp$-valued map.
\end{lemma}

\begin{proof}
For $m\ge1$ set
\[
 R_m(a):=a^{1/2}(a+m^{-1}I)^{-1/2}
       =f_m(a),
 \qquad
 f_m(\lambda)=\sqrt{\frac{\lambda}{\lambda+m^{-1}}}.
\]
Since trace-norm convergence implies operator-norm convergence and $f_m$ is
continuous on $[0,\|\Gamma\|_{\rm op}]$, the map
$a\mapsto R_m(a)$ is operator-norm continuous on $\mathcal I_\Gamma$.
Moreover $R_m(a)$ converges strongly to the support projection of $a$, whose
range is $\overline{\operatorname{Ran}a^{1/2}}$.  Hence
\[
 (a,u)\in\mathscr E_\Gamma^\sharp
 \quad\Longleftrightarrow\quad
 \|R_m(a)u-u\|_H\longrightarrow0.
\]
The right-hand condition is Borel in $(a,u)$, proving the first assertion.
The remaining statements transport this Borel structure through the
fiberwise isometries \eqref{eq:covariance-energy-fiber-embedding}.
\end{proof}

\subsection{Program-dependent visibility: covariance, area, and dynamics}
\label{sec:first-order-visibility}

\begin{remark}[Program-relative visibility of the area channel]
\label{cor:program-dependent-area-visibility}
At a martingale restart, Hilbert--Schmidt pairings of skew It\^o area are
centered and have second moment $O(h^2)$, so signed linear mean-area germs do
not enter this covariance derivative.  The invisibility is program-relative:
the quadratic response recovers $A=\operatorname{Anti}V$, forward flows read
it, and terminal data or drivers may use it environmentally.  Thus $\ker a$
is removed by an exact quotient, whereas omitting area requires an explicit
program restriction.
\end{remark}

Consequently $Q$, but not skew area, generates the fibres:
\begin{equation}\label{eq:state-to-covariance-energy-geometry}
 (X,A,Q)\longmapsto
 (\mathfrak a_t,\mathscr E_{\mathfrak a_t}).
\end{equation}

\subsection{Backward first-jet state selection}
\label{sec:backward-covariance-face}

The derivative leg reads $Q$ through
$(\mathfrak a_t,\mathscr E_{\mathfrak a_t})$ and factors as
\begin{equation}\label{eq:quadratic-to-backward-interface}
 (X,A,Q)\longmapsto(X,Q)
 \longmapsto (X,Q,\mathfrak a_t,\mathscr E_{\mathfrak a_t}).
\end{equation}
Area nevertheless remains environmental unless every terminal, driver, and
continuation response factors through $(X,A,Q)\mapsto(X,Q)$; it never enters
the covariance-energy fibre of $Z$.

We work on a nonempty law family $\varnothing\ne\cP\subset\Mmax$ and
formulate the backward interface for scalar responses.  The common stochastic
integral is the stable-causal closed embedding
\[
 I:\mathbb H_{\cP}^2\longrightarrow\mathbb S_{\cP}^2
\]
from \Cref{thm:uniform-BDG-extension}; two-sided BDG reconstructs the jet
topology.  Locally,
\[
 \|Y\|_{\mathbb S_{\cP}^2}
 :=\sup_{P\in\cP}\left(E^P\sup_{t\le T}|Y_t|^2\right)^{1/2},
\]
and $\mathbb H_{\cP}^2$ is the upper-energy completion of elementary
predictable scalar integrands.  For terminal variables write
\[
 \|\xi\|_{\mathcal L_{\cP}^2}
 :=\sup_{P\in\cP}(E^P|\xi|^2)^{1/2}.
\]

For an elementary predictable coefficient $Z$, define its active covariance
section by
\[
 \mathfrak J_{\mathfrak a}^0Z
 :=\mathfrak a^{1/2}Z.
\]
Since $\mathfrak a=a^P$ $\dd t\otimes P$-a.e.,
\begin{equation}\label{eq:active-coordinate-isometry}
 \|Z\|_{\mathbb H_{\cP}^2}
 =\sup_{P\in\cP}
 \left(E^P\int_0^T
 \|(\mathfrak J_{\mathfrak a}^0Z)_t\|_H^2\,\dd t\right)^{1/2}.
\end{equation}
First quotient the elementary active sections by the null space of the
displayed upper-energy seminorm, and let
$\mathbb U_{\mathfrak a,\cP}^2$ be the resulting completion.  Then
\eqref{eq:active-coordinate-isometry} extends uniquely to an isometric
bijection
\begin{equation}\label{eq:completed-active-coordinate-map}
 \mathfrak J_{\mathfrak a}:\mathbb H_{\cP}^2
 \longrightarrow\mathbb U_{\mathfrak a,\cP}^2.
\end{equation}
The target is this common active-section closure, so
$\mathfrak J_{\mathfrak a}$ is onto its declared common space.

\begin{remark}[Interfaces in the backward construction]
\label{rem:backward-construction-interfaces}
The selected covariance $Q$ determines the energy fibers.  On an admissible
raw testing core, covariance cross-responses construct the fiber-valued jet;
condition~\textup{(A2)} places its active section in one common upper-energy
class and gives its lawwise GKW interpretation.  Condition~\textup{(B4)}
identifies that class with a solver martingale coordinate.  Value-only
compression is governed by the closed-relation or restricted-range criteria.
\end{remark}

\subsubsection{Raw covariance-response jets}

\begin{proposition}[Covariance cross-responses and the orthogonal-martingale boundary]
\label{prop:orthogonal-martingale-derivative-boundary}
Fix $P\in\cP$ and let $M$ be a continuous square-integrable scalar
$P$-martingale.  Its Galtchouk--Kunita--Watanabe decomposition relative to the
closed stable subspace generated by $X$ is
\begin{equation}\label{eq:GKW-backward-boundary}
 M_t-M_0=\int_0^t\langle z_r^P,\dd X_r\rangle_H+N_t^P,
\end{equation}
where $N^P$ is strongly orthogonal to every square-integrable stochastic
integral against $X$.  In particular, for every $h\in H$,
\[
 [N^P,\langle X,h\rangle]=0,
 \qquad
 \frac{\dd[M,\langle X,h\rangle]_t}{\dd t}
 =\langle a_t^Pz_t^P,h\rangle_H
\]
whenever the displayed density is defined.  Hence covariance cross-responses
determine only the energy class of $z^P$ and are blind to $N^P$.

The value trajectory still contains $N^P$.  A common channel $I(Z)$ is exact
only if every $N^P=0$ and one $Z\in\mathbb H_{\cP}^2$ projects to every
$z^P$: respectively modelwise and common-energy representability.
\end{proposition}

\begin{proof}
This is the Kunita--Watanabe/Galtchouk orthogonal projection
\cite{KunitaWatanabe67,Galtchouk76}.  Strong orthogonality and
$\dd\qv{X}^P_t=a_t^P\dd t$ give the two bracket identities; the remaining
claims follow.
\end{proof}

Fix an at most countable orthonormal basis $(e_j)$ of $H$ and let
\[
 H_0:=\operatorname{span}_{\mathbb R}\{e_j:j\text{ is a basis index}\}
\]
be its dense algebraic finite-span.  For candidate trajectories, the testing
program supplies a pathwise algorithm producing
$[Y,\langle X,h\rangle]^{\rm pw}$ for $h\in H_0$, linearly in $h$.  At
points where these covariations are absolutely continuous, write
\begin{equation}\label{eq:backward-covariance-cotangent}
 \lambda_t^Y(h)
 :=\frac{\dd [Y,\langle X,h\rangle]^{\rm pw}_t}{\dd t},
 \qquad h\in H_0.
\end{equation}
Only the basis coordinates are used for measurability; real-linear extension
to $H_0$ is part of the testing structure.

\begin{proposition}[Energy-bounded Riesz reconstruction of the covariance first jet]
\label{prop:covariance-first-jet-riesz-forcing}
Fix a point $(t,\omega)$ at which $\mathfrak a_t$ and the linear functional
$\lambda_t^Y$ on $H_0$ are defined.  Assume that
\begin{equation}\label{eq:backward-cotangent-energy-bound}
 |\lambda_t^Y(h)|
 \le g_t\,\langle\mathfrak a_th,h\rangle_H^{1/2},
 \qquad h\in H_0,
\end{equation}
for some finite $g_t$.  Then $\lambda_t^Y$ vanishes on
$H_0\cap\ker\mathfrak a_t$, descends uniquely to a continuous functional on
the covariance-energy fiber $\mathscr E_{\mathfrak a_t}$, and there is a
unique
\[
 D_Q^{\rm pw}Y_t\in\mathscr E_{\mathfrak a_t}
\]
such that
\begin{equation}\label{eq:backward-first-jet-chart}
 \lambda_t^Y(h)
 =\left\langle D_Q^{\rm pw}Y_t,[h]\right\rangle_{\mathscr E_{\mathfrak a_t}},
 \qquad h\in H_0,
\end{equation}
with $\|D_Q^{\rm pw}Y_t\|_{\mathscr E_{\mathfrak a_t}}\le g_t$.
Since $\mathfrak a_t$ is bounded, the energy bound also makes
$\lambda_t^Y$ continuous for the $H$-norm.  Let $\bar\lambda_t^Y$ be its
unique $H$-continuous extension and let $c_t^Y\in H$ be the ordinary
$H$-Riesz representative of $\bar\lambda_t^Y$.  Under the fiber isometry
$\iota_{\mathfrak a_t}$, put
$u_t^Y:=\iota_{\mathfrak a_t}D_Q^{\rm pw}Y_t$.  Then
\begin{equation}\label{eq:backward-pseudoinverse-factorization}
 c_t^Y=\mathfrak a_t^{1/2}u_t^Y,
 \qquad
 u_t^Y=\lim_{\delta\downarrow0}
 (\mathfrak a_t+\delta I)^{-1/2}c_t^Y
 \quad\text{in }H.
\end{equation}
Moreover,
\begin{equation}\label{eq:backward-jet-exact-dual-norm}
 \|D_Q^{\rm pw}Y_t\|_{\mathscr E_{\mathfrak a_t}}
 =\sup_{\substack{h\in H_0\\
          \langle\mathfrak a_th,h\rangle\le1}}|\lambda_t^Y(h)|
 =\sup_{\substack{h\in H\\
          \langle\mathfrak a_th,h\rangle\le1}}|\bar\lambda_t^Y(h)|.
\end{equation}
This common value is the least constant $g$ for which
\eqref{eq:backward-cotangent-energy-bound} holds.  Conversely every fiber
vector defines a continuous covariance cotangent functional satisfying that
bound.

If $(t,\omega)\mapsto\mathfrak a_t(\omega)$ and the countable basis
coordinates $(t,\omega)\mapsto\lambda_t^Y(e_j)$ are predictable, then the
jet has a canonical total predictable realization.  Namely, define
\begin{align}
 c_t^{Y,\mathrm{tot}}
 &:=\operatorname{Lim}_H\left(
   \left(\sum_{j\le N}\lambda_t^Y(e_j)e_j\right)_{N\ge1}\right),
 \label{eq:backward-total-Riesz-vector}\\
 u_t^{Y,\mathrm{tot}}
 &:=\operatorname{Lim}_H\left(
   \left((\mathfrak a_t+n^{-1}I)^{-1/2}
          c_t^{Y,\mathrm{tot}}\right)_{n\ge1}\right),
 \label{eq:backward-total-pseudogradient}
\end{align}
and set $u_t^{Y,\mathrm{tot}}=0$ whenever
$u_t^{Y,\mathrm{tot}}\notin\overline{\operatorname{Ran}\mathfrak a_t^{1/2}}$.
Then $(\mathfrak a_t,u_t^{Y,\mathrm{tot}})$ is a predictable section of the
Borel bundle in \Cref{lem:Borel-covariance-energy-bundle}; on every point where
the energy bound holds it equals
$\iota_{\mathfrak a_t}D_Q^{\rm pw}Y_t$.  Thus no measurable selection is made
from an uncountable family of representatives.
\end{proposition}

\begin{proof}
The bound annihilates $H_0\cap\ker\mathfrak a_t$ and is continuous for
$h\mapsto\|\mathfrak a_t^{1/2}h\|$.  Density of
$\mathfrak a_t^{1/2}H_0$ in
$\overline{\operatorname{Ran}\mathfrak a_t^{1/2}}$ and Riesz representation
give the unique fibre vector and \eqref{eq:backward-jet-exact-dual-norm}.

The operator bound also extends $\lambda_t^Y$ to $H$, with Riesz vector
$c_t^Y$.  For $u_t^Y=\iota_{\mathfrak a_t}D_Q^{\rm pw}Y_t$ and then by
density,
\[
 \langle c_t^Y,h\rangle_H
 =\bar\lambda_t^Y(h)
 =\langle u_t^Y,\mathfrak a_t^{1/2}h\rangle_H
 =\langle\mathfrak a_t^{1/2}u_t^Y,h\rangle_H.
\]
Thus $c_t^Y=\mathfrak a_t^{1/2}u_t^Y$.  Functional calculus gives
\[
 (\mathfrak a_t+\delta I)^{-1/2}c_t^Y
 =\left(\frac{\mathfrak a_t}{\mathfrak a_t+\delta I}\right)^{1/2}u_t^Y
 \longrightarrow u_t^Y
\]
by spectral dominated convergence.  The converse is fibre Cauchy--Schwarz.

For predictability, each
\[
 c_t^{Y,N}:=\sum_{j\le N}\lambda_t^Y(e_j)e_j
\]
is predictable.  The Borel limit selector gives
\eqref{eq:backward-total-Riesz-vector}.  Borel functional calculus for
$(a,c)\mapsto(a+n^{-1}I)^{-1/2}c$ and a second limit selector give
\eqref{eq:backward-total-pseudogradient}, which equals $u_t^Y$ on the valid
locus.  Since $\mathscr E_\Gamma^\sharp$ is Borel by
\Cref{lem:Borel-covariance-energy-bundle}, zero extension yields the claimed
total predictable section.
\end{proof}

The preceding proposition gives the pointwise measurable Riesz
reconstruction.  The following definition adds the nondominated requirement
that its active section belong to one common completion.

\begin{definition}[Admissible backward testing core]
\label{def:admissible-backward-testing-core}
A linear operational class $\mathscr C_{\rm bw}^{1,2}$ of bounded
raw-adapted continuous trajectories is an \emph{admissible backward testing
core} if it is invariant under deterministic stopping and the following hold.
\begin{enumerate}[label=\textup{(A\arabic*)},leftmargin=2.8em]
\item Under every $P\in\cP$, each $Y\in\mathscr C_{\rm bw}^{1,2}$ is a
continuous square-integrable special semimartingale.  Write its normalized
canonical decomposition as
\[
 Y=Y_0+M^{Y,P}+A^{Y,P},
 \qquad M^{Y,P}\in\mathcal M^2(P),
\]
with $A^{Y,P}$ predictable of finite variation and both summands zero at
time zero.  For every $h\in H_0$, the pathwise covariation agrees $P$-a.s.
with the classical covariation of $M^{Y,P}$ with $\langle X,h\rangle$.
The agreement is required on one common full set for the countable basis
coordinates; real linearity then gives it simultaneously on $H_0$.
\item The covariations are pathwise absolutely continuous and their countable
basis densities $\lambda^Y(e_j)$ have predictable total versions.  For each
$Y$ there are a nonnegative predictable $g^Y$ and a $\cP$-polar set $N_Y$
such that, on $N_Y^c$, for
Lebesgue-almost every $t$, simultaneously for every $h\in H_0$,
\[
 |\lambda_t^Y(h)|
 \le g_t^Y\langle\mathfrak a_th,h\rangle_H^{1/2},
\]
and
\[
 \sup_{P\in\cP}E^P\int_0^T|g_t^Y|^2\,\dd t<\infty.
\]
Let $U^Y:=\iota_{\mathfrak a}D_Q^{\rm pw}Y$ be the predictable Riesz
section constructed in \Cref{prop:covariance-first-jet-riesz-forcing}.  It is
required that
\[
 [U^Y]\in\mathbb U_{\mathfrak a,\cP}^2
       =\operatorname{Ran}\mathfrak J_{\mathfrak a},
 \qquad
 [D_Q^{\rm pw}Y]_{\mathbb H_{\cP}^2}
 :=\mathfrak J_{\mathfrak a}^{-1}[U^Y].
\]
Equivalently, there are bounded elementary predictable integrands $Z^{Y,n}$
such that
\[
 \lim_{n\to\infty}\sup_{P\in\cP}E^P\int_0^T
 \|\mathfrak a_t^{1/2}Z_t^{Y,n}-U_t^Y\|_H^2\,\dd t=0.
\]
Thus the requirement is stronger than membership of $[U^Y]$ in the product of
the modelwise finite-energy spaces.
\item The covariation algorithm is linear and stopping compatible:
\begin{equation}\label{eq:backward-jet-stopping}
 D_Q^{\rm pw}(r_tY)=\one_{[0,t]}D_Q^{\rm pw}Y,
 \qquad 0\le t\le T.
\end{equation}
\item The observable testing graph defined in
\eqref{eq:backward-observable-testing-graph} is separable for the product
metric of $\mathbb S_{\cP}^2\times\mathbb H_{\cP}^2$.
\end{enumerate}
Thus \textup{(A2)} is a genuine common-completion assumption, stronger than
uniform modelwise energy; \textup{(A4)} only makes the completion causal
Polish.
\end{definition}

\begin{corollary}[Modelwise GKW identification of the covariance jet]
\label{cor:modelwise-GKW-covariance-jet}
Let $Y\in\mathscr C_{\rm bw}^{1,2}$ and fix $P\in\cP$.  Suppose the
$P$-martingale part of $Y$ has Galtchouk--Kunita--Watanabe decomposition
\[
 M_t^Y-M_0^Y
 =\int_0^t\langle z_r^P,\dd X_r\rangle_H+N_t^P
\]
as in \eqref{eq:GKW-backward-boundary}, and suppose the pathwise covariation
algorithm agrees with classical covariation under $P$.  Then, for
$\dd t\otimes P$-almost every point,
\begin{equation}\label{eq:modelwise-GKW-jet-identification}
 D_Q^{\rm pw}Y_t=[z_t^P]_{\mathscr E_{a_t^P}},
 \qquad
 \iota_{\mathfrak a_t}D_Q^{\rm pw}Y_t
 =(a_t^P)^{1/2}z_t^P,
\end{equation}
and
\begin{equation}\label{eq:modelwise-GKW-jet-energy}
 \|D_Q^{\rm pw}Y_t\|_{\mathscr E_{\mathfrak a_t}}^2
 =\langle a_t^Pz_t^P,z_t^P\rangle_H.
\end{equation}
In particular the jet is independent of the chosen representative of the GKW
integrand: two coefficients define the same covariance jet exactly when their
difference has zero $a_t^P$-energy.  The orthogonal martingale $N^P$ does not
enter the jet, although it remains visible in the value trajectory.
\end{corollary}

\begin{proof}
By \Cref{prop:common-predictable-covariance-density},
$\mathfrak a_t=a_t^P$ for $\dd t\otimes P$-almost every point.  The bracket
identity in \Cref{prop:orthogonal-martingale-derivative-boundary} holds on a
common $\dd t\otimes P$-full set for the countable basis coordinates.
Intersecting these full sets and using real linearity gives, simultaneously
for $h\in H_0$,
\[
 \lambda_t^Y(h)=\langle a_t^Pz_t^P,h\rangle_H
 =\langle[z_t^P],[h]\rangle_{\mathscr E_{a_t^P}}.
\]
Uniqueness in the Riesz reconstruction
\Cref{prop:covariance-first-jet-riesz-forcing} proves
\eqref{eq:modelwise-GKW-jet-identification}; the norm identity is the
definition of the covariance-energy fiber.  The final statements follow from
$[z]=[\widetilde z]$ if and only if
$(a_t^P)^{1/2}(z-\widetilde z)=0$ and from the strong orthogonality of $N^P$.
\end{proof}

\begin{corollary}[Identification with the common martingale integrand]
\label{cor:common-integrand-equals-covariance-jet}
Let $Y\in\mathscr C_{\rm bw}^{1,2}$ and suppose that its martingale part is
represented, under every $P\in\cP$, by the modelwise projection of one common
upper-energy class $Z\in\mathbb H_{\cP}^2$ and its integral class $I(Z)$.
If the pathwise covariation
algorithm agrees modelwise with classical covariation, then
\begin{equation}\label{eq:common-integrand-jet-identity}
 [D_Q^{\rm pw}Y]_{\mathbb H_{\cP}^2}=Z.
\end{equation}
Thus on the martingale-representation branch of the backward interface, the
response-generated covariance jet coincides with the conventional stochastic
integrand.
\end{corollary}

\begin{proof}
For fixed $P$, the GKW projection coefficient of the martingale part $I(Z)$ is
$j_P(Z)$ modulo zero $a_t^P$-energy.  By
\Cref{cor:modelwise-GKW-covariance-jet},
\[
 j_P(D_Q^{\rm pw}Y)=j_P(Z)
 \qquad\text{in }\mathbb H^2(P).
\]
Taking the supremum over $P$ in the energy-product embedding
\eqref{eq:energy-product-embedding} proves
\eqref{eq:common-integrand-jet-identity}.
\end{proof}

\begin{remark}[Intrinsic active covariance coordinate]
The identity $[D_Q^{\rm pw}Y]_{\mathbb H_{\cP}^2}=Z$ concerns common
upper-energy classes, whose intrinsic active coordinate is
\[
 [U^Y]=[\iota_{\mathfrak a}D_Q^{\rm pw}Y]
    =\mathfrak J_{\mathfrak a}Z,
\]
the common upper-$L^2$ limit of $\mathfrak a^{1/2}Z^n$ along elementary
representatives.  Only this coordinate is observable: neither a distinguished
$Z_t\in H$, nor a pointwise inverse of $\mathfrak a_t$, nor a representative
modulo $\ker\mathfrak a_t$ is selected.
The normalized innovation notation
$\mathfrak a^{\dagger/2}\!\cdot X$ would require additional integrability or
closed-range hypotheses and is not used here.
\end{remark}

Admissibility quotients $\ker\mathfrak a_t$ before completion and makes the
construction linear and stopping compatible; hence the following graph and
its closure are linear and truncation invariant.

Define the observable testing graph
\begin{equation}\label{eq:backward-observable-testing-graph}
 \mathfrak G_{Q,\cP}^{\rm bw,0}
 :=\left\{
   \bigl([Y]_{\mathbb S^2},[D_Q^{\rm pw}Y]_{\mathbb H^2}\bigr):
   Y\in\mathscr C_{\rm bw}^{1,2}\right\}
 \subset \mathbb S_{\cP}^2\times\mathbb H_{\cP}^2,
\end{equation}
where the two coordinates are respectively the quasi-sure process class and
the common upper-energy class; the fiber kernel has already been quotiented.
This raw relation is single-valued over its value projection: equality in
$\mathbb S_{\cP}^2$ gives indistinguishable continuous versions under every
$P$, hence equal classical covariations by admissibility and zero energy for
the jet difference under every law.

\subsubsection{Response-forced geometry, closed recovery, and compression}

\begin{definition}[Stable first-jet state and value--martingale response target]
\label{def:stable-backward-first-jet-state}
The \emph{stable backward first-jet space} is
\begin{equation}\label{eq:stable-backward-first-jet-space}
 \mathscr W_{Q,\cP}^{1,2}
 :=\overline{\mathfrak G_{Q,\cP}^{\rm bw,0}}
 ^{\,\mathbb S_{\cP}^2\times\mathbb H_{\cP}^2}.
\end{equation}
Its truncations are
\[
 \rho_t^{\mathscr W}(Y,Z):=(r_tY,\one_{[0,t]}Z),
 \qquad 0\le t\le T,
\]
and its graph metric is
\begin{equation}\label{eq:backward-graph-metric}
 d_{\mathscr W}((Y,Z),(\widetilde Y,\widetilde Z))
 :=\|Y-\widetilde Y\|_{\mathbb S_{\cP}^2}
   +\|Z-\widetilde Z\|_{\mathbb H_{\cP}^2}.
\end{equation}

The raw \emph{value--martingale response image} and its completion are
\begin{align}
 \mathscr R_{\rm VI}^{\rm bw,0}
 &:=\left\{\bigl([Y]_{\mathbb S^2},I(D_Q^{\rm pw}Y)\bigr):
       Y\in\mathscr C_{\rm bw}^{1,2}\right\},
 \label{eq:backward-VI-raw-response}\\
 \mathscr R_{\rm VI}^{\rm bw}
 &:=\overline{\mathscr R_{\rm VI}^{\rm bw,0}}
 ^{\,\mathbb S_{\cP}^2\times\mathbb S_{\cP}^2}.
 \label{eq:backward-VI-response-completion}
\end{align}
By the standing separability condition and completeness of the ambient Banach
products, both $\mathscr W_{Q,\cP}^{1,2}$ and
$\mathscr R_{\rm VI}^{\rm bw}$ are Polish.  The deterministic truncations
are continuous contractions, so these completions are causal Polish spaces in
the sense of \Cref{def:stable-causal}.

The response truncation is $(Y,M)\mapsto(r_tY,r_tM)$ and the response metric is
\begin{equation}\label{eq:backward-VI-response-metric}
 d_{\rm VI}((Y,M),(\widetilde Y,\widetilde M))
 :=\|Y-\widetilde Y\|_{\mathbb S_{\cP}^2}
   +\|M-\widetilde M\|_{\mathbb S_{\cP}^2}.
\end{equation}
Finally define the vertical space
\begin{equation}\label{eq:backward-vertical-space}
 \mathcal N_Q^{\rm vert}
 :=\{Z\in\mathbb H_{\cP}^2:(0,Z)\in\mathscr W_{Q,\cP}^{1,2}\}.
\end{equation}
\end{definition}

\begin{proposition}[Exact martingale calibration at deterministic and stopping horizons]
\label{prop:backward-exact-martingale-calibration}
Let $\tau:\Omega\to[0,T]$ be a bounded raw stopping time.  There is a
canonical contractive truncation
\[
 \mathsf T_\tau:\mathbb H_{\cP}^2\longrightarrow\mathbb H_{\cP}^2,
 \qquad
 \mathsf T_\tau Z=:\one_{[0,\tau]}Z,
\]
obtained as the upper-energy limit of dyadic predictable stopping
truncations.  It satisfies
\begin{equation}\label{eq:common-integral-random-stopping}
 I(\one_{[0,\tau]}Z)=r_\tau I(Z)
 \qquad\text{in }\mathbb S_{\cP}^2,
\end{equation}
and, for bounded raw stopping times $\sigma,\tau$,
\begin{equation}\label{eq:random-stopping-semigroup}
 \mathsf T_\sigma\mathsf T_\tau
 =\mathsf T_{\sigma\wedge\tau}.
\end{equation}
The truncation is continuous in the stopping horizon: if
$\|\tau_n-\tau\|_\infty\to0$, where
$\|\sigma-\tau\|_\infty:=\sup_{\omega\in\Omega}|\sigma(\omega)-\tau(\omega)|$
is the raw uniform norm, then, for every
$Z\in\mathbb H_{\cP}^2$,
\begin{equation}\label{eq:random-stopping-horizon-continuity}
 \|\mathsf T_{\tau_n}Z-\mathsf T_\tau Z\|_{\mathbb H_{\cP}^2}
 +\|r_{\tau_n}I(Z)-r_\tau I(Z)\|_{\mathbb S_{\cP}^2}
 \longrightarrow0.
\end{equation}
Moreover it obeys the exact stopping-horizon It\^o isometry
\begin{equation}\label{eq:backward-stopping-Ito-isometry}
 \boxed{
 \|\one_{[0,\tau]}Z\|_{\mathbb H_{\cP}^2}
 =\|I_\tau(Z)\|_{\mathcal L_{\cP}^2}.}
\end{equation}
For deterministic $\tau=t$ this is
\begin{equation}\label{eq:backward-terminal-Ito-isometry}
 \|\one_{[0,t]}Z\|_{\mathbb H_{\cP}^2}
 =\|I_t(Z)\|_{\mathcal L_{\cP}^2}.
\end{equation}
Consequently,
\begin{equation}\label{eq:backward-terminal-path-calibration}
 \|Z\|_{\mathbb H_{\cP}^2}
 =\|I_T(Z)\|_{\mathcal L_{\cP}^2}
 \le\|I(Z)\|_{\mathbb S_{\cP}^2}
 \le2\|Z\|_{\mathbb H_{\cP}^2}.
\end{equation}
In particular, on the stable first-jet graph,
\begin{equation}\label{eq:backward-exact-terminal-response-metric}
 d_{\mathscr W}((Y,Z),(\widetilde Y,\widetilde Z))
 =\|Y-\widetilde Y\|_{\mathbb S_{\cP}^2}
  +\|I_T(Z-\widetilde Z)\|_{\mathcal L_{\cP}^2}.
\end{equation}
Thus terminal and stopped martingale responses calibrate the jet geometry
isometrically; the whole-path martingale response is their Doob propagation.
\end{proposition}

\begin{proof}
Let $\Delta_n=T2^{-n}$ and let
\[
 \tau_n:=\min\{k\Delta_n:0\le k\le2^n,\ k\Delta_n\ge\tau\}
\]
be the dyadic ceiling of $\tau$.  Then $\tau_n\downarrow\tau$, and
$\one_{[0,\tau_n]}$ is elementary predictable: on
$(k\Delta_n,(k+1)\Delta_n]$ it is the
$\mathcal F_{k\Delta_n}^0$-measurable coefficient
$\one_{\{\tau>k\Delta_n\}}$.  Its multiplication operator is therefore
contractive on $\mathbb H_{\cP}^2$.  Given $Z$ there, choose bounded elementary
$H$ with
$\|Z-H\|_{\mathbb H_{\cP}^2}<\varepsilon$.  If
\[
 \sup_{(t,\omega)}\|H_t(\omega)\|_H\le M,
\]
the uniform trace clock $\Tr a_t^P\le\Lambda$ gives
\[
 \|\one_{(\tau,\tau_n]}H\|_{\mathbb H_{\cP}^2}^2
 \le M^2\Lambda\,\Delta_n.
\]
By contractivity,
\[
 \|\one_{[0,\tau_n]}Z-\one_{[0,\tau_m]}Z\|_{\mathbb H_{\cP}^2}
 \le2\varepsilon
   +\|\one_{[0,\tau_n]}H-\one_{[0,\tau_m]}H\|_{\mathbb H_{\cP}^2},
\]
so the Cauchy limit defines $\mathsf T_\tau Z$.  This uses only the trace
clock; the stronger operator envelope serves fiber compactness.  Also
$Y\mapsto r_\tau Y$ is contractive on $\mathbb S_{\cP}^2$.

For bounded elementary $H$ and bounded raw stopping times $\sigma,\tau$,
\begin{equation}\label{eq:elementary-stopping-horizon-bound}
 \|\mathsf T_\sigma H-\mathsf T_\tau H\|_{\mathbb H_{\cP}^2}^2
 \le M^2\Lambda\,\|\sigma-\tau\|_\infty.
\end{equation}
Indeed, their difference is supported on the interval between the horizons.
Density and contractivity give the first term in
\eqref{eq:random-stopping-horizon-continuity}.  The second follows from
\eqref{eq:common-integral-random-stopping} and the upper $L^2$ BDG/Doob bound
for the common integral.

Under every $P\in\cP$, the modelwise projection of $\mathsf T_\tau Z$ is the
classical stopped integrand $\one_{[0,\tau]}j_P(Z)$.  Classical stochastic
integration gives
\[
 \pi_P(I(\mathsf T_\tau Z))
 =\int_0^{\cdot\wedge\tau}j_P(Z)_r\,\dd X_r
 =\pi_P(r_\tau I(Z)).
\]
The product embedding proves \eqref{eq:common-integral-random-stopping} and
uniqueness of the common class with these projections, hence independence of
the dyadic approximation.  Applying the characterization twice gives
\[
 j_P(\mathsf T_\sigma\mathsf T_\tau Z)
 =\one_{[0,\sigma]}\one_{[0,\tau]}j_P(Z)
 =\one_{[0,\sigma\wedge\tau]}j_P(Z),
\]
and proves \eqref{eq:random-stopping-semigroup}.  Finally, It\^o isometry yields
\[
 E^P|I_\tau(Z)|^2
 =E^P\int_0^\tau
   \|(a_r^P)^{1/2}j_P(Z)_r\|_H^2\,\dd r
 =E^P\int_0^\tau
   \langle a_r^Pj_P(Z)_r,j_P(Z)_r\rangle_H\,\dd r.
\]
Taking $\sup_P$ proves \eqref{eq:backward-stopping-Ito-isometry}; $\tau=t$ is
the deterministic case.  Evaluation at $T$, Doob's inequality, and the same
identity for $Z-\widetilde Z$ give respectively
\eqref{eq:backward-terminal-path-calibration} and
\eqref{eq:backward-exact-terminal-response-metric}.
\end{proof}

\begin{proposition}[Response-forced geometry of the backward first jet]
\label{thm:backward-response-forced-geometry}
The map
\begin{equation}\label{eq:backward-VI-encoder}
 \mathcal J_{\rm VI}:\mathscr W_{Q,\cP}^{1,2}
 \longrightarrow\mathscr R_{\rm VI}^{\rm bw},
 \qquad
 \mathcal J_{\rm VI}(Y,Z):=(Y,I(Z)),
\end{equation}
is a stable-causal bi-Lipschitz homeomorphism.  More precisely, for every two
first-jet states,
\begin{align}
 &\|Y-\widetilde Y\|_{\mathbb S_{\cP}^2}
   +\|Z-\widetilde Z\|_{\mathbb H_{\cP}^2}
 \notag\\
 &\qquad\le
 d_{\rm VI}\bigl(\mathcal J_{\rm VI}(Y,Z),
                   \mathcal J_{\rm VI}(\widetilde Y,\widetilde Z)\bigr)
 \notag\\
 &\qquad\le
 \|Y-\widetilde Y\|_{\mathbb S_{\cP}^2}
   +2\|Z-\widetilde Z\|_{\mathbb H_{\cP}^2}.
 \label{eq:backward-VI-BDG-equivalence}
\end{align}
Its inverse is
\begin{equation}\label{eq:backward-VI-decoder}
 \mathcal D_{\rm VI}(Y,M):=(Y,I^{-1}M),
\end{equation}
where $I^{-1}$ is taken on the closed range of the common stochastic integral.
Consequently
\[
 \boxed{
 \text{value response}+\text{martingale response}
 \quad\Longrightarrow\quad
 \mathbb S_{\cP}^2\times\mathbb H_{\cP}^2
 \text{ first-jet geometry}.}
\]
In particular, the $\mathbb H_{\cP}^2$ topology is calibrated exactly by
the terminal martingale response and reconstructed from the full path response
by Doob propagation; it is not an independent ambient choice.
\end{proposition}

\begin{proof}
On the testing graph, \eqref{eq:backward-VI-BDG-equivalence} follows from
\Cref{prop:backward-exact-martingale-calibration}; hence graph and response
Cauchy sequences coincide and the encoder extends to the completions.  By
\Cref{thm:uniform-BDG-extension}, $I$ has closed range, bounded inverse there,
and commutes with stopping.  Thus \eqref{eq:backward-VI-decoder} is the inverse
on the entire response completion and both maps commute with truncations.
\end{proof}

\begin{corollary}[No proper lossless quotient of the backward first-jet state]
\label{cor:no-proper-backward-first-jet-quotient}
Let
$\pi:\mathscr W_{Q,\cP}^{1,2}\to\mathsf S$ be a stable-causal map into a
causal Polish space.  Suppose the complete value--martingale response factors
through $\pi$: there is a stable-causal
$\Psi:\mathsf S\to\mathscr R_{\rm VI}^{\rm bw}$ such that
\[
 \mathcal J_{\rm VI}=\Psi\circ\pi.
\]
Then $\pi$ has the stable-causal left inverse
$\mathcal D_{\rm VI}\circ\Psi$.  Hence $\pi$ is injective and a homeomorphism
onto its image.  No noninjective continuous causal quotient is lossless for
the complete backward first-jet response.  A value-only compression is
therefore possible exactly when the graph can be reconstructed continuously
from its value projection, as characterized in
\Cref{thm:backward-compression-retention}.
\end{corollary}

\begin{proof}
Apply \Cref{prop:lossless-chart-factorization}\textup{(iii)} to the lossless
chart $(\mathcal J_{\rm VI},\mathcal D_{\rm VI})$ constructed in
\Cref{thm:backward-response-forced-geometry}.  Its canonical left inverse is
$\mathcal D_{\rm VI}\Psi$.
\end{proof}

\begin{proposition}[Closed response relation, unbounded recovery, and genuine compression]
\label{thm:backward-compression-retention}
Regard
\[
 \mathscr W_{Q,\cP}^{1,2}
 \subset\mathbb S_{\cP}^2\times\mathbb H_{\cP}^2
\]
as the closed linear relation obtained from the raw testing relation.  Its
multivalued part is
\[
 \operatorname{mul}\mathscr W_{Q,\cP}^{1,2}
 :=\{Z:(0,Z)\in\mathscr W_{Q,\cP}^{1,2}\}
 =\mathcal N_Q^{\rm vert}.
\]

\emph{Closability and single-valued closed recovery.}  The following are
equivalent:
\begin{enumerate}[label=\textup{(\roman*)},leftmargin=2.5em]
\item $\mathcal N_Q^{\rm vert}=\{0\}$;
\item the closure of the raw covariance-differential relation is the graph of
a single-valued operator;
\item the raw covariance-differential relation is closable as an operator from
$\mathbb S_{\cP}^2$ to $\mathbb H_{\cP}^2$;
\item there is a unique closed linear operator
\[
 D_Q^{\rm cl}:\operatorname{Dom}(D_Q^{\rm cl})
 \subset\mathbb S_{\cP}^2\longrightarrow\mathbb H_{\cP}^2
\]
such that
\begin{equation}\label{eq:backward-closed-operator-graph}
 \mathscr W_{Q,\cP}^{1,2}=\operatorname{Graph}(D_Q^{\rm cl}),
 \qquad
 \operatorname{Dom}(D_Q^{\rm cl})
 =\pi_Y(\mathscr W_{Q,\cP}^{1,2}).
\end{equation}
\end{enumerate}
When these conditions hold, write $D_Q:=D_Q^{\rm cl}$.

\emph{Continuous compression in the inherited value topology.}  Under these
equivalent conditions, the following are equivalent:
\begin{enumerate}[label=\textup{(\alph*)},leftmargin=2.5em]
\item the inverse graph map
\[
 J_Q:\operatorname{Dom}(D_Q)\longrightarrow\mathscr W_{Q,\cP}^{1,2},
 \qquad J_QY=(Y,D_QY),
\]
is continuous when the domain carries the topology inherited from
$\mathbb S_{\cP}^2$;
\item $D_Q:\operatorname{Dom}(D_Q)\to\mathbb H_{\cP}^2$ is continuous for that
value topology;
\item whenever $Y^n,Y\in\operatorname{Dom}(D_Q)$ and
$Y^n\to Y$ in $\mathbb S_{\cP}^2$, one has
$D_QY^n\to D_QY$ in $\mathbb H_{\cP}^2$;
\item the value domain $\operatorname{Dom}(D_Q)$ is closed in
$\mathbb S_{\cP}^2$.
\end{enumerate}
Equivalently,
\begin{equation}\label{eq:backward-closed-domain-criterion}
 \boxed{
 \begin{aligned}
 &\text{value-process compression in the inherited }
       \mathbb S_{\cP}^2\text{ topology}\\
 &\quad\Longleftrightarrow\quad
 \mathcal N_Q^{\rm vert}=\{0\}
 \ \text{and}\ 
 \operatorname{Dom}(D_Q)\text{ is closed in }\mathbb S_{\cP}^2.
 \end{aligned}}
\end{equation}
Under these conditions $J_Q$ is a stable-causal homeomorphism from the
value domain onto the graph.  If $\mathcal N_Q^{\rm vert}=\{0\}$ but the value
domain is not closed, then $D_Q$ is a closed unbounded response operator for
the inherited value norm.  Its graph norm is
\begin{equation}\label{eq:backward-response-graph-norm}
 \|Y\|_{\operatorname{gr}D_Q}
 :=\|Y\|_{\mathbb S_{\cP}^2}+\|D_QY\|_{\mathbb H_{\cP}^2}.
\end{equation}
Relabelling the graph by $Y$ equipped with this stronger norm retains the full
jet geometry and is not a genuine value-process compression.

If instead $0\ne Z_*\in\mathcal N_Q^{\rm vert}$, no value-only realization can
reproduce the full backward response program: along a testing-graph sequence
\[
 Y^n\to0\quad\text{in }\mathbb S_{\cP}^2,
 \qquad
 D_Q^{\rm pw}Y^n\to Z_*\quad\text{in }\mathbb H_{\cP}^2,
\]
one has
\[
 I(D_Q^{\rm pw}Y^n)\longrightarrow I(Z_*)\ne0.
\]
\end{proposition}

\begin{proof}
A closed linear relation is an operator graph exactly when its multivalued
part vanishes; because $\mathscr W_{Q,\cP}^{1,2}$ closes the raw relation,
this is exactly its closability criterion.  In that case
\eqref{eq:backward-closed-operator-graph} defines the unique closed operator
$D_Q$.

Continuity of $J_Q$ is continuity of $D_Q$, equivalently its sequential
form.  A continuous $D_Q$ sends every domain sequence converging in
$\mathbb S_{\cP}^2$ to a Cauchy sequence; graph closedness then places the
limit in the domain, so that domain is closed.  Conversely, a closed value
domain is Banach and the closed graph theorem makes $D_Q$ bounded.  This
proves \eqref{eq:backward-closed-domain-criterion}; otherwise $D_Q$ is
unbounded and the transported graph metric is exactly
\eqref{eq:backward-response-graph-norm}.  Stopping stability and uniqueness give
\[
 D_Q(r_tY)=\one_{[0,t]}D_QY,
\]
so $J_Q$ commutes with the truncations and is stable-causal.

If $Z_*\ne0$ is vertical, exact terminal calibration gives
\[
 \|I_T(Z_*)\|_{\mathcal L_{\cP}^2}
 =\|Z_*\|_{\mathbb H_{\cP}^2}>0.
\]
Thus the integral response separates graph limits with the same value
coordinate, excluding value-only realization.
\end{proof}

\begin{proposition}[Value-process compression on a nonlinear solver range]
\label{prop:nonlinear-solver-range-compression}
Let
\[
 \mathscr S_{\rm sol}\subset\mathscr W_{Q,\cP}^{1,2}
\]
be a closed, not necessarily linear, solver range invariant under the
first-jet truncations, and give
$\mathscr Y_{\rm sol}:=\pi_Y(\mathscr S_{\rm sol})$ the topology inherited
from $\mathbb S_{\cP}^2$.  The following are equivalent:
\begin{enumerate}[label=\textup{(\roman*)},leftmargin=2.5em]
\item the restricted value projection
\[
 \pi_Y|_{\mathscr S_{\rm sol}}:
 \mathscr S_{\rm sol}\longrightarrow\mathscr Y_{\rm sol}
\]
is injective and has a stable-causal inverse;
\item there is a unique stable-causal map
\[
 \mathcal Z_{\rm sol}:\mathscr Y_{\rm sol}
 \longrightarrow\mathbb H_{\cP}^2
\]
such that
\begin{equation}\label{eq:nonlinear-solver-graph}
 \mathscr S_{\rm sol}
 =\{(Y,\mathcal Z_{\rm sol}(Y)):Y\in\mathscr Y_{\rm sol}\};
\end{equation}
\item the complete response restricted to $\mathscr S_{\rm sol}$ admits a
lossless realization on the value-process state with its inherited topology.
\end{enumerate}
If $\mathscr S_{\rm sol}$ is in addition a closed linear relation, these
conditions reduce to vanishing of its vertical space together with closedness
of its value domain.  Thus
\Cref{thm:backward-compression-retention} is the linear special case, not
an automatic criterion for a nonlinear solver image.  In particular, a fixed
nonlinear dynamics may admit continuous recovery on its own solution range
even when the ambient covariance differential is unbounded.
\end{proposition}

\begin{proof}
The first two statements are the graph characterization of continuous
single-valued recovery.  Put
$J_{\rm sol}(Y):=(Y,\mathcal Z_{\rm sol}(Y))$.  Then $J_{\rm sol}$ is the
canonical value decoder and the complete response factors as
\[
 \mathcal J_{\rm VI}|_{\mathscr S_{\rm sol}}
 =\bigl(\mathcal J_{\rm VI}\circ J_{\rm sol}\bigr)
    \circ\pi_Y|_{\mathscr S_{\rm sol}}.
\]
Truncation invariance and uniqueness give
\[
 \mathcal Z_{\rm sol}(r_tY)
 =\one_{[0,t]}\mathcal Z_{\rm sol}(Y),
\]
so the decoder is stable-causal.  Composition with
\Cref{thm:backward-response-forced-geometry} proves the forward implication in
the third equivalence.  Conversely, a lossless value-state realization,
followed by the integrand projection of its continuous inverse readout,
recovers $\mathcal Z_{\rm sol}$ and proves the reverse implication.
\Cref{thm:backward-compression-retention} gives the final assertion only
for a closed linear solver relation.
\end{proof}

\begin{theorem}[Raw covariance-response first-jet selection]
\label{thm:backward-first-jet-selection}
Assume an admissible testing core in the sense of
\Cref{def:admissible-backward-testing-core}.  Then the backward
value--martingale response program selects the stable first-jet state
$\mathscr W_{Q,\cP}^{1,2}$ with the following properties.
\begin{enumerate}[label=\textup{(\roman*)},leftmargin=2.5em]
\item Every energy-bounded covariance cross-response generates a unique
fiber-valued local jet $D_Q^{\rm pw}Y_t\in\mathscr E_{\mathfrak a_t}$ by Riesz
representation, with the regularized covariance-pseudoinverse formula
\eqref{eq:backward-pseudoinverse-factorization}.  The common-completion
condition \textup{(A2)} makes its active section a class in
$\mathbb H_{\cP}^2$, and under every law that class is the covariance-energy
class of the classical GKW coefficient.
\item The terminal martingale response calibrates the
$\mathbb H_{\cP}^2$ jet norm isometrically at every stopping horizon, while
the complete path response $(Y,I(Z))$ reconstructs the graph completion up to
stable-causal bi-Lipschitz equivalence.
\item No proper noninjective stable-causal quotient is lossless for
the complete first-jet response.
\item The completed first-jet state is a closed linear relation.  It is the
graph of a single-valued closed response operator exactly when its vertical
space vanishes.  Genuine compression to the inherited value-process topology
then occurs exactly when the value domain is closed; otherwise the recovery
operator is unbounded and the graph topology retains the jet.
\end{enumerate}
The derivative leg uses the covariance channel rather than skew area, while
the zero-order environmental state may still retain area.  Orthogonal
martingale components remain visible in the value trajectory but outside the
covariance jet.  Solver identification is governed separately by
\Cref{def:backward-compatible-interface}.
\end{theorem}

\begin{proof}[Proof of \Cref{thm:backward-first-jet-selection}]
Part~\textup{(i)} is
\Cref{prop:covariance-first-jet-riesz-forcing}, with modelwise identification
provided by \Cref{cor:modelwise-GKW-covariance-jet}; membership in the common
completion is precisely \textup{(A2)}.  The exact terminal calibration, its
pathwise propagation, and the global response decoder are
\Cref{prop:backward-exact-martingale-calibration,thm:backward-response-forced-geometry}.  The no-quotient statement is
\Cref{cor:no-proper-backward-first-jet-quotient}, and the exact distinction
between closed recovery and continuous value-only compression is
\Cref{thm:backward-compression-retention}.  Program-relative omission of
area from the derivative leg follows from
\Cref{cor:program-dependent-area-visibility}; the
orthogonal-martingale boundary is
\Cref{prop:orthogonal-martingale-derivative-boundary}.
\end{proof}

\subsubsection{Continuation, drivers, and solver-dependent consequences}

A backward-compatible continuation interface identifies the selected jet with
a solver's martingale coordinate and propagates terminal data.

\begin{definition}[Backward-compatible continuation interface]
\label{def:backward-compatible-interface}
A \emph{backward-compatible continuation interface} on the selected state
consists of the law family $\cP$ together with causal conditional response
operators $\mathcal E_{s,t}$ on a declared bounded operational class,
$0\le s\le t\le T$, such that:
\begin{enumerate}[label=\textup{(B\arabic*)},leftmargin=2.8em]
\item $\mathcal E_{s,t}$ is monotone, constant preserving and local at the
information available at time $s$;
\item the exact tower
\[
 \mathcal E_{r,s}\mathcal E_{s,t}=\mathcal E_{r,t},\qquad r\le s\le t,
\]
holds on the declared operational class and is compatible with deterministic
restart of the selected state;
\item under the usual right-continuous $P$-augmentation, every
$R$ in the declared response-supermartingale class is a square-integrable
continuous special semimartingale with normalized Doob--Meyer decomposition
\[
 R_t=R_0+M_t^{R,P}-K_t^{R,P},
 \qquad M_0^{R,P}=K_0^{R,P}=0.
\]
$K^{R,P}$ is predictable and nondecreasing.
\item for every declared $R$, the interface supplies a common upper-energy
class $Z^R\in\mathbb H_{\cP}^2$ such that
\[
 \pi_P\bigl(I(Z^R)\bigr)=M^{R,P}
 \qquad\text{in }\mathbb S^2(P),\qquad P\in\cP.
\]
Consequently the modelwise Galtchouk--Kunita--Watanabe component orthogonal to
the closed $X$-integral range vanishes, and $j_P(Z^R)$ is the modelwise GKW
energy class.  Whenever $R$ is supplied together with a declared first-jet
coordinate $Z$, compatibility of that declaration means $Z^R=Z$ in
$\mathbb H_{\cP}^2$.
\end{enumerate}
Clauses~\textup{(B1)}--\textup{(B2)} specify deterministic-time continuation
and restart compatibility, clause~\textup{(B3)} gives lawwise semimartingale
semantics, and clause~\textup{(B4)} supplies the common martingale channel.
\end{definition}

\begin{definition}[Admissible covariance-fiber driver]
\label{def:admissible-backward-driver}
Let $\mathcal P^0$ be the raw predictable sigma-field on
$[0,T]\times\Omega_H$, and set
\[
 \mathfrak B_{\mathfrak a}
 :=\left\{(t,\omega,y,u)\in[0,T]\times\Omega_H\times\mathbb R\times H:
       (\mathfrak a_t(\omega),u)\in\mathscr E_\Gamma^\sharp\right\}.
\]
This is measurable for
$\mathcal P^0\otimes\mathcal B(\mathbb R)\otimes\mathcal B(H)$.
An \emph{admissible backward driver} is a causal local law $F$ on the
covariance-energy fibers whose active-coordinate representation
\begin{equation}\label{eq:backward-driver-active-representation}
 \widehat F(t,\omega,y,u)
 :=F\bigl(t,\omega,y,
          \iota_{\mathfrak a_t(\omega)}^{-1}u\bigr),
 \qquad
 (\mathfrak a_t(\omega),u)\in\mathscr E_\Gamma^\sharp,
\end{equation}
is measurable for the restriction of
$\mathcal P^0\otimes\mathcal B(\mathbb R)\otimes\mathcal B(H)$ to
$\mathfrak B_{\mathfrak a}$.  For some $L<\infty$ it satisfies
\begin{align}
 |\widehat F(t,\omega,y,u)-\widehat F(t,\omega,\widetilde y,\widetilde u)|
 &\le L\bigl(|y-\widetilde y|+\|u-\widetilde u\|_H\bigr),
 \label{eq:backward-driver-active-Lipschitz}\\
 \sup_{P\in\cP}E^P\left(\int_0^T
 |\widehat F(t,\cdot,0,0)|\,\dd t\right)^2&<\infty.
 \label{eq:backward-driver-base-integrability}
\end{align}
Equivalently, the first inequality is the fiberwise Lipschitz estimate with
$\|z-\widetilde z\|_{\mathscr E_{\mathfrak a_t(\omega)}}$ in place of
$\|u-\widetilde u\|_H$.  On the admissible testing graph define the raw
running-source response
\begin{equation}\label{eq:backward-raw-running-channel}
 B_t^{F,0}[Y]
 :=\int_0^t\widehat F(r,\cdot,Y_r,U_r^Y)\,\dd r
 =\int_0^tF(r,\cdot,Y_r,D_Q^{\rm pw}Y_r)\,\dd r.
\end{equation}
The second equality is bundle notation: its measurable meaning is the first
active-coordinate formula.
\end{definition}

\begin{proposition}[Exact backward descent through the value and energy quotients]
\label{thm:backward-first-jet-descent}
Let $(Y,Z)$ and $(\widetilde Y,\widetilde Z)$ be two representatives of the
observable testing graph with
\[
 [Y]_{\mathbb S^2}=[\widetilde Y]_{\mathbb S^2},\qquad
 [Z]_{\mathbb H^2}=[\widetilde Z]_{\mathbb H^2}.
\]
Then:
\begin{enumerate}[label=\textup{(\roman*)},leftmargin=2.6em]
\item $I(Z)=I(\widetilde Z)$ in $\mathbb S_{\cP}^2$;
\item every admissible driver in the sense of
\Cref{def:admissible-backward-driver} assigns the same
$\mathbb S_{\cP}^2$ class to $B^{F,0}[Y]$ and
$B^{F,0}[\widetilde Y]$;
\item every continuous downstream constructor built from the value class, the
descended source and the common integral has the same output class.
\end{enumerate}
Conversely, the joint value/integral response is the restriction of the
completed encoder:
\begin{equation}\label{eq:backward-value-integral-response}
 \mathcal J_{\rm VI}\bigl(Y,D_Q^{\rm pw}Y\bigr)
 :=\bigl([Y]_{\mathbb S^2},I(D_Q^{\rm pw}Y)\bigr)
\end{equation}
separates the stable first-jet classes on the testing graph.  The local
covariance fiber, common upper-energy quotient, and process-version quotient
have already removed precisely the invisible directions.
\end{proposition}

\begin{proof}
The upper BDG bound proves \textup{(i)}; the active-coordinate Lipschitz bound,
Cauchy--Schwarz, and \eqref{eq:backward-driver-base-integrability} prove
\textup{(ii)}; and descent followed by continuous composition proves
\textup{(iii)}.  Conversely, equal value/integral responses and terminal
calibration give
\[
 \|D_Q^{\rm pw}Y-D_Q^{\rm pw}\widetilde Y\|_{\mathbb H_{\cP}^2}
 =\|I_T(D_Q^{\rm pw}Y-D_Q^{\rm pw}\widetilde Y)\|_{\mathcal L_{\cP}^2}
 =0,
\]
which proves equality of the observable jet classes.
\end{proof}

\begin{remark}[Closed recovery is weaker than stable propagation]
The first part does not imply the second.  On $\ell^2$, the diagonal operator
\[
 D(x_1,x_2,\ldots)=(x_1,2x_2,3x_3,\ldots)
\]
with its natural domain has a closed graph and zero vertical space, but its
domain is not closed in $\ell^2$; indeed $n^{-1}e_n\to0$ while
$D(n^{-1}e_n)=e_n\not\to0$.  The response-selected jet state therefore uses
the graph topology unless the closed-domain criterion has been verified.
\end{remark}

\begin{remark}[Backward coordinate hierarchy]
\label{rem:backward-coordinate-hierarchy}
The conventional BSDE coordinates have different response-selected roles:
\[
 \begin{aligned}
 Y&:\ \text{retained value/restart response},\\
 Z&:\ \text{Riesz-generated covariance-energy class, retained unless its}
       \\[-1mm]
   &\hspace{9mm}\text{recovery is continuous in the evolutionary value topology},\\
 K&:\ \text{downstream finite-variation exactness readout}.
 \end{aligned}
\]
The cross-response generates $D_Q^{\rm pw}Y$, while the closed relation decides
single-valuedness and value-topology compression.  An orthogonal martingale
remains visible in $Y$ but not in this covariance derivative, so the solver
interface imposes common-integral representability.  There $Z$ is the graph
coordinate by \Cref{cor:common-integrand-equals-covariance-jet}; recovery from
$Y$ still requires a compression criterion.  A program observing an
orthogonal martingale component requires a separate retained coordinate.
\end{remark}

\begin{example}[Boundary generation in the classical Brownian $L^2$ theory]
\label{ex:classical-terminal-jet-propagation}
Let $W$ generate the filtration, take the zero driver, and let
\[
 Y_t^\xi=E[\xi\mid\mathcal F_t],
 \qquad
 Y_t^\xi=E\xi+\int_0^t Z_s^\xi\,\dd W_s.
\]
For $\xi,\eta\in L^2$ the It\^o isometry gives
\begin{equation}\label{eq:classical-terminal-jet-L2-bound}
 \|Z^\xi-Z^\eta\|_{\mathbb H^2}^2
 =E\left| (\xi-\eta)-E(\xi-\eta)\right|^2
 \le E|\xi-\eta|^2.
\end{equation}
Thus the $L^2$ terminal topology propagates the jet continuously.  Moreover,
\[
 \|\xi-\eta\|_{L^2}
 =\|Y_T^\xi-Y_T^\eta\|_{L^2}
 \le\|Y^\xi-Y^\eta\|_{\mathbb S^2},
\]
so it is also continuously recoverable from the value trajectory on this
solver range.  This is the compressible branch, while $(Y,Z)$ remains a
lossless first-jet chart.
\end{example}

\subsubsection{Terminal generation and evolutionary jet retention}

Fix one backward-compatible continuation interface and one admissible driver
$F$.  Let $\mathscr T_0$ be an operational terminal core with a chosen value
metric $d_{\mathscr T}$.  For each $\xi\in\mathscr T_0$, suppose the declared
operational solver selects a backward first-jet trajectory
\begin{equation}\label{eq:backward-reference-continuation}
 \mathcal G_T^0(\xi):=(Y^\xi,Z^\xi)
 \in\mathscr W_{Q,\cP}^{1,2},
 \qquad Y_T^\xi=\xi.
\end{equation}
The fixed dependence on the continuation interface and $F$ is suppressed in
the notation.  For this subsection a solver trajectory is supplied for each
$\xi$; its construction and uniqueness are separate inputs.  The notation
$Z^\xi=D_QY^\xi$ is used once the relevant value projection is known to be
injective.

\begin{corollary}[Terminal generation versus terminal enhancement]
\label{thm:terminal-propagation}
The selected terminal map extends uniquely to a continuous map
\[
 \mathcal G_T:\overline{\mathscr T_0}
 \longrightarrow\mathscr W_{Q,\cP}^{1,2}
\]
if and only if $\mathcal G_T^0$ is Cauchy-continuous.  Equivalently, for every
$d_{\mathscr T}$-Cauchy sequence $(\xi_n)$, the first-jet sequence
\[
 \bigl((Y^{\xi_n},Z^{\xi_n})\bigr)_{n\ge1}
\]
is Cauchy in $d_{\mathscr W}$.  In particular, it is sufficient that a
nondecreasing
modulus $\omega$ with $\omega(0)=0$ and
$\lim_{r\downarrow0}\omega(r)=0$ satisfy
\begin{equation}\label{eq:terminal-to-jet}
 d_{\mathscr W}\bigl((Y^\xi,Z^\xi),(Y^\eta,Z^\eta)\bigr)
 \le \omega\bigl(d_{\mathscr T}(\xi,\eta)\bigr),
 \qquad \xi,\eta\in\mathscr T_0.
\end{equation}
If Cauchy continuity fails, the first-jet completion cannot be the graph of a single-valued
continuous map over the zero-order terminal completion.  This happens, in
particular, if one terminal Cauchy sequence has non-Cauchy continuation jets,
or if two equivalent terminal Cauchy sequences yield different stable
first-jet limits.  Relative to the declared terminal and first-jet response
program, any completion retaining all such response limits must then enhance
the terminal state by a first-order lift.
\end{corollary}

\begin{proof}
Necessity is immediate.  Conversely, map a completion point to the limit of
the images of a representing Cauchy sequence; interleaving equivalent
representatives proves well-definedness and continuity.  The modulus implies
Cauchy continuity, while the stated alternatives are precisely the
obstructions to a single-valued graph over the terminal completion.
\end{proof}

\begin{proposition}[Boundary generation versus evolutionary value compression]
\label{prop:boundary-generation-evolutionary-retention}
Put
\[
 \mathscr S_T:=\overline{\mathcal G_T^0(\mathscr T_0)}
 ^{\,\mathscr W_{Q,\cP}^{1,2}}
\]
and assume that this closed terminal solver range is invariant under the
first-jet truncations.  Terminal generation and evolutionary value compression
are different requirements:
\begin{enumerate}[label=\textup{(\roman*)},leftmargin=2.5em]
\item terminal generation asks whether
$\mathcal G_T^0:\mathscr T_0\to\mathscr S_T$ is Cauchy-continuous for the
terminal-data topology;
\item evolutionary compression asks whether
$\pi_Y|_{\mathscr S_T}$ is injective and has a stable-causal inverse
for the inherited $\mathbb S_{\cP}^2$ value topology.
\end{enumerate}
These conditions concern different maps, and neither implication is assumed
without additional comparison hypotheses.  The following comparison is
sufficient for boundary generation to
imply evolutionary compression.  Suppose
\eqref{eq:terminal-to-jet} holds and there is a nondecreasing modulus
$\chi$ with $\chi(0)=0$ and $\lim_{r\downarrow0}\chi(r)=0$ such that
\begin{equation}\label{eq:terminal-value-metric-comparison}
 d_{\mathscr T}(\xi,\eta)
 \le\chi\bigl(\|Y^\xi-Y^\eta\|_{\mathbb S_{\cP}^2}\bigr),
 \qquad \xi,\eta\in\mathscr T_0.
\end{equation}
Then
\begin{equation}\label{eq:terminal-generation-implies-value-recovery}
 \|Z^\xi-Z^\eta\|_{\mathbb H_{\cP}^2}
 \le(\omega\circ\chi)
       \bigl(\|Y^\xi-Y^\eta\|_{\mathbb S_{\cP}^2}\bigr),
\end{equation}
and the induced closed terminal solver range is value-process compressible by
\Cref{prop:nonlinear-solver-range-compression}.
\end{proposition}

\begin{proof}
Under the two modulus estimates,
\[
 \|Z^\xi-Z^\eta\|_{\mathbb H_{\cP}^2}
 \le d_{\mathscr W}\bigl(\mathcal G_T^0(\xi),
                          \mathcal G_T^0(\eta)\bigr)
 \le(\omega\circ\chi)
       \bigl(\|Y^\xi-Y^\eta\|_{\mathbb S_{\cP}^2}\bigr).
\]
Passing to the closure makes the value projection injective with continuous
inverse; truncation stability and uniqueness give stable causality.
\end{proof}

\begin{proposition}[Running-source stability and the intrinsic defect]
\label{prop:backward-intrinsic-defect}
Fix an admissible driver $F$ in the sense of
\Cref{def:admissible-backward-driver}.  The raw source map
$B^{F,0}$ on the testing graph has a unique continuous extension
\begin{equation}\label{eq:backward-running-channel}
 B^F:\mathscr W_{Q,\cP}^{1,2}\longrightarrow\mathbb S_{\cP}^2,
 \qquad B^F|_{\mathfrak G_{Q,\cP}^{\rm bw,0}}=B^{F,0}.
\end{equation}
For a stable lift $(Y,Z)$ define its intrinsic defect by
\begin{equation}\label{eq:backward-defect}
 K_t^F[Y,Z]:=Y_0-B_t^F[Y,Z]+I_t(Z)-Y_t.
\end{equation}
Then
\begin{align}
 &\|B^F[Y,Z]-B^F[\widetilde Y,\widetilde Z]\|_{\mathbb S_{\cP}^2}
 \notag\\
 &\qquad\le
 LT\|Y-\widetilde Y\|_{\mathbb S_{\cP}^2}
 +L\sqrt T\|Z-\widetilde Z\|_{\mathbb H_{\cP}^2},
 \label{eq:backward-running-source-stability}\\
 &\|K^F[Y,Z]-K^F[\widetilde Y,\widetilde Z]\|_{\mathbb S_{\cP}^2}
 \notag\\
 &\qquad\le
 (2+LT)\|Y-\widetilde Y\|_{\mathbb S_{\cP}^2}
 +(2+L\sqrt T)\|Z-\widetilde Z\|_{\mathbb H_{\cP}^2}.
 \label{eq:backward-defect-stability}
\end{align}
In particular $B^F$ is the response-completed running-source channel and
$B^F,K^F$ are continuous first-jet readouts.  For a completed lift, the occasional notation
$\int_0^tF(r,\cdot,Y_r,Z_r)\,\dd r$ means this continuous extension; it does
not assert a distinguished pointwise representative of the energy class $Z$.

The cone
\[
 \mathbb K_{\cP}^{2,\uparrow}
 :=\{K\in\mathbb S_{\cP}^2:K_0=0
       \text{ and }K\text{ has a quasi-sure continuous nondecreasing version}\}
\]
is closed in $\mathbb S_{\cP}^2$.  Suppose
$R:=Y+B^F[Y,Z]$ belongs to the declared response-supermartingale class under
every $P\in\cP$ and the common upper-energy class in \textup{(B4)} of
\Cref{def:backward-compatible-interface} is the specified first-jet
coordinate, $Z^R=Z$ in $\mathbb H_{\cP}^2$.  Then
$K^F[Y,Z]\in\mathbb K_{\cP}^{2,\uparrow}$.  Moreover, for every $P\in\cP$
and every deterministic $s\le t$,
\begin{equation}\label{eq:defect-gap}
 Y_s-E^P\!\left[Y_t+B_t^F-B_s^F\mid\mathcal F_s^0\right]
 =E^P[K_t^F-K_s^F\mid\mathcal F_s^0],
\end{equation}
so conditional exactness under $P$ is equivalent to $P$-flatness of the defect
on $[s,t]$.
\end{proposition}

\vspace{\smallskipamount}
\begin{proof}
Integrating \eqref{eq:backward-driver-active-Lipschitz} and applying
Cauchy--Schwarz gives the first estimate on the dense testing graph.  Descent
and completion give its unique extension, with $B_0^F=0$; the defect formula
and upper Doob/BDG bound give the second estimate.  The raw identity
\[
 r_tB^{F,0}[Y]
 =r_tB^{F,0}[r_tY]
\]
uses the truncated jet on the right and passes by continuity to
\[
 r_tB^F[Y,Z]
 =r_tB^F[r_tY,\one_{[0,t]}Z].
\]
Thus $B^F$, and algebraically $K^F$, are stable-causal.

To prove closedness of the increasing cone, let $K^n\to K$ in
$\mathbb S_{\cP}^2$ with every $K^n$ quasi-surely nondecreasing.  For rational
$s<t$ and $\varepsilon>0$,
\[
 c_{\cP}(K_s>K_t+\varepsilon)
 \le c_{\cP}\left(2\sup_{r\le T}|K_r-K_r^n|>\varepsilon\right)
 \le4\varepsilon^{-2}\|K-K^n\|_{\mathbb S_{\cP}^2}^2
 \longrightarrow0.
\]
A countable union over rational $s,t,\varepsilon$, followed by continuity,
gives one polar exceptional set; continuity of evaluation preserves $K_0=0$.

Under each $P$, the common integral is classical and $Z^R=Z$ makes $I(Z)$ the
Doob--Meyer martingale part of $R$; uniqueness identifies its increasing part
with $K^F$.  Monotonicity is a Borel path property, so this common process lies
in $\mathbb K_{\cP}^{2,\uparrow}$.  Conditional expectation in
$Y=Y_0-B^F+I(Z)-K^F$ gives \eqref{eq:defect-gap}; its nonnegative defect
increment is conditionally zero exactly when it vanishes $P$-a.s.
\end{proof}

\subsubsection{Backward response and completion limits}

\begin{definition}[Backward first-jet trajectory]
\label{def:canonical-backward-first-jet}
Let $\xi$ and an admissible driver $F$ be operational data.  A
\emph{backward first-jet trajectory}
 is a stable lift $(Y,Z)\in\mathscr W_{Q,\cP}^{1,2}$ with
$Y_T=\xi$.  For every declared continuation seam $s\le t$, require the payoff
\[
 G_{s,t}[Y,Z]:=Y_t+B_t^F[Y,Z]-B_s^F[Y,Z]
\]
to belong to the declared domain of $\mathcal E_{s,t}$ or to its specified
continuous $L^2$ extension, and require
\begin{equation}\label{eq:backward-response-law}
 Y_s=\mathcal E_{s,t}\!\left[G_{s,t}[Y,Z]\right],
\end{equation}
and $K^F[Y,Z]$ has nondecreasing paths outside a polar set.  Algebraically,
\begin{equation}\label{eq:backward-terminal-equation}
 Y_t=\xi+B_T^F[Y,Z]-B_t^F[Y,Z]
      -(I_T(Z)-I_t(Z))+K_T^F-K_t^F.
\end{equation}
On a compressed ambient linear branch this becomes
$(Y,Z,K)=(Y,D_QY,K^F[Y,D_QY])$; on a nonlinear solver-compressed branch,
$D_QY$ is replaced by the restricted recovery $\mathcal Z_{\rm sol}(Y)$ of
\Cref{prop:nonlinear-solver-range-compression}.
\end{definition}

Thus the state is the response-selected stable first-jet lift.  With one law,
linear conditional expectation, and flat defect, it is the usual
Pardoux--Peng BSDE coordinate system \cite{PardouxPeng90}.

\begin{corollary}[Fixed-interface, fixed-driver response-completion extension]
\label{thm:backward-universal-limit}
Fix the law family $\cP$, one backward-compatible continuation interface
$(\mathcal E_{s,t})$, and one admissible driver $F$ in the sense of
\Cref{def:admissible-backward-driver}.  Let $\mathscr D_0$ be a
response-quotiented operational data core for this fixed triple and suppose
every $d\in\mathscr D_0$ carries a backward first-jet trajectory with terminal
response $\xi^d$.  Assume that $d\mapsto\xi^d$ is part of the declared data
metric and is Cauchy-continuous into $\mathcal L_{\cP}^2$.  If the
operational solution map
\[
 \Phi_{\rm bw}^{\rm op}:d\longmapsto(Y^d,Z^d)
\]
is Cauchy-continuous from the declared data metric into
$\mathscr W_{Q,\cP}^{1,2}$, then it extends uniquely to the metric completion
$\overline{\mathscr D_0}$:
\begin{equation}\label{eq:backward-universal-map}
 \Phi_{\rm bw}:\overline{\mathscr D_0}
 \longrightarrow\mathscr W_{Q,\cP}^{1,2}.
\end{equation}
Assume also that $\mathscr D_0$ carries declared truncations
$\rho_u^{\mathscr D}$ which preserve $\mathscr D_0$, are Cauchy-continuous
for the data metric, and form a truncation system:
\[
 \rho_T^{\mathscr D}=I,
 \qquad
 \rho_u^{\mathscr D}\rho_v^{\mathscr D}
 =\rho_{u\wedge v}^{\mathscr D},
 \qquad u,v\in[0,T].
\]
Assume, moreover, that
\begin{equation}\label{eq:backward-data-restart-compatibility}
 \Phi_{\rm bw}^{\rm op}(\rho_u^{\mathscr D}d)
 =\rho_u^{\mathscr W}\Phi_{\rm bw}^{\rm op}(d),
 \qquad d\in\mathscr D_0,\quad u\in[0,T].
\end{equation}
Then the truncations and the displayed compatibility extend to the two
completions; the extended maps still satisfy the truncation identities.
The terminal response extends simultaneously; writing its extension again as
$d\mapsto\xi^d$, continuity of terminal evaluation gives
$Y_T^d=\xi^d$ on $\overline{\mathscr D_0}$.  Thus the limit is independent of
the operational approximation in the declared data metric and retains the
terminal condition.  The running sources $B^F[Y^d,Z^d]$, common stochastic integrals,
and intrinsic defects converge as continuous readouts, and the limiting defect
remains in $\mathbb K_{\cP}^{2,\uparrow}$.  For every declared seam $s\le t$
put
\[
 \mathscr G_{s,t}^0
 :=\{G_{s,t}[Y^d,Z^d]:d\in\mathscr D_0\}.
\]
If $\mathcal E_{s,t}$ extends continuously from $\mathscr G_{s,t}^0$ to its
closure in $\mathcal L_{\cP}^2$, with values in
$\mathcal L_{\cP}^2$, and these extensions preserve the exact tower on their
common closed seam domains, then the seam identity
\eqref{eq:backward-response-law} passes to the limit and remains
consistent with that closed deterministic tower.

Let
\[
 \mathscr S_{\rm sol}
 :=\overline{\Phi_{\rm bw}(\overline{\mathscr D_0})}
 ^{\,\mathscr W_{Q,\cP}^{1,2}}.
\]
Assume in addition that this completed solver range is invariant under the
first-jet truncations.
If the equivalent conditions of
\Cref{prop:nonlinear-solver-range-compression} hold on
$\mathscr S_{\rm sol}$, then the completed solution map admits the value-only
realization
\[
 \Phi_{\rm bw}(d)
 =\bigl(Y^d,\mathcal Z_{\rm sol}(Y^d)\bigr).
\]
Only when the solver range is a closed linear relation does this reduce to the
vertical-space and closed-value-domain criterion of
\Cref{thm:backward-compression-retention}.
\end{corollary}

\begin{proof}
Descent defines the map on the response quotient, and Cauchy continuity gives
its unique completion.  Cauchy continuity of $\rho_u^{\mathscr D}$ gives its
extension.  Since composition of continuous maps agrees on the dense data
core, the identities
$\rho_T^{\mathscr D}=I$ and
$\rho_u^{\mathscr D}\rho_v^{\mathscr D}
=\rho_{u\wedge v}^{\mathscr D}$ persist on the completion; the same density
argument passes \eqref{eq:backward-data-restart-compatibility} to the two
completed spaces.  Continuity
of the integral, source, defect, and terminal evaluation passes the remaining
readouts.  If $(Y^n,Z^n)\to(Y,Z)$ in the first-jet graph metric, then
\[
 \|G_{s,t}[Y^n,Z^n]-G_{s,t}[Y,Z]\|_{\mathcal L_{\cP}^2}
 \le \|Y^n-Y\|_{\mathbb S_{\cP}^2}
    +2\|B^F[Y^n,Z^n]-B^F[Y,Z]\|_{\mathbb S_{\cP}^2}
 \longrightarrow0.
\]
Thus the specified continuous seam extensions pass the response law to the
limit, while their common-domain hypothesis preserves the tower.  The
increasing cone is closed by \Cref{prop:backward-intrinsic-defect}.  The
last claim is \Cref{prop:nonlinear-solver-range-compression}; the
closed-domain criterion applies only to its closed linear special case.
\end{proof}

\begin{remark}[Data fixed by the backward completion]
\label{rem:backward-completion-data}
The theorem is a continuity extension for a fixed law family, continuation
interface, driver, testing core, and one-channel integral branch.  Solver
existence and uniqueness are inputs of that interface; variable drivers also
require running-source convergence in the data metric.  The quadratic state
$\mathcal A_2$ supplies the covariance fibers and common integral, while the
terminal data, continuation rule, and solver stability are specified
separately.  Accordingly, $D_Q^{\rm pw}Y$ is the GKW energy jet of the
backward response rather than a rough-path Gubinelli derivative.
\end{remark}

\section{Exterior and spectral representations}
\label{sec:further-exact-representations}

The exterior construction propagates the geometric quotient $(X,A)$
algebraically: it is lossless there, and its homogeneous energies recover area
spectral invariants.  Unlike the It\^o tower's scalar quasi-shuffle identities,
it uses the geometric Chen product.

\subsection{The exterior Chen face}

Let
\[
 \mathscr F^\wedge(H):=\bigoplus_{k\ge0}^{\ell^2}\bigwedge^kH
\]
be the Hilbert exterior Fock space.  We normalize its exterior powers so
that, for every orthonormal basis $(e_j)_j$ of $H$, the increasing wedges
$e_{i_1}\wedge\cdots\wedge e_{i_k}$ form an orthonormal basis of
$\bigwedge^kH$.  Exterior multiplication is the corresponding normalized
alternating product; equivalently, the algebraic map
$h_1\otimes\cdots\otimes h_k\mapsto h_1\wedge\cdots\wedge h_k$ is understood
with this Hilbert normalization.  All constants below use this convention.
If $A\in\mathcal S_2(H)_{\rm sk}$, its associated two-form is
\begin{equation}
 \label{eq:state-area-two-form}
 \alpha_A:=\sum_{i<j}\langle Ae_j,e_i\rangle
 e_i\wedge e_j\in\bigwedge^2H,
 \qquad
 \|\alpha_A\|^2=\frac12\|A\|_{\mathcal S_2}^2.
\end{equation}
The definition is independent of the orthonormal basis.  Write
$\exp_\wedge$ for the exponential with respect to exterior multiplication.
We work with the following \emph{exterior Chen image}:
\begin{equation}
\label{eq:exterior-Chen-image-group}
 \mathscr G_2^\wedge(H)
 :=\bigl\{\exp_\wedge(x+2\alpha_A):
          x\in H,\ A\in\mathcal S_2(H)_{\rm sk}\bigr\}
 \subset\mathscr F^\wedge(H),
\end{equation}
with the topology transported from $H\times\mathcal S_2(H)_{\rm sk}$ through
the degree-one and degree-two coordinates.
For $u=\sum_k u_k$ and $v=\sum_kv_k$ in the Fock space, the notation
$u\wedge v$ is used only when the degreewise Cauchy product
$w_n=\sum_{i+j=n}u_i\wedge v_j$ again belongs to
$\mathscr F^\wedge(H)$.  The theorem below proves that this partial product is
defined and closed on $\mathscr G_2^\wedge(H)$.

\begin{theorem}[Exterior Chen representation of the geometric state]
\label{thm:exterior-Chen-representation}
For $x\in H$ and $A\in\mathcal S_2(H)_{\rm sk}$, define
\begin{equation}
 \label{eq:exterior-state-exponential}
 \mathscr E(x,A):=\exp_\wedge(x+2\alpha_A).
\end{equation}
Then the following hold.

\textup{(i)} The series belongs to $\mathscr F^\wedge(H)$ and satisfies
\begin{equation}
 \label{eq:exterior-Fock-bound}
 \|\mathscr E(x,A)\|_{\mathscr F^\wedge}^2
 \le(1+\|x\|^2)\exp(4\|\alpha_A\|^2).
\end{equation}

\textup{(ii)} If
\[
 (x,A)\star(y,B)
 =\bigl(x+y,A+B+\Anti(x\otimes y)\bigr),
\]
then
\begin{equation}
 \label{eq:exterior-Chen-homomorphism}
 \mathscr E\bigl((x,A)\star(y,B)\bigr)
 =\mathscr E(x,A)\wedge\mathscr E(y,B).
\end{equation}
Thus the exterior product is closed on $\mathscr G_2^\wedge(H)$ and
$\mathscr E$ is a topological group isomorphism from the geometric step-two
state group onto its exterior Chen image, with the transported topology in
\eqref{eq:exterior-Chen-image-group}.  Its degree-one and degree-two coordinates
recover $x$ and $A$.

\textup{(iii)} Writing
$\mathscr E(x,A)=\sum_{k\ge0}\Xi_k(x,A)$ by homogeneous degree,
\begin{equation}
 \label{eq:even-odd-exterior-levels}
 \Xi_{2N}(x,A)=\frac{2^N}{N!}\alpha_A^{\wedge N},
 \qquad
 \Xi_{2N+1}(x,A)=\frac{2^N}{N!}
 x\wedge\alpha_A^{\wedge N}.
\end{equation}
For a state $z=(x,A,q)$, the joint map
\[
 z\longmapsto(\mathscr E(x,A),q)
\]
is a continuous lossless group isomorphism onto
$\mathscr G_2^\wedge(H)\times\mathcal S_1(H)_{\rm sa}$, where the second
factor is additive.  The covariance channel is exactly the information
forgotten by the exterior factor alone.
\end{theorem}

\begin{proof}
Since $x\wedge x=0$ and degree one commutes with degree two under exterior
multiplication,
\[
 \exp_\wedge(x+2\alpha_A)
 =(1+x)\wedge\sum_{N\ge0}\frac{2^N}{N!}\alpha_A^{\wedge N},
\]
which gives \eqref{eq:even-odd-exterior-levels}.  If
$a_1\ge a_2\ge\cdots$ are the positive skew singular values of $A$, then
\[
 \|\alpha_A^{\wedge N}\|^2=(N!)^2e_N(a_1^2,a_2^2,\ldots)
 \le N!\|\alpha_A\|^{2N}.
\]
Summation gives \eqref{eq:exterior-Fock-bound}.

For the product law, even exterior elements commute with homogeneous ones and
$x\wedge x=0$, so
\[
 \mathscr E(x,A)=(1+x)\wedge\exp_\wedge(2\alpha_A).
\]
In the next displays all products are first read degreewise: at each fixed
degree the Cauchy product is a finite algebraic sum.  Finite-degree truncation
therefore gives
\begin{align*}
 \mathscr E(x,A)\wedge\mathscr E(y,B)
 &=(1+x)\wedge(1+y)\wedge
   \exp_\wedge\bigl(2\alpha_A+2\alpha_B\bigr)\\
 &=(1+x+y+x\wedge y)\wedge
   \exp_\wedge\bigl(2\alpha_A+2\alpha_B\bigr).
\end{align*}
Now $x\wedge y$ is decomposable, so $(x\wedge y)^{\wedge2}=0$, and
\[
 (1+x+y)\wedge\exp_\wedge(x\wedge y)
 =1+x+y+x\wedge y.
\]
Since
$2\alpha_{\Anti(x\otimes y)}=x\wedge y$, the last two displays give
\[
 \mathscr E(x,A)\wedge\mathscr E(y,B)
 =\mathscr E\bigl(x+y,A+B+\Anti(x\otimes y)\bigr).
\]
By part~\textup{(i)} the candidate on the right belongs to
$\mathscr F^\wedge(H)$; hence the degreewise Cauchy product on the left is
defined in the sense preceding the theorem.  This proves
\eqref{eq:exterior-Chen-homomorphism} and closure on
$\mathscr G_2^\wedge(H)$.  The first two homogeneous coordinates recover $x$
and $2\alpha_A$, proving injectivity; the separate $q$ coordinate proves the
final assertion.
\end{proof}

\begin{remark}[Exterior quotient of the all-order geometric signature]
\label{rem:exterior-quotient-geometric-signature}
Let $\operatorname{Alt}_k:H^{\otimes_2 k}\to\bigwedge^kH$ denote the
continuous canonical quotient at degree $k$.  Interpreting
$\operatorname{Alt}\mathbf S^{S,\infty}$ levelwise, every state increment
satisfies
\[
 \bigl(\operatorname{Alt}_k S^{S,(k)}_{s,t}\bigr)_{k\ge0}
 =\mathscr E(x_{s,t},A_{s,t}).
\]
This follows from smooth finite-dimensional signatures and fixed-degree
continuity, equivalently uniqueness of the Chen homomorphism
\eqref{eq:exterior-Chen-homomorphism}.  The identity concerns the geometric
Chen product; the It\^o quasi-shuffle structure uses the separate covariance
channel.
\end{remark}

\begin{remark}[Compatibility with the little weak-trace endpoint]
\label{rem:exterior-endpoint-bridge}
Only the continuous area embedding into $\mathcal S_2(H)$ is used.  Hence the
stochastic endpoint of
\cite{ZhaoLevyArea26}, where
$A\in\mathcal S^0_{1,\infty}(H)\hookrightarrow\mathcal S_2(H)$, carries the
same readout and scalarization, although the Banach operational minimality
theorem and this endpoint realization use different topologies.
\end{remark}

\subsection{Even--odd exterior spectral response}

The full exterior element is lossless on $(x,A)$; its scalar energies are not.
Exactly, the even energies determine the area spectrum and the odd energies
locate the first-level increment relative to it.

\begin{theorem}[Even--odd exterior spectral decoding]
\label{thm:even-odd-spectral-decoding}
Let $A\in\mathcal S_2(H)_{\rm sk}$ have positive skew singular values
$a_1\ge a_2\ge\cdots$, and let
$\mathscr E(x,A)=\sum_{k\ge0}\Xi_k(x,A)$.  Define
\[
 \mathscr D_A(z):=\prod_{j\ge1}(1+za_j^2),
 \qquad z\in\mathbb C.
\]
Then
\begin{equation}
 \label{eq:even-exterior-spectral-series}
 \mathscr D_A(z)
 =\sum_{N\ge0}\frac{z^N}{4^N}\|\Xi_{2N}(x,A)\|^2,
 \qquad
 \det(\operatorname{Id}+zA^*A)=\mathscr D_A(z)^2.
\end{equation}
In particular, the even energy sequence determines the complete multiset of
nonzero positive skew singular values.

For $z>0$,
\begin{equation}
 \label{eq:odd-exterior-resolvent-series}
 \sum_{N\ge0}\frac{z^N}{4^N}\|\Xi_{2N+1}(x,A)\|^2
 =\mathscr D_A(z)
 \left\langle(\operatorname{Id}+zA^*A)^{-1}x,x\right\rangle.
\end{equation}
In particular, evaluating the even and odd series at $z=4$ gives the exact
Fock norm
\begin{equation}\label{eq:exterior-exact-Fock-norm}
 \boxed{
 \|\mathscr E(x,A)\|_{\mathscr F^\wedge}^2
 =\mathscr D_A(4)
 \left(1+\left\langle
   (\operatorname{Id}+4A^*A)^{-1}x,x
 \right\rangle\right).}
\end{equation}
Thus \eqref{eq:exterior-Fock-bound} is a direct consequence of
$\mathscr D_A(4)\le\exp(4\|\alpha_A\|^2)$ and
$(\operatorname{Id}+4A^*A)^{-1}\preceq I$.
Consequently, after the even sequence has determined $\mathscr D_A$, the odd
sequence determines the cyclic spectral measure
\[
 \nu_{x,A}(B):=
 \left\langle E_{A^*A}(B)x,x\right\rangle,
 \qquad B\subset[0,\infty)\text{ Borel},
\]
through
\begin{equation}
 \label{eq:cyclic-spectral-transform}
 \frac{\sum_{N\ge0}4^{-N}z^N\|\Xi_{2N+1}(x,A)\|^2}
      {\mathscr D_A(z)}
 =\int_{[0,\infty)}\frac{1}{1+z\lambda}\,
   \dd\nu_{x,A}(\lambda).
\end{equation}
Thus the scalar even--odd exterior response recovers the singular spectrum of
$A$ and the energy of $x$ in every $A^*A$-spectral subspace.  Directions inside a degenerate spectral subspace remain unresolved by these scalar spectral data.
\end{theorem}

\begin{proof}
Choose mutually orthogonal oriented two-planes $E_j$ so that
\[
 \alpha_A=\sum_{j\ge1}a_j\omega_j,
 \qquad
 \omega_j=e_{2j-1}\wedge e_{2j},
\]
and decompose
\[
 x=x_0+\sum_{j\ge1}x_j,
 \qquad
 x_0\in\ker A,
 \quad x_j\in E_j.
\]
The even formula in \eqref{eq:even-odd-exterior-levels} gives
\[
 4^{-N}\|\Xi_{2N}\|^2
 =e_N(a_1^2,a_2^2,\ldots),
\]
and summing the elementary-symmetric expansion proves
\eqref{eq:even-exterior-spectral-series}.

For the odd levels, orthogonality of the exterior basis gives
\begin{align*}
 4^{-N}\|\Xi_{2N+1}\|^2
 =\sum_{|J|=N}
 \left(\prod_{j\in J}a_j^2\right)
 \left(\|x_0\|^2+\sum_{k\notin J}\|x_k\|^2\right).
\end{align*}
The series is absolutely convergent locally uniformly in $z$ because it is
bounded coefficientwise by $\|x\|^2e_N(a_1^2,a_2^2,\ldots)$.  Summing first
the $x_0$ term and then the contribution of each fixed $x_k$ yields
\[
 \mathscr D_A(z)\left(
 \|x_0\|^2+\sum_{k\ge1}\frac{\|x_k\|^2}{1+za_k^2}
 \right),
\]
which is \eqref{eq:odd-exterior-resolvent-series}.  Adding the even and odd
series at $z=4$ proves \eqref{eq:exterior-exact-Fock-norm}.  The spectral
theorem gives \eqref{eq:cyclic-spectral-transform}; uniqueness of the
Stieltjes transform recovers $\nu_{x,A}$.
\end{proof}

\begin{corollary}[Exterior lossless representation]
\label{cor:exterior-lossless-face}
Equip the restarted exterior response image with the topology and
truncations transported from the corresponding $(x,A,q)$ increment-field
space through its degree-one and degree-two coordinates.
On the principal Hilbert operational state space, the restarted response
\[
 z\longmapsto\bigl((\mathscr E(x_{s,t},A_{s,t}))_{s,t},q\bigr)
\]
is stable-causal and lossless.  Hence it represents the same causal state as
the restarted It\^o signature.  Its scalar even--odd energy response is a
canonical nonlinear spectral scalarization: it is not lossless, but it
recovers exactly the unitary-invariant data described in
\Cref{thm:even-odd-spectral-decoding}.
\end{corollary}

\begin{proof}
Stable causality and the restart law follow from
\eqref{eq:exterior-Chen-homomorphism}.  The degree-one and degree-two
coordinates recover $x$ and $A$, while the displayed response retains $q$.
The lossless-response criterion and \Cref{thm:least-state} give the
least-state statement.  The spectral assertion is
\Cref{thm:even-odd-spectral-decoding}.
\end{proof}

\section{Conclusion}

Operational state selection separates causal descent and calibrated completion
from common probabilistic realization.  For the degree-two Hilbert It\^o
program, this procedure selects
\[
 \operatorname{ch}(\mathcal A_2)=(X,A,Q),
 \qquad A\in\mathcal S_2(H)_{\rm sk},\quad
 Q\in\mathcal S_1(H)_{\rm sa}.
\]
Brownian experiments identify the Hilbert--Schmidt/operator coefficient gauges
and hence the dual \(\mathcal S_2/\mathcal S_1\) state geometry; the orthogonal
rigidity theorem gives the converse characterization.  The lossless
encoder--decoder pair yields the direct least-state property
\[
 S\text{ jointly realizes the quadratic program}
 \quad\Longleftrightarrow\quad
 \mathcal A_2\preceqsc S,
\]
so \(\mathcal A_2\) is unique up to stable-causal equivalence.

A single sequence of total Borel causal approximants constructs the common
state, while the laws verify its classical semantics.  Defect--energy transfer
makes the uniform-\(\mathcal S_2\) primitive independent of the chosen causal
finite-variation regularization whenever the first-level upper-energy defect
vanishes.  On driftless subclasses, uniform spatial trace tightness yields
restart-stable compact capacity cores; a fixed trace-class covariance envelope
is a sufficient closed specialization.

Each fixed higher signature level and the joint exterior--covariance
presentation are readouts of the selected state.  Rough--Young maps propagate
it to finite-dimensional flows on the compact cores.  For deterministic
continuous coefficients satisfying the stated global state-Lipschitz and
linear-growth bounds, Euler schemes construct a total jointly Borel raw-causal
field whose flow law holds outside one parameter-independent polar set.  Under
the additional rough-vector-field hypotheses, that field agrees there with the
bracket-corrected rough flow simultaneously in all parameters.

On compact covariance-envelope subclasses, covariance cross-responses select a
fiber-valued backward jet.  Its active component belongs to the common
upper-energy space under~\textup{(A2)}, stopped martingale responses calibrate
its norm, and the common-channel clause~\textup{(B4)} identifies it with the
solver's martingale coordinate.  The vertical space and the closedness of the
value domain characterize compression for closed linear response relations;
restricted nonlinear solver ranges are governed by the corresponding graph
criterion.

\appendix

\section{Abstract operational completion and tangent response geometry}
\label{sec:response-generated-spaces}

This appendix implements the response congruence and completion in
\Cref{thm:descent-before-completion}: the joint profile separates the quotient
classes and its Hausdorff initial uniformity supplies the geometry to be
completed.  Quasi-sure identification enters only after realization.

\begin{definition}[Response presentation and response congruence]
\label{def:response-presentation}
A \emph{response presentation} is a family
\[
 \mathfrak R
 =\bigl(X,(\mathsf Y_\alpha,\mathcal U_\alpha,R_\alpha)_{\alpha\in A}\bigr),
\]
where $X$ is a set, each $(\mathsf Y_\alpha,\mathcal U_\alpha)$ is a complete
Hausdorff uniform space, and
$R_\alpha:X\to\mathsf Y_\alpha$ is a map.  No topology or uniformity is
assumed on $X$.

The \emph{joint response profile} is
\[
 \mathbf R:X\longrightarrow
 \mathsf Y:=\prod_{\alpha\in A}\mathsf Y_\alpha,
 \qquad
 \mathbf R(x):=(R_\alpha(x))_{\alpha\in A}.
\]
Two candidates are \emph{response-equivalent} if
\[
 x\sim_{\mathfrak R}y
 \quad\Longleftrightarrow\quad
 \mathbf R(x)=\mathbf R(y).
\]
Write $X_{\mathfrak R}:=X/\!\sim_{\mathfrak R}$.  The induced injective map
\[
 \iota_{\mathfrak R}^0:X_{\mathfrak R}\longrightarrow\mathsf Y
\]
is called the \emph{response profile embedding}.  The initial uniformity on
$X_{\mathfrak R}$ induced by its coordinate responses is the
\emph{response-generated uniformity}; its topology is the
\emph{response-generated topology}.
\end{definition}

\begin{definition}[Response completion and complete faithful realization]
\label{def:response-completion}
The \emph{response completion} of $\mathfrak R$ is
\begin{equation}\label{eq:canonical-response-completion}
 \mathsf Z_{\mathfrak R}
 :=\overline{\iota_{\mathfrak R}^0(X_{\mathfrak R})}^{\,\mathsf Y},
\end{equation}
with the subspace uniformity from the product $\mathsf Y$.

A \emph{complete faithful realization} of $\mathfrak R$ consists of a complete
Hausdorff uniform space $E$, a uniformly continuous map
$j:X_{\mathfrak R}\to E$ with dense range, and uniformly continuous maps
$\widetilde R_\alpha:E\to\mathsf Y_\alpha$ such that
\[
 \widetilde R_\alpha\circ j=R_\alpha
 \quad\text{on }X_{\mathfrak R},
\]
the family $(\widetilde R_\alpha)_\alpha$ separates points of $E$, and the
uniformity of $E$ is the initial uniformity generated by this family.
\end{definition}

\begin{theorem}[Response completion]
\label{thm:canonical-response-completion}
For every response presentation $\mathfrak R$:
\begin{enumerate}[label=\textup{(\roman*)},leftmargin=2.4em]
\item $\mathsf Z_{\mathfrak R}$ is a complete Hausdorff uniform space and
$\iota_{\mathfrak R}^0(X_{\mathfrak R})$ is dense in it;
\item every response $R_\alpha$ extends uniquely to the coordinate restriction
\[
 \overline R_\alpha:\mathsf Z_{\mathfrak R}\to\mathsf Y_\alpha;
\]
\item every complete faithful realization
$(E,j,(\widetilde R_\alpha)_\alpha)$ is canonically uniformly isomorphic to
$\mathsf Z_{\mathfrak R}$: there is a unique uniform isomorphism
\[
 \Theta:E\longrightarrow\mathsf Z_{\mathfrak R}
\]
with $\Theta\circ j=\iota_{\mathfrak R}^0$ and
$\overline R_\alpha\circ\Theta=\widetilde R_\alpha$ for every $\alpha$.
\end{enumerate}
\end{theorem}

\begin{proof}
The product of complete Hausdorff uniform spaces is complete and Hausdorff, so
its closed subspace $\mathsf Z_{\mathfrak R}$ has the same properties, proving
\textup{(i)}.  The coordinate projections of the product restrict to
$\mathsf Z_{\mathfrak R}$ and extend the original responses, giving
\textup{(ii)}; uniqueness follows from density and Hausdorffness.

For \textup{(iii)}, put
\[
 \widetilde{\mathbf R}:E\to\mathsf Y,
 \qquad
 \widetilde{\mathbf R}(e)
 :=(\widetilde R_\alpha(e))_{\alpha\in A}.
\]
Joint faithfulness and the initial-uniformity assumption make
$\widetilde{\mathbf R}$ a uniform embedding.  Since $E$ is complete, its image
is a complete subspace of the Hausdorff uniform space $\mathsf Y$ and hence is
closed.  Because $j(X_{\mathfrak R})$ is dense in $E$,
\[
 \widetilde{\mathbf R}(E)
 =\overline{\widetilde{\mathbf R}(j(X_{\mathfrak R}))}
 =\overline{\iota_{\mathfrak R}^0(X_{\mathfrak R})}
 =\mathsf Z_{\mathfrak R}.
\]
Thus $\widetilde{\mathbf R}$ is the required uniform isomorphism $\Theta$.
The intertwining identities give uniqueness.
\end{proof}

\begin{corollary}[Derived-response inheritance]
\label{cor:deterministic-response-inheritance}
Let $F_j:\mathsf Z_{\mathfrak R}\to\mathsf T_j$, $j\in J$, be uniformly
continuous maps into complete Hausdorff uniform spaces and adjoin to
$\mathfrak R$ the derived responses
\[
 x\longmapsto F_j\bigl(\iota_{\mathfrak R}^0([x])\bigr).
\]
Then the enlarged response presentation has the same response congruence,
the same response-generated uniformity, and the same response completion
up to the canonical identity isomorphism.
\end{corollary}

\begin{proof}
The original responses remain present, so the response congruence cannot
become coarser.  Every added response factors uniformly continuously through
$\mathsf Z_{\mathfrak R}$ and hence through the original response-generated
uniformity, so the enlarged family cannot generate a strictly finer
uniformity on $X_{\mathfrak R}$.  The uniformities coincide, and uniqueness of
completion in \Cref{thm:canonical-response-completion} identifies the completions.
\end{proof}

\begin{corollary}[Invariance under response-equivalent presentations]
\label{cor:response-presentation-invariance}
Suppose two response presentations have response quotients related by a
bijection that intertwines their joint response profiles after uniform
isomorphisms of the response targets.  Then their response-generated
uniformities correspond and their response completions are canonically
uniformly isomorphic.
\end{corollary}

\begin{proof}
After transporting the response targets by the given uniform isomorphisms,
the two profile embeddings have the same image up to the stated bijection.
Hence they induce the same initial uniformity on the response quotient.
Their closures in the corresponding product response spaces are therefore
identified by the product uniform isomorphism; equivalently apply the
uniqueness clause of \Cref{thm:canonical-response-completion}.
\end{proof}

\begin{remark}[Response-generated completion]
A topology alone does not determine a completion.  Here the response
presentation supplies the initial uniformity---the norm-induced uniformity in
the metric operator-ideal applications---and the selected state is the
completion of the response quotient in that uniformity.  The construction is
invariant under response-equivalent presentations, including presentations on
different carriers.
\end{remark}

\subsection{Sequential response reconstruction of restart algebra}
\label{sec:restart-representation-principle}

The lossless response recovers the state pointwise.  Sequential responses also
determine the restart multiplication.

Let $\Sigma$ be a bare set of elementary marked segments with observable
endpoint maps $\partial_-,\partial_+:\Sigma\to\mathsf Z_0$.  Let
$\mathsf W(\Sigma)$ be the free marked restart system of finite
seam-compatible words, including the empty word at each mark, with composition
given only by word concatenation.

\begin{definition}[Sequential response presentation and restart congruence]
\label{def:compositional-response-presentation}
A sequential response presentation is a family
$\mathcal O_\alpha:\mathsf W(\Sigma)\to\mathsf M_\alpha$ together with the
endpoint marks.  Write $\mathbf O$ for the full marked profile and put
\[
 w\sim_{\mathbf O}v\quad\Longleftrightarrow\quad \mathbf O(w)=\mathbf O(v).
\]
It is \emph{restart-congruent} if $w\sim_{\mathbf O}v$ implies
$awb\sim_{\mathbf O}avb$ whenever both contextual concatenations are
admissible; either contextual word may be empty.  No multiplication is assumed
on the response targets.
\end{definition}

\begin{theorem}[Restart composition from sequential responses]
\label{thm:algebraic-restart-reconstruction}
For every restart-congruent sequential presentation,
\[
 \mathfrak C_{\mathbf O}:=\mathsf W(\Sigma)/\!\sim_{\mathbf O}
\]
carries the unique marked partial semigroupoid product
\begin{equation}\label{eq:response-reconstructed-product}
 [w]\diamond_{\mathbf O}[v]:=[wv].
\end{equation}
The realized joint response image $\mathbf O(\mathsf W(\Sigma))$ carries the
transported product
\begin{equation}\label{eq:reconstructed-response-target-product}
 \mathbf O(w)\star_{\mathbf O}\mathbf O(v):=\mathbf O(wv).
\end{equation}
Moreover $\mathfrak C_{\mathbf O}$ is the least algebraic restart realization:
any restart realization through which all sequential responses factor admits a
unique surjective restart homomorphism onto $\mathfrak C_{\mathbf O}$ after
restriction to the generated sub-semigroupoid.
\end{theorem}

\begin{proof}
Restart congruence makes \eqref{eq:response-reconstructed-product} independent
of representatives.  Associativity and identities descend from free word
concatenation.  Since the joint response separates its own equivalence classes,
transporting the quotient product gives
\eqref{eq:reconstructed-response-target-product}.  If $J$ is another restart
realization and $J_*(w)=J_*(v)$, factorization of all responses through $J_*$
implies $w\sim_{\mathbf O}v$; hence $J_*(w)\mapsto[w]$ is the required unique
surjective homomorphism.
\end{proof}

\begin{proposition}[Quadratic restart rigidity]
\label{prop:quadratic-restart-reconstruction}
On the smooth finite-rank response skeleton, sequential concatenation obeys
\begin{equation}\label{eq:quadratic-signature-restart-law}
 x(wv)=x(w)+x(v),\qquad
 V(wv)=V(w)+V(v)+x(w)\otimes x(v).
\end{equation}
Therefore joint signature-response equivalence is a restart congruence.  Under
the global decoder \eqref{eq:primitive-decoder}, the reconstructed product is
\begin{equation}\label{eq:quadratic-restart-group-law}
 (x,A,q)\star(y,B,r)
 =\bigl(x+y,A+B+\Anti(x\otimes y),q+r\bigr).
\end{equation}
The first and covariance coordinates are autonomous additive blocks, whereas
the area channel is triangularly coupled to the first level.  The product is
Cauchy-continuous on the finite-rank skeleton and has a unique jointly
continuous extension, given by the same formula, to the complete state and
coherent response targets.
\end{proposition}

\begin{proof}
Equation \eqref{eq:quadratic-signature-restart-law} is the step-two tensor
concatenation law.  Applying
\Cref{thm:algebraic-restart-reconstruction} reconstructs the product on the
joint signature response.  The decoder gives
\eqref{eq:quadratic-restart-group-law}; its area cross term is
$\Anti(x\otimes y)$ and its covariance coordinate is additive.  To extend
this product from the incomplete skeleton, it remains to verify the Cauchy
property.  If $(x_n)$ and $(y_n)$ are Cauchy in $H$,
then they are bounded and
\[
 \|\Anti(x_n\otimes y_n)-\Anti(x_m\otimes y_m)\|_{\mathcal S_2}
 \le
 \|x_n-x_m\|_H\,\|y_n\|_H
 +\|x_m\|_H\,\|y_n-y_m\|_H.
\]
Thus the cross term is Cauchy, and the same estimate shows that its limit is
independent of the two approximating skeleton sequences.  Formula
\eqref{eq:quadratic-restart-group-law} therefore defines the unique product on
the completion.  The rank-one estimate also proves joint continuity, locally
uniformly on bounded sets.
\end{proof}

\begin{remark}[Observable origin of Chen and covariance laws]
The area Chen relation and covariance additivity are the coordinate form of
the reconstructed sequential product.  Probabilistic realization then selects
random trajectories with values in this response-generated algebra.
\end{remark}

\subsection{Infinitesimal response geometry}
\label{sec:infinitesimal-response-geometry}

The response uniformity measures finite errors; first-order geometry also
requires their scale along short experiments.  A declared scale therefore
quotients controlled germs by infinitesimal response indistinguishability.
We work with the metrizable subpresentations used in the operator-ideal
applications.

Let $z\in\mathsf Z_{\mathfrak R}$ and let
$(\mathsf Y_\beta,d_\beta)$, $\beta\in\mathfrak B$, be metric response targets
with extended readouts $\overline R_\beta$.  Put
\[
 \delta_\beta(u,v)
 :=d_\beta\bigl(\overline R_\beta(u),\overline R_\beta(v)\bigr).
\]

\begin{definition}[Scaled response germs]
\label{def:scaled-response-germs}
A \emph{scale system} at $z$ is a family
\[
 r_\beta:(0,1]\to(0,\infty),
 \qquad r_\beta(\eps)\downarrow0,
 \quad \beta\in\mathfrak B.
\]
A pointed curve $\gamma=(\gamma_\eps)_{\eps\downarrow0}$ in
$\mathsf Z_{\mathfrak R}$ is an \emph{$r$-controlled response germ at $z$} if
$\gamma_\eps\to z$ in the response-generated uniformity and
\begin{equation}\label{eq:scaled-response-germ-control}
 \limsup_{\eps\downarrow0}
 \frac{\delta_\beta(\gamma_\eps,z)}{r_\beta(\eps)}<\infty
 \qquad\text{for every }\beta.
\end{equation}
For two controlled germs set
\begin{equation}\label{eq:infinitesimal-response-equivalence}
 \gamma\sim_{z,r}^{(1)}\eta
 \quad\Longleftrightarrow\quad
 \frac{\delta_\beta(\gamma_\eps,\eta_\eps)}{r_\beta(\eps)}\longrightarrow0
 \quad\text{for every }\beta.
\end{equation}
The quotient
\begin{equation}\label{eq:response-tangent-quotient}
 T_{z,r}^{\rm resp}\mathsf Z_{\mathfrak R}
 :=\operatorname{Germ}_{z,r}(\mathsf Z_{\mathfrak R})/
      \!\sim_{z,r}^{(1)}
\end{equation}
is the \emph{scaled response tangent set}.  At this level it is a pointed set.
Linear, graded, or Lie structure descends from short-experiment operations
that are asymptotically congruent.
\end{definition}

\begin{definition}[Response differential]
\label{def:response-differential}
Let $F:\mathsf Z_{\mathfrak R}\to V$ take values in a Hausdorff topological
vector space and let $r_F(\eps)\downarrow0$.  We call $F$
\emph{response-differentiable at $z$ relative to $(r,r_F)$} if, for every
$r$-controlled germ $\gamma$, the limit
\begin{equation}\label{eq:response-differential}
 d^{\rm resp}F_z([\gamma])
 :=\lim_{\eps\downarrow0}
   \frac{F(\gamma_\eps)-F(z)}{r_F(\eps)}
\end{equation}
exists in $V$ and depends only on the class $[\gamma]$ in
$T_{z,r}^{\rm resp}\mathsf Z_{\mathfrak R}$.  A scalar-valued such functional
is a \emph{cotangent response functional}.  A gradient is defined after a
Hilbertian tangent geometry has been reconstructed.
\end{definition}

\begin{proposition}[Infinitesimal response quotient principle]
\label{thm:infinitesimal-response-quotient}
For every metric response subpresentation and every scale system:
\begin{enumerate}[label=\textup{(\roman*)},leftmargin=2.5em]
\item \eqref{eq:infinitesimal-response-equivalence} is an equivalence relation,
and $T_{z,r}^{\rm resp}\mathsf Z_{\mathfrak R}$ is the least quotient through
which every germ-level readout constant on infinitesimal response classes
factors;
\item every finitary operation on controlled germs which is asymptotically
congruent descends uniquely to the response tangent quotient, together with
all algebraic identities satisfied before quotienting;
\item every response-differentiable observable factors uniquely through the
tangent quotient.  If a descended block is linear and normalized increments
respect addition and scalar multiplication, then the resulting cotangent
functional is linear on that block.
\end{enumerate}
\end{proposition}

\begin{proof}
Reflexivity and symmetry are immediate, while transitivity follows from the
triangle inequality after division by $r_\beta(\eps)$.  The factorization
claim is the universal property of quotienting by
\eqref{eq:infinitesimal-response-equivalence}.  Asymptotic congruence makes an
operation independent of representatives, so it descends together with its
identities.  The last assertion is exactly class-invariance of the normalized
increment, with linearity inherited from the descended first-order operations.
\end{proof}

\begin{corollary}[Hilbertian response differential]
\label{cor:hilbertian-response-gradient}
Suppose a linear tangent block is generated by a real vector space $V$ with a
positive semidefinite bilinear form $g_z$.  Let
\[
 N_z:=\{v:g_z(v,v)=0\},\qquad
 \mathscr H_z:=\overline{V/N_z}^{\,\|\cdot\|_{g_z}},\qquad
 \|[v]\|_{g_z}^2=g_z(v,v).
\]
If a scalar linear cotangent response $\lambda_z$ satisfies
\begin{equation}\label{eq:response-cotangent-energy-bound}
 |\lambda_z(v)|\le C_z\|[v]\|_{g_z},
\end{equation}
then there is a unique $\nabla_g^{\rm resp}F(z)\in\mathscr H_z$ such that
\begin{equation}\label{eq:response-Riesz-gradient}
 \lambda_z(v)
 =\langle\nabla_g^{\rm resp}F(z),[v]\rangle_{\mathscr H_z}.
\end{equation}
Thus $\lambda_z$ is the cotangent datum, whereas
$\nabla_g^{\rm resp}F(z)$ is its Riesz representative relative to $g_z$.
\end{corollary}

\begin{proof}
The bound annihilates $N_z$ and makes $\lambda_z$ continuous on $V/N_z$.
Extend to the Hilbert completion and apply the Riesz representation theorem.
\end{proof}

\begin{remark}[First-order scale]
The response uniformity alone does not determine a Brownian, parabolic, or
rough first-order scale.  After fixing a scale, the tangent is the
corresponding germ quotient; on a Hilbertian tangent block, gradients are
Riesz representatives of cotangent responses.
\end{remark}

\section{Capacity, quasi-sure identification, and common upper-expectation spaces}
\label{sec:common-objects}

Throughout this section, \(\Omega\) is a Polish space and
\(\cP\subset\mathfrak P(\Omega)\) is nonempty.  For an arbitrary
\(A\subset\Omega\), define the upper capacity
\begin{equation}\label{eq:capacity}
  c_{\cP}(A):=\sup_{P\in\cP}P^*(A),
\end{equation}
where \(P^*\) denotes outer probability.  A set \(N\) is
\(\cP\)-\emph{polar} if \(c_{\cP}(N)=0\).  A property holds
\(\cP\)-quasi surely if it holds outside a polar set, and a set is
\(\cP\)-\emph{full} when its complement is \(\cP\)-polar.
The family \(\cP\) is \emph{uniformly tight} if, for every \(\eps>0\),
there is a compact \(K\subset\Omega\) with
\(\sup_{P\in\cP}P(K^c)<\eps\).

\begin{remark}[Polar equality and information]
The polar ideal compares versions of a common Borel field.  Upper-norm
completion provides the construction mechanism for common fields used below.
\end{remark}

\begin{definition}[Simultaneous lawwise realization]
\label{def:common-representation}
Let \(E\) be Polish, let \(Z:\Omega\to E\) be Borel, and for every
\(P\in\cP\) let \(Z^P\) be an \(E\)-valued random variable defined up to
\(P\)-almost-sure equality.  We say that \(Z\) \emph{realizes the family
simultaneously} if
\[
  Z=Z^P\qquad P\text{-almost surely for every }P\in\cP.
\]
Two simultaneous Borel realizations are identified if they agree
\(\cP\)-quasi surely.  Existence of a simultaneous Borel realization is an
additional compatibility property of the modelwise family.
\end{definition}

\begin{definition}[Compact capacity core]
\label{def:compact-core}
An increasing sequence of compact sets \((K_R)_{R\ge1}\) is a
\emph{compact capacity core} for \(\cP\) if
\[
  \tau_R:=c_{\cP}(K_R^c)\longrightarrow0.
\]
If \(E\) is Polish, a Borel map \(F:\Omega\to E\) is
\emph{core-continuous} if
\(F|_{K_R}\) is continuous for every \(R\).
\end{definition}

The union \(\bigcup_RK_R\) is a common full-capacity domain, compactly
exhausted by continuity sets shared by all models.

\begin{proposition}[Common full-measure event]
\label{prop:common-event}
If \(G\subset\Omega\) is universally measurable and $P(G)=1$ for every
$P\in\cP$, then \(G^c\) is \(\cP\)-polar.
\end{proposition}

\begin{proof}
Universal measurability gives \(P^*(G^c)=P(G^c)=0\) for every \(P\).
\end{proof}

\begin{remark}
The equality \(c_{\cP}(G)=1\) alone does not imply the common full-measure
condition \(c_{\cP}(G^c)=0\).
\end{remark}

\subsection{Regular upper-expectation completions}
\label{sec:regular-completion}

We use the upper-$L^p$ regular completion, the Banach-space counterpart of
quasi-continuous common-version spaces; see
\cite{DenisHuPeng11,SonerTouziZhang11,Cohen12,Nutz12}.

Let $B$ be a separable Banach space and $1\le p<\infty$.  For a Borel map
$Z:\Omega\to B$, set
\begin{equation}\label{eq:upper-Lp-norm}
  \norm{Z}_{\mathcal L_{\cP}^p(B)}
  :=\sup_{P\in\cP}\left(E^P\norm{Z}^p\right)^{1/p}.
\end{equation}
Identify maps that agree $\cP$-quasi surely and denote by
$\mathcal L_{\cP}^p(B)$ the resulting space of finite-norm classes.  The
closure of $C_b(\Omega;B)$ in this norm is denoted by
$\mathbb L_{\cP}^p(B)$.  The latter is the regular capacitary completion; it
need not contain every Borel variable with finite upper moment.

A Borel map $Z:\Omega\to B$ is called $\cP$-\emph{quasi-continuous} if,
for every $\eps>0$, there is an open set $O\subset\Omega$ such that
$c_{\cP}(O)<\eps$ and $Z|_{O^c}$ is continuous.  An equivalence class is
quasi-continuous if it has such a representative.  The following theorem
characterizes the finite-upper-moment variables in the regular completion.

\begin{theorem}[Completeness of the upper-expectation space]
\label{thm:upper-Lp-complete}
For every separable Banach space $B$ and $1\le p<\infty$,
$\mathcal L_{\cP}^p(B)$ is a Banach space.  Consequently,
$\mathbb L_{\cP}^p(B)$ is a closed Banach subspace.  If $\cP=\{P\}$, these
spaces reduce to the usual $L^p(P;B)$ space and the $L^p(P)$ closure of
$C_b(\Omega;B)$, respectively.
\end{theorem}

\begin{proof}
Let $(Z_n)$ be Cauchy.  Choose a subsequence $(Z_{n_k})$ such that
\[
  \norm{Z_{n_{k+1}}-Z_{n_k}}_{\mathcal L_{\cP}^p(B)}
  \le 2^{-2k},\qquad k\ge1.
\]
By Markov's inequality,
\[
  c_{\cP}\!\left(
    \norm{Z_{n_{k+1}}-Z_{n_k}}>2^{-k}
  \right)
  \le 2^{-kp}.
\]
The capacity Borel--Cantelli argument gives a common full-capacity set on
which $(Z_{n_k})$ is Cauchy in $B$.  Define $Z$ as its pointwise limit on
that Borel convergence set and as zero outside it.  Separability of $B$
ensures that $Z$ is Borel.  For each $k$, Fatou's lemma gives
\[
  \norm{Z-Z_{n_k}}_{\mathcal L_{\cP}^p(B)}
  \le \liminf_{\ell\to\infty}
      \norm{Z_{n_\ell}-Z_{n_k}}_{\mathcal L_{\cP}^p(B)},
\]
which tends to zero.  The full sequence converges because it is Cauchy.
Closedness of the regular completion follows from completeness.  When
$\cP=\{P\}$, bounded continuous $B$-valued maps are dense in
$L^p(P;B)$ because $\Omega$ is Polish and $B$ is separable.
\end{proof}

\begin{theorem}[Quasi-continuous characterization]
\label{thm:quasi-continuous-characterization}
Let $B$ be a separable Banach space and $1\le p<\infty$.  For
$Z\in\mathcal L_{\cP}^p(B)$, the following are equivalent:
\begin{enumerate}[label=(\roman*),leftmargin=2.2em]
\item $Z\in\mathbb L_{\cP}^p(B)$;
\item $Z$ has a $\cP$-quasi-continuous representative and
\begin{equation}\label{eq:uniform-p-tail}
  \lim_{N\to\infty}\sup_{P\in\cP}
  E^P\!\left[\norm{Z}^p\one_{\{\norm{Z}>N\}}\right]=0.
\end{equation}
\end{enumerate}
If $\cP$ is uniformly tight, quasi-continuity is equivalent to the existence
of a compact capacity core $(K_R)$ on which $Z$ is core-continuous.
\end{theorem}

\begin{proof}
Suppose first that $Z_n\in C_b(\Omega;B)$ and $Z_n\to Z$ in
$\mathcal L_{\cP}^p(B)$.  Pass to a subsequence, still denoted by $(Z_n)$,
such that
\[
  \norm{Z_{n+1}-Z_n}_{\mathcal L_{\cP}^p(B)}\le 2^{-2n}.
\]
The sets
\[
  A_n:=\{\norm{Z_{n+1}-Z_n}>2^{-n}\}
\]
are open and satisfy $c_{\cP}(A_n)\le2^{-np}$.  For
$O_m:=\bigcup_{n\ge m}A_n$, the capacity of $O_m$ tends to zero, while
$(Z_n)$ is uniformly Cauchy on $O_m^c$.  Its pointwise limit on
$\bigcup_mO_m^c$, extended by one fixed base point to the polar complement, is Borel and
continuous on every $O_m^c$.  Since $Z_n\to Z$ in
$\mathcal L_{\cP}^p(B)$, a further subsequence converges to $Z$ quasi surely.
On $\bigcup_mO_m^c$ that subsequence also converges to the uniform limit just
constructed, so the limit agrees with $Z$ quasi surely.  Thus $Z$ has a
quasi-continuous representative.

To prove \eqref{eq:uniform-p-tail}, fix $n$ and take
$N>2\norm{Z_n}_\infty$.  On $\{\norm{Z}>N\}$ one has
$\norm{Z}\le2\norm{Z-Z_n}$, and therefore
\[
  \sup_{P\in\cP}E^P\!\left[
    \norm{Z}^p\one_{\{\norm{Z}>N\}}
  \right]
  \le 2^p\norm{Z-Z_n}_{\mathcal L_{\cP}^p(B)}^p.
\]
First let $N\to\infty$ with $n$ fixed and then let $n\to\infty$.

Conversely, assume (ii) and choose a quasi-continuous representative.  Let
$T_N:B\to B$ be radial truncation to the closed ball of radius $N$.  Then
\[
 \norm{Z-T_NZ}_{\mathcal L_{\cP}^p(B)}^p
 \le \sup_{P\in\cP}E^P\!\left[
   \norm{Z}^p\one_{\{\norm{Z}>N\}}
 \right]\longrightarrow0.
\]
Given $N$ and $\eps>0$, choose an open set $O$ of arbitrarily small
capacity such that $T_NZ$ is continuous on the closed set $O^c$.
The Banach-valued Tietze--Dugundji extension theorem \cite{Dugundji51} gives
$G\in C_b(\Omega;B)$ with $G=T_NZ$ on $O^c$ and
$\norm{G}_\infty\le N$.  Hence
\[
  \norm{T_NZ-G}_{\mathcal L_{\cP}^p(B)}
  \le 2N c_{\cP}(O)^{1/p}.
\]
After choosing $N$ and then $O$, this proves that $Z$ lies in the
$\mathcal L_{\cP}^p$-closure of $C_b(\Omega;B)$.

Finally suppose that $\cP$ is uniformly tight.  For each $m$, choose a
compact $C_m$ and an open $O_m$ such that
\[
  c_{\cP}(C_m^c)+c_{\cP}(O_m)\le2^{-m},
  \qquad Z|_{O_m^c}\ \text{is continuous}.
\]
The finite unions
\[
  K_R:=\bigcup_{m=1}^R(C_m\cap O_m^c)
\]
are compact, increase with $R$, and satisfy
$c_{\cP}(K_R^c)\le2^{-R}$.  The finite pasting lemma shows that $Z|_{K_R}$
is continuous.  The converse follows by taking $O=\Omega\setminus K_R$.
\end{proof}

\subsection{Uniform BDG and common stochastic integration}

Let $\Omega=C_0([0,T];H)$ with coordinate process $X$, and let $\cP$ be a
family of laws under which $X$ is a continuous local martingale.  For a
bounded elementary predictable process
\[
  H_t=\sum_{k=0}^{N-1}H_k\one_{(t_k,t_{k+1}]}(t),
\]
require $H_k$ to be Borel, bounded,
$\cF_{t_k}^0:=\sigma(X_s:s\le t_k)$-measurable, and valued in
$\mathcal L(H,\R^m)$.  Define the common step
integral pathwise by
\[
  I_t(H)
  :=\sum_{k=0}^{N-1}H_k
     \bigl(X_{t\wedge t_{k+1}}-X_{t\wedge t_k}\bigr).
\]
For $1\le p<\infty$, introduce
\begin{align}
  \norm{Y}_{\mathbb S_{\cP}^p}
  &:=\sup_{P\in\cP}
      \left(E^P\sup_{t\le T}\abs{Y_t}^p\right)^{1/p},
  \label{eq:S-p-norm}\\
  \norm{H}_{\mathbb H_{\cP}^p}
  &:=\sup_{P\in\cP}
      \left[
        E^P\left(
          \int_0^T
          \operatorname{tr}\!\left(
            H_t\,\dd\qv{X}^P_t\,H_t^*
          \right)
        \right)^{p/2}
      \right]^{1/p}.
  \label{eq:H-p-norm}
\end{align}
Here $\qv{X}^P$ is the positive trace-class operator bracket under $P$.  Define
$\mathbb S_{\cP}^p$ as the closed subspace of
$\mathcal L_{\cP}^p(C([0,T];\R^m))$ generated by finite-norm Borel
continuous-path processes adapted to the raw canonical filtration.  Thus an
element of $\mathbb S_{\cP}^p$ has an adapted continuous version under each
$P$, although the ambient class need not have a distinguished causal
representative.  The expression in \eqref{eq:H-p-norm} is a seminorm on
elementary predictable integrands.  First quotient by its null space and
then take the completion; the resulting Banach space is denoted by
$\mathbb H_{\cP}^p$.  Only elementary processes of finite seminorm enter
this completion.  For each $P\in\cP$, the natural contractions will be
denoted by
\[
 \pi_P:\mathbb S_{\cP}^p\to\mathbb S^p(P),
 \qquad
 j_P:\mathbb H_{\cP}^p\to\mathbb H^p(P).
\]
The second map is obtained by completing the identity on elementary
integrands; both maps have norm at most one.  Moreover,
\begin{equation}\label{eq:energy-product-embedding}
 \norm{H}_{\mathbb H_{\cP}^p}
 =\sup_{P\in\cP}\norm{j_P(H)}_{\mathbb H^p(P)},
 \qquad H\in\mathbb H_{\cP}^p.
\end{equation}
Indeed, the identity holds before completion, and if $H_n\to H$ then
\[
 \sup_{P\in\cP}\left|
  \|j_PH_n\|_{\mathbb H^p(P)}-\|j_PH\|_{\mathbb H^p(P)}
 \right|
 \le\|H_n-H\|_{\mathbb H_{\cP}^p},
\]
so it passes to the quotient completion.

\begin{theorem}[Two-sided uniform BDG extension]
\label{thm:uniform-BDG-extension}
For every $1\le p<\infty$, there are constants
$0<c_p\le C_p<\infty$, independent of
$\cP$, such that
\begin{equation}\label{eq:upper-BDG}
  c_p\norm{H}_{\mathbb H_{\cP}^p}
  \le \norm{I(H)}_{\mathbb S_{\cP}^p}
  \le C_p\norm{H}_{\mathbb H_{\cP}^p}
\end{equation}
for every finite-upper-energy elementary predictable $H$.  Hence $I$
extends uniquely to a
bounded injective linear map with closed range,
\[
  I:\mathbb H_{\cP}^p
  \longrightarrow
  \mathbb S_{\cP}^p.
\]
For every $P\in\cP$ and $H\in\mathbb H_{\cP}^p$, one has the commuting
identity
\begin{equation}\label{eq:modelwise-integral-projection}
 \pi_P(I(H))
 =\int_0^{\cdot}j_P(H)_s\,\dd X_s
 \qquad\text{in }\mathbb S^p(P).
\end{equation}
Moreover, for every deterministic $t$,
\begin{equation}\label{eq:common-integral-stopping}
 I(\one_{[0,t]}H)=r_tI(H).
\end{equation}
Thus, with the natural truncations on integrands and continuous paths, the
common stochastic integral is a stable-causal closed embedding.
\end{theorem}

\begin{proof}
Under each fixed $P$, the pathwise step integral is the classical elementary
It\^o integral and
\[
  \operatorname{tr}\qv{I(H)}_T
  =\int_0^T\operatorname{tr}\!\left(
    H_t\,\dd\qv{X}^P_t\,H_t^*
  \right).
\]
The two-sided classical BDG inequality \cite{RevuzYor99}, with constants
depending only on $p$, gives $c_pb_P(H)\le a_P(H)\le C_pb_P(H)$, where
\begin{align*}
 a_P(H)&:=\left(E^P\sup_{t\le T}\abs{I_t(H)}^p\right)^{1/p},\\
 b_P(H)&:=\left[
   E^P\left(
     \int_0^T\operatorname{tr}\!\left(
       H_t\,\dd\qv{X}^P_t\,H_t^*
     \right)
   \right)^{p/2}
 \right]^{1/p}.
\end{align*}
Taking the supremum proves \eqref{eq:upper-BDG}.  Completeness gives the
extension, while the lower estimate makes it injective and its range closed.
If $H_n$ is elementary and converges to $H$ in the upper
$\mathbb H^p$ norm, it converges in the corresponding $P$-norm for every
$P$.  The classical It\^o integrals therefore converge in $\mathbb S^p(P)$
to the usual integral, while \eqref{eq:upper-BDG} identifies the same limit
with the common extension.  Continuity of $\pi_P$ and $j_P$ gives
\eqref{eq:modelwise-integral-projection}.  Identity
\eqref{eq:common-integral-stopping} holds for elementary integrands by the
pathwise step-integral formula.  The truncation
$H\mapsto\one_{[0,t]}H$ is contractive in the upper energy norm and
$Y\mapsto r_tY$ is contractive in the upper process norm, so the identity
passes to the two completions.
\end{proof}

\begin{corollary}[Compatibility across integrability exponents]
\label{cor:common-integral-exponent-compatibility}
If $1\le p\le q<\infty$, Lyapunov's inequality and completion of the
identity on elementary classes give canonical contractions
\[
 \iota_{q,p}^{\mathbb S}:\mathbb S_{\cP}^q\longrightarrow\mathbb S_{\cP}^p,
 \qquad
 \iota_{q,p}^{\mathbb H}:\mathbb H_{\cP}^q\longrightarrow\mathbb H_{\cP}^p.
\]
Writing $I_p$ for the common integral at exponent $p$, these maps satisfy
\begin{equation}\label{eq:common-integral-exponent-compatibility}
 I_p\iota_{q,p}^{\mathbb H}
 =\iota_{q,p}^{\mathbb S}I_q.
\end{equation}
\end{corollary}

\begin{proof}
The contractions hold on elementary processes by Lyapunov's inequality and
therefore extend to the completions.  Both sides of
\eqref{eq:common-integral-exponent-compatibility} agree on elementary
integrands, which are dense in $\mathbb H_{\cP}^q$.
\end{proof}

Thus $I$ is exactly the closed extension from common elementary integrands in
upper energy.

\section{Common Euler fields and one-set perfection}
\label{sec:common-Euler-fields}

This appendix proves the three inputs used in
\Cref{thm:Euler-state-interface}.  Fixed-parameter Euler completion uses only
the trace clock $\Tr a_t^P\le\Lambda$.  The spatial profile enters later,
through uniform tightness and compact-core perfection, and through comparison
with the corewise continuous state propagation (and, under
\Cref{cor:full-state-flow-readout}, the full-state readout).  Throughout, solution paths started from
$(s,y)$ are frozen at $y$ before $s$.

\begin{lemma}[Euler completion under a trace clock]
\label{lem:trace-clock-Euler-sections}
Let
$\varnothing\ne\cP\subset\mathfrak S_{\Lambda,0}^{0,T}(H)$, and let
\[
 b:[0,T]\times\mathbb R^m\to\mathbb R^m,
 \qquad
 \sigma:[0,T]\times\mathbb R^m\to\mathcal L(H,\mathbb R^m)
\]
be deterministic and continuous.  Suppose that, for constants $L,K<\infty$,
\begin{align}
 &|b(t,z)-b(t,z')|
  +\|\sigma(t,z)-\sigma(t,z')\|_{\rm op}
  \le L|z-z'|,
 \label{eq:Euler-coefficient-Lipschitz}\\
 &|b(t,z)|+\|\sigma(t,z)\|_{\rm op}
  \le K(1+|z|).
 \label{eq:Euler-coefficient-growth}
\end{align}
For every deterministic $(s,y)$, every $q\ge2$, and every sequence of
partitions $\pi_n$ of $[s,T]$ with $|\pi_n|\to0$, the raw maps
$Y^{\pi_n;s,y}$ from \eqref{eq:raw-Euler-starting-pair} converge in
$\mathbb S_{\cP}^q$ to a limit $Y^{s,y}$ independent of the partitions.
If
\begin{equation}\label{eq:Euler-step-integrand}
 H_u^{\pi;s,y}
 :=\one_{(s,T]}(u)\,
 \sigma\bigl(\ell_\pi(u),Y_{\ell_\pi(u)}^{\pi;s,y}\bigr),
\end{equation}
then these integrands converge in $\mathbb H_{\cP}^q$ to the unique element
$\Sigma^{s,y}$ satisfying
\begin{equation}\label{eq:Euler-limit-modelwise-integrand}
 j_P\Sigma^{s,y}
 =\one_{(s,T]}\sigma(\cdot,\pi_PY^{s,y})
 \quad\text{in }\mathbb H^q(P),
 \qquad P\in\cP.
\end{equation}
The pair $(Y^{s,y},\Sigma^{s,y})$ solves
\begin{equation}\label{eq:Euler-common-solution-pair}
 Y_t^{s,y}
 =y+\int_s^{t\vee s}b(u,Y_u^{s,y})\,\dd u+I_t(\Sigma^{s,y})
 \quad\text{in }\mathbb S_{\cP}^q,
\end{equation}
and is the unique pair with the modelwise prescription
\eqref{eq:Euler-limit-modelwise-integrand}.  Under each $P\in\cP$ its
projection is the unique classical strong It\^o solution.

Moreover,
\begin{equation}\label{eq:Euler-section-regular-completion}
 Y^{s,y}\in
 \mathbb L_{\cP}^q\bigl(C([0,T];\mathbb R^m)\bigr),
\end{equation}
and, for every $R<\infty$, there are constants depending only on the displayed
parameters such that
\begin{align}
 \sup_{\substack{s\in[0,T]\\|y|\le R}}
 \|Y^{s,y}\|_{\mathbb S_{\cP}^q}
 &\le C_{q,R},
 \label{eq:Euler-section-moment}\\
 \|Y^{s,y}-Y^{r,z}\|_{\mathbb S_{\cP}^q}
 &\le C_{q,R}\bigl(|s-r|^{1/2}+|y-z|\bigr),
 \qquad s,r\in[0,T],\quad |y|,|z|\le R.
 \label{eq:Euler-section-parameter-increment}
\end{align}
The conclusions in
\eqref{eq:Euler-section-regular-completion}--\eqref{eq:Euler-section-parameter-increment}
also hold for every $1\le q<2$.  In that range, define
$(Y^{s,y},\Sigma^{s,y})$ as the image of the exponent-two pair under the
canonical contractions of
\Cref{cor:common-integral-exponent-compatibility}.  Then
\eqref{eq:common-integral-exponent-compatibility} gives
\eqref{eq:Euler-common-solution-pair} in $\mathbb S_{\cP}^q$, and
Lyapunov's inequality gives the stated bounds.  These lower-exponent pairs are
the limits of the same elementary Euler sequence and are therefore compatible
across all exponents.
\end{lemma}

\begin{proof}
The common Cauchy estimate uses only the trace-clock bound and therefore
requires no fixed operator envelope.  For a positive trace-class operator $a$
and $G\in\mathcal L(H,\mathbb R^m)$,
\begin{equation}\label{eq:trace-clock-integrand-bound}
 \operatorname{tr}(GaG^*)
 =\|Ga^{1/2}\|_{\mathcal S_2}^2
 \le\|G\|_{\rm op}^2\Tr a.
\end{equation}
Thus every stochastic estimate below uses only
$\Tr a_t^P\le\Lambda$.

Fix $(s,y)$ and abbreviate $Y^n=Y^{\pi_n;s,y}$,
$H^n=H^{\pi_n;s,y}$, and $\ell_n=\ell_{\pi_n}$.  Finite-step BDG,
\eqref{eq:trace-clock-integrand-bound}, and Gronwall give, for every
$q\ge2$,
\begin{align}
 \sup_n\|Y^n\|_{\mathbb S_{\cP}^q}
 &\le C_q(1+|y|),
 \label{eq:Euler-uniform-upper-moment}\\
 \sup_n\sup_{P\in\cP}E^P\int_s^T
 |Y_u^n-Y_{\ell_n(u)}^n|^q\,\dd u
 &\le C_q(1+|y|^q)|\pi_n|^{q/2}.
 \label{eq:Euler-grid-motion}
\end{align}
Although $H^n$ need not be bounded, it belongs to
$\mathbb H_{\cP}^q$.  Indeed, radial truncations $T_MH^n$ are bounded
elementary predictable integrands.  Choose $q'>q$.  From
\eqref{eq:Euler-uniform-upper-moment}, linear growth, and
\eqref{eq:trace-clock-integrand-bound},
\[
 \|T_MH^n-T_NH^n\|_{\mathbb H_{\cP}^q}^q
 \le C N^{q-q'}
 \sup_{P\in\cP}E^P\int_s^T(1+|Y_{\ell_n(u)}^n|^{q'})\,\dd u
 \longrightarrow0
\]
as $M>N\to\infty$.  Hence the common step integral of $H^n$ is exactly the
raw sum in \eqref{eq:raw-Euler-starting-pair}.

To remove the time discretization, set, for $R\ge1$,
\[
 \omega_R(h):=
 \sup_{\substack{|u-v|\le h\\|z|\le R}}
 \left(|b(u,z)-b(v,z)|
 +\|\sigma(u,z)-\sigma(v,z)\|_{\rm op}\right).
\]
Continuity gives $\omega_R(h)\downarrow0$.  Combining this local modulus with
\eqref{eq:Euler-grid-motion}, and using a higher moment from
\eqref{eq:Euler-uniform-upper-moment} on the set where
$|Y_{\ell_n(u)}^n|>R$, gives
\begin{equation}\label{eq:Euler-consistency-profile}
 \delta_n^{(q)}:=\sup_{P\in\cP}E^P\int_s^T
 \begin{aligned}[t]
 &\bigl|b(\ell_n(u),Y_{\ell_n(u)}^n)-b(u,Y_u^n)\bigr|^q\\[-1mm]
 &+\bigl\|\sigma(\ell_n(u),Y_{\ell_n(u)}^n)
          -\sigma(u,Y_u^n)\bigr\|_{\rm op}^q
 \end{aligned}
 \dd u\longrightarrow0.
\end{equation}
Here one first lets $R\to\infty$ and then $n\to\infty$.

For
\[
 d_t(Y,Z):=\sup_{P\in\cP}
 \left(E^P\sup_{v\le t}|Y_v-Z_v|^q\right)^{1/q},
\]
the drift estimate, the uniform BDG inequality, H\"older's inequality, and
\eqref{eq:trace-clock-integrand-bound} yield
\begin{equation}\label{eq:Euler-common-Cauchy-Gronwall}
 d_t(Y^n,Y^k)^q
 \le C\int_s^t d_u(Y^n,Y^k)^q\,\dd u
     +C\bigl(\delta_n^{(q)}+\delta_k^{(q)}\bigr).
\end{equation}
Gronwall proves that $(Y^n)$ is Cauchy in $\mathbb S_{\cP}^q$.  The same
decomposition, now in the energy norm, gives
\begin{equation}\label{eq:Euler-integrand-Cauchy}
 \|H^n-H^k\|_{\mathbb H_{\cP}^q}^q
 \le C\int_s^T d_u(Y^n,Y^k)^q\,\dd u
     +C\bigl(\delta_n^{(q)}+\delta_k^{(q)}\bigr),
\end{equation}
so $(H^n)$ converges to some $\Sigma^{s,y}$.  Passing to the limit in the
Euler identity by \Cref{thm:uniform-BDG-extension} proves
\eqref{eq:Euler-common-solution-pair}.

For fixed $P$, the right-hand side of
\eqref{eq:Euler-integrand-Cauchy}, with the limiting process in place of
$Y^k$, also proves
$j_PH^n\to\one_{(s,T]}\sigma(\cdot,\pi_PY^{s,y})$ in
$\mathbb H^q(P)$.  This gives
\eqref{eq:Euler-limit-modelwise-integrand}.  Conversely, the norm identity
\eqref{eq:energy-product-embedding} makes the projection map injective, even
for degenerate brackets.  Its range consists of the compatible modelwise
families arising from common upper-energy elements.
The same BDG--Gronwall estimate proves uniqueness of the common pair.
Equation \eqref{eq:modelwise-integral-projection} then identifies each
projection with the classical It\^o solution; classical pathwise uniqueness
gives the strong-solution statement.  Running the construction at a higher
moment exponent and using the continuous embeddings, uniqueness also shows
that the process limits obtained at different exponents are compatible.

Each finite Euler map is continuous from the raw uniform path space into
$C([0,T];\mathbb R^m)$.  Its upper $q$-tail is uniformly negligible by
\eqref{eq:Euler-uniform-upper-moment} at a higher exponent.  Hence
\Cref{thm:quasi-continuous-characterization} puts every $Y^n$ in the regular
upper-$L^q$ completion.  Closedness and $Y^n\to Y^{s,y}$ give
\eqref{eq:Euler-section-regular-completion}.

Finally, the preceding moment argument gives
\eqref{eq:Euler-section-moment}.  Initial-state stability gives the
$|y-z|$ part of \eqref{eq:Euler-section-parameter-increment}.  If $s<r$,
BDG and the growth bound give
\[
 \sup_{P\in\cP}
 \left(E^P\sup_{u\in[s,r]}|Y_u^{s,y}-y|^q\right)^{1/q}
 \le C_{q,R}|r-s|^{1/2}.
\]
From time $r$ onward, apply initial-state stability to
$Y_r^{s,y}$ and $z$.  Together with the frozen paths this proves
\eqref{eq:Euler-section-parameter-increment}.  The exponent-$q$ results imply
the assertions below exponent two by Lyapunov's inequality and continuity of
the regular-completion embedding.
\end{proof}

\begin{lemma}[Deterministic anisotropic grid chaining]
\label{lem:deterministic-anisotropic-chaining}
Let $E$ be a Banach space,
$D=[0,T]\times[-R,R]^m$, and
\[
 \varrho((s,y),(r,z))=|s-r|^{1/2}+|y-z|.
\]
Use the nested affine dyadic tensor grids
\[
 \mathcal G_n
 :=\{kT4^{-n}:0\le k\le4^n\}
 \times
 \prod_{j=1}^m
 \{-R+2R\ell2^{-n}:0\le\ell\le2^n\}.
\]
Let $\mathcal E_n$ contain the nearest-neighbour level-$n$ edges and, for
$n\ge1$, the edges joining each new vertex to all corners of one containing
parent rectangle in $\mathcal G_{n-1}$.  With $d=m+2$,
\[
 |\mathcal E_n|\le C_{R,m,T}2^{dn},
 \qquad
 \varrho(v,w)\le C_{R,m,T}2^{-n}
 \quad ((v,w)\in\mathcal E_n).
\]
Assign compatible values $z_v\in E$ on the union of the grids, let $U_n$ be
their multilinear interpolant, and put
\[
 M_n:=\max_{(v,w)\in\mathcal E_n}\|z_v-z_w\|_E.
\]
Then
\begin{equation}\label{eq:deterministic-grid-interpolant-difference}
 \|U_{n+1}-U_n\|_{C(D;E)}\le C_{R,m,T}M_{n+1}.
\end{equation}
If $\sum_nM_n<\infty$, the interpolants converge uniformly to $U$ and, for
every $0<\delta<1$,
\begin{equation}\label{eq:deterministic-anisotropic-chain}
 [U]_{C_\varrho^\delta(D;E)}
 \le C_{R,m,T,\delta}
 \left(M_0+\sum_{n\ge0}2^{n\delta}M_n\right).
\end{equation}
\end{lemma}

\begin{proof}
The time mesh is $T4^{-n}$ and every spatial mesh is $2R2^{-n}$, so the edge
count and $\varrho$-diameter estimates follow directly.  On each fine
rectangle, the norm of a multilinear interpolant is bounded by the largest
norm of its vertex values.  At old vertices $U_{n+1}-U_n=0$, while at every
new vertex the parent-corner edges give
\[
 \|z_v-U_n(v)\|_E\le C_{R,m,T}M_{n+1}.
\]
Consequently
\[
 \|U_{n+1}-U_n\|_\infty
 \le\max_{v\in\mathcal G_{n+1}}\|z_v-U_n(v)\|_E
 \le C_{R,m,T}M_{n+1},
\]
which proves \eqref{eq:deterministic-grid-interpolant-difference}.

If $2^{-(n+1)}<\varrho(a,b)\le2^{-n}$, then $a$ and $b$ cross at most a
bounded number of level-$n$ rectangles in the time direction and in each
spatial coordinate.  Joining their surrounding vertices by nearest-neighbour
edges and using multilinearity therefore gives
\[
 \|U_n(a)-U_n(b)\|_E\le C_{R,m,T}M_n.
\]
Moreover,
\[
 \|U-U_n\|_\infty\le C\sum_{k>n}M_k.
\]
Combining the chain estimate with this tail, dividing by
$\varrho(a,b)^\delta$, and using
$2^{n\delta}M_k\le2^{k\delta}M_k$ for $k\ge n$ proves
\eqref{eq:deterministic-anisotropic-chain} for $\varrho(a,b)\le1$.
When $\varrho(a,b)>1$, a level-zero coordinate chain gives
$\|U_0(a)-U_0(b)\|_E\le C_{R,m,T}M_0$, and the same uniform tail handles the
two endpoints.  Since $\varrho(a,b)^\delta\ge1$, this also yields the stated
bound at large scales.
\end{proof}

\begin{lemma}[Euler-generated parameter-field completion]
\label{lem:Euler-parameter-field-completion}
Assume the hypotheses of \Cref{lem:trace-clock-Euler-sections} and suppose
also that $\cP$ is uniformly tight on
$\Omega=C_0([0,T];H)$.  The fixed-parameter classes admit compatible
completions
\begin{equation}\label{eq:Euler-compatible-field-completions}
 \Phi^R\in\bigcap_{1\le q<\infty}
 \mathbb L_{\cP}^q(\mathscr B_R),
 \qquad R\in\mathbb N.
\end{equation}
Put
\[
 \varrho((s,y),(r,z)):=|s-r|^{1/2}+|y-z|.
\]
For every $q>m+2$ and every
$0<\delta<1-(m+2)/q$,
\begin{equation}\label{eq:Euler-anisotropic-Kolmogorov-bound}
 \sup_{P\in\cP}E^P\left[
  \sup_{\substack{(s,y)\ne(r,z)\\|y|,|z|\le R}}
  \frac{\|\Phi^R(s,y)-\Phi^R(r,z)\|_\infty^q}
       {\varrho((s,y),(r,z))^{\delta q}}
 \right]<\infty.
\end{equation}

There are a total jointly Borel raw-causal field
\begin{equation}\label{eq:Euler-total-field}
 \Phi:\Omega\times\Delta_T\times\mathbb R^m\to\mathbb R^m
\end{equation}
and one sequence $(F_n)$ with the following properties.  Each $F_n$ is a
continuous finite-parameter interpolation field assembled from finitely many
finite-step raw Euler maps.  There is one Borel full-capacity set
$G_{\rm reg}$, independent of the radius, moment exponent, and parameters,
such that
\begin{equation}\label{eq:Euler-raw-and-completed-field-agree}
 \Phi_{s,t}(x,y)=\Phi^R(x)(s,y)(t)
 \quad
 (x\in G_{\rm reg},\ |y|\le R,\ 0\le s\le t\le T),
\end{equation}
while the right-hand path is frozen before $s$.  In particular,
the field is locally jointly continuous in $(s,t,y)$ on $G_{\rm reg}$.
There is an increasing compact capacity core $(K_j^{\rm E})_{j\ge1}$ on
each member of which $(x,s,t,y)\mapsto\Phi_{s,t}(x,y)$ is jointly continuous
when the parameters range over compact sets.
\end{lemma}

\begin{proof}
Let $E=C([0,T];\mathbb R^m)$ and first work on the tensor-product cube
\[
 D_R^\square:=[0,T]\times[-R,R]^m.
\]
The desired set $D_R=[0,T]\times\overline B_R$ is a closed subset of a
cube of this form, so restriction will give the stated completion.  Under
the metric $\varrho$, the time coordinate has weight two, and the homogeneous
parameter dimension is $d=m+2$.  After adjusting the radius in the constant,
\eqref{eq:Euler-section-parameter-increment} reads
\begin{equation}\label{eq:Euler-field-increment-input}
 \|Y^a-Y^{a'}\|_{\mathcal L_{\cP}^q(E)}
 \le C_{q,R}\varrho(a,a'),
 \qquad a,a'\in D_R^\square.
\end{equation}

Use nested tensor-product grids with time mesh $2^{-2n}$ and state mesh
$2^{-n}$.  On the countable union of the grids over all integer radii and
levels, choose once and for all one Borel representative of each
fixed-parameter solution class, reusing the same representative whenever a
node recurs.  Let $\mathcal E_n^R$ be the augmented edge set from
\Cref{lem:deterministic-anisotropic-chaining}; it contains the
nearest-neighbour and parent-corner edges.  There are at most
$C_R2^{dn}$ such edges, all of $\varrho$-length at most $C_R2^{-n}$.  With
\[
 M_n^R:=\max_{e\in\mathcal E_n^R}
 \|Y^{e^+}-Y^{e^-}\|_E,
\]
\eqref{eq:Euler-field-increment-input} and
$\max_i z_i^q\le\sum_i z_i^q$ give the upper-moment estimate
\begin{equation}\label{eq:Euler-edge-maximum-moment}
 \|M_n^R\|_{\mathcal L_{\cP}^q}
 \le\left(
   \sum_{e\in\mathcal E_n^R}
   \|Y^{e^+}-Y^{e^-}\|_{\mathcal L_{\cP}^q(E)}^q
 \right)^{1/q}
 \le C_{q,R}2^{-n(1-d/q)}.
\end{equation}
Let $U_n^R$ be the multilinear interpolant of the grid values.  Applying
\Cref{lem:deterministic-anisotropic-chaining} gives
\begin{align}
 \|U_{n+1}^R-U_n^R\|_{C(D_R^\square;E)}
 &\le C_RM_{n+1}^R,\label{eq:Euler-interpolant-difference}\\
 [U^R]_{C_\varrho^\delta(D_R^\square;E)}
 &\le C_R\left(M_0^R+\sum_{n\ge0}2^{n\delta}M_n^R\right),
 \label{eq:Euler-interpolant-Holder-chain}
\end{align}
where $U^R$ is the uniform limit whenever the series is finite.  Since
$\delta<1-d/q$, \eqref{eq:Euler-edge-maximum-moment} and Minkowski's
inequality make both series summable in $\mathcal L_{\cP}^q$.  Thus
$(U_n^R)$ is Cauchy in
$\mathcal L_{\cP}^q(C(D_R^\square;E))$, and its limit has the
$q$-integrable H\"older seminorm asserted in
\eqref{eq:Euler-anisotropic-Kolmogorov-bound}.  Restricting to $D_R$ gives
the claimed ball-indexed field.  Since every grid value belongs to the regular
upper-$L^q$ completion, the finite interpolants and then their limit belong to
the same closed regular completion.

For an arbitrary $a\in D_R^\square$, choose nested-grid points
$a_n\to a$.  The interpolation limit is continuous and agrees with
$Y^{a_n}$ at every recurring grid node, while
\eqref{eq:Euler-field-increment-input} gives
$Y^{a_n}\to Y^a$ in $\mathcal L_{\cP}^q(E)$.  Hence evaluation of the limit
at $a$ is exactly the prescribed fixed-parameter class $Y^a$.

Apply this construction for integer $R$ and integer $q>d$.  Uniqueness of the
fixed-parameter solution implies that two resulting completions agree on the
countable union of their grids; their parameter continuity then makes them
compatible after restriction in $R$ and passage to lower moment exponents.
Higher integer moments and Lyapunov's inequality give all finite real
exponents.  This proves
\eqref{eq:Euler-compatible-field-completions} and
\eqref{eq:Euler-anisotropic-Kolmogorov-bound}.

It remains to select one raw-causal representative.  For each $N$, first
approximate the compatible fields $\Phi^R$, $R\le N$, by one sufficiently
fine finite parameter interpolant.  Only finitely many start--state nodes are
used.  By \Cref{lem:trace-clock-Euler-sections}, choose a common deterministic
partition refinement such that the finite-step Euler map at every one of
these nodes approximates its solution section in every upper-$L^k(E)$ norm,
$k\le N$.  Replacing the nodes by those Euler maps and using the same
nonnegative multilinear weights produces a field $F_N$.  Coordinatewise
truncation of the initial state extends it to all $y$, and subtracting its
diagonal value and adding $y$ enforces $F_{N,s,s}(x,y)=y$.  These operations
are followed by setting the associated path equal to $y$ on $[0,s]$.
Continuity at the seam follows from the diagonal normalization.  The
operations preserve joint Borel measurability, raw causality, and continuity of
\[
 x\longmapsto F_N(x)|_{D_R}
 \quad\text{from }\Omega\text{ into }\mathscr B_R.
\]
Choose the approximations fast enough that, for every $R,k\le N$,
\begin{equation}\label{eq:Euler-field-fast-diagonal}
 \|F_N-\Phi^R\|_{\mathcal L_{\cP}^k(\mathscr B_R)}
 \le2^{-4N}.
\end{equation}
The resulting deterministic sequence satisfies all countably many constraints
simultaneously.

By Markov's inequality and \eqref{eq:Euler-field-fast-diagonal}, the sequence
is uniformly Cauchy in every $\mathscr B_R$ outside one Borel polar set; call
the common convergence set $G_{\rm reg}$.  On this set let $L^R$ be the
$\mathscr B_R$-limit and use the identity field off $G_{\rm reg}$.  Fatou's
lemma and \eqref{eq:Euler-field-fast-diagonal} show that this Borel map
represents $\Phi^R$ in every finite upper moment, compatibly in $R$.

To preserve raw causality on all paths, not only on $G_{\rm reg}$, use the
Borel selector $\operatorname{Lim}_{\mathbb R^m}$ of
\Cref{lem:causal-Borel-limit}.  For $s\le t$ put
\begin{equation}\label{eq:Euler-pointwise-Lim-totalization}
 \Phi_{s,t}(x,y)
 :=y+\operatorname{Lim}_{\mathbb R^m}
       \bigl((F_{N,s,t}(x,y)-y)_{N\ge1}\bigr).
\end{equation}
This field is total and jointly Borel.  It is raw causal pointwise on all of
$\Omega$: for every $N$, the two sequences in
\eqref{eq:Euler-pointwise-Lim-totalization} have identical terms at $x$ and
at $r_tx$.  On $G_{\rm reg}$, the convergence is locally uniform in all
parameters and $\Phi=L^R$ for $|y|\le R$, where $L^R$ is the chosen Borel
representative of the class $\Phi^R$.  This proves
\eqref{eq:Euler-raw-and-completed-field-agree} and the asserted local
parameter continuity.

Uniform tightness also yields compact continuity cores.  Choose compact
$C_j\subset\Omega$ with
$c_{\cP}(C_j^c)\le2^{-j-2}$.  For every pair $(R,j)$, the fast diagonal and
Markov's inequality allow a tail index $N_{R,j}$ such that
\[
 \sum_{n\ge N_{R,j}}
 c_{\cP}\!\left(
  \|F_{n+1}-F_n\|_{\mathscr B_R}>2^{-n}
 \right)
 \le2^{-R-j-2}.
\]
Define
\[
 G_j:=C_j\cap
 \bigcap_{R\ge1}\ \bigcap_{n\ge N_{R,j}}
 \left\{\|F_{n+1}-F_n\|_{\mathscr B_R}\le2^{-n}\right\}.
\]
Every set in the intersection is closed because the approximating field maps
are continuous; hence $G_j$ is compact.  Countable subadditivity gives
$c_{\cP}(G_j^c)\le2^{-j-1}$.  On $G_j$, for every $R$, the sequence
$F_n|_{D_R}$ is uniformly Cauchy after its $R$-dependent tail index.  The
finite unions
$K_j^{\rm E}:=\bigcup_{i=1}^jG_i$ therefore form an increasing compact
capacity core on each member of which the limit maps
$x\mapsto\Phi(x)|_{D_R}$ are continuous for every $R$.  Evaluation gives the
final joint-continuity assertion on all compact parameter sets.
\end{proof}

\begin{lemma}[Dense-skeleton one-set perfection]
\label{lem:dense-skeleton-one-set-perfection}
Let $(\Theta,d_\Theta)$ be a separable metric space, let $(E,d_E)$ be a
metric space, and let $F,G:\Omega\times\Theta\to E$ be jointly Borel fields.
Suppose there is a Borel $\cP$-full set $G_0$ on which the two parameter
sections are continuous.  If, for one countable dense set
$D\subset\Theta$,
\[
 F(\cdot,\theta)=G(\cdot,\theta)
 \quad\cP\text{-quasi surely},\qquad \theta\in D,
\]
then there is one Borel $\cP$-polar set $N$, independent of $\theta$, such
that
\[
 F(x,\theta)=G(x,\theta),
 \qquad x\in G_0\setminus N,\quad \theta\in\Theta.
\]
The same conclusion holds for a jointly Borel residual
$\mathfrak R:\Omega\times\Theta\to E$ whose parameter sections are
continuous on $G_0$: if $\mathfrak R(\cdot,\theta)=e_0$ quasi surely on
$D$, then it equals $e_0$ on $(G_0\setminus N)\times\Theta$.
\end{lemma}

\begin{proof}
For each $\theta\in D$, the Borel mismatch event is polar.  Let $N$ be their
countable union together with the polar continuity-exception set $G_0^c$.
Equality on the dense set then extends to all of $\Theta$ by continuity.  The
residual formulation is the special case $G\equiv e_0$.
\end{proof}

\begin{lemma}[Random-initial restart for the Euler field]
\label{lem:Euler-random-initial-restart}
Let $\Phi$ be the field of
\Cref{lem:Euler-parameter-field-completion}.  Fix $P\in\cP$,
$u\in[0,T]$, $q\ge2$, and
$Z\in L^q(P,\mathcal F_u^0;\mathbb R^m)$.  Choose an
$\mathcal F_u^0$-measurable Borel representative, still denoted by $Z$.
Define
$t\mapsto\Phi_{u,t}(\cdot,Z)$ by pointwise composition on $G_{\rm reg}$ and
by the constant path $Z$ off $G_{\rm reg}$.  This is a Borel
$C([u,T];\mathbb R^m)$-valued variable and represents the unique classical
solution started from $Z$; a different representative gives the same
$P$-equivalence class.  Moreover, for
$Z,Z'\in L^q(P,\mathcal F_u^0;\mathbb R^m)$, with similarly chosen Borel
representatives,
\begin{equation}\label{eq:Euler-random-initial-stability}
 \left(E^P\sup_{t\in[u,T]}
 |\Phi_{u,t}(\cdot,Z)-\Phi_{u,t}(\cdot,Z')|^q\right)^{1/q}
 \le C_q\|Z-Z'\|_{L^q(P)}.
\end{equation}
\end{lemma}

\begin{proof}
Approximate $Z$---and, for the stability estimate, $Z'$---almost surely and
in $L^q(P)$ by finite-valued $\mathcal F_u^0$-measurable variables.  For the
identification of the first composition, denote the approximants of $Z$ by
$Z_n$.  The declared pointwise
composition is Borel; changing it on the polar complement of $G_{\rm reg}$
does not alter any modelwise class.  Locality of the Lebesgue and It\^o
integrals shows that $\Phi_{u,\cdot}(\cdot,Z_n)$ is the classical solution
started from $Z_n$.  The usual initial-state BDG--Gronwall estimate gives
\eqref{eq:Euler-random-initial-stability}.  On $G_{\rm reg}$, continuity in
the initial state identifies the limit of the finite-valued compositions
with the pointwise composition by $Z$.  Hence that composition is the unique
strong solution started from $Z$.
\end{proof}

\begin{lemma}[One-set flow and comparison perfection]
\label{lem:one-set-Euler-perfection}
Let $\Phi$ be the field of
\Cref{lem:Euler-parameter-field-completion}.  There is a Borel full-capacity
set $G_{\rm fl}\subset G_{\rm reg}$, independent of all flow parameters,
such that
\begin{equation}\label{eq:perfect-Euler-flow}
 \Phi_{s,t}(x,y)
 =\Phi_{u,t}\bigl(x,\Phi_{s,u}(x,y)\bigr)
\end{equation}
for every $x\in G_{\rm fl}$,
$0\le s\le u\le t\le T$, and $y\in\mathbb R^m$.

More generally, let
$\widetilde\Phi:\Omega\times\Delta_T\times\mathbb R^m\to\mathbb R^m$
be jointly Borel.  Suppose that outside one Borel polar set both parameter
fields are locally continuous and that, for every $P\in\cP$ and every fixed
$(s,y)$, the processes
$\Phi_{s,\cdot}(\cdot,y)$ and
$\widetilde\Phi_{s,\cdot}(\cdot,y)$ are indistinguishable under $P$.
Then there is one Borel $\cP$-polar set $N$, independent of
$P,s,t,y$, such that
\begin{equation}\label{eq:field-comparison-perfection}
 \Phi_{s,t}(x,y)=\widetilde\Phi_{s,t}(x,y)
\end{equation}
for every $x\notin N$, $0\le s\le t\le T$, and
$y\in\mathbb R^m$.
\end{lemma}

\begin{proof}
We first perfect the flow law.  Fix $P\in\cP$, take rational $s\le u$ and
rational $y$, and put
$Z=\Phi_{s,u}(\cdot,y)$.  Raw causality makes $Z$
$\mathcal F_u^0$-measurable, and \eqref{eq:Euler-section-moment} gives all
finite moments.  Projecting \eqref{eq:Euler-common-solution-pair} under $P$
and splitting the resulting classical equation at $u$, then using
\Cref{lem:Euler-random-initial-restart} and pathwise uniqueness, gives
\eqref{eq:perfect-Euler-flow} $P$-almost surely for this triple and
all terminal times.  Since the preceding argument applies to arbitrary
$P\in\cP$, the Borel exceptional event is polar for each rational
$(s,u,y)$.  Apply the residual clause of
\Cref{lem:dense-skeleton-one-set-perfection} on
\[
 \Theta_{\rm fl}
 :=\{(s,u,t,y):0\le s\le u\le t\le T,\ y\in\mathbb R^m\}
\]
to the difference between the two sides of
\eqref{eq:perfect-Euler-flow}.  Local joint continuity on
$G_{\rm reg}$ gives one parameter-independent full-capacity set
$G_{\rm fl}$ on which the flow identity holds everywhere.

For comparison, fixed-parameter indistinguishability makes the mismatch event
polar on a countable dense subset of
$\Delta_T\times\mathbb R^m$.  The two-field assertion of
\Cref{lem:dense-skeleton-one-set-perfection}, applied on the common parameter
continuity set, gives one polar set independent of $(s,t,y)$ and proves
\eqref{eq:field-comparison-perfection}.
\end{proof}

\end{document}